\documentclass{article}
\usepackage[english]{babel}
\usepackage[letterpaper,top=2cm,bottom=2cm,left=3cm,right=3cm,marginparwidth=1.75cm]{geometry}

\usepackage{amsmath, amssymb, array, booktabs, longtable, enumitem}
\usepackage{multicol}
\usepackage{tikz}
\usetikzlibrary{calc, arrows.meta, positioning, decorations.markings, shapes.misc}
\usepackage{authblk}
\usepackage{array}
\usepackage{appendix}
\usepackage{booktabs}
\usepackage{multirow}
\usepackage{mathrsfs}
\usepackage{etoolbox}
\usepackage{dynkin-diagrams}
\usepackage{subcaption}
\usepackage{float}
\usepackage{amsmath}
\usepackage{float}
\usepackage{graphicx}
\usepackage[colorlinks=true, allcolors=blue]{hyperref}
\usepackage{tikz}
\usetikzlibrary{intersections}
\usepackage{amssymb}
\usepackage{amsmath}
\usepackage{amsthm}
\usepackage{tikz-cd}
\usepackage{geometry}
\newtheorem{Thm}{Theorem}[section]
\newtheorem{Lem}[Thm]{Lemma}
\newtheorem{Prop}[Thm]{Proposition}
\newtheorem{Cor}[Thm]{Corollary}
\newtheorem{Con}[Thm]{Conjecture}

\theoremstyle{definition}
\newtheorem{Def}[Thm]{Definition}
\newtheorem{Rem}[Thm]{Remark}

\newtheorem{Ex}[Thm]{Example}
\newtheorem{Problem}[Thm]{Problem}

\newtheoremstyle{named}{}{}{\itshape}{}{\bfseries}{.}{.5em}{#1 #3}
\theoremstyle{named}

\def\R{\mathbb{R}}
\def\Q{\mathbb{Q}}
\def\F{\mathbb{F}}
\def\N{\mathbb{N}}

\def\Z{\mathbb{Z}}

\def\bbS{\mathbb{S}}

\def\cA{\mathcal{A}}

\def\cK{\mathcal{K}}
\def\cL{\mathcal{L}}

\def\cP{\mathcal{P}}

\def\cR{\mathcal{R}}
\def\cS{\mathcal{S}}
\def\cT{\mathcal{T}}

\def\cX{\mathcal{X}}

\def\a{\alpha}

\def\c{\gamma}
\def\G{\Gamma}

\def\n{\eta}

\def\bo{\mathbf{o}}

\def\bx{\mathbf{x}}

\def\0{\mathbf{0}}

\def\=>{\Longrightarrow}

\def\to{\rightarrow}

\def\o+{\oplus}
\def\bo+{\bigoplus}
\def\x{\times}
\def\<{\langle}
\def\>{\rangle}

\def\^{\wedge}
\def\+{\dagger}

\def\dis{\displaystyle}
\def\dd[#1,#2]{\frac{d#1}{d#2}}
\def\del[#1,#2]{\frac{\partial #1}{\partial #2}}

\def\And{\quad\mbox{and}\quad}

\def\tab{\;\;\;\;\;\;}

\newcommand{\case}[2][lllllllllllllllllllllllllllllllllllll]{\left\{\begin{array}{#1}#2 \\ \end{array}\right.}
\newcommand{\Eq}[1]{\begin{align}#1\end{align}}
\newcommand{\Eqn}[1]{\begin{align*}#1\end{align*}}

\tikzset{>=latex}
\tikzstyle{vthick}=[line width=1.8pt]

\newcommand\drawpath[2]{%
  \foreach \too [count=\c from 1] in {#1}
  {
  \ifthenelse{\c=1}
  {\xdef\from{\too}}
  {\path (\from) edge [->, #2] (\too);
    \xdef\from{\too}}
  };
}

\title{On the expansion formulas of cluster varieties from surfaces and their combinatorial properties}
\author{Vu Tung Lam Dinh*, \quad Ivan Chi-Ho Ip**}

\begin{document}
\maketitle

\begin{abstract}
This paper explores the cluster algebra structure of the moduli space $\mathscr{A}_{\mathrm{SL}_{n+1},\bbS}$ of twisted $\mathrm{SL}_{n+1}$-local systems on a surface. We derive general recurrence relations for cluster variables arising from flips of a triangulation, corresponding to specific sequences of mutations. Our approach is grounded in a detailed combinatorial analysis over the standard $n$-triangulated $m$-gon (with explicit calculations for $n=1,2$). As a generalization, the non-simply-laced $G_2$ type is also considered. We prove the \emph{well-triangulated} property for cluster mutations under flips, which provides a combinatorial framework for understanding the stability and transformation rules of these cluster algebra structures, and compute the monomial counts for the cluster expansion formula.
\end{abstract}

\tableofcontents

\vspace{2mm} 
\hrule
\vspace{1mm}
\begin{flushleft}
\small
Department of Mathematics, Hong Kong University of Science and Technology \\
*Email: vtldinh@connect.ust.hk\\
**Email: ivan.ip@ust.hk 
\end{flushleft}

\section{Introduction}
Cluster algebras, introduced by Fomin and Zelevinsky \cite{FZ}, have become a fundamental structure in modern mathematics. These commutative rings, equipped with a combinatorial seed structure governed by mutation, have found significant applications in diverse areas such as representation theory \cite{GS24, Ip, Ke08, SS}, Poisson geometry \cite{GSV}, and integrable systems \cite{GK, IIKKN}. A cornerstone result, the Laurent phenomenon \cite{FZ}, guarantees that every cluster variable can be expressed as a Laurent polynomial in the initial cluster variables. The positivity conjecture, asserting the non-negativity of coefficients in these polynomials, was resolved in full generality \cite{GHKK, LS}. Simpler, combinatorial proofs exist in special cases, such as for cluster algebras from surfaces using perfect matchings and snake graphs \cite{MSW08, MSW09}. 

In this paper, we investigate the expansion formulas associated with the cluster algebra structure on the moduli space $\mathscr{A}_{G, \bbS}$ of twisted $G$-local systems on a decorated Riemann surface $\bbS$ as introduced by Fock, Goncharov, and Shen \cite{FG, GS24} for a split semi-simple simply-connected algebraic group $G$. The case $G=\text{SL}_2$ recovers the well-studied cluster algebra from surfaces, which is also connected to Teichm\"uller theory \cite{FST, Pen}. In this paper, our primary focus is on the case $G = \text{SL}_{n+1}$.  We aim to derive general formulas by analyzing configurations corresponding to the standard $n$-triangulated $m$-gon. Subsequently, we also examine mutation sequences for quadrilateral flips in the case of quivers associated with type $G_2$, leading to some elementary polynomial identities.

The comprehensive framework developed by Goncharov and Shen \cite{GS24} provides the geometric foundation for our investigation. The work of Shapiro--Schrader \cite{SS} provides the crucial quantum group connection that guides our approach. Their explicit embedding of the quantum group $U_q(\mathfrak{sl}_{n+1})$ into quantum cluster algebras associated to moduli spaces of framed $\mathrm{SL}_{n+1}$-local systems reveals the geometry underlying the representation-theoretic aspects of these cluster structures. Notably, the $R$-matrix action is realized as a sequence of quiver mutations that produces the geometric action of a half-Dehn twist in a twice-punctured disk. The cluster realization of quantum group embeddings and $R$-matrices was later generalized to all Lie types in \cite{Ip}; in particular, one of the explicit realizations motivates the computation for the case of type $G_2$ in Section~\ref{sec:g2-cluster-realization}.

The $G=\text{SL}_2$ case associates the cluster algebra from surfaces with a fundamental combinatorial model via perfect matchings on snake graphs \cite{MSW08, MSW09}, which yield explicit expansion formulas for the cluster variables. For unpunctured surfaces with an initial triangulation $\cT$, the cluster variable $x_\gamma$ associated to an arc $\gamma\not\in\cT$ admits the  expansion formula \cite{MSW08, Sch07}:
\Eq{
x_\gamma = \frac{1}{\text{cross}(\gamma)} \sum_{P \in \text{Match}(\widetilde{G}_\gamma)} x(P) y(P),
}
where  $\widetilde{G}_\gamma$ is the snake graph associated to the arc $\gamma$, $\text{cross}(\gamma)$ is a product of cluster variables corresponding to arcs crossed by $\gamma$,
the sum is over all perfect matchings $P$ of $\widetilde{G}_\gamma$, and $x(P)$ and $y(P)$ are monomials in the initial cluster variables and coefficients respectively. This formula establishes the Laurent positivity property for cluster variables in unpunctured surfaces. 

Alternatively, a combinatorial model using $T$-paths also yields an equivalent expansion \cite{CS}. For a diagonal $M_{a,b}$ in a polygon (considered as a disk with $m$ marked points):
\Eq{
x_M = \sum_{\alpha \in \mathbb{W}_T(a,b)} x(\alpha),
}
where each $T$-path $\alpha$ contributes a monomial $x(\alpha)$ with coefficients either $0$ or $1$. This illustrates the fundamental connection between cluster variables and combinatorial paths in triangulated polygons.

In this work, we attempt to generalize the combinatorial model to higher rank, and derive explicit expansion formulas for the cluster algebra structure of $\mathscr{A}_{G, \bbS}$ associated to a sequence of cluster mutations corresponding to flips of diagonal in \emph{higher triangulations}, with particular emphasis on $G = \text{SL}_{n+1}$. Below, we summarize the key innovations and main theorems as follows.

\begin{enumerate}
    \item[(i)] \emph{Well-triangulated property}: We introduce and study the \textit{well-triangulated} property for triangles in $n$-triangulated polygons (Definition~\ref{def:well_triangulated}). This combinatorial condition is characterized by a specific monomial pattern $x^{bc}y^{ca}z^{ab}$ for the vertex values, and the main result of the paper shows that the property is preserved under arbitrary sequences of flips.

    \item[(ii)] \emph{Framework for $n$-triangulated $m$-gon}: We develop a comprehensive framework for $n$-triangulated $m$-gons (Definitions~\ref{Def:trilad}) with explicit labeling and mutation sequences, generalizing the type $A_1$ case to higher rank.

    \item[(iii)] \emph{Stair path combinatorics}: We introduce \textit{$n$-stair paths} and \textit{$n$-reversed stair paths} (Definitions~\ref{def:n_stairs}, \ref{def:n_rev_stairs}) as fundamental combinatorial objects encoding the structure of cluster variable expansions, visualized as alternating sequences of left/up or right/down segments. The recurrence relation established resembles the discrete 4D Hirota--Miwa equation (Proposition \ref{4DHM}).

    \item[(iv)] \emph{Mutation sequences for general flips}: We establish explicit mutation sequences for flips within $n$-triangulated quadrilaterals (Definition~\ref{subsec:flip}), revealing the intricate layered structure of these transformations.

    \item[(v)] \emph{Reflected polygon construction}: Definition~\ref{Def:ref_poly} introduces \textit{reflected polygons} $\overline{\mathcal{P}}(p_1,p_2,\ldots,p_k)$, providing a symmetry reduction tool for calculating the expansion formulas.
\end{enumerate}
The Main Theorems in this paper are as follows.
\begin{Thm}[Well-triangulated preservation (Theorem~\ref{thm:well_triangulated_preservation})]
For a triangulation $\cT$ of a well-triangulated polygon $\mathcal{P}$ with the cluster realization of type $A_n$, after any sequence of flips, the resulting triangulation is also well-triangulated.
\end{Thm}

Next, we obtain a closed-form expression for the exponents of the expansion formula of cluster variables along the diagonal of a triangulated square, which serves as a building block.
\begin{Thm}[Exponent Formula (Theorem~\ref{thm:exponents})]
The exponents in the expansion formula at unity satisfy $a_{i,j} = j(i-j+1)$ for all $i\geq j\geq 1$, leading to the explicit evaluation:
\Eq{
\cK_n = (2^n, 2^{2n-2}, \ldots, 2^{t(n-t+1)}, \ldots, 2^n).
}
\end{Thm}
Finally, we further generalize our calculations to arbitrary $m$-gons:
\begin{Thm}[Number of monomials (Theorem~\ref{thm:monomial_count})]
Consider the cluster realization of $\mathscr{A}_{\text{SL}_{n+1},\bbS}$ for a triangulation $\cT$ of a surface $\bbS$ with boundary and without punctures, and fix any non $\cT$-diagonal $\gamma$. If the number of monomials (counting multiplicities) in the Laurent polynomial of the expansion formula of $\gamma$ for the case $n=1$ is $K$, then the number of monomials (counting multiplicities) in the Laurent polynomial in each coordinate of the expansion formula of $\gamma$ for general $n$ is of the form $(K^n, K^{2n-2}, \ldots, K^{t(n+1-t)}, \ldots, K^{2n-2}, K^n)$.
\end{Thm}
After verifying the main theorems, we provide ways to label all vertices for convenience in setting up recurrences, hence we need to separate into $12$ different cases (depending on parity modulo $3$ and modulo $2$) where the order of diagonals in a polygon changes, allowing us to deduce the following formulas more concisely.
\begin{Thm}[Recurrence Formulas for General Polygons (Theorems~\ref{thm:1tri}, \ref{thm:2tri})]
We establish complete recurrence relations for expansion formulas in the general $1$-triangulated (Theorem~\ref{thm:1tri}) and $2$-triangulated (Theorem~\ref{thm:2tri}) $m$-gons, providing computational algorithms for cluster variables in arbitrary $n$-triangulated polygons.
\end{Thm}
From the number theory perspective, we further find the connections between the number of terms in the expansion formulas and \emph{continued fractions}.
\begin{Cor}[Term Count Formulas (Corollaries~\ref{cor:diagonal-terms}, \ref{cor:inner-terms})]
For any $n$-triangulated polygon $\mathcal{P}(p_1,p_2,\ldots,p_N)$:
\begin{enumerate}
    \item[(i)] The diagonal formula $D^{[p_1,p_2,\ldots,p_N]}_l$ has precisely $a(p_1,p_2,\ldots,p_N)^{l(n+1-l)}$ terms,
    \item[(ii)] The inner vertex formula $I^{[p_1,p_2,\ldots,p_N]}_{(i,j)}$ has precisely $a(p_1,p_2,\ldots,p_N)^{i(n+1-i-j)}a(p_1,p_2,\ldots,p_{N-1})^{j(n+1-i-j)}$ terms,
\end{enumerate}
where $a(p_1,p_2,\ldots,p_N)$ is the numerator of the continued fraction $[1;p_1,p_2,\ldots,p_N]$.
\end{Cor}
We provide a comprehensive combinatorial interpretation of cluster expansions through explicit formulas for $\psi_{t,k}(x_{i,j})$ in terms of stair path sums (Propositions~\ref{prop:stairs_combinatorial}, \ref{prop:reversed_stairs}), and the recurrence relations connecting different levels of the triangulation hierarchy, mediated via enumerating functions $\cS(x_{i,j},x_{k,l})$ and $\cR(x_{i,j},x_{k,l})$ counting stair paths and reversed-stair paths (Section~\ref{subsec:combinatorial_interpretation}). These results offer a complete combinatorial description of cluster variable expansions in type $A_n$ local system from surfaces, with explicit algorithms for computing the Laurent expansions in arbitrary $n$-triangulated polygons. This also lays the groundwork for further exploration in $G_2$ and other finite type cluster algebras.

By thoroughly examining the standard $n$-triangulated $m$-gon, this work makes progress toward a universal understanding of expansion formulas within the $\mathscr{A}_{G,\bbS}$ framework. The specific cases examined have revealed the deep links between the algebraic and the graphical aspects, providing a foundation for future research and tools for theoretical and computational applications in cluster algebra theory. In particular, the well-triangulated property plays a key role in understanding the stability and transformation rules within these cluster algebraic structures.

Future research will aim to derive complete expansion formulas for single flips across all finite irreducible Dynkin types, seeking to elucidate both universal and type-dependent behaviors. This will involve developing refined combinatorial strategies to solve recurrence relations, especially for other Lie types, thereby advancing the understanding of the cluster structure on $\mathscr{A}_{G, \bbS}$ in general.\\

\textbf{Organization of the paper.} The paper is organized as follows. Section \ref{sec:prelim} reviews the basic definitions and properties of cluster algebras and the cluster structure on moduli spaces of twisted $G$-local systems detailing the type $A_n$ case. Section \ref{sec:type_an_m_gon} provides the combinatorial interpretation of the expansion formula associated to general flip of diagonal over a quadrilateral via the $n$-stair path model and its associated recurrence relations. Section \ref{sec:ntrimgon} presents a comprehensive analysis of the cluster structure of $\mathscr{A}_{G,\bbS}$ over the standard $n$-triangulated $m$-gon, including the well-triangulated property, monomial counts, recurrence relations of the expansion formulas, and its connections to continued fractions. Section \ref{sec:g2-cluster-realization} details the study of cluster structure over mutation sequences for quadrilateral flips in type $G_2$ quivers, revealing new polynomial identities. Finally, the Appendices provide further observations and verification related to the recurrences discussed in Proposition  \ref{prop:g_func}.\\

\textbf{Acknowledgment.} This study was conducted under the Undergraduate Research Opportunities Program (UROP) at The Hong Kong University of Science and Technology. The second author is supported by the Hong Kong RGC General Research Funds [GRF \#16305122].

\section{Preliminaries}\label{sec:prelim}
\subsection{Cluster algebra} In this section we summarize the definitions and notation on cluster algebra \cite{FZ} needed in this paper.
\begin{Def}
A \emph{quiver} is a tuple $Q = (Q_{0}, Q_{1}, s, t)$, where $Q_{0}$ and $Q_{1}$ are finite sets, and $s, t:Q_{1}\to Q_{0}$ are set maps. $Q_0$ is the set of \emph{vertices}, $Q_1$ is the set of \emph{arrows}, $s$ is the \emph{source} map, and $t$ is the \emph{target} map. We assume $Q$ has no 1-cycles (or simply say $Q$ has no loops), i.e. $s(\a) \neq t(\a)$ for any $\a \in Q_{1}$, and $Q$ has no 2-cycles, i.e. for any $\a_{1}, \a_{2} \in Q_{1}$ such that $t(\a_{1}) = s(\a_{2})$, we require $t(\a_{2}) \neq s(\a_{1})$. We will present an arrow $\a\in Q_1$ with $s(\a)=i$ and $t(\a)=j$ by $i\to j$.

A quiver $Q$ is called an \emph{ice quiver} if we further partition the vertex set $Q_0=M\sqcup P$ into two sets (possibly empty), called \emph{mutable} and \emph{frozen} vertices respectively, such that no arrows connect two frozen vertices.
\end{Def}

In this section, let us denote $n:=|Q_0|$.
\begin{Def}
The \emph{signed adjacency matrix} of $Q$ is the $n\x n$ integer matrix $B = B(Q) = (b_{ij})$ indexed by $Q_{0}$, such that 
\Eq{
b_{ij} := \#\{\mbox{arrows from } i \to j\} - \#\{\mbox{arrows from } j \to i\}.
}
By definition $B$ is a skew-symmetric matrix, i.e. $B = -B^{T}$. In particular, all diagonal entries satisfy $b_{ii}=0$ for all $i \in Q_0$.
\end{Def}

\begin{Def}\label{mutQ}
Given an ice quiver $Q$ and a mutable vertex $k\in Q_0$, a quiver mutation $\mu_{k}$ is an operation transforming $Q$ into a new quiver $\mu_{k}(Q)$ by the following three steps:

\begin{itemize}
\item[(1)] For two arrows $i \to k \to j$ where not both $i, j$ are frozen, add a new arrow $i \to j$, counting with multiplicities.
\item[(2)] Reverse the direction of all arrows incident to $k$.
\item[(3)] Remove all oriented 2-cycles.
\end{itemize}
If $Q$ can be transformed into $Q'$ by a finite sequence of mutations up to permutation of the vertices, we call those 2 quivers \emph{mutation equivalent} and write $Q \sim Q'$. 
\end{Def}
The set of all quivers that are mutation equivalent to $Q$ is the \emph{mutation equivalence class} $[Q]$.

\begin{Lem}
The signed adjacency matrix $B' = (b'_{ij})$ of the mutated quiver $Q':=\mu_{k}(Q)$ satisfies:
\Eq{
b'_{ij}= 
\begin{cases}
    -b_{ij} & \text{if } k \in \{i, j\},\\
    b_{ij} + b_{ik}.b_{kj} & \text{if } b_{ik}>0, b_{kj}>0,\\
    b_{ij} - b_{ik}.b_{kj} & \text{if } b_{ik}<0, b_{kj}<0,\\
    b_{ij} & \text{otherwise.}
\end{cases}
}
Equivalently, we can rewrite this as
\Eq{
b'_{ij}= 
\begin{cases}
    -b_{ij} & \text{if } k \in \{i, j\},\\
    b_{ij} + \frac{|b_{ik}|b_{kj}+b_{ik}|b_{kj}|}{2} & \text{otherwise.}
\end{cases}
}
\end{Lem}
Let $\F:= \Q(u_{1}, u_{2}, ..., u_{n})$ be the ambient field of rational functions in $n$ independent variables. 

\begin{Def} 
A \emph{cluster} is an $n$-tuple $\bx = (x_{1}, x_{2}, ..., x_{n}) \in \F^{n}$ of algebraically independent variables. The elements in a cluster $\bx \in \F^{n}$ are called \emph{cluster variables}. A \emph{seed} is a pair $(\bx, Q)$ where $\bx \in \F^{n}$ is a cluster and $Q$ is an ice quiver with vertices $Q_{0} = \{1, 2, ..., n\}$ (without loops and 2-cycles).

A cluster variable $x_i$ is \emph{mutable} (or \emph{frozen}) if the vertex $i\in Q_0$ is mutable (or frozen) respectively.
\end{Def}

\begin{Def}
Let $k\in Q_0$ be mutable. The \emph{mutation of seed} $(\bx, Q)$ at $k$ is a seed $\mu_{k}(\bx, Q) := (\mu_{k}(\bx), \mu_{k}(Q))$ where $\mu_{k}(Q)$ is the mutation of $Q$ defined in Definition \ref{mutQ}, $\mu_{k}(\bx) := (x'_{1}, x'_{2}, ..., x'_{n}) \in \F^{n}$ is the cluster defined by:
\Eq{
x'_{l}:= 
\begin{cases}
    x_{l} & \text{if } l \neq k,\\
    \dis\frac{1}{x_{k}} \left(\prod_{k\to i} x_i + \prod_{j \to k} x_j\right) & \text{if } l = k
\end{cases}
}
where the product is taken over all arrows in $Q_1$ that start or end in vertex $k$, respectively (counted with multiplicity); the product is understood to be $1$ if there are no such arrows.
\end{Def}

Equivalently, if $B := B(Q)$ is the signed-adjacency matrix of $Q$, the equation can be rewritten as:
\Eq{
x_{k}x'_{k}= \prod_{\substack{i \in \{1, 2, ..., n\} \\ b_{ik} > 0}} x_i^{b_{ik}} + \prod_{\substack{i \in \{1, 2, ..., n\} \\ b_{ik} < 0}} x_i^{-b_{ik}}.
}
We shall call this equation the \emph{exchange relation} of the cluster variables.
\begin{Def} Let $(\bx, Q)$ be a seed. The $\Z$-subalgebra $\cA(Q)\subset \F$ generated by all the cluster variables obtained by all possible finite sequences of cluster mutations from $(\bx, Q)$ is called the \emph{cluster algebra} associated to the initial seed $(\bx, Q)$. The algebra is independent of the mutation equivalence class of $Q$.
\end{Def}

Let $(\bx, Q)$ be an initial seed and $\cA(Q)$ the corresponding cluster algebra. Let $\{1,...,n\}$ be mutable and $\{n+1,...,m\}$ be frozen. The Laurent phenomenon \cite{FZ} states that
\Eq{
\cA(Q)\subset \Z[x_1^{\pm1}, x_2^{\pm1}, ..., x_n^{\pm1}, x_{n+1}, x_{n+2}, . . . , x_m].
}
One of the main recent breakthroughs is the \emph{Positivity Theorem}, which was proven in full generality in \cite{GHKK}.
\begin{Thm}[Positivity Theorem \cite{GHKK}] 
\label{thm:pos_thm}
For any cluster variables written as Laurent polynomial of the initial seed as above, all the coefficients lie in $\Z_{>0}$.
\end{Thm}
The special case of the theorem which motivates the combinatorial understanding of the current work comes from cluster algebras arising from surfaces, which was proved in \cite{MSW08, MSW09} by describing the coefficients as counting the number of perfect matchings associated to certain graphs called \emph{snake diagrams}. It turns out that they can also be described in terms of continued fractions \cite{CS}. We recall some notations that are needed in this paper.

\begin{Def}
The \emph{standard continued fraction} representation of $x\in\R$ is defined as:
\Eq{
x = b_0 + \cfrac{a_1}{b_1 + \cfrac{a_2}{b_2 + \cfrac{a_3}{b_3 + \cdots}}}
}
understood as the limit of the sequence of \emph{convergents}
\Eq{
x_n := b_0+\cfrac{a_1}{b_1+\cfrac{a_2}{b_2+\cfrac{a_3}{\ddots + \cfrac{a_n}{b_n}}}}=\frac{A_n}{B_n}.
}
A specific example with numerical values for $\pi$ is:
\Eq{
\pi = 3 + \cfrac{1^2}{6 + \cfrac{3^2}{6 + \cfrac{5^2}{6 + \cfrac{7^2}{6 + \cdots}}}}.
}
When all $a_i = 1$, we will write the continued fraction as 
\Eq{
x=[b_0; b_1, b_2, \dots].
}
\end{Def}

\begin{Prop}[Fundamental theorem of continued fractions]
\label{prop:fun_con} The convergents $\dfrac{A_n}{B_n}$ satisfy the initial conditions:
\begin{align*}
A_{-1} = 1,\quad A_0 = b_0, \quad B_{-1} = 0,\quad B_0 = 1
\end{align*}
and for $n \geq 1$, the recurrence relations:
\Eq{
A_n &= b_n A_{n-1} + a_n A_{n-2}, \nonumber\\
B_n &= b_n B_{n-1} + a_n B_{n-2}.
}
\end{Prop}

\subsection{Moduli spaces of $G$-local systems from surfaces}
We recall some definitions and terminology related to the moduli spaces of $G$-local systems over a bordered surface with marked points.
\label{subsec:modspsurface}
\begin{Def}
A pair $(\bbS, M)$ is called a \emph{bordered surface with marked points} if it satisfies:
\begin{enumerate}
\item[(i)] $\bbS$ is a connected, orientable 2-dimensional surface but possibly with boundary $\partial \bbS$;
\item[(ii)] $M \subseteq \partial \bbS$ is a non-empty finite set of marked points;
\item[(iii)] Each connected boundary component contains at least one marked point.
\end{enumerate}
A disk with $n$ marked points is called an \emph{$n$-gon}.
\end{Def}
In this article, we will assume that $\bbS$ has no punctures. See Section \ref{subsec:punctured_case} for further discussions on the punctured surface cases.

\begin{Def}
An \emph{arc} in $(\bbS, M)$ is an isotopy class of curves $\gamma$ satisfying:
\begin{enumerate}
\item[(i)] The endpoints of $\gamma$ lie in $M$;
\item[(ii)] $\gamma$ does not self-intersect except at endpoints;
\item[(iii)] $\gamma$ is disjoint from $M$ and $\partial \bbS$ except at endpoints.
\end{enumerate}
Two arcs are called \emph{compatible} if they do not intersect in the interior of $\bbS$. The set of all arcs is denoted by $A^{0}(\bbS, M)$. A maximal collection of distinct pairwise compatible arcs is called an \emph{ideal triangulation}. The arcs of an ideal triangulation partition $\bbS$ into \emph{ideal triangles} $T\in\cT$.
\end{Def}
Since we only consider surfaces without punctures, there will be no self-folded triangles.
\begin{Prop}
Any ideal triangulation $\cT$ of $\bbS$ contains exactly
\Eq{
n = 6g + 3b + c - 6
}
arcs, where $g$ denotes the \emph{genus} of $\bbS$, $b$ the number of \emph{boundary components}, and $c$ the number of \emph{marked points} on $\partial \bbS$.
\end{Prop}
\begin{Def}
Let $\cT$ be a triangulation containing a diagonal $\gamma$. Then there exists a unique quadrilateral in $\cT$ having $\gamma$ as a diagonal. The operation of replacing $\gamma$ with the other diagonal $\gamma'$ of this quadrilateral to obtain a new triangulation $\cT'$ is called a \emph{flip} (see Figure \ref{def:flip}).
\end{Def}

\begin{figure}[H]
\centering
\begin{tikzpicture}[node distance=2cm]
  \coordinate (A1) at (0,0);
  \coordinate (B1) at (3,0);
  \coordinate (C1) at (4,1.5);
  \coordinate (D1) at (1,1.5);
  \draw (A1) -- (B1) -- (C1) -- (D1) -- cycle;
  \draw[thick] (A1) -- (C1) node[midway, below left] {$\gamma$};
  \node[below left] at (A1) {$A$};
  \node[below right] at (B1) {$B$};
  \node[above right] at (C1) {$C$};
  \node[above left] at (D1) {$D$};

  \coordinate (A2) at (6,0);
  \coordinate (B2) at (9,0);
  \coordinate (C2) at (10,1.5);
  \coordinate (D2) at (7,1.5);
  \draw (A2) -- (B2) -- (C2) -- (D2) -- cycle;
  \draw[thick] (B2) -- (D2);
  \node[below left] at (A2) {$A$};
  \node[below right] at (B2) {$B$};
  \node[above right] at (C2) {$C$};
  \node[above left] at (D2) {$D$};
  
  \node at (8,1) {$\gamma'$};

  \draw[->, thick] (4.5,0.75) -- (5.5,0.75) node[midway, above] {flip};
\end{tikzpicture}
\caption{The flip of a quadrilateral}
\label{def:flip}
\end{figure}
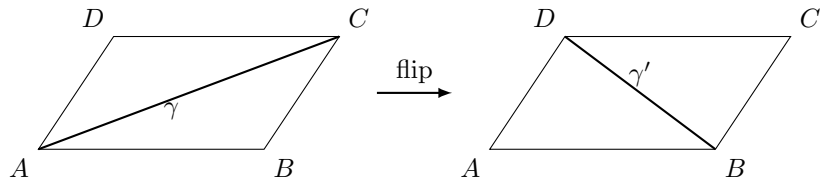

Let $G$ be a split semisimple simply-connected group over $\Q$ and $\bbS$ a bordered surface with marked points as above. The \emph{moduli space of $G$-local systems} $\mathscr{A}_{G,\bbS}$ and its variants are first considered by Fock--Goncharov \cite{FG} and further refined in Goncharov--Shen \cite{GS24}. Typically a $G$-local system consists of a principal $G$-bundle $\cL$ on $\bbS$ with flat connection and certain extra data provided by decorated flags in $\cA\in G/U$, where $U$ is the unipotent radical of a fixed Borel subgroup $B\subset G$.

Together with another decorated moduli space $\mathscr{P}_{G^{ad},\bbS}$ where $G^{ad}$ is the adjoint group, this carries a cluster Poisson structure, i.e. corresponds to a cluster $\cX$-variety, and the pair $(\mathscr{P}_{G^{ad},\bbS},\mathscr{A}_{G,\bbS})$ forms a so-called cluster ensemble. There is an atlas for each choice of ideal triangulation for $\bbS$, along with a choice of special vertex and also a choice of reduced word of the longest element $w_0$ of the Weyl group of $G$ for each triangle of the ideal triangulation which encodes the seed of the cluster structure. Each of these atlases is related by a sequence of mutations.

Consequently, in the simply-laced case, the cluster structure of $\mathscr{A}_{G,\bbS}$ above can be encoded by a generalized quiver $Q$, which is the result of amalgamation of the \emph{basic quivers}, first constructed in \cite{FG} for type $A_n$ and \cite{Ip, Le, GS24} for other Lie types, associated to each triangle $T\in \cT$ of the triangulation of $\bbS$. They satisfy certain features:
\begin{itemize}
\item[(1)] For each triangle $T\in\cT$ the \emph{basic quiver} $Q_T$ depends on the choice of reduced word of $w_0$ and the special vertex, and different choices are related by a sequence of quiver mutations.
\item[(2)] The basic quiver consists of $n$ frozen vertices on each boundary edge of $T$, where $n$ is the rank of $G$. Furthermore, we allow ``half-weighted arrows'' between frozen vertices, so that the signed-adjacency matrix $B$ has $\frac12\Z$ coefficients.
\item[(3)] The quiver $Q$ encoding the cluster structure of $\mathscr{A}_{G,\bbS}$ is obtained by \emph{amalgamation}, i.e. gluing the basic quivers, according to the triangulations, so that the frozen vertices on the glued boundary become mutable, and the weights of the arrows add accordingly.
\end{itemize}
The half arrows between frozen vertices facilitate the amalgamation between basic quivers, but otherwise they do not play a role in the cluster structure of the frozen variables of $\mathscr{A}_{G,\bbS}$.
\begin{Rem}The non-simply-laced case is also available, which is described using quivers with multipliers.\end{Rem}

In particular, when $G=\text{SL}_2$, the cluster structure is the same as the standard \emph{cluster algebra from surfaces} \cite{FST}, where the basic quiver assigned to each triangle $T$ is just a counterclockwise quiver with vertex on the midpoint of each boundary edge of $T$. The flipping of a diagonal corresponds to a single cluster mutation, and the combinatorics of cluster variables and their expansion formula under these sequences of mutations are completely understood via the combinatorics of snake diagrams and perfect matching \cite{MSW08,MSW09}.

In this paper, we will be interested in understanding the combinatorics of the cluster variables and their expansion formula for higher rank algebraic group $G$ under sequences of cluster mutations arising from flipping the diagonal of triangulations on the surface of $m$-gons, which is considerably much more complicated.

\subsection{Cluster realization of $\mathscr{A}_{G,\bbS}$ in type $A_n$}
\label{subsec:cluster_realization}
Recall the moduli space $\mathscr{A}_{G,\bbS}$ of twisted $G$-local systems, whose cluster $\cA$-variety structure is encoded by a certain quiver $Q$ associated to the triangulation of the surface. In this subsection, we explain how the quivers are constructed for $G=\text{SL}_{n+1}$ of type $A_n$, and also recall the mutation sequences associated with flips.

We first consider a single triangle, i.e. a disk with 3 marked points on the boundary.
\begin{Def}
\label{Def:trilad} An \emph{$n$-triangulated quiver} of a triangle $T$ is constructed through the following procedure:
\begin{enumerate}[label=(\arabic*)]
\item Subdivide each edge of $T$ by placing $n$ equally spaced vertices, creating $n+1$ equal segments per edge. These vertices are frozen.

\item Construct a triangulation by drawing lines parallel to the edges between corresponding vertices. Their intersections form the mutable vertices.

\item Classify triangles into two types:
\begin{enumerate}
    \item The three corner triangles (containing vertices of the original triangle) are designated as \emph{unshaded}.
    \item The remaining triangles alternate between \emph{shaded} and \emph{unshaded}, with adjacent triangles having opposite types.
\end{enumerate}

\item Construct the corresponding quiver where:
\begin{enumerate}
    \item Shaded triangles contribute 3 arrows forming counterclockwise cycles.
    \item Unshaded triangles contribute 3 arrows forming clockwise cycles.
    \item Red dotted arrows represent half-arrows in the quiver along the boundaries.
\end{enumerate}
\end{enumerate}
\end{Def}

We can now construct the quiver on a general marked surface by amalgamating (gluing) the corresponding quivers associated to the triangles of a given triangulation.
\begin{Def}
An \emph{$n$-triangulated quiver of a quadrilateral} is formed by gluing two $n$-triangulated triangles along a common edge. An \emph{$n$-triangulated quiver of a surface} is obtained by assembling $n$-triangulated triangles according to a surface triangulation.
\end{Def}

\begin{Def}
\label{subsec:flip}
A \emph{flip} of an $n$-triangulated quiver of a quadrilateral is a sequence of quiver mutations performed in $n$ steps. For each step $k$ with $1 \leq k \leq n$:
\begin{enumerate}[label=(\arabic*)]
\item Inscribe a $k \times (n+1-k)$ rectangle such that its vertices coincide with boundary vertices of the $n$-triangulated quadrilateral, while the side of length $n+1-k$ lies along the diagonal to be flipped.
\item Divide this rectangle into $k(n+1-k)$ unit squares.
\item Perform quiver mutations at the center of each square.
\end{enumerate}
\end{Def}
The order of mutations within each step is irrelevant, since no arrows exist between the mutated vertices after the previous steps, hence the mutations commute with each other. Figure~\ref{fig:78} shows the vertices (in blue) to mutate at each step in case $n=7$.
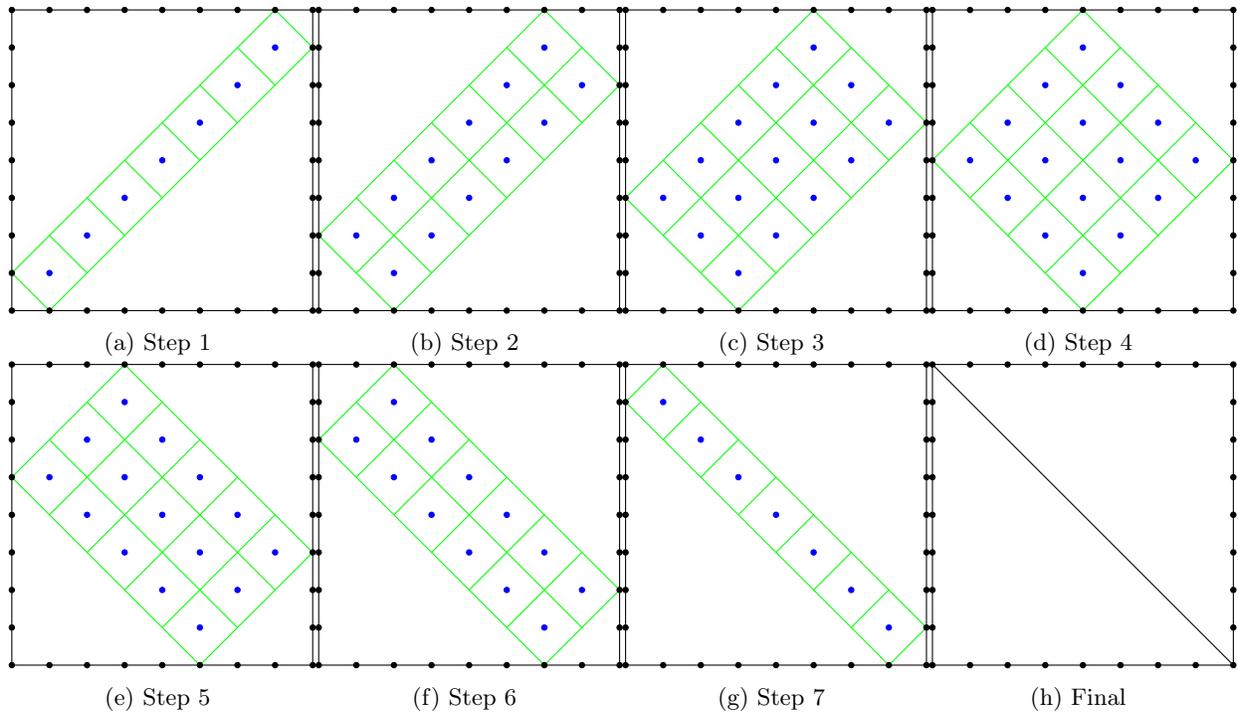
\begin{figure}[H]
\centering
\newcommand{\drawboundary}{
    \foreach \x in {0,...,8} {
        \foreach \y in {0,8} {
            \filldraw (\x,\y) circle (2pt);
        }
    }
    \foreach \y in {1,...,7} {
        \foreach \x in {0,8} {
            \filldraw (\x,\y) circle (2pt);
        }
    }
}

\begin{subfigure}[b]{0.24\textwidth}
\centering
\begin{tikzpicture}[scale=0.5]
\draw (0,0) rectangle (8,8);
\draw[thin,green] (1,0)--(0,1)--(7,8)--(8,7)--cycle;
\foreach \x in {1,...,6} {
	\draw[thin,green] (\x,\x+1)--(\x+1,\x);
}
\drawboundary
\foreach \i in {1,...,7} {
    \filldraw[blue] (\i,\i) circle (2pt);
}
\end{tikzpicture}
\caption{Step 1}
\end{subfigure}
\begin{subfigure}[b]{0.24\textwidth}
\centering
\begin{tikzpicture}[scale=0.5]
\draw (0,0) rectangle (8,8);
\draw[thin,green] (2,0)--(0,2)--(6,8)--(8,6)--cycle;
\draw[thin,green] (1,1)--(7,7);
\foreach \x in {1,...,5} {
	\draw[thin,green] (\x,\x+2)--(\x+2,\x);
}
\drawboundary
\foreach \i in {1,...,6} {
    \filldraw[blue] (\i+1,\i) circle (2pt);
    \filldraw[blue] (\i,\i+1) circle (2pt);
}
\end{tikzpicture}
\caption{Step 2}
\end{subfigure}
\begin{subfigure}[b]{0.24\textwidth}
\centering
\begin{tikzpicture}[scale=0.5]
\draw (0,0) rectangle (8,8);
\draw[thin,green] (3,0)--(0,3)--(5,8)--(8,5)--cycle;
\draw[thin,green] (1,2)--(6,7);
\draw[thin,green] (2,1)--(7,6);
\foreach \x in {1,...,4} {
	\draw[thin,green] (\x,\x+3)--(\x+3,\x);
}
\drawboundary
\foreach \i in {1,...,5} {
    \filldraw[blue] (\i+2,\i) circle (2pt);
    \filldraw[blue] (\i+1,\i+1) circle (2pt);
    \filldraw[blue] (\i,\i+2) circle (2pt);
}
\end{tikzpicture}
\caption{Step 3}
\end{subfigure}
\begin{subfigure}[b]{0.24\textwidth}
\centering
\begin{tikzpicture}[scale=0.5]
\draw (0,0) rectangle (8,8);
\draw[thin,green] (4,0)--(0,4)--(4,8)--(8,4)--cycle;
\draw[thin,green] (1,3)--(5,7);
\draw[thin,green] (2,2)--(6,6);
\draw[thin,green] (3,1)--(7,5);
\foreach \x in {1,...,3} {
	\draw[thin,green] (\x,\x+4)--(\x+4,\x);
}
\drawboundary
\foreach \i in {1,...,4} {
    \filldraw[blue] (\i+3,\i) circle (2pt);
    \filldraw[blue] (\i+2,\i+1) circle (2pt);
    \filldraw[blue] (\i+1,\i+2) circle (2pt);
    \filldraw[blue] (\i,\i+3) circle (2pt);
}
\end{tikzpicture}
\caption{Step 4}
\end{subfigure}

\begin{subfigure}[b]{0.24\textwidth}
\centering
\begin{tikzpicture}[scale=0.5]
\draw (0,0) rectangle (8,8);
\draw[thin,green] (5,0)--(0,5)--(3,8)--(8,3)--cycle;
\draw[thin,green] (1,4)--(4,7);
\draw[thin,green] (2,3)--(5,6);
\draw[thin,green] (3,2)--(6,5);
\draw[thin,green] (4,1)--(7,4);
\foreach \x in {1,...,2} {
	\draw[thin,green] (\x,\x+5)--(\x+5,\x);
}
\drawboundary
\foreach \i in {1,...,3} {
    \filldraw[blue] (\i+4,\i) circle (2pt);
    \filldraw[blue] (\i+3,\i+1) circle (2pt);
    \filldraw[blue] (\i+2,\i+2) circle (2pt);
    \filldraw[blue] (\i+1,\i+3) circle (2pt);
    \filldraw[blue] (\i,\i+4) circle (2pt);
}
\end{tikzpicture}
\caption{Step 5}
\end{subfigure}
\begin{subfigure}[b]{0.24\textwidth}
\centering
\begin{tikzpicture}[scale=0.5]
\draw (0,0) rectangle (8,8);
\draw[thin,green] (6,0)--(0,6)--(2,8)--(8,2)--cycle;
\draw[thin,green] (1,5)--(3,7);
\draw[thin,green] (2,4)--(4,6);
\draw[thin,green] (3,3)--(5,5);
\draw[thin,green] (4,2)--(6,4);
\draw[thin,green] (5,1)--(7,3);
\foreach \x in {1,...,1} {
	\draw[thin,green] (\x,\x+6)--(\x+6,\x);
}
\drawboundary
\foreach \i in {1,...,2} {
    \filldraw[blue] (\i+5,\i) circle (2pt);
    \filldraw[blue] (\i+4,\i+1) circle (2pt);
    \filldraw[blue] (\i+3,\i+2) circle (2pt);
    \filldraw[blue] (\i+2,\i+3) circle (2pt);
    \filldraw[blue] (\i+1,\i+4) circle (2pt);
    \filldraw[blue] (\i,\i+5) circle (2pt);
}
\end{tikzpicture}
\caption{Step 6}
\end{subfigure}
\begin{subfigure}[b]{0.24\textwidth}
\centering
\begin{tikzpicture}[scale=0.5]
\draw (0,0) rectangle (8,8);
\draw[thin,green] (7,0)--(0,7)--(1,8)--(8,1)--cycle;
\draw[thin,green] (1,6)--(2,7);
\draw[thin,green] (2,5)--(3,6);
\draw[thin,green] (3,4)--(4,5);
\draw[thin,green] (4,3)--(5,4);
\draw[thin,green] (5,2)--(6,3);
\draw[thin,green] (6,1)--(7,2);
\drawboundary
\foreach \i in {1,...,1} {
    \filldraw[blue] (\i+6,\i) circle (2pt);
    \filldraw[blue] (\i+5,\i+1) circle (2pt);
    \filldraw[blue] (\i+4,\i+2) circle (2pt);
    \filldraw[blue] (\i+3,\i+3) circle (2pt);
    \filldraw[blue] (\i+2,\i+4) circle (2pt);
    \filldraw[blue] (\i+1,\i+5) circle (2pt);
    \filldraw[blue] (\i,\i+6) circle (2pt);
}
\end{tikzpicture}
\caption{Step 7}
\end{subfigure}
\begin{subfigure}[b]{0.24\textwidth}
\centering
\begin{tikzpicture}[scale=0.5]
\draw (0,0) rectangle (8,8);
\drawboundary
\draw (0,8) -- (8,0);
\end{tikzpicture}
\caption{Final}
\end{subfigure}
\caption{Step-by-step construction of a flip in a $7$-triangulated quadrilateral}
\label{fig:78}
\end{figure}

\begin{Ex}
For $n=2$, each edge of a triangle is divided into 3 equal parts using 2 vertices. We get the $2$-triangulated triangle with its corresponding quiver (see Figure~\ref{fig:2_tri_quiv}) and the resulting quadrilateral after flip (see Figure~\ref{fig:3}).
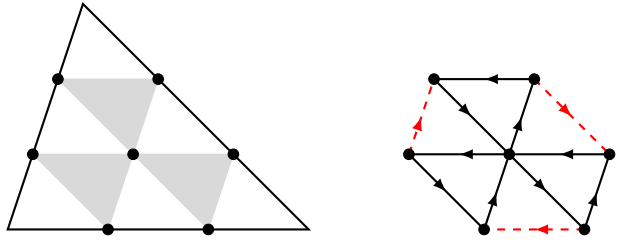
\begin{figure}[H]
\centering
\begin{tikzpicture}[scale=1,
    mid arrow/.style={
        postaction={decorate},
        decoration={
            markings,
            mark=at position 0.5 with {\arrow{>}}
        }
    }
]
   
  \draw[thick] (5,0) -- (6,3) -- (9,0) -- cycle;
  
  \filldraw[gray!30] (5.667,2) -- (7,2) -- (6.667,1) -- cycle;
  \filldraw[gray!30] (6.667,1) -- (6.333,0) -- (5.333,1) -- cycle;
  \filldraw[gray!30] (6.667,1) -- (7.667,0) -- (8,1) -- cycle;
  
  \filldraw[black] (5.333,1) circle (2pt);
  \filldraw[black] (5.667,2) circle (2pt);
  \filldraw[black] (6.333,0) circle (2pt);
  \filldraw[black] (7.667,0) circle (2pt);
  \filldraw[black] (7,2) circle (2pt);
  \filldraw[black] (8,1) circle (2pt);
  \filldraw[black] (6.667,1) circle (2pt);
  
  
  \draw[mid arrow, red, thick, dashed] (10.333,1) -- (10.667,2);
  \draw[mid arrow, red, thick, dashed] (12,2) -- (13,1);
  \draw[mid arrow, red, thick, dashed] (12.667,0) -- (11.333,0);
  
  \draw[mid arrow, thick] (12,2) -- (10.667,2);
  \draw[mid arrow, thick] (13,1) -- (11.667,1);
  \draw[mid arrow, thick] (11.667,1) -- (10.333,1);
  \draw[mid arrow, thick] (10.333,1) -- (11.333,0);
  \draw[mid arrow, thick] (10.667,2) -- (11.667,1);
  \draw[mid arrow, thick] (11.667,1) -- (12.667,0);
  \draw[mid arrow, thick] (11.333,0) -- (11.667,1);
  \draw[mid arrow, thick] (11.667,1) -- (12,2);
  \draw[mid arrow, thick] (12.667,0) -- (13,1);
  
  \filldraw[black] (10.333,1) circle (2pt);
  \filldraw[black] (10.667,2) circle (2pt);
  \filldraw[black] (11.333,0) circle (2pt);
  \filldraw[black] (12.667,0) circle (2pt);
  \filldraw[black] (12,2) circle (2pt);
  \filldraw[black] (13,1) circle (2pt);
  \filldraw[black] (11.667,1) circle (2pt);
\end{tikzpicture}
\caption{A $2$-triangulated triangle and the corresponding quiver}
\label{fig:2_tri_quiv}
\end{figure}

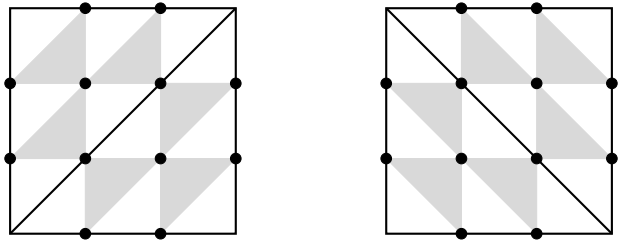
\begin{figure}[H]
\centering
\begin{tikzpicture}[scale=1]
  \draw[thick] (0,0) -- (0,3) -- (3,3) -- (3,0) -- cycle;
  \draw[thick] (5,0) -- (5,3) -- (8,3) -- (8,0) -- cycle;
  \draw[thick] (0,0) -- (3,3);
  \draw[thick] (5,3) -- (8,0);
  
  \filldraw[gray!30] (0,1) -- (1,2) -- (1,1) -- cycle;
  \filldraw[gray!30] (0,2) -- (1,3) -- (1,2) -- cycle;
  \filldraw[gray!30] (1,2) -- (2,3) -- (2,2) -- cycle;
  \filldraw[gray!30] (1,1) -- (2,1) -- (1,0) -- cycle;
  \filldraw[gray!30] (2,1) -- (3,1) -- (2,0) -- cycle;
  \filldraw[gray!30] (2,2) -- (3,2) -- (2,1) -- cycle;
  
  \filldraw[gray!30] (5,1) -- (6,1) -- (6,0) -- cycle;
  \filldraw[gray!30] (5,2) -- (6,2) -- (6,1) -- cycle;
  \filldraw[gray!30] (6,1) -- (7,1) -- (7,0) -- cycle;
  \filldraw[gray!30] (6,3) -- (6,2) -- (7,2) -- cycle;
  \filldraw[gray!30] (7,3) -- (7,2) -- (8,2) -- cycle;
  \filldraw[gray!30] (7,2) -- (7,1) -- (8,1) -- cycle;   
  
  \filldraw[black] (0,1) circle (2pt);
  \filldraw[black] (0,2) circle (2pt);
  \filldraw[black] (1,0) circle (2pt);
  \filldraw[black] (2,0) circle (2pt);
  \filldraw[black] (3,1) circle (2pt);
  \filldraw[black] (3,2) circle (2pt);
  \filldraw[black] (1,3) circle (2pt);
  \filldraw[black] (2,3) circle (2pt); 
  \filldraw[black] (1,1) circle (2pt);
  \filldraw[black] (1,2) circle (2pt);
  \filldraw[black] (2,1) circle (2pt);
  \filldraw[black] (2,2) circle (2pt);
  
  \filldraw[black] (5,1) circle (2pt);
  \filldraw[black] (5,2) circle (2pt);
  \filldraw[black] (6,0) circle (2pt);
  \filldraw[black] (7,0) circle (2pt);
  \filldraw[black] (8,1) circle (2pt);
  \filldraw[black] (8,2) circle (2pt);
  \filldraw[black] (6,3) circle (2pt);
  \filldraw[black] (7,3) circle (2pt); 
  \filldraw[black] (6,1) circle (2pt);
  \filldraw[black] (6,2) circle (2pt);
  \filldraw[black] (7,1) circle (2pt);
  \filldraw[black] (7,2) circle (2pt);   
\end{tikzpicture}
\caption{Original quadrilateral (left) and after flip (right)}
\label{fig:3}
\end{figure}
Note that no arrows exist between vertices lying on the same main diagonal. The flip is achieved through sequences of mutations in two steps (refer to the labeling in Figure \ref{quadquiv}):
\begin{enumerate}[label=Step \arabic*:]
\item Mutate at $x_1$ and $x_2$;
\item Mutate at $x_3$ and $x_4$.
\end{enumerate}
Since no edges connect vertices within the same step, the mutation order does not affect the result. Without loss of generality, we can let the sequence be $\mu = \{x_2 \rightarrow x_1 \rightarrow x_4 \rightarrow x_3 \}$.
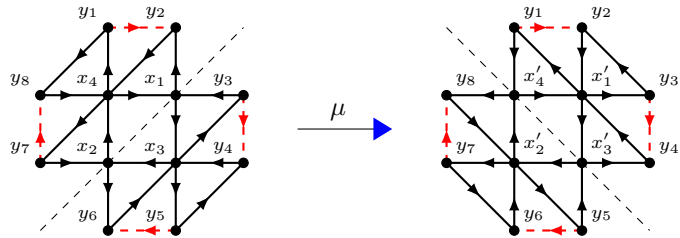
\begin{figure}[H]
\centering
\begin{tikzpicture}[scale=0.9,
    mid arrow/.style={
        postaction={decorate},
        decoration={
            markings,
            mark=at position 0.5 with {\arrow{>}}
        }
    }
]
  \draw[thin,dashed] (0,0) -- (3,3);
  \draw[thin,dashed] (5+1,3) -- (8+1,0);
  \draw[-{Triangle[blue,scale=2,line width=1.5pt]}] (3.8,1.5) -- (5.2,1.5);
  \node[above] at (4.4,1.5) {$\mu$};

\draw[mid arrow, thick] (1,2) -- (0,1);
\draw[mid arrow, thick] (1,1) -- (1,2);
\draw[mid arrow, thick] (0,1) -- (1,1);

\draw[mid arrow, thick] (1,3) -- (0,2);
\draw[mid arrow, thick] (1,2) -- (1,3);
\draw[mid arrow, thick] (0,2) -- (1,2);

\draw[mid arrow, thick] (2,3) -- (1,2);
\draw[mid arrow, thick] (2,2) -- (2,3);
\draw[mid arrow, thick] (1,2) -- (2,2);

\draw[mid arrow, thick] (2,1) -- (1,1);
\draw[mid arrow, thick] (1,0) -- (2,1);
\draw[mid arrow, thick] (1,1) -- (1,0);

\draw[mid arrow, thick] (3,1) -- (2,1);
\draw[mid arrow, thick] (2,0) -- (3,1);
\draw[mid arrow, thick] (2,1) -- (2,0);

\draw[mid arrow, thick] (3,2) -- (2,2);
\draw[mid arrow, thick] (2,1) -- (3,2);
\draw[mid arrow, thick] (2,2) -- (2,1);

\draw[mid arrow, thick] (6+1,1) -- (5+1,1);
\draw[mid arrow, thick] (6+1,0) -- (6+1,1);
\draw[mid arrow, thick] (5+1,1) -- (6+1,0);

\draw[mid arrow, thick] (6+1,2) -- (5+1,2);
\draw[mid arrow, thick] (6+1,1) -- (6+1,2);
\draw[mid arrow, thick] (5+1,2) -- (6+1,1);

\draw[mid arrow, thick] (7+1,1) -- (6+1,1);
\draw[mid arrow, thick] (7+1,0) -- (7+1,1);
\draw[mid arrow, thick] (6+1,1) -- (7+1,0);

\draw[mid arrow, thick] (6+1,3) -- (6+1,2);
\draw[mid arrow, thick] (6+1,2) -- (7+1,2);
\draw[mid arrow, thick] (7+1,2) -- (6+1,3);
  
\draw[mid arrow, thick] (7+1,3) -- (7+1,2);
\draw[mid arrow, thick] (7+1,2) -- (8+1,2);
\draw[mid arrow, thick] (8+1,2) -- (7+1,3);
  
\draw[mid arrow, thick] (7+1,2) -- (7+1,1);
\draw[mid arrow, thick] (7+1,1) -- (8+1,1); 
\draw[mid arrow, thick] (8+1,1) -- (7+1,2);
  
  \draw[mid arrow, red,thick,dashed] (0,1) -- (0,2);
  \draw[mid arrow, red,thick,dashed] (1,3) -- (2,3);
  \draw[mid arrow, red,thick,dashed] (3,2) -- (3,1); 
  \draw[mid arrow, red,thick,dashed] (2,0) -- (1,0);
  
  \draw[mid arrow, red,thick,dashed] (6+1,3) -- (7+1,3);
  \draw[mid arrow, red,thick,dashed] (8+1,2) -- (8+1,1);
  \draw[mid arrow, red,thick,dashed] (7+1,0) -- (6+1,0); 
  \draw[mid arrow, red,thick,dashed] (5+1,1) -- (5+1,2); 
  
  \filldraw[black] (0,1) circle (2pt);
  \node[above left, font=\scriptsize] at (0,1) {$y_{7}$};
  
  \filldraw[black] (0,2) circle (2pt);
  \node[above left, font=\scriptsize] at (0,2) {$y_{8}$};
  
  \filldraw[black] (1,0) circle (2pt);
  \node[above left, font=\scriptsize] at (1,0) {$y_{6}$};
  
  \filldraw[black] (2,0) circle (2pt);
  \node[above left, font=\scriptsize] at (2,0) {$y_{5}$};
  
  \filldraw[black] (3,1) circle (2pt);
  \node[above left, font=\scriptsize] at (3,1) {$y_{4}$};
  
  \filldraw[black] (3,2) circle (2pt);
  \node[above left, font=\scriptsize] at (3,2) {$y_{3}$};
  
  \filldraw[black] (1,3) circle (2pt);
  \node[above left, font=\scriptsize] at (1,3) {$y_{1}$};
  
  \filldraw[black] (2,3) circle (2pt); 
  \node[above left, font=\scriptsize] at (2,3) {$y_{2}$};
  
  \filldraw[black] (1,1) circle (2pt);
  \node[above left, font=\scriptsize] at (1,1) {$x_{2}$};
  
  \filldraw[black] (1,2) circle (2pt);
  \node[above left, font=\scriptsize] at (1,2) {$x_{4}$};
  
  \filldraw[black] (2,1) circle (2pt);
  \node[above left, font=\scriptsize] at (2,1) {$x_{3}$};
  
  \filldraw[black] (2,2) circle (2pt);
  \node[above left, font=\scriptsize] at (2,2) {$x_{1}$};

  \filldraw[black] (5+1,1) circle (2pt);
  \node[above right, font=\scriptsize] at (5+1,1) {$y_{7}$};
  
  \filldraw[black] (5+1,2) circle (2pt);
  \node[above right, font=\scriptsize] at (5+1,2) {$y_{8}$};
  
  \filldraw[black] (6+1,0) circle (2pt);
  \node[above right, font=\scriptsize] at (6+1,0) {$y_{6}$};
  
  \filldraw[black] (7+1,0) circle (2pt);
  \node[above right, font=\scriptsize] at (7+1,0) {$y_{5}$};
  
  \filldraw[black] (8+1,1) circle (2pt);
  \node[above right, font=\scriptsize] at (8+1,1) {$y_{4}$};
  
  \filldraw[black] (8+1,2) circle (2pt);
  \node[above right, font=\scriptsize] at (8+1,2) {$y_{3}$};
  
  \filldraw[black] (6+1,3) circle (2pt);
  \node[above right, font=\scriptsize] at (6+1,3) {$y_{1}$};
  
  \filldraw[black] (7+1,3) circle (2pt);
  \node[above right, font=\scriptsize] at (7+1,3) {$y_{2}$};
  
  \filldraw[black] (6+1,1) circle (2pt);
  \node[above right, font=\scriptsize] at (6+1,1) {$x'_{2}$};
  
  \filldraw[black] (6+1,2) circle (2pt);
  \node[above right, font=\scriptsize] at (6+1,2) {$x'_{4}$};
  
  \filldraw[black] (7+1,1) circle (2pt);
  \node[above right, font=\scriptsize] at (7+1,1) {$x'_{3}$};
  
  \filldraw[black] (7+1,2) circle (2pt);
  \node[above right, font=\scriptsize] at (7+1,2) {$x'_{1}$};   
\end{tikzpicture}
\caption{Quivers corresponding to the quadrilaterals in Figure~\ref{fig:3}}\label{quadquiv}
\end{figure}

\begin{figure}[H]
\centering
\begin{tikzpicture}[scale=0.72,
    mid arrow/.style={
        postaction={decorate},
        decoration={
            markings,
            mark=at position 0.5 with {\arrow{>}}
        }
    }
]
\foreach \i in {0,1,2,3,4} {
  
  \draw[mid arrow, red,thick,dashed] (5*\i+0,1) -- (5*\i+0,2);
  \draw[mid arrow, red,thick,dashed] (5*\i+1,3) -- (5*\i+2,3);
  \draw[mid arrow, red,thick,dashed] (5*\i+3,2) -- (5*\i+3,1); 
  \draw[mid arrow, red,thick,dashed] (5*\i+2,0) -- (5*\i+1,0);
  
  \filldraw[black] (5*\i+0,1) circle (2pt);
  \node[left, font=\tiny] at (5*\i+0,1) {$y_{7}$};
  
  \filldraw[black] (5*\i+0,2) circle (2pt);
  \node[left, font=\tiny] at (5*\i+0,2) {$y_{8}$};
  
  \filldraw[black] (5*\i+1,0) circle (2pt);
  \node[below, font=\tiny] at (5*\i+1,0) {$y_{6}$};
  
  \filldraw[black] (5*\i+2,0) circle (2pt);
  \node[below, font=\tiny] at (5*\i+2,0) {$y_{5}$};
  
  \filldraw[black] (5*\i+3,1) circle (2pt);
  \node[right, font=\tiny] at (5*\i+3,1) {$y_{4}$};
  
  \filldraw[black] (5*\i+3,2) circle (2pt);
  \node[right, font=\tiny] at (5*\i+3,2) {$y_{3}$};
  
  \filldraw[black] (5*\i+1,3) circle (2pt);
  \node[above, font=\tiny] at (5*\i+1,3) {$y_{1}$};
  
  \filldraw[black] (5*\i+2,3) circle (2pt); 
  \node[above, font=\tiny] at (5*\i+2,3) {$y_{2}$};
  
  \filldraw[black] (5*\i+1,1) circle (2pt); 
  \filldraw[black] (5*\i+2,2) circle (2pt);
  \filldraw[black] (5*\i+1,2) circle (2pt);
  \filldraw[black] (5*\i+2,1) circle (2pt);
}

\foreach \i in {0,1,2,3} {
  \draw[-{Triangle[blue,scale=2,line width=1.5pt]}] (5*\i+3.6,1.5) -- (5*\i+4.4,1.5); 
}
  
\draw[mid arrow, thick] (1,2) -- (0,1);
\draw[mid arrow, thick] (1,1) -- (1,2);
\draw[mid arrow, thick] (0,1) -- (1,1);

\draw[mid arrow, thick] (1,3) -- (0,2);
\draw[mid arrow, thick] (1,2) -- (1,3);
\draw[mid arrow, thick] (0,2) -- (1,2);

\draw[mid arrow, thick] (2,3) -- (1,2);
\draw[mid arrow, thick] (2,2) -- (2,3);
\draw[mid arrow, thick] (1,2) -- (2,2);

\draw[mid arrow, thick] (2,1) -- (1,1);
\draw[mid arrow, thick] (1,0) -- (2,1);
\draw[mid arrow, thick] (1,1) -- (1,0);

\draw[mid arrow, thick] (3,1) -- (2,1);
\draw[mid arrow, thick] (2,0) -- (3,1);
\draw[mid arrow, thick] (2,1) -- (2,0);

\draw[mid arrow, thick] (3,2) -- (2,2);
\draw[mid arrow, thick] (2,1) -- (3,2);
\draw[mid arrow, thick] (2,2) -- (2,1);

\draw[mid arrow, thick] (5*1+2,1) -- (5*1+1,2);
\draw[mid arrow, thick] (5*1+1,2) -- (5*1+1,1);
\draw[mid arrow, thick] (5*1+1,1) -- (5*1+0,1);

\draw[mid arrow, thick] (5*1+1,3) -- (5*1+0,2);
\draw[mid arrow, thick] (5*1+1,2) -- (5*1+1,3);
\draw[mid arrow, thick] (5*1+0,2) -- (5*1+1,2);

\draw[mid arrow, thick] (5*1+2,3) -- (5*1+1,2);
\draw[mid arrow, thick] (5*1+2,2) -- (5*1+2,3);
\draw[mid arrow, thick] (5*1+1,2) -- (5*1+2,2);

\draw[mid arrow, thick] (5*1+1,1) -- (5*1+2,1);
\draw[mid arrow, thick] (5*1+0,1) -- (5*1+1,0);
\draw[mid arrow, thick] (5*1+1,0) -- (5*1+1,1);

\draw[mid arrow, thick] (5*1+3,1) -- (5*1+2,1);
\draw[mid arrow, thick] (5*1+2,0) -- (5*1+3,1);
\draw[mid arrow, thick] (5*1+2,1) -- (5*1+2,0);

\draw[mid arrow, thick] (5*1+3,2) -- (5*1+2,2);
\draw[mid arrow, thick] (5*1+2,1) -- (5*1+3,2);
\draw[mid arrow, thick] (5*1+2,2) -- (5*1+2,1);

\draw[mid arrow, thick] (5*2+1,2) -- (5*2+1,1);
\draw[mid arrow, thick] (5*2+1,1) -- (5*2+0,1);

\draw[mid arrow, thick] (5*2+1,3) -- (5*2+0,2);
\draw[mid arrow, thick] (5*2+1,2) -- (5*2+1,3);
\draw[mid arrow, thick] (5*2+0,2) -- (5*2+1,2);

\draw[mid arrow, thick] (5*2+3,2) -- (5*2+2,3);
\draw[mid arrow, thick] (5*2+2,3) -- (5*2+2,2);
\draw[mid arrow, thick] (5*2+2,2) -- (5*2+1,2);

\draw[mid arrow, thick] (5*2+1,1) -- (5*2+2,1);
\draw[mid arrow, thick] (5*2+0,1) -- (5*2+1,0);
\draw[mid arrow, thick] (5*2+1,0) -- (5*2+1,1);

\draw[mid arrow, thick] (5*2+3,1) -- (5*2+2,1);
\draw[mid arrow, thick] (5*2+2,0) -- (5*2+3,1);
\draw[mid arrow, thick] (5*2+2,1) -- (5*2+2,0);

\draw[mid arrow, thick] (5*2+2,2) -- (5*2+3,2);
\draw[mid arrow, thick] (5*2+2,1) -- (5*2+2,2);

\draw[mid arrow, thick] (5*3+2,2) -- (5*3+1,3);
\draw[mid arrow, thick] (5*3+1,1) -- (5*3+1,2);
\draw[mid arrow, thick] (5*3+1,1) -- (5*3+0,1);

\draw[mid arrow, thick] (5*3+0,2) -- (5*3+1,1);
\draw[mid arrow, thick] (5*3+1,3) -- (5*3+1,2);
\draw[mid arrow, thick] (5*3+1,2) -- (5*3+0,2);

\draw[mid arrow, thick] (5*3+3,2) -- (5*3+2,3);
\draw[mid arrow, thick] (5*3+2,3) -- (5*3+2,2);
\draw[mid arrow, thick] (5*3+1,2) -- (5*3+2,2);

\draw[mid arrow, thick] (5*3+1,1) -- (5*3+2,1);
\draw[mid arrow,thick] (5*3+0,1) -- (5*3+1,0);
\draw[mid arrow, thick] (5*3+1,0) -- (5*3+1,1);

\draw[mid arrow, thick] (5*3+3,1) -- (5*3+2,1);
\draw[mid arrow, thick] (5*3+2,0) -- (5*3+3,1);
\draw[mid arrow, thick] (5*3+2,1) -- (5*3+2,0);

\draw[mid arrow, thick] (5*3+2,2) -- (5*3+3,2);
\draw[mid arrow,thick] (5*3+2,1) -- (5*3+2,2);
\draw[mid arrow, thick] (5*3+2,2) -- (5*3+1,1);

\draw[mid arrow, thick] (5*4+2,2) -- (5*4+1,3);
\draw[mid arrow, thick] (5*4+1,1) -- (5*4+1,2);
\draw[mid arrow, thick] (5*4+1,1) -- (5*4+0,1);

\draw[mid arrow, thick] (5*4+0,2) -- (5*4+1,1);
\draw[mid arrow, thick] (5*4+1,3) -- (5*4+1,2);
\draw[mid arrow, thick] (5*4+1,2) -- (5*4+0,2);

\draw[mid arrow, thick] (5*4+3,2) -- (5*4+2,3);
\draw[mid arrow, thick] (5*4+2,3) -- (5*4+2,2);
\draw[mid arrow, thick] (5*4+1,2) -- (5*4+2,2);

\draw[mid arrow, thick] (5*4+2,1) -- (5*4+1,1);
\draw[mid arrow, thick] (5*4+0,1) -- (5*4+1,0);
\draw[mid arrow, thick] (5*4+1,0) -- (5*4+1,1);

\draw[mid arrow, thick] (5*4+2,1) -- (5*4+3,1);
\draw[mid arrow, thick] (5*4+1,1) -- (5*4+2,0);
\draw[mid arrow, thick] (5*4+2,0) -- (5*4+2,1);

\draw[mid arrow, thick] (5*4+2,2) -- (5*4+3,2);
\draw[mid arrow, thick] (5*4+2,2) -- (5*4+2,1);
\draw[mid arrow, thick] (5*4+3,1) -- (5*4+2,2);

\filldraw[blue] (5*1+1,1) circle (2pt);
\node[blue,below right, font=\tiny] at (5*1+1,1) {$x'_{2}$};
\filldraw[blue] (5*2+2,2) circle (2pt);
\node[blue,above right, font=\tiny] at (5*2+2,2) {$x'_{1}$};
\filldraw[blue] (5*3+1,2) circle (2pt);
\node[blue,above left, font=\tiny] at (5*3+1,2) {$x'_{4}$};
\filldraw[blue] (5*4+2,1) circle (2pt);
\node[blue,below right, font=\tiny] at (5*4+2,1) {$x'_{3}$};

\node[below left, font=\tiny] at (1,1) {$x_{2}$};
\node[above left, font=\tiny] at (11,1) {$x'_{2}$};
\node[below right, font=\tiny] at (16,1) {$x'_{2}$};
\node[above right, font=\tiny] at (21,1) {$x'_{2}$};

\node[below right, font=\tiny] at (1,2) {$x_{4}$};
\node[below left, font=\tiny] at (6,2) {$x_{4}$};
\node[above right, font=\tiny] at (11,2) {$x_{4}$};
\node[above left, font=\tiny] at (21,2) {$x'_{4}$};

\node[below right, font=\tiny] at (2,1) {$x_{3}$};
\node[below right, font=\tiny] at (7,1) {$x_{3}$};
\node[above right, font=\tiny] at (12,1) {$x_{3}$};
\node[above right, font=\tiny] at (17,1) {$x_{3}$};

\node[above right, font=\tiny] at (2,2) {$x_{1}$};
\node[above right, font=\tiny] at (7,2) {$x_{1}$};
\node[above right, font=\tiny] at (17,2) {$x'_{1}$};
\node[above right, font=\tiny] at (22,2) {$x'_{1}$};

\end{tikzpicture}
\caption{Mutation sequence for the flip.}
\end{figure}
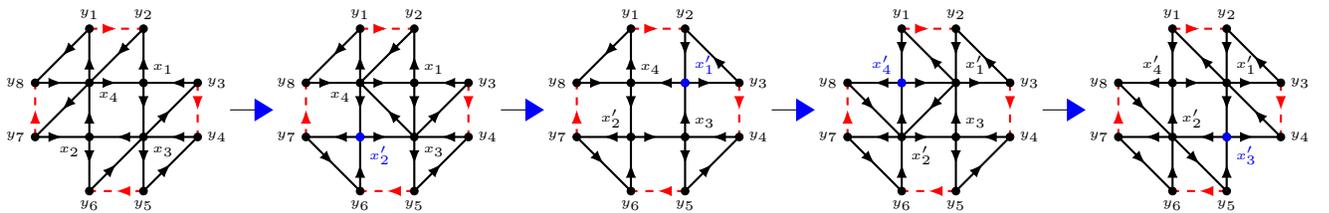
Starting with the initial cluster variables $\bx=(y_1,...,y_{8},x_1,...,x_4)$, the resulting cluster variables after the flip can be expressed in terms of the initial cluster variables in Laurent polynomials as:
\Eqn{
x'_2 &= \frac{y_7 x_3 + y_6 x_4}{x_2};\\
x'_1 &= \frac{y_3 x_4 + y_2 x_3}{x_1};\\
x'_4 &= \frac{y_8 x'_1 + y_1 x'_2}{x_4}=\frac{y_3 y_8 x_2 x_4  + y_2 y_8 x_2 x_3 + y_1 y_7 x_1 x_3 + y_1 y_6 x_1 x_4}{x_1 x_2 x_4};\\
x'_3 &= \frac{y_4 x'_2 + y_5 x'_1}{x_3}=\frac{y_4 y_7 x_1 x_3 + y_4 y_6 x_1 x_4 + y_3 y_5 x_2 x_4 + y_2 y_5 x_2 x_3}{x_1 x_2 x_3}.
}
\end{Ex}

\begin{Ex}
Consider a pentagon with a given triangulation. Each triangle is assigned the $2$-triangulated quiver, yielding the initial quiver:
\begin{figure}[H]
\centering
\begin{tikzpicture}[scale=3,
    mid arrow/.style={
        postaction={decorate},
        decoration={
            markings,
            mark=at position 0.5 with {\arrow{>}}
        }
    }
]
  \coordinate (A') at (0, 1.539);
  \coordinate (B') at (0.809, 0.951);
  \coordinate (C') at (0.5, 0);
  \coordinate (D') at (-0.5, 0);
  \coordinate (E') at (-0.809, 0.951);
  \node[above, font=\scriptsize] at (A') {$A$};
  \node[right, font=\scriptsize] at (B') {$B$};
  \node[right, font=\scriptsize] at (C') {$C$};
  \node[left, font=\scriptsize] at (D') {$D$};
  \node[left, font=\scriptsize] at (E') {$E$};
  
  \draw[thick] (A') -- (B') -- (C') -- (D') -- (E') -- cycle;
  \draw[thick] (C') -- (A') -- (D');
  
  \coordinate (A) at (2+0, 1.539);
  \coordinate (B) at (2+0.809, 0.951);
  \coordinate (C) at (2+0.5, 0);
  \coordinate (D) at (2-0.5, 0);
  \coordinate (E) at (2-0.809, 0.951);
  
  \coordinate (X1) at ($(A)!0.333!(B)$);
  \coordinate (X2) at ($(A)!0.667!(B)$);
  \coordinate (X3) at ($(B)!0.333!(C)$);
  \coordinate (X4) at ($(B)!0.667!(C)$);
  \coordinate (X5) at ($(C)!0.333!(D)$);
  \coordinate (X6) at ($(C)!0.667!(D)$);
  \coordinate (X7) at ($(D)!0.333!(E)$);
  \coordinate (X8) at ($(D)!0.667!(E)$);
  \coordinate (X9) at ($(E)!0.333!(A)$);
  \coordinate (X10) at ($(E)!0.667!(A)$);
  \coordinate (Y1) at ($(A)!0.333!(D)$);
  \coordinate (Y2) at ($(A)!0.333!(C)$);
  \coordinate (Y3) at ($(X8)!0.5!(Y1)$);
  \coordinate (Y4) at ($(X3)!0.5!(Y2)$);
  \coordinate (Y5) at ($(X5)!0.5!(Y1)$);
  \coordinate (Y6) at ($(A)!0.667!(D)$);
  \coordinate (Y7) at ($(A)!0.667!(C)$);
  
  \draw[mid arrow, thick] (X2) -- (Y4);
  \draw[mid arrow, thick] (Y4) -- (X3);
  \draw[mid arrow, thick] (X3) -- (X2);
  
  \draw[mid arrow, thick] (X1) -- (Y2);
  \draw[mid arrow, thick] (Y2) -- (Y4);
  \draw[mid arrow, thick] (Y4) -- (X1);
  
  \draw[mid arrow, thick] (Y4) -- (Y7);
  \draw[mid arrow, thick] (Y7) -- (X4);
  \draw[mid arrow, thick] (X4) -- (Y4);
  
  \draw[mid arrow, thick] (Y1) -- (Y5);
  \draw[mid arrow, thick] (Y5) -- (Y2);
  \draw[mid arrow, thick] (Y2) -- (Y1);
  
  \draw[mid arrow, thick] (Y6) -- (X6);
  \draw[mid arrow, thick] (X6) -- (Y5);
  \draw[mid arrow, thick] (Y5) -- (Y6);
  
  \draw[mid arrow, thick] (Y5) -- (X5);
  \draw[mid arrow, thick] (X5) -- (Y7);
  \draw[mid arrow, thick] (Y7) -- (Y5);

  \draw[mid arrow, thick] (X8) -- (Y3);
  \draw[mid arrow, thick] (Y3) -- (X9);
  \draw[mid arrow, thick] (X9) -- (X8);
  
  \draw[mid arrow, thick] (X7) -- (Y6);
  \draw[mid arrow, thick] (Y6) -- (Y3);
  \draw[mid arrow, thick] (Y3) -- (X7);
  
  \draw[mid arrow, thick] (Y3) -- (Y1);
  \draw[mid arrow, thick] (Y1) -- (X10);
  \draw[mid arrow, thick] (X10) -- (Y3);
  
  \draw[mid arrow, red,thick,dashed] (X1) -- (X2);
  \draw[mid arrow, red,thick,dashed] (X3) -- (X4);
  \draw[mid arrow, red,thick,dashed] (X5) -- (X6); 
  \draw[mid arrow, red,thick,dashed] (X7) -- (X8);
  \draw[mid arrow, red,thick,dashed] (X9) -- (X10);
  
  \draw[thin,dashed] (A) -- (C);
  \draw[thin,dashed] (A) -- (D);
  	\foreach \i in {5,6,7,8,9,10} { 
  \filldraw[black] (X\i) circle (1pt);
  \node[above left, font=\scriptsize] at (X\i) {$y_{\i}$};
}

  	\foreach \i in {1,2,3,5,6,7} { 
  \filldraw[black] (Y\i) circle (1pt);
  \node[above left, font=\scriptsize] at (Y\i) {$x_{\i}$};
}

	\foreach \i in {1,2,3,4} { 
  \filldraw[black] (X\i) circle (1pt);
  \node[above right, font=\scriptsize] at (X\i) {$y_{\i}$};
}

	\foreach \i in {4,...,4} { 
  \filldraw[black] (Y\i) circle (1pt);
  \node[above right, font=\scriptsize] at (Y\i) {$x_{\i}$};
}

\end{tikzpicture}
\caption{Pentagon with triangulation (left) and the initial quiver for the pentagon (right)}
\end{figure}
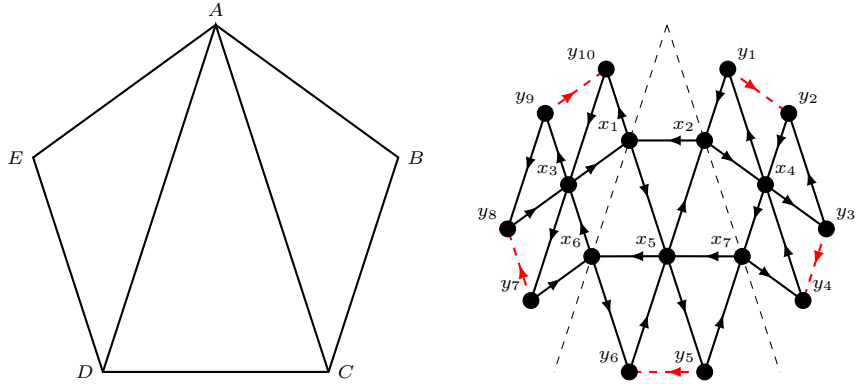
To obtain the diagonal $BE$, we consider two flip sequences:
\begin{figure}[H]
\centering
\begin{tikzpicture}[scale=1.7,
    mid arrow/.style={
        postaction={decorate},
        decoration={
            markings,
            mark=at position 0.5 with {\arrow{>}}
        }
    }
]
  \node[left, font=\scriptsize] at (-1.2, 0.7695) {$\mu_1$:};
  \node[left, font=\scriptsize] at (-1.2, 0.7695-2) {$\mu_2$:};
  
  \foreach \i in {0,1,2} {
  
  \coordinate (A\i) at (3.5*\i+0, 1.539);
  \coordinate (B\i) at (3.5*\i+0.809, 0.951);
  \coordinate (C\i) at (3.5*\i+0.5, 0);
  \coordinate (D\i) at (3.5*\i-0.5, 0);
  \coordinate (E\i) at (3.5*\i-0.809, 0.951);
  \coordinate (A'\i) at (3.5*\i+0, 1.539-2);
  \coordinate (B'\i) at (3.5*\i+0.809, 0.951-2);
  \coordinate (C'\i) at (3.5*\i+0.5, 0-2);
  \coordinate (D'\i) at (3.5*\i-0.5, 0-2);
  \coordinate (E'\i) at (3.5*\i-0.809, 0.951-2);
  
  \node[above, font=\scriptsize] at (A\i) {$A$};
  \node[right, font=\scriptsize] at (B\i) {$B$};
  \node[right, font=\scriptsize] at (C\i) {$C$};
  \node[left, font=\scriptsize] at (D\i) {$D$};
  \node[left, font=\scriptsize] at (E\i) {$E$};
  \node[above, font=\scriptsize] at (A'\i) {$A$};
  \node[right, font=\scriptsize] at (B'\i) {$B$};
  \node[right, font=\scriptsize] at (C'\i) {$C$};
  \node[left, font=\scriptsize] at (D'\i) {$D$};
  \node[left, font=\scriptsize] at (E'\i) {$E$};
}
   \foreach \i in {0,1} {
  \draw[-{Triangle[blue,scale=2,line width=1.5pt]}] (3.5*\i+1.3,0.77) -- (3.5*\i+2.3,0.77);
  \draw[-{Triangle[blue,scale=2,line width=1.5pt]}] (3.5*\i+1.3,0.77-2) -- (3.5*\i+2.3,0.77-2);
}
  \filldraw[cyan] (A0) -- (B0) -- (C0) -- (D0) -- cycle;
  \filldraw[cyan] (A1) -- (B1) -- (D1) -- (E1) -- cycle;
  \filldraw[cyan] (A'0) -- (C'0) -- (D'0) -- (E'0) -- cycle;
  \filldraw[cyan] (A'1) -- (B'1) -- (C'1) -- (E'1) -- cycle; 
  
  \draw[thick] (A0) -- (C0);
  \draw[thick] (A0) -- (D0);
  \draw[thick] (A1) -- (D1);
  \draw[thick] (B1) -- (D1);
  \draw[red,thick] (B2) -- (E2);
  \draw[thick] (B2) -- (D2);
  
  \draw[thick] (A'0) -- (C'0);
  \draw[thick] (A'0) -- (D'0);
  \draw[thick] (A'1) -- (C'1);
  \draw[thick] (C'1) -- (E'1);
  \draw[red,thick] (B'2) -- (E'2);
  \draw[thick] (C'2) -- (E'2);
  
  \foreach \i in {0,1,2} {
  
  \draw[thick] (A\i) -- (B\i) -- (C\i) -- (D\i) -- (E\i) -- cycle;
  \draw[thick] (A'\i) -- (B'\i) -- (C'\i) -- (D'\i) -- (E'\i) -- cycle;
}
\end{tikzpicture}
\caption{Both possible flip sequences}
\label{fig:both_flip}
\end{figure}
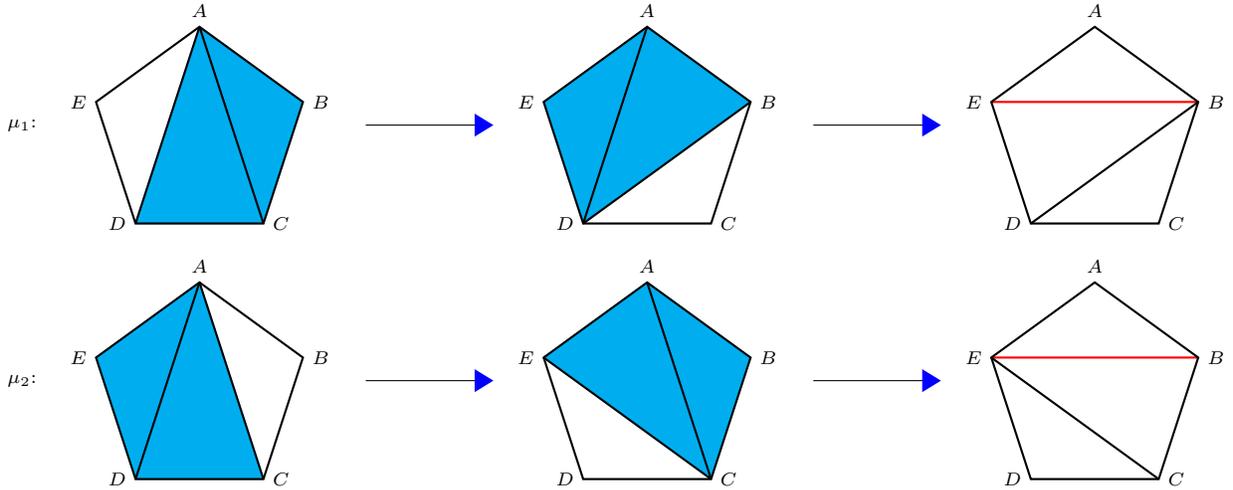
Note that in the $2$-triangulated construction, we do not assume that the main diagonal is equally divided into 3 parts. The mutation sequences for both flip paths are as follow (the order of each flip sequence is not necessary unique):
\begin{gather*}
	\mu_1 = \{x_2 \rightarrow x_7 \rightarrow x_4 \rightarrow x_5 \rightarrow x_1 \rightarrow x_6 \rightarrow x_3 \rightarrow x_2 \}; 
	\\
	\mu_2 = \{x_1 \rightarrow x_6 \rightarrow x_3 \rightarrow x_5 \rightarrow x_2 \rightarrow x_7 \rightarrow x_4 \rightarrow x_1 \}.
\end{gather*}
Below we will verify that the cluster expansion of the resulting cluster variables corresponding to both mutation sequences coincides up to reindexing of the variables.

For the first flip sequence $\mu_1$ (top of Figure~\ref{fig:both_flip}), we calculate:
{\small\begin{align*}
    x'_{2} = \frac{x_1x_4+y_1x_5}{x_2};\quad&  x'_{7} = \frac{y_4x_5+y_5x_4}{x_7};\\
    x'_{4} = \frac{y_3x'_2+y_2x'_7}{x_4} =&\ \frac{y_3x_1x_4x_7+y_3y_1x_5x_7+y_2y_4x_2x_5+y_2y_5x_2x_4}{x_2x_4x_7};\\
    x'_{5} = \frac{x_6x'_7+y_6x'_2}{x_5} =&\ \frac{y_4x_2x_5x_6+y_5x_2x_4x_6+y_6x_1x_4x_7+y_1y_6x_5x_7}{x_2x_5x_7};\\
    x'_{1} = \frac{y_1x_3+y_{10}x'_2}{x_1} = &\ \frac{y_1x_2x_3+y_{10}x_1x_4+y_1y_{10}x_5}{x_1x_2};\\
    x'_{6} = \frac{x_3x'_5+y_{7}x'_2}{x_6} = &\ \frac{y_4x_2x_3x_5x_6+y_{5}x_2x_3x_4x_6+y_6x_1x_3x_4x_7+y_1y_6x_3x_5x_7+y_7x_1x_4x_5x_7+y_1y_7x^2_5x_7}{x_2x_5x_6x_7};\\
    x'_{3} = \frac{y_8x'_1+y_{9}x'_6}{x_3} =&\ \frac{y_1y_8x_2x_3x_5x_6x_7 + y_8y_{10}x_1x_4x_5x_6x_7 + y_1y_8y_{10}x^2_5x_6x_7 + y_4y_9x_1x_2x_3x_5x_6}{x_1x_2x_3x_5x_6x_7} \\
            &+ \frac{y_5y_9x_1x_2x_3x_4x_6 + y_6y_{9}x^2_1x_3x_4x_7 + y_1y_6y_9x_1x_3x_5x_7 + y_7y_9x^2_1x_4x_5x_7 + y_1y_7y_9x_1x^2_5x_7}{x_1x_2x_3x_5x_6x_7};\\
    x''_{2} = \frac{y_2x'_6+x'_{4}x'_1}{x'_2} =&\ \frac{y_2y_4x^2_2x_3x_5x_6 + y_2y_5x^2_2x_3x_4x_6 + y_2y_6x_1x_2x_3x_4x_7 + y_2y_7x_1x_2x_4x_5x_7 + y_3y_{10}x_1x_4x_5x_6x_7}{x_1x_2x_4x_5x_6x_7} \\
            &+ \frac{y_1y_3y_{10}x^2_5x_6x_7 + y_2y_4y_{10}x_2x^2_5x_6 + y_2y_5y_{10}x_2x_4x_5x_6 + y_1y_3x_2x_3x_5x_6x_7}{x_1x_2x_4x_5x_6x_7}.
\end{align*}
}
For the second flip sequence $\mu_2$ (bottom of Figure~\ref{fig:both_flip}), we calculate:
{\small\begin{align*}
    x'_{1} = \frac{x_2x_3+y_{10}x_5}{x_1};\quad&  x'_{6} = \frac{y_7x_5+y_6x_3}{x_6};\\
    x'_{3} = \frac{y_8x'_1+y_9x'_6}{x_3} =&\ \frac{y_8x_2x_3x_6+y_8y_{10}x_5x_6+y_7y_9x_1x_5+y_6y_9x_1x_3}{x_1x_3x_6};\\
    x'_{5} = \frac{x_7x'_6+y_5x'_1}{x_5} =&\ \frac{y_7x_1x_5x_7 + y_6x_1x_3x_7 + y_5x_2x_3x_6 + y_5y_{10}x_5x_6}{x_1x_5x_6};\\
    x'_{2} = \frac{y_{10}x_4+y_{1}x'_1}{x_2} =&\ \frac{y_{10}x_1x_4 + y_1x_2x_3 + y_1y_{10}x_5}{x_1x_2};\\          
    x'_{7} = \frac{x_4x'_5+y_{4}x'_1}{x_7} =&\ \frac{y_7x_1x_4x_5x_7 + y_6x_1x_3x_4x_7 + y_5x_2x_3x_4x_6 + y_5y_{10}x_4x_5x_6 + y_4x_2x_3x_5x_6 + y_4y_{10}x^2_5x_6}{x_1x_5x_6x_7};\\         
    x'_{4} = \frac{y_3x'_2+y_{2}x'_7}{x_4} =&\ \frac{y_3y_{10}x_1x_4x_5x_6x_7 + y_1y_3x_2x_3x_5x_6x_7 + y_1y_3y_{10}x^2_5x_6x_7 + y_2y_7x_1x_2x_4x_5x_7}{x_1x_2x_4x_5x_6x_7} \\
            &+ \frac{y_2y_6x_1x_2x_3x_4x_7 + y_2y_5x^2_2x_3x_4x_6 + y_2y_5y_{10}x_2x_4x_5x_6 + y_2y_4x^2_2x_3x_5x_6 + y_2y_4y_{10}x_2x^2_5x_6}{x_1x_2x_4x_5x_6x_7};\\            
    x''_{1} = \frac{y_9x'_7+x'_{3}x'_2}{x'_1} =&\ \frac{y_5y_9x_1x_2x_3x_4x_6 + y_7y_9x^2_1x_4x_5x_7 + y_6y_9x^2_1x_3x_4x_7 + y_4y_9x_1x_2x_3x_5x_6 + y_8y_{10}x_1x_4x_5x_6x_7}{x_1x_2x_3x_5x_6x_7} \\
            &+ \frac{y_1y_8x_2x_3x_5x_6x_7 + y_1y_8y_{10}x^2_5x_6x_7 + y_1y_7y_9x_1x^2_5x_7 + y_1y_6y_9x_1x_3x_5x_7}{x_1x_2x_3x_5x_6x_7}.
\end{align*}}

For the first sequence, the dotted diagonal corresponds to expansion variables $x''_2$ and $x'_3$; for the second sequence, it corresponds to $x''_1$ and $x'_4$. Our calculations show that $x'_3 = x''_1$ and $x''_2 = x'_4$, confirming that both sequences yield the same result up to vertex relabeling.
\end{Ex}

\begin{Ex}
For $n=3$, each edge of a triangle is divided into 4 equal parts using 3 vertices. The construction follows similarly:
\begin{figure}[H]
\centering
\begin{tikzpicture}[scale=1,
    mid arrow/.style={
        postaction={decorate},
        decoration={
            markings,
            mark=at position 0.5 with {\arrow{>}}
        }
    }
]  
  \draw[thick] (6,0) -- (7,4) -- (11,0) -- cycle;
  
  \filldraw[gray!30] (6.75,3) -- (8,3) -- (7.75,2) -- cycle;
  \filldraw[gray!30] (6.5,2) -- (7.75,2) -- (7.5,1) -- cycle;
  \filldraw[gray!30] (7.75,2) -- (9,2) -- (8.75,1) -- cycle;
  \filldraw[gray!30] (6.25,1) -- (7.5,1) -- (7.25,0) -- cycle;
  \filldraw[gray!30] (7.5,1) -- (8.75,1) -- (8.5,0) -- cycle;
  \filldraw[gray!30] (8.75,1) -- (10,1) -- (9.75,0) -- cycle;
  
  \filldraw[black] (6.25,1) circle (2pt);
  \filldraw[black] (6.5,2) circle (2pt);
  \filldraw[black] (6.75,3) circle (2pt);

  \filldraw[black] (8,3) circle (2pt);
  \filldraw[black] (9,2) circle (2pt);
  \filldraw[black] (10,1) circle (2pt);
  
  \filldraw[black] (7.25,0) circle (2pt);
  \filldraw[black] (8.5,0) circle (2pt);
  \filldraw[black] (9.75,0) circle (2pt);
  
  \filldraw[black] (7.75,2) circle (2pt);
  \filldraw[black] (7.5,1) circle (2pt);
  \filldraw[black] (8.75,1) circle (2pt);
  
  
  \draw[mid arrow, red, thick, dashed] (12.25,1) -- (12.5,2);
  \draw[mid arrow, red, thick, dashed] (12.5,2) -- (12.75,3);
  \draw[mid arrow, red, thick, dashed] (14,3) -- (15,2);
  \draw[mid arrow, red, thick, dashed] (15,2) -- (16,1);
  \draw[mid arrow, red, thick, dashed] (15.75,0) -- (14.5,0);
  \draw[mid arrow, red, thick, dashed] (14.5,0) -- (13.25,0);
  
  \draw[mid arrow, thick] (14,3) -- (12.75,3);
  \draw[mid arrow, thick] (15,2) -- (13.75,2);
  \draw[mid arrow, thick] (13.75,2) -- (12.5,2);
  \draw[mid arrow, thick] (16,1) -- (14.75,1);
  \draw[mid arrow, thick] (14.75,1) -- (13.5,1);
  \draw[mid arrow, thick] (13.5,1) -- (12.25,1);
  
  \draw[mid arrow, thick] (12.25,1) -- (13.25,0);
  \draw[mid arrow, thick] (12.5,2) -- (13.5,1);
  \draw[mid arrow, thick] (13.5,1) -- (14.5,0);
  \draw[mid arrow, thick] (12.75,3) -- (13.75,2);
  \draw[mid arrow, thick] (13.75,2) -- (14.75,1);
  \draw[mid arrow, thick] (14.75,1) -- (15.75,0);
  
  \draw[mid arrow, thick] (13.25,0) -- (13.5,1);
  \draw[mid arrow, thick] (13.5,1) -- (13.75,2);
  \draw[mid arrow, thick] (13.75,2) -- (14,3);
  \draw[mid arrow, thick] (14.5,0) -- (14.75,1);
  \draw[mid arrow, thick] (14.75,1) -- (15,2);
  \draw[mid arrow, thick] (15.75,0) -- (16,1);
  
  \filldraw[black] (12.25,1) circle (2pt);
  \filldraw[black] (12.5,2) circle (2pt);
  \filldraw[black] (12.75,3) circle (2pt);

  \filldraw[black] (14,3) circle (2pt);
  \filldraw[black] (15,2) circle (2pt);
  \filldraw[black] (16,1) circle (2pt);
  
  \filldraw[black] (13.25,0) circle (2pt);
  \filldraw[black] (14.5,0) circle (2pt);
  \filldraw[black] (15.75,0) circle (2pt);
  
  \filldraw[black] (13.75,2) circle (2pt);
  \filldraw[black] (13.5,1) circle (2pt);
  \filldraw[black] (14.75,1) circle (2pt);  
\end{tikzpicture}
\caption{A triangle and its corresponding $3$-triangulated triangle}
\end{figure}
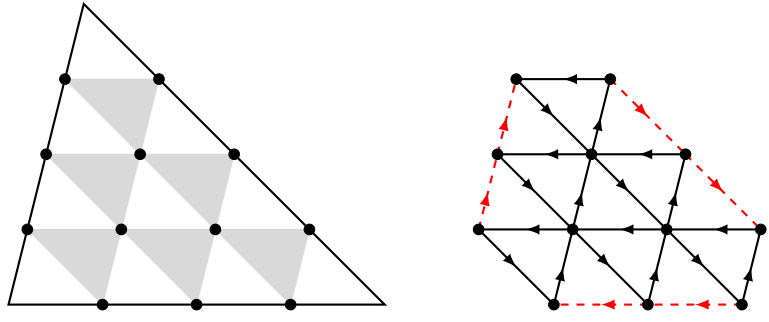

Consider the $3$-triangulated quiver of a quadrilateral, labeled as in the left of Figure \ref{3tq}.
\begin{figure}[H]
\centering
\begin{tikzpicture}[scale=1,
    mid arrow/.style={
        postaction={decorate},
        decoration={
            markings,
            mark=at position 0.5 with {\arrow{>}}
        }
    }
]
  \draw[thin,dashed] (0,0) -- (4,4);
  \draw[thin,dashed] (5+2,4) -- (9+2,0);
  \draw[-{Triangle[blue,scale=2,line width=1.5pt]}] (4.8,2) -- (6.2,2);
  \node[above] at (5.4,2) {$\mu$};
  
  \draw[mid arrow, red,thick,dashed] (0,1) -- (0,2);
  \draw[mid arrow, red,thick,dashed] (0,2) -- (0,3);
  \draw[mid arrow, red,thick,dashed] (1,4) -- (2,4); 
  \draw[mid arrow, red,thick,dashed] (2,4) -- (3,4);
  \draw[mid arrow, red,thick,dashed] (4,3) -- (4,2);
  \draw[mid arrow, red,thick,dashed] (4,2) -- (4,1);
  \draw[mid arrow, red,thick,dashed] (3,0) -- (2,0); 
  \draw[mid arrow, red,thick,dashed] (2,0) -- (1,0); 
 
  \draw[mid arrow, red,thick,dashed] (5+2,1) -- (5+2,2);
  \draw[mid arrow, red,thick,dashed] (5+2,2) -- (5+2,3);
  \draw[mid arrow, red,thick,dashed] (6+2,4) -- (7+2,4); 
  \draw[mid arrow, red,thick,dashed] (7+2,4) -- (8+2,4);
  \draw[mid arrow, red,thick,dashed] (9+2,3) -- (9+2,2);
  \draw[mid arrow, red,thick,dashed] (9+2,2) -- (9+2,1);
  \draw[mid arrow, red,thick,dashed] (8+2,0) -- (7+2,0); 
  \draw[mid arrow, red,thick,dashed] (7+2,0) -- (6+2,0);
  
  \foreach \i in {1,2,3} {
    \foreach \j in {0,...,\the\numexpr\i-1\relax} {
      \draw[mid arrow, thick] (\j,\i) -- (\j+1,\i);
    }
  }
 
  \foreach \i in {1,2,3} {
    \foreach \j in {\i,...,3} {
      \draw[mid arrow, thick] (\j+1,\i) -- (\j,\i);
    }
  }  
  
  \draw[mid arrow, thick] (1,4) -- (0,3);
  \draw[mid arrow, thick] (2,4) -- (1,3);
  \draw[mid arrow, thick] (1,3) -- (0,2);
  \draw[mid arrow, thick] (3,4) -- (2,3);
  \draw[mid arrow, thick] (2,3) -- (1,2);
  \draw[mid arrow, thick] (1,2) -- (0,1);

  \foreach \i in {1,2,3} {
    \foreach \j in {\i,...,3} {
      \draw[mid arrow, thick] (\j-\i+6+2,\i) -- (\j-\i+5+2,\i);
    }
  }
 
  \foreach \i in {1,2,3} {
    \foreach \j in {0,...,\the\numexpr\i-1\relax} {
      \draw[mid arrow, thick] (\j-\i+9+2,\i) -- (\j-\i+10+2,\i);
    }
  }
  
  \foreach \i in {1,2,3} {
    \foreach \j in {\i,...,3} {
      \draw[mid arrow, thick] (\i,\j) -- (\i,\j+1);
    }
  }
  
  \foreach \i in {1,2,3} {
    \foreach \j in {0,...,\the\numexpr\i-1\relax} {
      \draw[mid arrow, thick] (\i,\j+1) -- (\i,\j);
    }
  }
  
  \foreach \i in {1,2,3} {
    \foreach \j in {\i,...,3} {
      \draw[mid arrow, thick] (5+\i+2,3-\j) -- (5+\i+2,4-\j);
    }
  }
  
  \foreach \i in {1,2,3} {
    \foreach \j in {0,...,\the\numexpr\i-1\relax} {
      \draw[mid arrow, thick] (5+\i+2,4-\j) -- (5+\i+2,3-\j);
    }
  }
  
  \foreach \i in {1,2,3} {
    \foreach \j in {\i,...,3} {
      \draw[mid arrow, thick] (\j,\i-1) -- (\j+1,\i);
    }
  }  

  \foreach \i in {1,2,3} {
    \foreach \j in {\i,...,3} {
      \draw[mid arrow, thick] (\j-\i+5+2,\i) -- (\j-\i+6+2,\i-1);
    }
  }
  
  \foreach \i in {1,2,3} {
    \foreach \j in {0,...,\the\numexpr\i-1\relax} {
      \draw[mid arrow, thick] (9-\j+2,\i) -- (8-\j+2,\i+1);
    }
  }
        
  \filldraw[black] (0,1) circle (2pt);
  \node[above left, font=\scriptsize] at (0,1) {$y_{10}$};
  
  \filldraw[black] (0,2) circle (2pt);
  \node[above left, font=\scriptsize] at (0,2) {$y_{11}$};
  
  \filldraw[black] (0,3) circle (2pt);
  \node[above left, font=\scriptsize] at (0,3) {$y_{12}$};
  
  \filldraw[black] (1,0) circle (2pt);
  \node[above left, font=\scriptsize] at (1,0) {$y_9$};
  
  \filldraw[black] (2,0) circle (2pt);
  \node[above left, font=\scriptsize] at (2,0) {$y_8$};
  
  \filldraw[black] (3,0) circle (2pt);
  \node[above left, font=\scriptsize] at (3,0) {$y_7$};
  
  \filldraw[black] (4,1) circle (2pt);
  \node[above left, font=\scriptsize] at (4,1) {$y_6$};
  
  \filldraw[black] (4,2) circle (2pt); 
  \node[above left, font=\scriptsize] at (4,2) {$y_5$};
  
  \filldraw[black] (4,3) circle (2pt);
  \node[above left, font=\scriptsize] at (4,3) {$y_4$};
  
  \filldraw[black] (1,4) circle (2pt);
  \node[above left, font=\scriptsize] at (1,4) {$y_1$};
  
  \filldraw[black] (2,4) circle (2pt);
  \node[above left, font=\scriptsize] at (2,4) {$y_2$};
  
  \filldraw[black] (3,4) circle (2pt);
  \node[above left, font=\scriptsize] at (3,4) {$y_3$};
  
  \filldraw[black] (1,1) circle (2pt);
  \node[above left, font=\scriptsize] at (1,1) {$x_3$};
  
  \filldraw[black] (2,1) circle (2pt);
  \node[above left, font=\scriptsize] at (2,1) {$x_8$};
  
  \filldraw[black] (3,1) circle (2pt);
  \node[above left, font=\scriptsize] at (3,1) {$x_4$};
  
  \filldraw[black] (1,2) circle (2pt);
  \node[above left, font=\scriptsize] at (1,2) {$x_9$};
  
  \filldraw[black] (2,2) circle (2pt); 
  \node[above left, font=\scriptsize] at (2,2) {$x_2$};
  
  \filldraw[black] (3,2) circle (2pt);
  \node[above left, font=\scriptsize] at (3,2) {$x_7$};
  
  \filldraw[black] (1,3) circle (2pt);
  \node[above left, font=\scriptsize] at (1,3) {$x_5$};
  
  \filldraw[black] (2,3) circle (2pt);
  \node[above left, font=\scriptsize] at (2,3) {$x_6$};
  
  \filldraw[black] (3,3) circle (2pt);
  \node[above left, font=\scriptsize] at (3,3) {$x_1$};

  \filldraw[black] (5+2,1) circle (2pt);
  \node[above right, font=\scriptsize] at (5+2,1) {$y_{10}$};
  
  \filldraw[black] (5+2,2) circle (2pt);
  \node[above right, font=\scriptsize] at (5+2,2) {$y_{11}$};
  
  \filldraw[black] (5+2,3) circle (2pt);
  \node[above right, font=\scriptsize] at (5+2,3) {$y_{12}$};
  
  \filldraw[black] (6+2,0) circle (2pt);
  \node[above right, font=\scriptsize] at (6+2,0) {$y_9$};
  
  \filldraw[black] (7+2,0) circle (2pt);
  \node[above right, font=\scriptsize] at (7+2,0) {$y_8$};
  
  \filldraw[black] (8+2,0) circle (2pt);
  \node[above right, font=\scriptsize] at (8+2,0) {$y_7$};
  
  \filldraw[black] (9+2,1) circle (2pt);
  \node[above right, font=\scriptsize] at (9+2,1) {$y_6$};
  
  \filldraw[black] (9+2,2) circle (2pt); 
  \node[above right, font=\scriptsize] at (9+2,2) {$y_5$};
  
  \filldraw[black] (9+2,3) circle (2pt);
  \node[above right, font=\scriptsize] at (9+2,3) {$y_4$};
  
  \filldraw[black] (6+2,4) circle (2pt);
  \node[above right, font=\scriptsize] at (6+2,4) {$y_1$};
  
  \filldraw[black] (7+2,4) circle (2pt);
  \node[above right, font=\scriptsize] at (7+2,4) {$y_2$};
  
  \filldraw[black] (8+2,4) circle (2pt);
  \node[above right, font=\scriptsize] at (8+2,4) {$y_3$};
  
  \filldraw[black] (6+2,1) circle (2pt);
  \node[above right, font=\scriptsize] at (6+2,1) {$x_3$};
  
  \filldraw[black] (7+2,1) circle (2pt);
  \node[above right, font=\scriptsize] at (7+2,1) {$x_8$};
  
  \filldraw[black] (8+2,1) circle (2pt);
  \node[above right, font=\scriptsize] at (8+2,1) {$x_4$};
  
  \filldraw[black] (6+2,2) circle (2pt);
  \node[above right, font=\scriptsize] at (6+2,2) {$x_9$};
  
  \filldraw[black] (7+2,2) circle (2pt); 
  \node[above right, font=\scriptsize] at (7+2,2) {$x_2$};
  
  \filldraw[black] (8+2,2) circle (2pt);
  \node[above right, font=\scriptsize] at (8+2,2) {$x_7$};
  
  \filldraw[black] (6+2,3) circle (2pt);
  \node[above right, font=\scriptsize] at (6+2,3) {$x_5$};
  
  \filldraw[black] (7+2,3) circle (2pt);
  \node[above right, font=\scriptsize] at (7+2,3) {$x_6$};
  
  \filldraw[black] (8+2,3) circle (2pt);
  \node[above right, font=\scriptsize] at (8+2,3) {$x_1$};
\end{tikzpicture} 
\caption{Quivers before and after flipping a quadrilateral}
\label{3tq}
\end{figure}
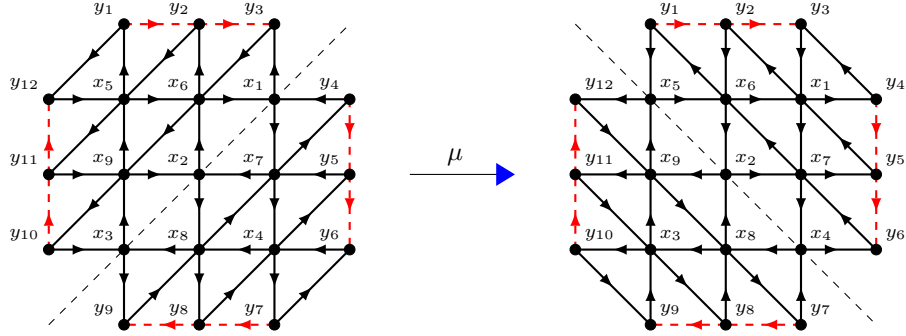
The flip of the quiver corresponding to the flip of diagonal of the quadrilateral can be performed in 3 mutation steps:
\begin{enumerate}[label=Step \arabic*:]
\item Mutate at $x_1$, $x_2$, $x_3$;
\item Mutate at $x_6$, $x_7$, $x_8$, $x_9$;  
\item Mutate at $x_2$, $x_4$, $x_5$.
\end{enumerate}
Without loss of generality, we can consider the following flip sequence:
\begin{gather*}
	\mu = \{x_3 \rightarrow x_2 \rightarrow x_1 \rightarrow x_9 \rightarrow x_8 \rightarrow x_6 \rightarrow x_7 \rightarrow x_5 \rightarrow x_4 \rightarrow x_2 \}.
\end{gather*}

The resulting cluster variables are then given by:
{\small
\begin{align*}
    x'_3 &= \frac{y_{10} x_8 + y_9 x_9}{x_3}; \quad 
    x'_2 = \frac{x_9 x_7 + x_6 x_8}{x_2}; \quad 
    x'_1 = \frac{x_6 y_4 + x_7 y_3}{x_1};\\
    x'_9 &= \frac{y_{11} x'_2 + x_5 x'_3}{x_9} = \frac{x_3 x_9 x_7 y_{11} + x_3 x_8 x_6 y_{11} + x_2 x_8 x_5 y_{10} + x_2 x_9 x_5 y_9}{x_3 x_2 x_9};\\
    x'_8 &= \frac{y_8 x'_2 + x_4 x'_3}{x_8} = \frac{x_3 x_9 x_7 y_8 + x_3 x_8 x_6 y_8 + x_2 x_8 x_4 y_{10} + x_2 x_9 x_4 y_9}{x_3 x_2 x_8};\\
    x'_6 &= \frac{x_5 x'_1 + y_2 x'_2}{x_6} = \frac{x_2 x_6 x_5 y_4 + x_2 x_7 x_5 y_3 + x_1 x_9 x_7 y_2 + x_1 x_8 x_6 y_2}{x_2 x_1 x_6};\\
    x'_7 &= \frac{x_4 x'_1 + y_5 x'_2}{x_7} = \frac{x_2 x_6 x_4 y_4 + x_2 x_7 x_4 y_3 + x_1 x_9 x_7 y_5 + x_1 x_8 x_6 y_5}{x_2 x_1 x_7};
\end{align*}

\begin{align*}
    x'_5 = \frac{y_{12} x'_6 + y_1 x'_9}{x_5} =&\ \frac{x_3 x_1 x_8 x_6^2 y_1 y_{11} + x_3 x_1 x_9^2 x_7 y_2 y_{12} + x_3 x_1 x_9 x_6 x_7 y_1 y_{11} + x_3 x_1 x_9 x_8 x_6 y_2 y_{12}}{x_3 x_2 x_1 x_9 x_6 x_5} \\
         &+ \frac{x_2 x_1 x_8 x_6 x_5 y_1 y_{10} + x_2 x_1 x_9 x_6 x_5 y_1 y_9 + x_3 x_2 x_9 x_6 x_5 y_4 y_{12} + x_3 x_2 x_9 x_7 x_5 y_3 y_{12}}{x_3 x_2 x_1 x_9 x_6 x_5};\\
    x'_4 = \frac{y_6 x'_8 + y_7 x'_7}{x_4} =&\ \frac{x_3 x_1 x_9 x_7^2 y_8 y_6 + x_3 x_1 x_8^2 x_6 y_7 y_5 + x_3 x_1 x_8 x_9 x_7 y_7 y_5 + x_3 x_1 x_8 x_7 x_6 y_8 y_6}{x_3 x_2 x_1 x_8 x_7 x_4} \\
         &+ \frac{x_2 x_1 x_8 x_7 x_4 y_{10} y_6 + x_2 x_1 x_9 x_7 x_4 y_9 y_6 + x_3 x_2 x_8 x_6 x_4 y_7 y_4 + x_3 x_2 x_8 x_7 x_4 y_7 y_3}{x_3 x_2 x_1 x_8 x_7 x_4};
\end{align*}

\begin{align*}
    x''_2 = \frac{x'_9 x'_7 + x'_6 x'_8}{x'_2} =&\ \frac{x_3 x_1 x_8 x_6 (x_9 x_7 + x_6 x_8) y_{11} y_5 + x_3 x_2 x_8 x_6^2 x_4 y_4 y_{11} + x_3 x_2 x_8 x_6 x_7 x_4 y_{11} y_3}{x_3 x_2 x_1 x_9 x_8 x_6 x_7} \\
          &+ \frac{x_2 x_1 x_8^2 x_6 x_5 y_{10} y_5 + x_2^2 x_8 x_6 x_5 x_4 y_{10} y_4 + x_2^2 x_8 x_7 x_5 x_4 y_{10} y_3 + x_2 x_1 x_9 x_8 x_6 x_5 y_5 y_9}{x_3 x_2 x_1 x_9 x_8 x_6 x_7} \\
          &+ \frac{x_2^2 x_9 x_6 x_5 x_4 y_4 y_9 + x_2^2 x_9 x_7 x_5 x_4 y_3 y_9 + x_3 x_2 x_9 x_6 x_7 x_5 y_4 y_8 + x_3 x_2 x_9 x_7^2 x_5 y_8 y_3}{x_3 x_2 x_1 x_9 x_8 x_6 x_7} \\
          &+ \frac{x_3 x_1 x_9 x_7 (x_9 x_7 + x_6 x_8) y_2 y_8 + x_2 x_1 x_9 x_8 x_7 x_4 y_{10} y_2 + x_2 x_1 x_9^2 x_7 x_4 y_2 y_9}{x_3 x_2 x_1 x_9 x_8 x_6 x_7}.
\end{align*}
}
\end{Ex}

\section{Cluster expansion formulas for general flips}
\label{sec:type_an_m_gon}
Recall the moduli spaces $\mathscr{A}_{\text{SL}_{n+1},\bbS}$ of $G$-local system of type $A_n$. In this section, we first derive the general expansion formula corresponding to the cluster mutation sequences associated with a flip of diagonal when $\bbS$ is a quadrilateral. Based on the formula, we will derive the recurrence relation in the case of general $n$-triangulated $m$-gon in the next section.

As established in Definition~\ref{subsec:flip}, the flip of a quadrilateral corresponds to a specific sequence of mutations in the associated quiver. We consider the moduli space $\mathscr{A}_{\text{SL}_{n+1},\bbS}$ where gluing maps are applied to pairs of triangles such that the diagonal of the resulting quadrilateral is formed by two corresponding colored boundary intervals with opposite orientations \cite{GS19, GS24}. We will consider the unpunctured case as assumed, and give some remarks on the punctured surface cases in Section \ref{subsec:punctured_case}.

\subsection{Unpunctured surface case}
\label{subsec:unpunctured_case}
 Let $\bbS$ be a quadrilateral with vertices $V_1$, $V_2$, $V_3$, $V_4$ labeled in clockwise order, with $V_1V_3$ as the initial diagonal. These vertices correspond to flags as defined in Section~\ref{subsec:modspsurface}. For convenience, we treat the quadrilateral $\bbS$ as a square. The initial cluster seed consists of an $(n^2+4n)$-tuple of cluster variables, indexed by the \emph{boundary variables} $y_i$ for $i = 1, 2, \dots, 4n$, and \emph{interior variables} $x_j$ for $j = 1, 2, \dots, n^2$ assigned to the vertices of the $n$-triangulation. The labeling $(x_1,...,x_{n^2},y_1,...,y_{4n})$ of the cluster seed follows the convention below:
\begin{enumerate}
    \item[(a)] \emph{Boundary edges:}
    \begin{itemize}
        \item Edge $V_4V_1$: $(y_1, y_2, \dots, y_n)$
        \item Edge $V_iV_{i+1}$ for $i = 1, 2, 3$: (ordered from $V_i$ to $V_{i+1}$) $(y_{in+1}, y_{in+2}, \dots, y_{(i+1)n})$ 
    \end{itemize}
    
    \item[(b)] \emph{Diagonal variables:}
    \begin{itemize}
        \item Initial diagonal $V_1V_3$: (ordered from $V_1$ to $V_3$) $(x_1, x_2, \dots, x_n)$
        \item Diagonal $V_2V_4$: (ordered from $V_2$ to $V_4$)
        		\Eq{
        			\begin{cases}
    (x_{n+1}, x_{n+2}, \dots, x_{2n}) & \text{if } n \text{ is even}, \\
    (x_{n+1}, \dots, x_{\frac{3n-1}{2}}, x_{\frac{n+1}{2}}, x_{\frac{3n-1}{2}}, \dots, x_{2n-1}) & \text{if } n \text{ is odd}.
\end{cases}
        		}
    \end{itemize}
    
    \item[(c)] \emph{Interior variables:} The remaining $x_j$ variables are labeled by:
    \begin{enumerate}[label=(\roman*)]
        \item Considering concentric layers of vertices inside the square (excluding the four edges);
        \item Starting from the outermost layer containing $n^2$ vertices;
        \item Labeling vertices in clockwise order, beginning from the edge parallel to $V_4V_1$ that is closest to $V_4V_1$;
        \item Proceeding inward to subsequent layers, and repeating the labeling procedure.
    \end{enumerate}
\end{enumerate}
\begin{figure}[H]
\centering
\begin{tikzpicture}[scale=1,
    mid arrow/.style={
        postaction={decorate},
        decoration={
            markings,
            mark=at position 0.5 with {\arrow{>}}
        }
    }
]
  \coordinate (A3) at (0,0);
  \coordinate (A4) at (0,4);
  \coordinate (A1) at (4,4);
  \coordinate (A2) at (4,0);
  \coordinate (B3) at (6,0);
  \coordinate (B4) at (6,5);
  \coordinate (B1) at (11,5);
  \coordinate (B2) at (11,0);
  \node[above, font=\scriptsize] at (A1) {$V_1$};
  \node[below, font=\scriptsize] at (A2) {$V_2$};
  \node[below, font=\scriptsize] at (A3) {$V_3$};
  \node[above, font=\scriptsize] at (A4) {$V_4$};
  \node[above, font=\scriptsize] at (B1) {$V_1$};
  \node[below, font=\scriptsize] at (B2) {$V_2$};
  \node[below, font=\scriptsize] at (B3) {$V_3$};
  \node[above, font=\scriptsize] at (B4) {$V_4$};
        
  \draw[thick] (A1) -- (A2) -- (A3) -- (A4) -- cycle;
  \draw[thick] (B1) -- (B2) -- (B3) -- (B4) -- cycle;
  \draw[thick] (A1) -- (A3);
  \draw[thin, dashed] (A2) -- (A4);
  \draw[thick] (B1) -- (B3);
  \draw[thin, dashed] (B2) -- (B4);
  \draw[green, thin, dashed] (5,0) -- (5,5);
  
  \filldraw[blue] (1,4) circle (2pt);
  \node[above, font=\scriptsize] at (1,4) {$y_{1}$};
  \filldraw[blue] (2,4) circle (2pt);
  \node[above, font=\scriptsize] at (2,4) {$y_{2}$}; 
  \filldraw[blue] (3,4) circle (2pt);
  \node[above, font=\scriptsize] at (3,4) {$y_{3}$};
  \filldraw[blue] (4,3) circle (2pt);
  \node[right, font=\scriptsize] at (4,3) {$y_{4}$};
  \filldraw[blue] (4,2) circle (2pt);
  \node[right, font=\scriptsize] at (4,2) {$y_{5}$}; 
  \filldraw[blue] (4,1) circle (2pt);
  \node[right, font=\scriptsize] at (4,1) {$y_{6}$};
  \filldraw[blue] (3,0) circle (2pt);
  \node[below, font=\scriptsize] at (3,0) {$y_{7}$};
  \filldraw[blue] (2,0) circle (2pt);
  \node[below, font=\scriptsize] at (2,0) {$y_{8}$}; 
  \filldraw[blue] (1,0) circle (2pt);
  \node[below, font=\scriptsize] at (1,0) {$y_{9}$};
  \filldraw[blue] (0,1) circle (2pt);
  \node[left, font=\scriptsize] at (0,1) {$y_{10}$};
  \filldraw[blue] (0,2) circle (2pt);
  \node[left, font=\scriptsize] at (0,2) {$y_{11}$}; 
  \filldraw[blue] (0,3) circle (2pt);
  \node[left, font=\scriptsize] at (0,3) {$y_{12}$};
  
  \filldraw[blue] (3,3) circle (2pt);
  \node[above, font=\scriptsize] at (3,3) {$x_{1}$};
  \filldraw[blue] (2,2) circle (2pt);
  \node[above, font=\scriptsize] at (2,2) {$x_{2}$}; 
  \filldraw[blue] (1,1) circle (2pt);
  \node[above, font=\scriptsize] at (1,1) {$x_{3}$};
  \filldraw[blue] (3,1) circle (2pt);
  \node[above, font=\scriptsize] at (3,1) {$x_{4}$};
  \filldraw[blue] (1,3) circle (2pt);
  \node[above, font=\scriptsize] at (1,3) {$x_{5}$}; 
  \filldraw[blue] (2,3) circle (2pt);
  \node[above, font=\scriptsize] at (2,3) {$x_{6}$};
  \filldraw[blue] (3,2) circle (2pt);
  \node[above, font=\scriptsize] at (3,2) {$x_{7}$};
  \filldraw[blue] (2,1) circle (2pt);
  \node[above, font=\scriptsize] at (2,1) {$x_{8}$}; 
  \filldraw[blue] (1,2) circle (2pt);
  \node[above, font=\scriptsize] at (1,2) {$x_{9}$};
  
  \filldraw[blue] (7,5) circle (2pt);
  \node[above, font=\scriptsize] at (7,5) {$y_{1}$};
  \filldraw[blue] (8,5) circle (2pt);
  \node[above, font=\scriptsize] at (8,5) {$y_{2}$}; 
  \filldraw[blue] (9,5) circle (2pt);
  \node[above, font=\scriptsize] at (9,5) {$y_{3}$};
  \filldraw[blue] (10,5) circle (2pt);
  \node[above, font=\scriptsize] at (10,5) {$y_{4}$};
  \filldraw[blue] (11,4) circle (2pt);
  \node[right, font=\scriptsize] at (11,4) {$y_{5}$};
  \filldraw[blue] (11,3) circle (2pt);
  \node[right, font=\scriptsize] at (11,3) {$y_{6}$}; 
  \filldraw[blue] (11,2) circle (2pt);
  \node[right, font=\scriptsize] at (11,2) {$y_{7}$};
  \filldraw[blue] (11,1) circle (2pt);
  \node[right, font=\scriptsize] at (11,1) {$y_{8}$};
  \filldraw[blue] (10,0) circle (2pt);
  \node[below, font=\scriptsize] at (10,0) {$y_{9}$};
  \filldraw[blue] (9,0) circle (2pt);
  \node[below, font=\scriptsize] at (9,0) {$y_{10}$}; 
  \filldraw[blue] (8,0) circle (2pt);
  \node[below, font=\scriptsize] at (8,0) {$y_{11}$};
  \filldraw[blue] (7,0) circle (2pt);
  \node[below, font=\scriptsize] at (7,0) {$y_{12}$};
  \filldraw[blue] (6,1) circle (2pt);
  \node[left, font=\scriptsize] at (6,1) {$y_{13}$};
  \filldraw[blue] (6,2) circle (2pt);
  \node[left, font=\scriptsize] at (6,2) {$y_{14}$}; 
  \filldraw[blue] (6,3) circle (2pt);
  \node[left, font=\scriptsize] at (6,3) {$y_{15}$};
  \filldraw[blue] (6,4) circle (2pt);
  \node[left, font=\scriptsize] at (6,4) {$y_{16}$};
  
  \filldraw[blue] (10,4) circle (2pt);
  \node[above, font=\scriptsize] at (10,4) {$x_{1}$};
  \filldraw[blue] (9,3) circle (2pt);
  \node[above, font=\scriptsize] at (9,3) {$x_{2}$}; 
  \filldraw[blue] (8,2) circle (2pt);
  \node[above, font=\scriptsize] at (8,2) {$x_{3}$};
  \filldraw[blue] (7,1) circle (2pt);
  \node[above, font=\scriptsize] at (7,1) {$x_{4}$};
  \filldraw[blue] (10,1) circle (2pt);
  \node[above, font=\scriptsize] at (10,1) {$x_{5}$};
  \filldraw[blue] (9,2) circle (2pt);
  \node[above, font=\scriptsize] at (9,2) {$x_{6}$}; 
  \filldraw[blue] (8,3) circle (2pt);
  \node[above, font=\scriptsize] at (8,3) {$x_{7}$};
  \filldraw[blue] (7,4) circle (2pt);
  \node[above, font=\scriptsize] at (7,4) {$x_{8}$};
  \filldraw[blue] (8,4) circle (2pt);
  \node[above, font=\scriptsize] at (8,4) {$x_{9}$};
  \filldraw[blue] (9,4) circle (2pt);
  \node[above, font=\scriptsize] at (9,4) {$x_{10}$}; 
  \filldraw[blue] (10,3) circle (2pt);
  \node[above, font=\scriptsize] at (10,3) {$x_{11}$};
  \filldraw[blue] (10,2) circle (2pt);
  \node[above, font=\scriptsize] at (10,2) {$x_{12}$};
  \filldraw[blue] (9,1) circle (2pt);
  \node[above, font=\scriptsize] at (9,1) {$x_{13}$};
  \filldraw[blue] (8,1) circle (2pt);
  \node[above, font=\scriptsize] at (8,1) {$x_{14}$}; 
  \filldraw[blue] (7,2) circle (2pt);
  \node[above, font=\scriptsize] at (7,2) {$x_{15}$};
  \filldraw[blue] (7,3) circle (2pt);
  \node[above, font=\scriptsize] at (7,3) {$x_{16}$};        
\end{tikzpicture}
\caption{Vertex labeling examples for $n = 3$ and $n = 4$}
\label{fig:vertex_labeling}
\end{figure}
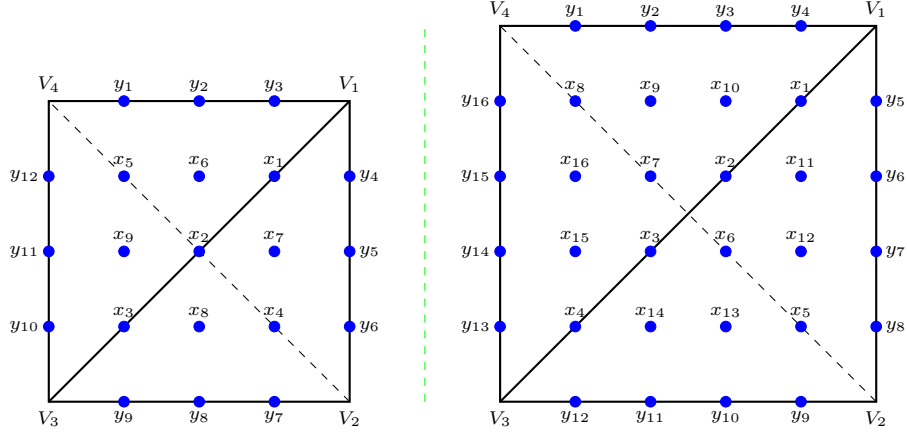
The vertices are organized in concentric \emph{layers}, where the \emph{$k$-th layer} (also called \emph{layer $k$} for convenience) contains $(n-2k)^2$ vertices. After the sequence of mutation corresponding to flipping the diagonal from $V_1V_3$ to $V_2V_4$, each vertex in the $k$-th layer undergoes exactly $k$ mutations, starting from the outermost layer ($k=1$) and proceeding inward.

Explicitly, consider the cluster realization of $\mathscr{A}_{\text{SL}_{n+1},\bbS}$ over the square $\Box V_1V_2V_3V_4$, i.e. with the associated $n$-triangulated quiver. Let $\Upsilon(V_1V_2V_3V_4)$ be the $n$-dimensional vector where each component is a Laurent polynomial in $\Z[x_1^\pm, \dots, x_{n^2}^\pm, y_1, \dots, y_{4n}]$ representing the expansion formula of the $n$ cluster variables labeled along $V_2V_4$ with respect to the initial variables after flipping the diagonal from $V_1V_3$ to $V_2V_4$. Then we have:
\Eq{\label{vertex}
    \Upsilon(V_1V_2V_3V_4)= 
\begin{cases}
    (x_{n+1}^{(1)}, \dots, x_{n+k}^{(k)}, x_{n+k+1}^{(k)}, \dots, x_{2n}^{(1)}) & \text{if } n = 2k \text{ is even}, \\
    (x_{n+1}^{(1)}, \dots, x_{n+k}^{(k)}, x_{k+1}^{(k+1)}, x_{n+k+1}^{(k)}, \dots, x_{2n-1}^{(1)}) & \text{if } n = 2k+1 \text{ is odd}
\end{cases}
}
where $x_{i}^{(j)}$ denotes that vertex $x_i$ is mutated exactly $j$ times in the flip mutation sequence. For convenience, we will also call each variable $x_{i}^{(j)}$ that appears as a coordinate in \eqref{vertex} to be a \emph{vertex} of $\Upsilon(V_1V_2V_3V_4)$. For notation convenience, we sometimes write $x'_j$ for $x_j^{(1)}$, $x''_j$ for $x_j^{(2)}$, $x'''_j$ for $x_j^{(3)}$, etc. 

The following proposition can be verified directly from the definition by considering the last two mutation steps:
\begin{Prop}
\label{prop:flip_properties}
Consider the cluster realization of $\mathscr{A}_{\text{SL}_{n+1},\bbS}$ over the square $\Box V_1V_2V_3V_4$ with vertices $x_i$ and $y_j$ as defined. Define the mutated variables of the sub-square of 1 smaller rank to be:
\Eq{
    \Upsilon(V_1y_{2n}x_{n}y_1) &= (X_{1,1}, X_{1,2}, \dots, X_{1,n-1}); \quad \Upsilon(x_1y_{2n+1}V_3y_{4n}) = (X_{2,1}, X_{2,2}, \dots, X_{2,n-1}).
}
Then the mutated variables along the diagonal $V_2V_4$ satisfy:
    \begin{enumerate}[label=(\alph*)]
    		\item For $n = 1$, we have only one mutable variable $x_1$ with 4 frozen variables $y_1,...,y_4$. Then
    			$$\Upsilon(V_1V_2V_3V_4) = x'_1 := \frac{y_1y_3+y_2y_4}{x_1}.$$
    			For $n = 2$, we have 4 mutable variables $x_1,...,x_4$ and 8 frozen variables $y_1,...,y_8$. Then
    	$$\Upsilon(V_1V_2V_3V_4) = (x'_3, x'_4)$$
    	where
    	$$x'_3 := \frac{y_2y_5}{x_1} + \frac{y_4y_7}{x_2} + \frac{y_3y_5x_4}{x_1x_3} + \frac{y_4y_6x_4}{x_2x_3}; \quad\mbox{and}\quad x'_4 := \frac{y_3y_8}{x_1} + \frac{y_1y_6}{x_2} + \frac{y_2y_8x_3}{x_1x_4} + \frac{y_1y_7x_3}{x_2x_4}.$$
    		
        \item If $n = 2k \geq 4$ is even, then $V_2V_4$ consists of vertices $x_{n+1}, x_{n+2}, \dots, x_{2n}$ with:
\begin{equation}
\begin{aligned}\nonumber
\Upsilon(x_{1}x_{n+1}x_{n}x_{2n}) &= (x^{(1)}_{n+2},\ldots,x^{(k-1)}_{n+k},x^{(k-1)}_{n+k+1},\ldots,x^{(1)}_{2n-1}); \\
\Upsilon(V_{1}V_{2}V_{3}V_{4}) &= (x^{(1)}_{n+1},\ldots,x^{(k)}_{n+k},x^{(k)}_{n+k+1},\ldots,x^{(1)}_{2n}); \\
x^{\prime}_{n+1}=x^{(1)}_{n+1} &=\frac{y_{2n}X_{2,1}+y_{2n+1}X_{1,1}}{x_{n+1}}; \\
 x^{\prime}_{2n}=x^{(1)}_{2n} &=\frac{y_{1}X_{2,n-1}+y_{4n}X_{1,n-1}}{x_{2n}};  
 \end{aligned}\nonumber
\end{equation}
   and for $t = 0,1,\dots,k-2$:
   \begin{equation}
\begin{aligned}
x^{(2+t)}_{n+2+t} &=\frac{X_{1,1+t}X_{2,2+t}+X_{1,2+t}X_{2,1+t}}{x^{(1+t)}_{n+2+t}}; \\
x^{(2+t)}_{2n-1-t} &=\frac{X_{1,n-2-t}X_{2,n-1-t}+X_{1,n-1-t}X_{2,n-2-t}}{x^{(1+t)}_{2n-1-t}}.
\end{aligned}
\label{eq:upsilon_even}
\end{equation}

        \item If $n = 2k+1 \geq 3$ is odd, then $V_2V_4$ consists of vertices $x_{n+1}, \dots, x_{n+k}, x_{k+1}, x_{n+k+1}, \dots, x_{2n-1}$ with:
\begin{equation}
\begin{aligned}
    \Upsilon(x_{1}x_{n+1}x_{n}x_{2n-1}) &=\case{(x_2^{(1)})&k=1\\ (x_{n+2}^{(1)}, \dots, x_{n+k}^{(k-1)}, x_{k+1}^{(k)}, x_{n+k+1}^{(k-1)}, \dots, x_{2n-2}^{(1)})&k>1}; \\
    \Upsilon(V_1V_2V_3V_4) &= (x_{n+1}^{(1)}; \dots, x_{n+k}^{(k)}, x_{k+1}^{(k+1)}, x_{n+k+1}^{(k)}, \dots, x_{2n-1}^{(1)}); \\
    x'_{n+1} = x_{n+1}^{(1)} &= \frac{y_{2n}X_{2,1} + y_{2n+1}X_{1,1}}{x_{n+1}}; \\
     x'_{2n-1} = x_{2n-1}^{(1)} &= \frac{y_{1}X_{2,n-1} + y_{4n}X_{1,n-1}}{x_{2n-1}}; \\
    x_{k+1}^{(k+1)} &= \frac{X_{1,k}X_{2,k+1} + X_{1,k+1}X_{2,k}}{x_{k+1}^{(k)}}; 
   \end{aligned}\nonumber
\end{equation}
   and if $k>1$, for $t = 0,1,\dots,k-2$:
   \begin{equation}
\begin{aligned}
    x_{n+2+t}^{(2+t)} &= \frac{X_{1,1+t}X_{2,2+t} + X_{1,2+t}X_{2,1+t}}{x_{n+2+t}^{(1+t)}}; \\
    x_{2n-2-t}^{(2+t)} &= \frac{X_{1,n-2-t}X_{2,n-1-t} + X_{1,n-1-t}X_{2,n-2-t}}{x_{2n-2-t}^{(1+t)}}.
\end{aligned}
\label{eq:upsilon_odd}
\end{equation}
\end{enumerate}
\end{Prop}
By the Laurent phenomenon, every coordinate of $\Upsilon$ is a Laurent polynomial in the initial variables $x_i$ and $y_j$.

\begin{Ex}
\label{ex:SS}
For $n = 4$, we shall compute:
\begin{align*}
    \Upsilon(x_{1}x_{5}x_{4}x_{8}) &= (x'_{6}, x'_{7}); \\
    x'_{6} &= \frac{x_{10}x_{13}}{x_2} + \frac{x_{12}x_{15}}{x_3} + \frac{x_{7}x_{11}x_{13}}{x_2x_6} + \frac{x_{7}x_{12}x_{14}}{x_3x_{6}};\\
    x'_{7} &= \frac{x_{11}x_{16}}{x_2} + \frac{x_{9}x_{14}}{x_3} + \frac{x_{6}x_{10}x_{16}}{x_2x_7} + \frac{x_{6}x_{9}x_{15}}{x_3x_7}.
\end{align*}
Applying Proposition~\ref{prop:flip_properties} yields:
\begin{align*}
    \Upsilon(V_1y_{8}x_{4}y_1) &= (X_{1,1}, X_{1,2}, X_{1,3});
\end{align*}
{\small
\begin{align*}
    X_{1,1}=x'_{12}=&\frac{y_4x_5}{x_1}+\frac{y_7x_{15}}{x_3}+\frac{y_5x_5x_{10}}{x_1x_{11}}+\frac{y_7x_7x_{14}}{x_3x_6} +\frac{y_6x_5x_7}{x_2x_{12}}+\frac{y_7x_{10}x_{13}}{x_2x_{12}}+\frac{y_7x_7x_{11}x_{13}}{x_2x_6x_{12}}+\frac{y_6x_5x_6x_{10}}{x_2x_{11}x_{12}};\\
    X_{1,2}=x''_2=&\frac{y_3x_{13}}{x_2}+\frac{y_6x_{16}}{x_2}+\frac{y_4x_{12}x_{16}}{x_1x_7}+\frac{y_6x_{9}x_{14}}{x_3x_{11}}+\frac{y_5x_{9}x_{13}}{x_1x_6} +\frac{y_3x_{12}x_{15}}{x_3x_{10}}+\frac{y_6x_{6}x_{10}x_{16}}{x_2x_7x_{11}}+\frac{y_3x_7x_{11}x_{13}}{x_2x_6x_{10}} \\
    &+\frac{y_5x_{10}x_{12}x_{16}}{x_1x_7x_{11}} +\frac{y_6x_{6}x_{9}x_{15}}{x_3x_7x_{11}}+\frac{y_4x_{9}x_{11}x_{13}}{x_1x_6x_{10}}+\frac{y_3x_{7}x_{12}x_{14}}{x_3x_6x_{10}}+\frac{y_5x_{2}x_9x_{12}x_{15}}{x_1x_3x_7x_{11}}\\
    &+\frac{y_4x_{2}x_9x_{12}x_{15}}{x_1x_3x_7x_{10}}+\frac{y_4x_{2}x_9x_{12}x_{14}}{x_1x_3x_6x_{10}}+\frac{y_5x_{2}x_9x_{12}x_{14}}{x_1x_3x_6x_{11}};\\
    X_{1,3}=x'_9=&\frac{y_2x_{14}}{x_3}+\frac{y_5x_{8}}{x_1}+\frac{y_4x_{8}x_{11}}{x_1x_{10}}+\frac{y_2x_{6}x_{15}}{x_3x_7} +\frac{y_2x_{11}x_{16}}{x_2x_9}+\frac{y_3x_{6}x_{8}}{x_2x_9}+\frac{y_2x_{6}x_{10}x_{16}}{x_2x_7x_9}+\frac{y_3x_{7}x_8x_{11}}{x_2x_9x_{10}};
\end{align*}
}

\begin{align*}      
    \Upsilon(x_1y_{9}V_3y_{16})=(X_{2,1}, X_{2,2}, X_{2,3});
\end{align*}
{\small
\begin{align*}    
    X_{2,1}=x'_{13}=&\frac{y_{10}x_{10}}{x_2}+\frac{y_{13}x_{5}}{x_4}+\frac{y_{10}x_7x_{11}}{x_2x_6}+\frac{y_{12}x_5x_{15}}{x_4x_{14}} +\frac{y_{10}x_{12}x_{15}}{x_3x_{13}}+\frac{y_{11}x_5x_7}{x_3x_{13}}+\frac{y_{11}x_5x_6x_{15}}{x_3x_{13}x_{14}}+\frac{y_{10}x_7x_{12}x_{14}}{x_3x_{6}x_{13}};\\
    X_{2,2}=x''_3=&\frac{y_{11}x_{9}}{x_3}+\frac{y_{14}x_{12}}{x_3}+\frac{y_{14}x_{10}x_{13}}{x_2x_{15}}+\frac{y_{12}x_{12}x_{16}}{x_4x_{6}}+\frac{y_{11}x_{11}x_{16}}{x_2x_{14}} +\frac{y_{13}x_{9}x_{13}}{x_4x_{7}}+\frac{y_{14}x_{7}x_{12}x_{14}}{x_3x_6x_{15}}+\frac{y_{11}x_6x_{9}x_{15}}{x_3x_7x_{14}}\\
    &+\frac{y_{14}x_{7}x_{11}x_{13}}{x_2x_6x_{15}} + \frac{y_{13}x_{12}x_{14}x_{16}}{x_4x_6x_{15}}+\frac{y_{11}x_{6}x_{10}x_{16}}{x_2x_7x_{14}}+\frac{y_{12}x_{9}x_{13}x_{15}}{x_4x_7x_{14}}+\frac{y_{13}x_{3}x_{11}x_{13}x_{16}}{x_2x_4x_6x_{15}}\\
    &+\frac{y_{13}x_{3}x_{10}x_{13}x_{16}}{x_2x_4x_7x_{15}}+\frac{y_{12}x_{3}x_{10}x_{13}x_{16}}{x_2x_4x_7x_{14}}+\frac{y_{12}x_{3}x_{11}x_{13}x_{16}}{x_2x_4x_6x_{14}};\\
    X_{2,3}=x'_{16}=&\frac{y_{12}x_8}{x_4}+\frac{y_{15}x_{11}}{x_2}+\frac{y_{15}x_{6}x_{10}}{x_2x_7}+\frac{y_{13}x_{8}x_{14}}{x_4x_{15}} +\frac{y_{14}x_{6}x_8}{x_3x_{16}}+\frac{y_{15}x_{9}x_{14}}{x_3x_{16}}+\frac{y_{14}x_{7}x_8x_{14}}{x_3x_{15}x_{16}}+\frac{y_{15}x_{6}x_9x_{15}}{x_3x_{7}x_{16}}.
\end{align*}
}
Therefore, the expansion formula for $n=4$ of the full square $\Box V_1V_2V_3V_4$ is:
\begin{align*}
    \Upsilon(V_1V_2V_3V_4) &= (x'_5, x''_{6}, x''_{7}, x'_{8})
\end{align*}
where
{\small
\begin{align*}    
    x'_5 =& \frac{y_{8}X_{2,1}+y_{9}X_{1,1}}{x_{5}}\\
    =&\frac{y_8y_{13}}{x_4}+\frac{y_4y_9}{x_1}+\frac{y_5y_9x_{10}}{x_{1}x_{11}}+\frac{y_6y_9x_{7}}{x_{2}x_{12}}+\frac{y_8y_{11}x_{7}}{x_{3}x_{13}}+\frac{y_8y_{12}x_{15}}{x_{4}x_{14}}+\frac{y_7y_9x_{15}}{x_{3}x_{5}}+\frac{y_8y_{10}x_{10}}{x_{2}x_{5}}+\frac{y_6y_9x_6x_{10}}{x_{2}x_{11}x_{12}}\\
    &+\frac{y_8y_{11}x_6x_{15}}{x_{3}x_{13}x_{14}}+\frac{y_7y_9x_7x_{14}}{x_{3}x_{5}x_{6}}+\frac{y_7y_9x_{10}x_{13}}{x_{2}x_{5}x_{12}}+\frac{y_8y_{10}x_7x_{11}}{x_{2}x_{5}x_{6}}+\frac{y_8y_{10}x_{12}x_{15}}{x_{3}x_{5}x_{13}}+\frac{y_7y_9x_7x_{11}x_{13}}{x_{2}x_{5}x_{6}x_{12}}+\frac{y_8y_{10}x_7x_{12}x_{14}}{x_{3}x_{5}x_{6}x_{13}};
\end{align*}

\begin{align*}
    x'_8 =& \frac{y_{1}X_{2,3}+y_{16}X_{1,3}}{x_{8}}\\
    =&\frac{y_1y_{12}}{x_4}+\frac{y_6y_{15}}{x_1}+\frac{y_1y_{13}x_{14}}{x_{4}x_{15}}+\frac{y_1y_{14}x_{6}}{x_{3}x_{16}}+\frac{y_3y_{16}x_{7}}{x_{2}x_{9}}+\frac{y_4y_{16}x_{11}}{x_{1}x_{10}}+\frac{y_1y_{15}x_{11}}{x_{2}x_{8}}+\frac{y_2y_{16}x_{14}}{x_{3}x_{8}}+\frac{y_1y_{14}x_7x_{14}}{x_{3}x_{15}x_{16}}\\
    &+\frac{y_3y_{16}x_7x_{11}}{x_{2}x_{9}x_{10}}+\frac{y_1y_{15}x_6x_{10}}{x_{2}x_{7}x_{8}}+\frac{y_1y_{15}x_{9}x_{14}}{x_{3}x_{8}x_{16}}+\frac{y_2y_{16}x_6x_{15}}{x_{3}x_{7}x_{8}}+\frac{y_2y_{16}x_{11}x_{16}}{x_{2}x_{8}x_{9}}+\frac{y_{1}y_{15}x_6x_{9}x_{15}}{x_{3}x_{7}x_{8}x_{16}}+\frac{y_2y_{16}x_6x_{10}x_{16}}{x_{2}x_{7}x_{8}x_{9}};
\end{align*}

\begin{align*}    
    x''_6 =& \frac{X_{1,1}X_{2,2}+X_{1,2}X_{2,1}}{x'_6}\\
    =&\frac{y_{3} y_{10}}{x_{2}}+\frac{y_{7} y_{14}}{x_{3}}+\frac{y_{3} y_{13}x_{5}}{x_{4} x_{10}}+\frac{y_{4} y_{14}x_{5}}{x_{1} x_{15}}+\frac{y_{5} y_{10}x_{9}}{x_{1}x_{6}}+\frac{y_{6} y_{10}x_{16}}{x_{2} x_{13}}+\frac{y_{7} y_{11}x_{9}}{x_{3} x_{12}}+\frac{y_{7} y_{12}x_{16}}{x_{4}x_{6}}\\
    &+\frac{y_{3} y_{10}x_{7}x_{11}}{x_{2} x_{6}x_{10}}+\frac{y_{3} y_{10}x_{12} x_{15}}{x_{3} x_{10} x_{13}}+\frac{y_{3} y_{11}x_{5}x_{7}}{x_{3}x_{10}x_{13}}+\frac{y_{3} y_{12} x_{5}x_{15}}{x_{4}x_{10}x_{14}}+\frac{ y_{4} y_{10}x_{9}x_{11}}{x_{1}x_{6}x_{10}}+\frac{y_{4}y_{10}x_{12}x_{16}}{x_{1} x_{7} x_{13}}+\frac{y_{5}y_{14}x_{5}x_{10}}{x_{1} x_{11} x_{15}}\\
    &+\frac{y_{6} y_{10}x_{9} x_{14}}{x_{3} x_{11} x_{13}}+\frac{y_{6} y_{14}x_{5}x_{7}}{x_{2} x_{12} x_{15}}+\frac{y_{7} y_{11}x_{11} x_{16} }{x_{2} x_{12} x_{14}}+\frac{y_{7} y_{13} x_{9}x_{13}}{x_{4} x_{7}x_{12}}+\frac{y_{7} y_{13}x_{14} x_{16} }{x_{4}x_{6} x_{15}}+\frac{y_{7} y_{14}x_{7}x_{14}}{x_{3}x_{6}x_{15}}+\frac{ y_{7} y_{14}x_{10} x_{13}}{x_{2} x_{12} x_{15}}\\
    &+\frac{y_{3} y_{10}x_{7}x_{12} x_{14}}{ x_{3} x_{6}x_{10} x_{13}}+\frac{y_{3} y_{11}x_{5}x_{6} x_{15}}{x_{3} x_{10} x_{13} x_{14}}+\frac{y_{4} y_{11}x_{2} x_{5} x_{9}}{x_{1}x_{3}x_{10} x_{13}}+\frac{y_{4} y_{11}x_{5}x_{6}x_{16}}{x_{1} x_{7} x_{13}x_{14}}+\frac{y_{4} y_{12}x_{3} x_{5} x_{16}}{x_{1} x_{4}x_{7} x_{14}}+\frac{y_{4} y_{13}x_{2} x_{5} x_{9}}{x_{1} x_{4} x_{7} x_{10}}\\
    &+\frac{y_{4} y_{13}x_{3} x_{5} x_{16} }{x_{1}x_{4} x_{7}x_{15}}+\frac{y_{5} y_{10}x_{10}x_{12}x_{16}}{x_{1}x_{7}x_{11}x_{13}}+\frac{y_{5} y_{11}x_{2} x_{5} x_{9}}{x_{1} x_{3} x_{11} x_{13}}+\frac{y_{5}y_{13}x_{2} x_{5} x_{9}}{x_{1}x_{4} x_{7}x_{11}}+\frac{y_{6} y_{10}x_{6}  x_{9}x_{15}}{x_{3}x_{7} x_{11}x_{13}}+\frac{y_{6} y_{10}x_{6} x_{10} x_{16}}{x_{2}x_{7} x_{11} x_{13}}\\
    &+\frac{y_{6} y_{11}x_{5}x_{6}x_{9}}{x_{3} x_{11} x_{12} x_{13}}+\frac{y_{6} y_{11}x_{5}x_{6}x_{16}}{x_{2} x_{12} x_{13}x_{14}}+\frac{y_{6} y_{12}x_{3} x_{5} x_{16}}{x_{2}x_{4} x_{12}  x_{14}}+\frac{y_{6} y_{13}x_{3} x_{5} x_{16} }{x_{2}x_{4}x_{12} x_{15}}+\frac{y_{6} y_{13}x_{5}x_{6}x_{9}}{x_{4} x_{7}x_{11} x_{12}}+\frac{y_{6} y_{14}x_{5}x_{6} x_{10}}{x_{2} x_{11} x_{12} x_{15}}\\
    &+\frac{y_{7} y_{11}x_{6}x_{9}x_{15}}{x_{3}x_{7} x_{12} x_{14}}+\frac{y_{7} y_{11}x_{6} x_{10} x_{16}}{x_{2} x_{7}x_{12}x_{14}}+\frac{y_{7} y_{12}x_{9}x_{13} x_{15}}{x_{4} x_{7}x_{12}x_{14}}+\frac{y_{7} y_{14}x_{7}x_{11}x_{13}}{x_{2}x_{6}x_{12}x_{15}}+\frac{ y_{4} y_{10}x_{2}x_{9}x_{12}x_{14}}{ x_{1}x_{3}x_{6}x_{10} x_{13}}+\frac{y_{4} y_{10}x_{2}x_{9} x_{12} x_{15}}{x_{1}x_{3}x_{7} x_{10} x_{13}}\\
    &+\frac{y_{4} y_{12}x_{2}x_{5}x_{9}x_{15}}{x_{1} x_{4}  x_{7} x_{10}x_{14}}+\frac{y_{5} y_{10}x_{2} x_{9} x_{12} x_{14}}{x_{1}x_{3}x_{6}x_{11} x_{13}}+\frac{y_{5} y_{10}x_{2} x_{9}x_{12} x_{15}}{x_{1}x_{3}x_{7}x_{11}  x_{13}}+\frac{y_{5} y_{11}x_{5}x_{6}x_{10}x_{16}}{x_{1}  x_{7}x_{11}x_{13}x_{14}}+\frac{y_{5} y_{12}x_{2}x_{5}x_{9}x_{15}}{x_{1}  x_{4}x_{7}x_{11} x_{14}}\\
    &+\frac{y_{5} y_{12}x_{3}x_{5}x_{10}x_{16}}{x_{1}  x_{4}x_{7}x_{11} x_{14}}+\frac{y_{5} y_{13}x_{3}x_{5}x_{10}x_{16}}{x_{1}  x_{4} x_{7}x_{11} x_{15}}+\frac{y_{6} y_{12}x_{5}x_{6} x_{9} x_{15}}{x_{4}  x_{7}x_{11} x_{12}x_{14}}+\frac{y_{7} y_{12}x_{3} x_{10} x_{13} x_{16}}{x_{2}x_{4}x_{7} x_{12} x_{14}}+\frac{ y_{7} y_{12}x_{3} x_{11} x_{13} x_{16}}{x_{2}x_{4}x_{6} x_{12} x_{14}}\\
    &+\frac{y_{7} y_{13}x_{3} x_{10} x_{13} x_{16} }{x_{2} x_{4} x_{7}x_{12} x_{15}}+\frac{y_{7} y_{13}x_{3} x_{11} x_{13} x_{16}}{x_{2}x_{4}x_{6}x_{12} x_{15}}+\frac{y_{6} y_{11}x_{6}^{2} x_{5} x_{9} x_{15}}{x_{3}x_{7}x_{11} x_{12}x_{13} x_{14}}+\frac{y_{6} y_{11}x_{6}^{2} x_{5}x_{10}x_{16}}{x_{2}x_{7} x_{11} x_{12}x_{13}x_{14}}\\
    &+\frac{y_{4} y_{11}x_{2}  x_{5}x_{6}  x_{9}x_{15} }{x_{1}x_{3}x_{7}x_{10} x_{13}x_{14}}+\frac{y_{5} y_{11} x_{2}  x_{5}x_{6} x_{9}x_{15} }{x_{1}x_{3}x_{7}  x_{11} x_{13}x_{14}}+\frac{y_{6} y_{12} x_{3}  x_{5}x_{6} x_{10}x_{16} }{x_{2} x_{4}x_{7}x_{11} x_{12}  x_{14} }+\frac{ y_{6} y_{13} x_{3}  x_{5}x_{6}x_{10} x_{16}}{x_{2}x_{4} x_{7} x_{11} x_{12} x_{15}};
\end{align*}

\begin{align*}
    x''_7 =& \frac{X_{1,2}X_{2,3}+X_{1,3}X_{2,2}}{x'_7}\\
    =&\frac{y_{2} y_{11}}{x_{3}}+\frac{y_{6} y_{15}}{x_{2}}+\frac{ y_{2} y_{13}x_{13}}{x_{4}x_{7}}+\frac{y_{2} y_{14}x_{12} }{x_{3} x_{9}}+\frac{y_{3} y_{15}x_{13} }{x_{2} x_{16}}+\frac{ y_{4} y_{15}x_{12}}{x_{1}x_{7}}+\frac{ y_{5} y_{11}x_{8}}{x_{1} x_{14}}+\frac{y_{6} y_{12}x_{8} }{x_{4}x_{11}}
    \\
    &+\frac{y_{2} y_{11}x_{6} x_{15} }{x_{3}x_{7} x_{14}}+\frac{ y_{2} y_{11}x_{11} x_{16}}{x_{2} x_{9} x_{14}}+\frac{y_{2} y_{12}x_{12}x_{16}}{x_{4}x_{6}x_{9}}+\frac{y_{2} y_{12}x_{13}x_{15}}{x_{4}x_{7}x_{14}}+\frac{y_{2} y_{14}x_{10} x_{13}}{x_{2} x_{9}x_{15} }+\frac{y_{3} y_{11}x_{6} x_{8}}{x_{2} x_{9} x_{14}}+\frac{y_{3} y_{15} x_{12}x_{15} }{x_{3} x_{10} x_{16}}
    \\
    &+\frac{y_{4} y_{11}x_{8} x_{11}}{x_{1} x_{10} x_{14}}+\frac{y_{5} y_{15}x_{9} x_{13} }{x_{1} x_{6} x_{16}}+\frac{y_{5} y_{15}x_{10} x_{12} }{x_{1}x_{7}x_{11}}+\frac{y_{6} y_{13}x_{8}x_{14}}{x_{4}x_{11}x_{15}}+\frac{y_{6} y_{14}x_{6} x_{8}}{x_{3} x_{11} x_{16}}+\frac{ y_{6} y_{15}x_{6} x_{10}}{x_{2}x_{7}x_{11}}+\frac{y_{6} y_{15}x_{9} x_{14} }{x_{3} x_{11} x_{16}}
    \\
    &+\frac{y_{2} y_{11}x_{6} x_{10} x_{16} }{x_{2}x_{7}x_{9}x_{14}}+\frac{ y_{2} y_{13}  x_{12}x_{14}x_{16}}{x_{4}x_{6} x_{9} x_{15}}+\frac{ y_{2} y_{14}x_{7} x_{11} x_{13}}{x_{2} x_{6} x_{9}x_{15}}+\frac{y_{2} y_{14}x_{7}x_{12}x_{14}}{x_{3} x_{6}  x_{9}x_{15}}+\frac{y_{3} y_{11}x_{7}x_{8}x_{11}}{x_{2} x_{9} x_{10}x_{14}}+\frac{ y_{3} y_{12}x_{3} x_{8} x_{13}}{x_{2} x_{4} x_{9} x_{14}}
    \\
    &+\frac{y_{3} y_{12}x_{7} x_{8} x_{12} }{x_{4}x_{6}x_{9} x_{10}}+\frac{ y_{3} y_{13}x_{3} x_{8} x_{13}}{x_{2} x_{4} x_{9}x_{15}}+\frac{y_{3} y_{14}x_{7} x_{8} x_{12} }{x_{3}x_{9} x_{10} x_{16} }+\frac{y_{3} y_{14}x_{7} x_{8} x_{13}}{x_{2}x_{9} x_{15} x_{16} }+\frac{y_{3} y_{15}x_{7} x_{11} x_{13} }{x_{2} x_{6} x_{10} x_{16}}+\frac{y_{3} y_{15}x_{7} x_{12}x_{14}}{x_{3} x_{6} x_{10} x_{16}}
    \\
    &+\frac{y_{4} y_{12}x_{2} x_{8} x_{12} }{x_{1}x_{4} x_{6} x_{10}}+\frac{ y_{4} y_{14}x_{2} x_{8} x_{12}}{x_{1}x_{3}x_{10} x_{16}}+\frac{y_{4} y_{15} x_{9} x_{11}x_{13} }{x_{1} x_{6} x_{10} x_{16}}+\frac{y_{5} y_{12}x_{2} x_{8} x_{12} }{x_{1} x_{4}x_{6} x_{11} }+\frac{y_{5} y_{12}x_{3} x_{8} x_{13} }{x_{1}x_{4} x_{6} x_{14}}+\frac{ y_{5} y_{13}x_{3} x_{8} x_{13}}{x_{1}x_{4} x_{6} x_{15}}
    \\
    &+\frac{y_{5} y_{14}x_{2} x_{8} x_{12} }{x_{1}x_{3}  x_{11} x_{16}}+\frac{y_{5} y_{14}x_{7} x_{8} x_{13} }{x_{1} x_{6} x_{15} x_{16}}+\frac{y_{6} y_{14}x_{7}x_{8} x_{14}}{x_{3} x_{11} x_{15} x_{16}}+\frac{y_{6} y_{15}x_{6} x_{9} x_{15}}{x_{3}x_{7} x_{11} x_{16}}+\frac{y_{2} y_{12}x_{3} x_{10}  x_{13}x_{16} }{x_{2} x_{4}x_{7} x_{9} x_{14}}+\frac{y_{2} y_{12}x_{3} x_{11} x_{13}x_{16}}{x_{2}x_{4} x_{6}  x_{9} x_{14}}
    \\
    &+\frac{y_{2} y_{13}x_{3} x_{10} x_{13}x_{16} }{ x_{2} x_{4} x_{7} x_{9}x_{15}}+\frac{y_{2} y_{13}x_{3} x_{11}x_{13}  x_{16} }{x_{2}  x_{4}x_{6}x_{9} x_{15}}+\frac{y_{3} y_{13}x_{7}x_{8} x_{12}x_{14}}{x_{4}x_{6}x_{9} x_{10}  x_{15}}+\frac{y_{4} y_{12}x_{3} x_{8}x_{11}  x_{13}}{x_{1}x_{4} x_{6} x_{10}  x_{14}}+\frac{y_{4} y_{13}x_{2} x_{8}x_{12}x_{14}}{x_{1}x_{4} x_{6} x_{10}  x_{15}}
    \\
    &+\frac{y_{4} y_{13}x_{3}x_{8} x_{11}  x_{13} }{x_{1}x_{4} x_{6} x_{10}  x_{15}}+\frac{y_{4} y_{14}x_{7}  x_{8}x_{11} x_{13} }{x_{1} x_{6} x_{10} x_{15} x_{16}}+\frac{y_{4} y_{15}x_{2} x_{9} x_{12}x_{14}}{x_{1}x_{3}  x_{6} x_{10} x_{16}}+\frac{y_{4} y_{15}x_{2} x_{9} x_{12}x_{15}}{x_{1} x_{3} x_{7} x_{10} x_{16}}+\frac{y_{5} y_{13}x_{2}x_{8} x_{12} x_{14} }{x_{1}x_{4}x_{6} x_{11}x_{15}}
    \\
    &+\frac{y_{5} y_{15}x_{2} x_{9} x_{12}x_{14}}{x_{1}x_{3}x_{6}  x_{11}x_{16}}+\frac{y_{5} y_{15}x_{2} x_{9} x_{12}x_{15} }{x_{1} x_{3} x_{7} x_{11} x_{16}}+\frac{y_{3} y_{14}x_{7}^{2} x_{8} x_{11} x_{13} }{x_{2} x_{6}x_{9} x_{10} x_{15} x_{16} }+\frac{y_{3} y_{14}x_{7}^{2}x_{8} x_{12} x_{14}}{x_{3} x_{6}x_{9} x_{10} x_{15} x_{16}}
    \\
    &+\frac{y_{3} y_{12}x_{3}x_{7} x_{8} x_{11}  x_{13} }{x_{2}x_{4} x_{6} x_{9} x_{10}  x_{14}}+\frac{y_{3} y_{13}x_{3}x_{7}  x_{8} x_{11} x_{13}}{x_{2}x_{4} x_{6} x_{9}x_{10} x_{15}}
    +\frac{y_{4} y_{14}x_{2}x_{7} x_{8} x_{12} x_{14}}{x_{1}x_{3}  x_{6} x_{10} x_{15} x_{16}}
    +\frac{y_{5} y_{14}x_{2}x_{7} x_{8} x_{12} x_{14}}{x_{1}x_{3} x_{6} x_{11}  x_{15} x_{16}}.\\
\end{align*}
}
The explicit expressions for these variables are computed similarly to the $n=1$ case, confirming the general pattern.
\end{Ex}
As in the type $A_1$ case, i.e. cluster algebra on surfaces, the flip operation is an involution, since the reverse mutation sequence of a flip produces the original configuration. Moreover, flips from disjoint quadrilaterals commute.

The number of monomials (counting multiplicities) in the Laurent polynomial for each coordinate of $\Upsilon(V_1V_2V_3V_4)$ equals:
\Eq{
\cK_n:=\Upsilon(V_1V_2V_3V_4)|_{x_i=y_j=1}
} where $n$ stands for the rank. The following proposition follows directly from Proposition~\ref{prop:flip_properties}:
\begin{Prop}
\label{prop:unity_value}
For the general cluster realization of $\mathscr{A}_{\text{SL}_{n+1},\bbS}$ on a square, the value of the expansion formula at unity is:
\Eq{
    \cK_n = (2^{a_{n,1}}, 2^{a_{n,2}}, \dots, 2^{a_{n,n}})
}

where the positive integers $a_{i,j}$ ($i \geq j \geq 1$) satisfy the recurrence relations:
\Eq{
    a_{n,1} &= a_{n,n} = n;\nonumber\\
    a_{n+2,k} &= 1 + a_{n+1,k} + a_{n+1,k-1} - a_{n,k-1} \quad \mbox{for } n \geq 1, k \geq 2.
}
\end{Prop}
\begin{proof}From Proposition~\ref{prop:flip_properties}(a), we have $\cK_1 =(2)$ and $\cK_2 = (4,4) = (2^2, 2^2)$. 
	
	We shall prove the result by induction on $n$. For $n = 1$, then $\cK_1 = 2^{a_{1,1}}$ with $a_{1,1} = 1$. For $n = 2$, then $\cK_2 = (2^{a_{2,1}}, 2^{a_{2,2}})$ with $a_{2,1} = a_{2,2} = 2$, the claim holds. Suppose that the claim holds for $n-1 \geq 2$, then consider $\cK_n = (b_1, b_2, ..., b_n)$ in the general cluster realization of $\mathscr{A}_{SL_{n+1},\bbS}$.
	\begin{enumerate}[label=(\arabic*)]
		\item \textbf{If} $n = 2k$ \textbf{is even}, consider Proposition~\ref{prop:flip_properties}(b). Applying the inductive hypothesis, at all variables $x_i = y_j = 1$, we get:
		\begin{align*}
			X_{1,l} &= X_{2,l} = 2^{a_{n-1,l}} \text{ for } l = 1,2,...,(n-1);\\
			a_{n-1,1} &= a_{n-1,n-1} = n-1;\\
			x_{n+2+t}^{(1+t)} &= 2^{a_{n-2,t+1}}; \quad x_{2n-1-t}^{(1+t)} = 2^{a_{n-2,n-2-t}} \text{ for } t = 0,1,...,(k-2).
		\end{align*}
		Hence, we can compute the following:
		\begin{align*}
			b_1 &= x_{n+1}^{(1)} = X_{2,1} + X_{1,1} = 2 \cdot 2^{a_{n-1,1}} = 2 \cdot 2^{n-1} = 2^n;\\
			b_n &= x_{2n}^{(1)} = X_{2,n-1} + X_{1,n-1} = 2 \cdot 2^{a_{n-1,n-1}} = 2 \cdot 2^{n-1} = 2^n;\\
			b_{2+t} &= \frac{X_{1,1+t}X_{2,2+t}+X_{1,2+t}X_{2,1+t}}{x^{(1+t)}_{n+2+t}} = 2 \cdot 2^{a_{n-1,t+1}+a_{n-1,t+2}-a_{n-2,t+1}} = 2^{1+a_{n-1,t+1}+a_{n-1,t+2}-a_{n-2,t+1}};\\
			b_{n-1-t} &= \frac{X_{1,n-2-t}X_{2,n-1-t}+X_{1,n-1-t}X_{2,n-2-t}}{x^{(1+t)}_{2n-1-t}} = 2 \cdot 2^{a_{n-1,n-2-t}+a_{n-1,n-1-t}-a_{n-2,n-2-t}} \\
			&= 2^{1+a_{n-1,n-2-t}+a_{n-1,n-1-t}-a_{n-2,n-2-t}} \text{ for all } t = 0,1,...,(k-2)
		\end{align*}
which implies each $b_l = 2^{a_{n,l}}$, where $a_{n,1} = a_{n,n} = n$ and $a_{n,r} = 1 + a_{n-1,r} + a_{n-1,r-1} - a_{n-2,r-1}$ for all $r = 2,3,...,(n-1)$. This satisfies the claim.
		
		\item \textbf{If} $n = 2k+1$ \textbf{is odd}, consider Proposition~\ref{prop:flip_properties}(c). Applying the inductive hypothesis, at all variables $x_i = y_j = 1$, we get:
		\begin{align*}
			X_{1,l} &= X_{2,l} = 2^{a_{n-1,l}} \text{ for } l = 1,2,...,(n-1); \quad a_{n-1,1} = a_{n-1,n-1} = n-1;\\
			x_{k+1}^{(k)} &= 2^{a_{n-2,k}}; \quad x_{n+2+t}^{(1+t)} = 2^{a_{n-2,t+1}}; \quad x_{2n-2-t}^{(1+t)} = 2^{a_{n-2,n-2-t}} \text{ for } t = 0,1,...,(k-2).
		\end{align*}
		Hence, we can compute the following:
		\begin{align*}
			b_1 &= x_{n+1}^{(1)} = X_{2,1} + X_{1,1} = 2 \cdot 2^{a_{n-1,1}} = 2 \cdot 2^{n-1} = 2^n;\\
			b_n &= x_{2n-1}^{(1)} = X_{2,n-1} + X_{1,n-1} = 2 \cdot 2^{a_{n-1,n-1}} = 2 \cdot 2^{n-1} = 2^n;\\
			b_{k+1} &= \frac{X_{1,k}X_{2,k+1} + X_{1,k+1}X_{2,k}}{x_{k+1}^{(k)}} = 2 \cdot 2^{a_{n-1,k+1}+a_{n-1,k}-a_{n-2,k}} = 2^{1+a_{n-1,k+1}+a_{n-1,k}-a_{n-2,k}};\\
			b_{2+t} &= \frac{X_{1,1+t}X_{2,2+t}+X_{1,2+t}X_{2,1+t}}{x^{(1+t)}_{n+2+t}} = 2 \cdot 2^{a_{n-1,t+1}+a_{n-1,t+2}-a_{n-2,t+1}} = 2^{1+a_{n-1,t+1}+a_{n-1,t+2}-a_{n-2,t+1}};\\
			b_{n-1-t} &= \frac{X_{1,n-2-t}X_{2,n-1-t}+X_{1,n-1-t}X_{2,n-2-t}}{x^{(1+t)}_{2n-2-t}} = 2 \cdot 2^{a_{n-1,n-2-t}+a_{n-1,n-1-t}-a_{n-2,n-2-t}} \\
			&= 2^{1+a_{n-1,n-2-t}+a_{n-1,n-1-t}-a_{n-2,n-2-t}} \text{ for all } t = 0,1,...,(k-2)
		\end{align*}
which implies each $b_l = 2^{a_{n,l}}$, where $a_{n,1} = a_{n,n} = n$ and $a_{n,r} = 1 + a_{n-1,r} + a_{n-1,r-1} - a_{n-2,r-1}$ for all $r = 2,3,...,(n-1)$. This satisfies the claim.
	\end{enumerate}
Therefore, the inductive step works for all positive integer $n$. This concludes the proof.
\end{proof}
It turns out the recurrence relations can be solved explicitly, and as a consequence we obtain:
\begin{Thm}
\label{thm:exponents}
The exponents satisfy $a_{i,j} = j(i-j+1)$ for all $i \geq j \geq 1$. Therefore:
\Eq{
   \cK_n= (2^{n}, 2^{2n-2}, \dots, 2^{t(n-t+1)}, \dots, 2^n).}
\end{Thm}
\begin{proof}Firstly, we shall prove that $a_{i,j} = j(i-j+1)$ for all $i \geq j \geq 1$ by induction on $j$. Indeed, for $j = 1$, then $a_{i,1} = i = 1 \cdot (i-1+1)$, satisfied. Suppose that we already have $a_{i,j} = j(i-j+1)$ for all $i \geq j$ for $j \geq 1$, then for $l \geq j+2$, we have:
	\begin{align*}
		a_{j+1,j+1} &= j+1 = (j+1) \cdot 1;\\
		a_{l,j+1} &= 1 + a_{l-1,j+1} + a_{l-1,j} - a_{l-2,j} = 1 + a_{l-1,j+1} + j(l-j) - j(l-j-1) = 1 + j + a_{l-1,j+1}.
	\end{align*}
Hence by induction on $l \geq j+1$, we conclude that $a_{l,j+1} = (j+1)(l-j)$, and hence the claim also holds for all $j$, whence $a_{i,j} = j(i-j+1)$ for all $i \geq j \geq 1$.

As a consequence, as $a_{n,t}=t(n-t+1)$, we conclude $\cK_n = (2^{n}, 2^{2n-2}, \dots, 2^{t(n-t+1)}, \dots, 2^n)$.
\end{proof}
The following Proposition illustrates the relations in case we set several variables to be equal.
\begin{Prop}
\label{prop:x_y_z}
	Given the cluster realization of $\mathscr{A}_{\text{SL}_{n+1},\bbS}$ for the quadrilateral $V_1V_2V_3V_4$ with diagonal $V_1V_3$ and the corresponding $n^2+4n$ variables $x_i,y_j$ defined as above, we shall let parameters $x,y,z$ and then set all $x_i = x$ and $y_j = y$ for $j = 1,2,...,2n$ while $y_j = z$ for $j = 2n+1,2n+2,...,4n$. Then we get the expansion formula for diagonal $V_2V_4$ of the form $(F_1, F_2, ..., F_n)$ with:
\Eq{
	F_k = 2^{k(n+1-k)}\frac{yz}{x}
}
for $k = 1, \dots, n$. As a consequence, in case we set all variables to $x_i = x$ and $y_j = y$ (or $y = z$), we have:
\Eq{
	F_k = 2^{k(n+1-k)}\frac{y^2}{x}
}
for $k = 1, \dots, n$.
\end{Prop}

\begin{proof}
	We prove the claim by induction on $n$. For $n=1$, then $F_1 = \frac{y_1y_3+y_2y_4}{x_1} = \frac{2yz}{x}$. For $n=2$, we have:
\begin{align*}
	F_1 &= x'_3 = \frac{y_2y_5}{x_1} + \frac{y_4y_7}{x_2} + \frac{y_3y_5x_4}{x_1x_3} + \frac{y_4y_6x_4}{x_2x_3} = \frac{4yz}{x};\\
	F_2 &= x'_4 = \frac{y_3y_8}{x_1} + \frac{y_1y_6}{x_2} + \frac{y_2y_8x_3}{x_1x_4} + \frac{y_1y_7x_3}{x_2x_4} = \frac{4yz}{x}.
\end{align*}	
Suppose that the claim holds for $n-1 \geq 2$, we shall consider the different cases.
	\begin{enumerate}[label=(\arabic*)]
		\item \textbf{If} $n = 2k$ \textbf{is even}, consider Proposition~\ref{prop:flip_properties}(b). Applying the inductive hypothesis, at all variables $x_i = y_j = 1$, we get:
			\begin{align*}
			X_{1,l} &= 2^{l(n-l)}\frac{xy}{x} = 2^{l(n-l)} \cdot y; \quad X_{2,l} = 2^{l(n-l)}\frac{xz}{x} = 2^{l(n-l)} \cdot z \tab\text{ for } l = 1,2,...,(n-1);\\
			x_{n+2+t}^{(1+t)} &= 2^{(t+1)(n-2-t)}\frac{x^2}{x} = 2^{(t+1)(n-2-t)} \cdot x;\\
			x_{2n-1-t}^{(1+t)} &= 2^{(n-2-t)(t+1)}\frac{x^2}{x} = 2^{(n-2-t)(t+1)} \cdot x \tab \text{ for } t = 0,1,...,(k-2).
		\end{align*}
		Hence, we can compute the following:
		\begin{align*}
			F_1 &= x_{n+1}^{(1)} = \frac{y_{2n}X_{2,1}+y_{2n+1}X_{1,1}}{x_{n+1}} = 2\cdot 2^{n-1}\frac{yz}{x} = 2^n\frac{yz}{x};\\
			F_n &= x_{2n}^{(1)} = \frac{y_{1}X_{2,n-1}+y_{4n}X_{1,n-1}}{x_{2n}} = 2\cdot 2^{n-1}\frac{yz}{x} = 2^n\frac{yz}{x};\\
			F_{2+t} &= \frac{X_{1,1+t}X_{2,2+t}+X_{1,2+t}X_{2,1+t}}{x^{(1+t)}_{n+2+t}} = 2\cdot 2^{(t+1)(n-1-t)+(t+2)(n-2-t)-(t+1)(n-2-t)}\frac{yz}{x} = 2^{(t+2)(n-1-t)}\frac{yz}{x};\\
			F_{n-1-t} &= \frac{X_{1,n-2-t}X_{2,n-1-t}+X_{1,n-1-t}X_{2,n-2-t}}{x^{(1+t)}_{2n-1-t}} = 2\cdot 2^{(t+1)(n-1-t)+(t+2)(n-2-t)-(t+1)(n-2-t)}\frac{yz}{x} \\
			&= 2^{(t+2)(n-1-t)}\frac{yz}{x}
		\end{align*}
for all $t = 0,1,...,(k-2)$, which satisfies the claim.
		\item \textbf{If} $n = 2k+1$ \textbf{is odd}, consider Proposition~\ref{prop:flip_properties}(c). Applying the inductive hypothesis, at all variables $x_i = y_j = 1$, we get:
		\begin{align*}
			X_{1,l} &= 2^{l(n-l)}\frac{xy}{x} = 2^{l(n-l)} \cdot y; \quad X_{2,l} = 2^{l(n-l)}\frac{xz}{x} = 2^{l(n-l)} \cdot z \tab\text{ for } l = 1,2,...,(n-1);\\
			x_{k+1}^{(k)} &= 2^{k^2}\frac{x^2}{x} = 2^{k^2} \cdot x; \quad x_{n+2+t}^{(1+t)} = 2^{(t+1)(n-2-t)}\frac{x^2}{x} = 2^{(t+1)(n-2-t)} \cdot x;\\
			x_{2n-2-t}^{(1+t)} &= 2^{(n-2-t)(t+1)}\frac{x^2}{x} = 2^{(n-2-t)(t+1)} \cdot x \tab\text{ for } t = 0,1,...,(k-2).
		\end{align*}
		Hence, we can compute the following:
		\begin{align*}
			F_1 &= x_{n+1}^{(1)} = \frac{y_{2n}X_{2,1}+y_{2n+1}X_{1,1}}{x_{n+1}} = 2\cdot 2^{n-1}\frac{yz}{x} = 2^n\frac{yz}{x};\\
			F_n &= x_{2n-1}^{(1)} = \frac{y_{1}X_{2,n-1}+y_{4n}X_{1,n-1}}{x_{2n-1}} = 2\cdot 2^{n-1}\frac{yz}{x} = 2^n\frac{yz}{x};\\
			F_{k+1} &= \frac{X_{1,k}X_{2,k+1} + X_{1,k+1}X_{2,k}}{x_{k+1}^{(k)}} = 2\cdot 2^{k(k+1)+k(k+1)-k^2}\frac{yz}{x} = 2^{(k+1)^2}\frac{yz}{x};\\
			F_{2+t} &= \frac{X_{1,1+t}X_{2,2+t}+X_{1,2+t}X_{2,1+t}}{x^{(1+t)}_{n+2+t}} = 2\cdot 2^{(t+1)(n-1-t)+(t+2)(n-2-t)-(t+1)(n-2-t)}\frac{yz}{x} = 2^{(t+2)(n-1-t)}\frac{yz}{x};\\
			F_{n-1-t} &= \frac{X_{1,n-2-t}X_{2,n-1-t}+X_{1,n-1-t}X_{2,n-2-t}}{x^{(1+t)}_{2n-2-t}} = 2\cdot 2^{(t+1)(n-1-t)+(t+2)(n-2-t)-(t+1)(n-2-t)}\frac{yz}{x} \\
			&= 2^{(t+2)(n-1-t)}\frac{yz}{x}
		\end{align*}
for all $t = 0,1,...,(k-2)$, which satisfies the claim.
	\end{enumerate}
Therefore, in both cases, the claim holds for $n$. Then the consequence trivially holds by letting $y = z$. This concludes the proof.
\end{proof}

\subsection{Combinatorial interpretation via stair paths}
\label{subsec:combinatorial_interpretation}
While the labels \(x_i\) and \(y_j\) provide a basis for analysis, they are inadequate for deriving a closed-form recursion for \(\Upsilon(ABCD)\) for general \(n\), due to inconsistencies in their indexing under variation of \(n\). To formulate a universal solution, we make a transition to an infinite integer grid. This setting permits a generalized indexing scheme, obtained by a family of Laurent polynomial maps \(\psi_{t,k}: \{x_{i,j}\} \rightarrow \mathbb{Z}[x_{i,j}^{\pm 1}]\). Afterwards, we shall introduce the combinatorial constructions \emph{$n$-stair path} and \emph{$n$-reversed-stair path} to help solve in the particular values, namely $\psi_{1,k}$ and $\psi_{k,k}$, then give further connections among these combinatorial objects. As a result, note that the new notations in this section are independent of the labels $x_i$ and $y_j$ from the previous sections.

Consider an infinite integer grid with variable $x_{i,j}$ assigned to each point $(i,j) \in \mathbb{Z}^2$. For each $t,k\in\Z_{\geq 0}$ with $0\leq t\leq k+1$, define maps $\psi_{t,k}: \{x_{i,j}\} \rightarrow \mathbb{Z}[x_{i,j}^{\pm 1}]$ by:
\begin{equation}
\begin{aligned}
\psi_{0,k}(x_{i,j})&=\psi_{k+1,k}(x_{i,j}):=x_{i,j};\\
\psi_{1,1}(x_{i,j})&:=\frac{x_{i-1,j}x_{i+1,j}+x_{i,j-1}x_{i,j+1}}{x_{i,j}};\\
\psi_{t+1,k+2}(x_{i,j})&:=\frac{\psi_{t+1,k+1}(x_{i-1,j})\psi_{t,k+1}(x_{i+1,j})+\psi_{t,k+1}(x_{i,j-1})\psi_{t+1,k+1}(x_{i,j+1})}{\psi_{t,k}(x_{i,j})}.
\end{aligned}
\label{eq:crfs}
\end{equation}
Using Proposition~\ref{prop:flip_properties}, $\psi_{t,k}(x_{i,j})$ equals the $t$-th coordinate of the index of $V_2V_4$ diagonal of the expansion after flipping the square $A_{i,j}^{t,k}B_{i,j}^{t,k}C_{i,j}^{t,k}D_{i,j}^{t,k}$, where we place the vertices on the infinite grid as
\begin{gather*}
    A_{i,j}^{t,k} = (i+t,j-t+k+1), \quad B_{i,j}^{t,k} = (i+t,j-t), \\
    C_{i,j}^{t,k} = (i+t-k-1,j-t), \quad D_{i,j}^{t,k} = (i+t-k-1,j-t+k+1).
\end{gather*}
In particular, by the Laurent phenomenon, each $\psi_{t,k}(x_{i,j})$ is a Laurent polynomial in the variables $x_{i,j}$, hence the maps are well-defined. When all variables $x_{i,j} = x$, applying Proposition~\ref{prop:x_y_z} yields:
\Eq{
	\psi_{t,k}(x) = 2^{t(k+1-t)}x.
}
\begin{Def}
\label{def:n_stairs}
Given an infinite integer grid with variables $x_{i,j}$ assigned at each point $(i,j)$, an \emph{$n$-stair path of} $(i,j)$ is an alternating sequence of left and up segments with total length $n$ satisfying:
\begin{enumerate}[label=(\arabic*)]
    \item The first segment contains $(i,j)$ as its second point;
    \item Consecutive segments connect via $\sqrt{2}$-length diagonals:
    \begin{enumerate}
        \item The head of a left segment is $\sqrt{2}$ northeast of the tail of the next up segment;
        \item The head of an up segment is $\sqrt{2}$ southwest of the tail of next left segment;
    \end{enumerate}
    \item The \emph{ending point} $E$ is one step away from the head of the last segment in the sequence. Equivalently, $E$ is the point reached by starting at $(i,j)$ and following the alternating left/up path according to the connection rules (2), until the sum of segment lengths equals $n$.
\end{enumerate}
For any integer point $(k,l) \neq (i,j)$ such that $k \leq i$ and $l \geq j$, by a \emph{stair path from $(i,j)$ to $(k,l)$}, we mean an $s$-stair path of $(i,j)$ (where $s = (i-k) + (l-j)$) such that $E = (k,l)$.  
\end{Def}
See Figure~\ref{fig:stairs} for examples. In both examples, the ending point $E = (i-7, j+7)$ is marked in green and each example is a stair path from $(i,j)$ to $E$. The left example has segment lengths $3$, $4$, $4$, $3$ (sum $14$); the right has lengths $3$, $3$, $2$, $1$, $2$, $3$ (sum $14$).
\begin{figure}[H]
\centering
\begin{tikzpicture}[scale=0.8,
    node/.style={circle, draw, fill=white, inner sep=1.5pt},
    arrow/.style={-Stealth, thick},
    gridline/.style={gray, thin},
    pathline/.style={thick, blue}
]
	\draw[gridline] (0,0) grid (9,8);
	\draw[gridline] (11,0) grid (20,8);
	\draw[green,thin,dashed] (10,0) -- (10,8);
	
	\node[node, fill=red!20] at (8,1) {};
    \node[above right] at (8,1) {$(i,j)$};
    \node[node, fill=red!20] at (19,1) {};
    \node[above right] at (19,1) {$(i,j)$};
    
    \draw[pathline, arrow] (9,1) -- (6,1);
    \draw[pathline, arrow] (5,0) -- (5,4);
    \draw[pathline, arrow] (6,5) -- (2,5);
    \draw[pathline, arrow] (1,4) -- (1,7);
    
    \node[node, fill=green!20] at (1,8) {};
    \node[above right] at (1,8) {$E$};
    
    \draw[pathline, arrow] (19,0) -- (19,3);
    \draw[pathline, arrow] (20,4) -- (17,4);
    \draw[pathline, arrow] (16,3) -- (16,5);
    \draw[pathline, arrow] (17,6) -- (16,6);
    \draw[pathline, arrow] (15,5) -- (15,7);
    \draw[pathline, arrow] (16,8) -- (13,8);
    
    \node[node, fill=green!20] at (12,8) {};
    \node[above right] at (12,8) {$E$};
    
\end{tikzpicture}
\caption{Examples of $14$-stair paths starting from $(i,j)$}
\label{fig:stairs}
\end{figure}
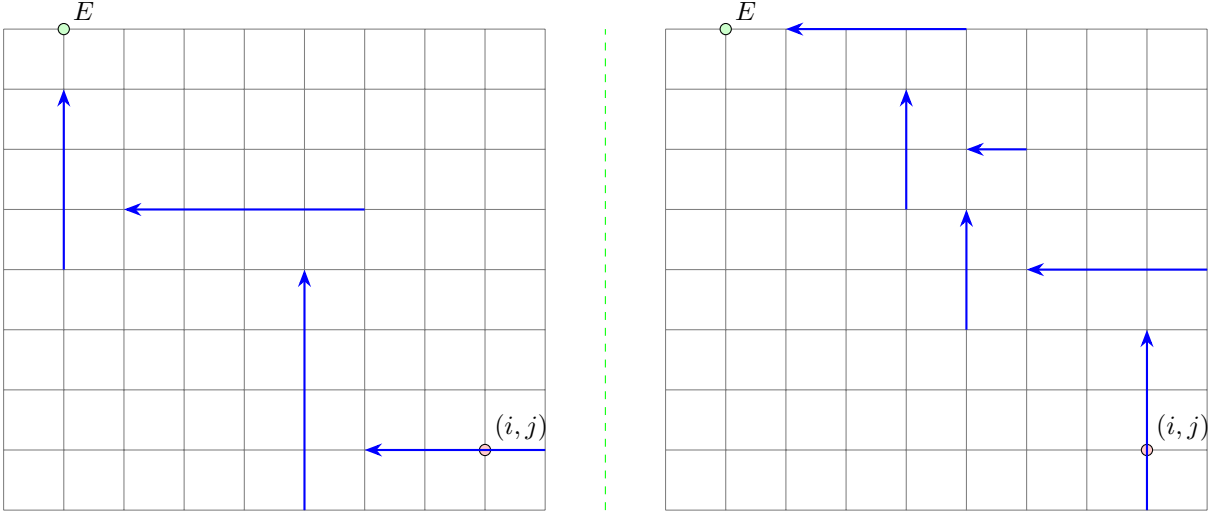
For two integer points $a,b$ with the same horizontal or vertical coordinate with the corresponding segment arrow $a \rightarrow b$, regardless of the direction (either left ($\leftarrow$), right ($\rightarrow$), up ($\uparrow$), or down ($\downarrow$)), assign a \emph{weight} $\frac{x_a}{x_b}$ of Laurent monomial to that segment. It is clear that for consecutive points sharing the same direction, the product of weights equals the weight of the single segment connecting all points. Let $f_n(x_{i,j}) := \psi_{1,n}(x_{i,j})$ for all $n \geq 0$, then from the recurrence (\ref{eq:crfs}), we have: 
\Eq{
    f_0(x_{i,j}) = x_{i,j}; \quad f_{n+1}(x_{i,j}) = \frac{f_n(x_{i-1,j})x_{i+1,j} + x_{i,j-1}f_n(x_{i,j+1})}{x_{i,j}}.
}
Define $c_{i,j} := \frac{x_{i+1,j}}{x_{i,j}}$ and $r_{i,j} := \frac{x_{i,j-1}}{x_{i,j}}$, then the recurrence (\ref{eq:crfs}) implies:
\Eq{
    f_0(x_{i,j}) = x_{i,j}; \quad f_{n+1}(x_{i,j}) = c_{i,j}f_n(x_{i-1,j}) + r_{i,j}f_n(x_{i,j+1}). \label{eq:recf}
}
This leads to a combinatorial interpretation:
\begin{Prop}
\label{prop:stairs_combinatorial}
The value $f_n(x_{i,j}) = \psi_{1,n}(x_{i,j})$ equals the sum of weights of all possible $n$-stairs paths starting from $(i,j)$, where the weight of each path is the product of its segment weights multiplied by the variable at its ending point $E$. More explicitly, the general formula can be written as:
\begin{equation}
\begin{split}
    f_n(x_{i,j}) &= \psi_{1,n}(x_{i,j}) \\
    &= \sum_{\substack{(l_1+1),l_i,u_i,(u_k+1),k \in \mathbb{N}: \\ \sum_{i=1}^{k}(l_i+u_i) = n}} \left(x_{i-\sum_{s=1}^{k}l_s,j+\sum_{t=1}^{k}u_t} \prod_{v=1}^{k} \frac{x_{i+1-\sum_{s=1}^{v-1}l_s,j+\sum_{t=1}^{v-1}u_t}}{x_{i+1-\sum_{s=1}^{v}l_s,j+\sum_{t=1}^{v-1}u_t}} \frac{x_{i-\sum_{s=1}^{v}l_s,j-1+\sum_{t=1}^{v-1}u_t}}{x_{i-\sum_{s=1}^{v}l_s,j-1+\sum_{t=1}^{v}u_t}}\right).
\end{split}
\label{eq:fn_psi}
\end{equation}
Here and thereafter, the notation for the summation index $(l_1+1)\in\N$ etc. is equivalent to adding up all nonnegative index $l_1\in\N\cup\{0\}$.
\end{Prop}

\begin{proof}
We shall prove by induction on $n$ that:
\Eq{
    f_n(x_{i,j}) = \sum_{\tau_1, \tau_2, \dots, \tau_n \in \{r,c\}} (\tau_{1})_{i,j}(\tau_{2})_{\sigma_1(i,j)} \cdots (\tau_{n})_{(\sigma_{n-1} \circ \cdots \circ \sigma_1)(i,j)} x_{(\sigma_{n} \circ \cdots \circ \sigma_1)(i,j)}
    \label{eq:paramet}
}
where $\sigma_k(i',j') = (i'-1,j')$ if $\tau_k = c$, and $\sigma_k(i',j') = (i',j'+1)$ if $\tau_k = r$ for $1 \leq k \leq n$. 

For $n=1$, then $f_1(x_{i,j}) = r_{i,j}x_{i,j+1} + c_{i,j}x_{i-1,j}$, satisfied.

Suppose that the formula holds for $n \geq 1$, then consider the $(n+1)$ case, from recurrence (\ref{eq:recf}), we can write:
\begin{align*}
	f_{n+1}(x_{i,j}) &= c_{i,j}f_n(x_{i-1,j}) + r_{i,j}f_n(x_{i,j+1}) \\
	&= c_{i,j}.\sum_{\tau_1, \tau_2, \dots, \tau_n \in \{r,c\}} (\tau_{1})_{i-1,j}(\tau_{2})_{\sigma_1(i-1,j)} \cdots (\tau_{n})_{(\sigma_{n-1} \circ \cdots \circ \sigma_1)(i-1,j)} x_{(\sigma_{n} \circ \cdots \circ \sigma_1)(i-1,j)} \\
	&+ r_{i,j}\sum_{\tau_1, \tau_2, \dots, \tau_n \in \{r,c\}} (\tau_{1})_{i,j+1}(\tau_{2})_{\sigma_1(i,j+1)} \cdots (\tau_{n})_{(\sigma_{n-1} \circ \cdots \circ \sigma_1)(i,j+1)} x_{(\sigma_{n} \circ \cdots \circ \sigma_1)(i,j+1)} \\
	&= c_{i,j}.\sum_{\tau_2, \tau_3, \dots, \tau_{n+1} \in \{r,c\}} (\tau_{2})_{i-1,j}(\tau_{3})_{\sigma_2(i-1,j)} \cdots (\tau_{n+1})_{(\sigma_{n} \circ \cdots \circ \sigma_2)(i-1,j)} x_{(\sigma_{n+1} \circ \cdots \circ \sigma_2)(i-1,j)} \\
	&+ r_{i,j}\sum_{\tau_2, \tau_3, \dots, \tau_{n+1} \in \{r,c\}} (\tau_{2})_{i,j+1}(\tau_{3})_{\sigma_2(i,j+1)} \cdots (\tau_{n+1})_{(\sigma_{n} \circ \cdots \circ \sigma_2)(i,j+1)} x_{(\sigma_{n+1} \circ \cdots \circ \sigma_2)(i,j+1)} \\
	&= \sum_{\tau_1 = c; \tau_2, \dots, \tau_{n+1} \in \{r,c\}} (\tau_{1})_{i,j}(\tau_{2})_{\sigma_1(i,j)} \cdots (\tau_{n+1})_{(\sigma_{n} \circ \cdots \circ \sigma_1)(i,j)} x_{(\sigma_{n+1} \circ \cdots \circ \sigma_1)(i,j)} \\
	&+ \sum_{\tau_1 = r; \tau_2, \dots, \tau_{n+1} \in \{r,c\}} (\tau_{1})_{i,j}(\tau_{2})_{\sigma_1(i,j)} \cdots (\tau_{n+1})_{(\sigma_{n} \circ \cdots \circ \sigma_1)(i,j)} x_{(\sigma_{n+1} \circ \cdots \circ \sigma_1)(i,j)} \\
	&= \sum_{\tau_1, \tau_2, \dots, \tau_{n+1} \in \{r,c\}} (\tau_{1})_{i,j}(\tau_{2})_{\sigma_1(i,j)} \cdots (\tau_{n+1})_{(\sigma_{n} \circ \cdots \circ \sigma_1)(i,j)} x_{(\sigma_{n+1} \circ \cdots \circ \sigma_1)(i,j)}
\end{align*}
which completes the inductive proof.

Each $c_{i',j'}$ corresponds to the weight of the left unit segment $(i'+1,j') \rightarrow (i',j')$, while $r_{i',j'}$ corresponds to the weight of the up unit segment $(i',j'-1) \rightarrow (i',j')$, then for each choice of $(\tau_1,\tau_2,...,\tau_n)$, the term of the summation above represents an $n$-stairs path with the variable $x_{(\sigma_{n} \circ \cdots \circ \sigma_1)(i,j)}$ representing the label of the ending vertex $E$. This establishes the combinatorial interpretation. The other formula in the statement is the parametrization for equation (\ref{eq:paramet}), with each value $(k,l_i,u_i)$ corresponding to an $n$-stair path.
\end{proof}
For small values we compute:
\begin{align*}
    f_1(x_{i,j}) &= \psi_{1,1}(x_{i,j}) = \frac{x_{i-1,j}x_{i+1,j} + x_{i,j-1}x_{i,j+1}}{x_{i,j}}, \\
    f_2(x_{i,j}) &= \psi_{1,2}(x_{i,j}) = \frac{x_{i-2,j}x_{i+1,j}}{x_{i-1,j}} + \frac{x_{i,j-1}x_{i,j+2}}{x_{i,j+1}} + \frac{x_{i-1,j+1}x_{i,j}x_{i+1,j+1}}{x_{i,j}x_{i,j+1}} + \frac{x_{i-1,j-1}x_{i-1,j+1}x_{i+1,j}}{x_{i-1,j}x_{i,j}}.
\end{align*}
Figure~\ref{fig:f3} illustrates the calculation of $f_3(x_{i,j})$ using Proposition~\ref{prop:stairs_combinatorial}, yielding:
\begin{align*}
    \psi_{1,3}(x_{i,j}) &= \frac{x_{i,j-1}x_{i,j+3}}{x_{i,j+2}} + \frac{x_{i,j-1}x_{i+1,j+2}x_{i-1,j+2}}{x_{i,j+1}x_{i,j+2}} + \frac{x_{i,j-1}x_{i+1,j+1}x_{i-1,j}x_{i-1,j+2}}{x_{i,j}x_{i,j+1}x_{i-1,j+1}} + \frac{x_{i,j-1}x_{i+1,j+1}x_{i-2,j+1}}{x_{i,j}x_{i-1,j+1}} \\
    &\quad + \frac{x_{i+1,j}x_{i-3,j}}{x_{i-2,j}} + \frac{x_{i+1,j}x_{i-2,j-1}x_{i-2,j+1}}{x_{i-1,j}x_{i-2,j}} + \frac{x_{i+1,j}x_{i-1,j-1}x_{i,j+1}x_{i-2,j+1}}{x_{i,j}x_{i-1,j}x_{i-1,j+1}} + \frac{x_{i+1,j}x_{i-1,j-1}x_{i-1,j+2}}{x_{i,j}x_{i-1,j+1}}.
\end{align*}
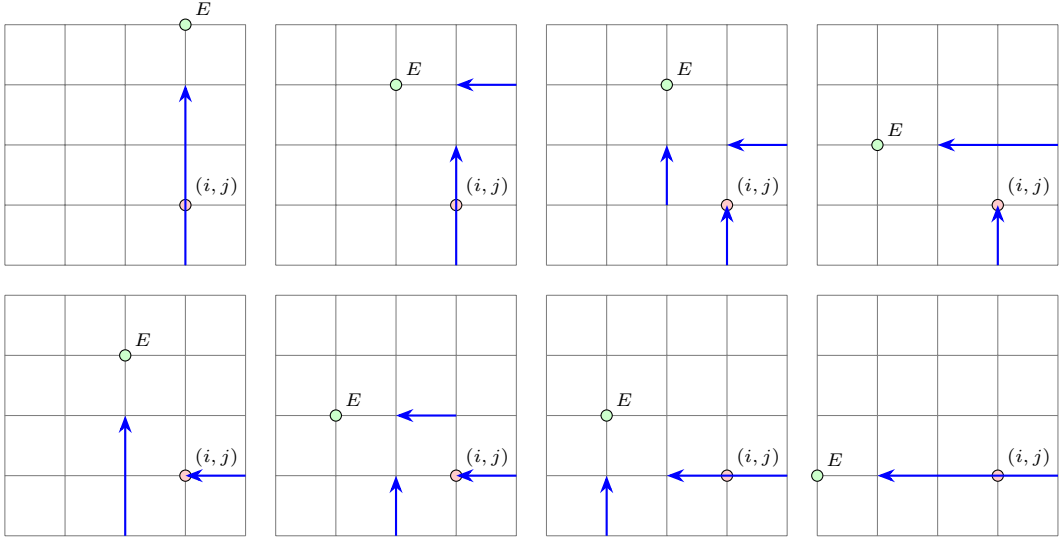
\begin{figure}[H]
\centering
\begin{tikzpicture}[scale=0.8,
    node/.style={circle, draw, fill=white, inner sep=1.5pt},
    arrow/.style={-Stealth, thick},
    gridline/.style={gray, thin},
    pathline/.style={thick, blue}
]

\foreach \row in {0,1} {
    \foreach \col in {0,1,2,3} {
        \begin{scope}[shift={(\col*4.5, -\row*4.5)}]
        
        \draw[gridline] (0,0) grid (4,4);
        
        \node[node, fill=red!20] at (3,1) {};
        \node[above right, font=\scriptsize] at (3,1) {$(i,j)$};
        
        \pgfmathtruncatemacro{\pathindex}{\row*4 + \col + 1}
        
        \ifnum\pathindex=1 
            \draw[pathline, arrow] (3,0) -- (3,3);
        		\node[node, fill=green!20] at (3,4) {};
        		\node[above right, font=\scriptsize] at (3,4) {$E$};
            
        \else\ifnum\pathindex=2 
            \draw[pathline, arrow] (3,0) -- (3,2);
            \draw[pathline, arrow] (4,3) -- (3,3);
        		\node[node, fill=green!20] at (2,3) {};
        		\node[above right, font=\scriptsize] at (2,3) {$E$};
            
        \else\ifnum\pathindex=3 
            \draw[pathline, arrow] (3,0) -- (3,1);
            \draw[pathline, arrow] (4,2) -- (3,2);
            \draw[pathline, arrow] (2,1) -- (2,2);
        		\node[node, fill=green!20] at (2,3) {};
        		\node[above right, font=\scriptsize] at (2,3) {$E$};
        		            
        \else\ifnum\pathindex=4 
            \draw[pathline, arrow] (3,0) -- (3,1);
            \draw[pathline, arrow] (4,2) -- (2,2);
        		\node[node, fill=green!20] at (1,2) {};
        		\node[above right, font=\scriptsize] at (1,2) {$E$};
            
        \else\ifnum\pathindex=5 
            \draw[pathline, arrow] (4,1) -- (3,1);
            \draw[pathline, arrow] (2,0) -- (2,2);
        		\node[node, fill=green!20] at (2,3) {};
        		\node[above right, font=\scriptsize] at (2,3) {$E$};
            
        \else\ifnum\pathindex=6 
            \draw[pathline, arrow] (4,1) -- (3,1);
            \draw[pathline, arrow] (2,0) -- (2,1);
            \draw[pathline, arrow] (3,2) -- (2,2);
        		\node[node, fill=green!20] at (1,2) {};
        		\node[above right, font=\scriptsize] at (1,2) {$E$};
        		            
        \else\ifnum\pathindex=7 
            \draw[pathline, arrow] (4,1) -- (2,1);
            \draw[pathline, arrow] (1,0) -- (1,1);
        		\node[node, fill=green!20] at (1,2) {};
        		\node[above right, font=\scriptsize] at (1,2) {$E$};
            
        \else\ifnum\pathindex=8 
            \draw[pathline, arrow] (4,1) -- (1,1);
        		\node[node, fill=green!20] at (0,1) {};
        		\node[above right, font=\scriptsize] at (0,1) {$E$};
        \fi\fi\fi\fi\fi\fi\fi\fi
        
        \end{scope}
    }
}
\end{tikzpicture}
\caption{All $2^3 = 8$ possible $3$-stair paths starting from $(i,j)$}
\label{fig:f3}
\end{figure}
Similarly, we can also define the reversed version:
\begin{Def}
\label{def:n_rev_stairs}
An \emph{$n$-reversed-stair path} is defined similarly by reflection: an alternating sequence of right and down segments with total length $n$, where:
\begin{enumerate}[label=(\arabic*)]
    \item The first segment contains $(i,j)$ as its second point;
    \item Consecutive segments connect via $\sqrt{2}$-length diagonals:
    \begin{enumerate}
        \item The head of a right segment ends $\sqrt{2}$ southwest of the tail of the next down segment;
        \item The head of a down segment ends $\sqrt{2}$ northeast of the tail of the next right segment;
    \end{enumerate}
    
    \item The \emph{ending point} $E$ is one step away from the head of the last segment in the sequence. Equivalently, $E$ is the point reached by starting at $(i,j)$ and following the alternating right/down path according to the connection rules (2), until the sum of segment lengths equals $n$.
\end{enumerate}
For any integer point $(k,l) \neq (i,j)$ such that $k \geq i$ and $l \leq j$, by a \emph{reversed-stair path from $(i,j)$ to $(k,l)$}, we mean a $r$-reversed-stair path of $(i,j)$ (where $r = (k-i) + (j-l)$) such that $E = (k,l)$.
\end{Def}
The notion allows us to compute the following case:
\begin{Prop}
\label{prop:reversed_stairs}
$\psi_{n,n}(x_{i,j})$ equals the sum of weights of all possible $n$-reversed-stair paths starting from $(i,j)$, where the weight of each path is the product of its segment weights multiplied by the variable at its ending point $E$. More explicitly, the general formula can be written as:
\begin{equation}
    \psi_{n,n}(x_{i,j}) = \sum_{\substack{(r_1+1),r_i,d_i,(d_k+1),k \in \mathbb{N}: \\ \sum_{i=1}^{k}(r_i+d_i) = n}} \left(x_{i+\sum_{s=1}^{k}r_s,j-\sum_{t=1}^{k}d_t} \prod_{v=1}^{k} \frac{x_{i-1+\sum_{s=1}^{v-1}r_s,j-\sum_{t=1}^{v-1}d_t}}{x_{i-1+\sum_{s=1}^{v}r_s,j-\sum_{t=1}^{v-1}d_t}} \frac{x_{i+\sum_{s=1}^{v}r_s,j+1-\sum_{t=1}^{v-1}d_t}}{x_{i+\sum_{s=1}^{v}r_s,j+1-\sum_{t=1}^{v}d_t}}\right).
\end{equation}
\end{Prop}

\begin{proof}
Setting $t = k+1$, we obtain the recurrence:
\Eq{
    \psi_{k+2,k+2}(x_{i,j}) = \frac{x_{i-1,j}\psi_{k+1,k+1}(x_{i+1,j}) + \psi_{k+1,k+1}(x_{i,j-1})x_{i,j+1}}{x_{i,j}},
}
which admits a similar combinatorial interpretation as Proposition~\ref{prop:stairs_combinatorial}.
\end{proof}
Figure~\ref{fig:f3_reversed} illustrates the calculation of $\psi_{3,3}(x_{i,j})$:
\begin{align*}
    \psi_{3,3}(x_{i,j}) &= \frac{x_{i,j+1}x_{i,j-3}}{x_{i,j-2}} + \frac{x_{i,j+1}x_{i-1,j-2}x_{i+1,j-2}}{x_{i,j-1}x_{i,j-2}} + \frac{x_{i,j+1}x_{i-1,j-1}x_{i+1,j}x_{i+1,j-2}}{x_{i,j}x_{i,j-1}x_{i+1,j-1}} + \frac{x_{i,j+1}x_{i-1,j-1}x_{i+2,j-1}}{x_{i,j}x_{i+1,j-1}} \\
    &\quad + \frac{x_{i-1,j}x_{i+3,j}}{x_{i+2,j}} + \frac{x_{i-1,j}x_{i+2,j+1}x_{i+2,j-1}}{x_{i+1,j}x_{i+2,j}} + \frac{x_{i-1,j}x_{i+1,j+1}x_{i,j-1}x_{i+2,j-1}}{x_{i,j}x_{i+1,j}x_{i+1,j-1}} + \frac{x_{i-1,j}x_{i+1,j+1}x_{i+1,j-2}}{x_{i,j}x_{i+1,j-1}}.
\end{align*}

\begin{figure}[H]
\centering
\begin{tikzpicture}[scale=0.8,
    node/.style={circle, draw, fill=white, inner sep=1.5pt},
    arrow/.style={-Stealth, thick},
    gridline/.style={gray, thin},
    pathline/.style={thick, blue}
]

\foreach \row in {0,1} {
    \foreach \col in {0,1,2,3} {
        \begin{scope}[shift={(\col*4.5, -\row*4.5)}]
        
        \draw[gridline] (0,0) grid (4,4);
        
        \node[node, fill=red!20] at (1,3) {};
        \node[above left, font=\scriptsize] at (1,3) {$(i,j)$};
        
        \pgfmathtruncatemacro{\pathindex}{\row*4 + \col + 1}
        
        \ifnum\pathindex=1 
            \draw[pathline, arrow] (0,3) -- (3,3);
        		\node[node, fill=green!20] at (4,3) {};
        		\node[above left, font=\scriptsize] at (4,3) {$E$};
            
        \else\ifnum\pathindex=2 
            \draw[pathline, arrow] (0,3) -- (2,3);
            \draw[pathline, arrow] (3,4) -- (3,3);
        		\node[node, fill=green!20] at (3,2) {};
        		\node[above left, font=\scriptsize] at (3,2) {$E$};
            
        \else\ifnum\pathindex=3 
            \draw[pathline, arrow] (0,3) -- (1,3);
            \draw[pathline, arrow] (2,4) -- (2,3);
            \draw[pathline, arrow] (1,2) -- (2,2);
        		\node[node, fill=green!20] at (3,2) {};
        		\node[above left, font=\scriptsize] at (3,2) {$E$};
        		            
        \else\ifnum\pathindex=4 
            \draw[pathline, arrow] (0,3) -- (1,3);
            \draw[pathline, arrow] (2,4) -- (2,2);
        		\node[node, fill=green!20] at (2,1) {};
        		\node[above left, font=\scriptsize] at (2,1) {$E$};
            
        \else\ifnum\pathindex=5 
            \draw[pathline, arrow] (1,4) -- (1,3);
            \draw[pathline, arrow] (0,2) -- (2,2);
        		\node[node, fill=green!20] at (3,2) {};
        		\node[above left, font=\scriptsize] at (3,2) {$E$};
            
        \else\ifnum\pathindex=6 
            \draw[pathline, arrow] (1,4) -- (1,3);
            \draw[pathline, arrow] (0,2) -- (1,2);
            \draw[pathline, arrow] (2,3) -- (2,2);
        		\node[node, fill=green!20] at (2,1) {};
        		\node[above left, font=\scriptsize] at (2,1) {$E$};
        		            
        \else\ifnum\pathindex=7 
            \draw[pathline, arrow] (1,4) -- (1,2);
            \draw[pathline, arrow] (0,1) -- (1,1);
        		\node[node, fill=green!20] at (2,1) {};
        		\node[above left, font=\scriptsize] at (2,1) {$E$};
            
        \else\ifnum\pathindex=8 
            \draw[pathline, arrow] (1,4) -- (1,1);
        		\node[node, fill=green!20] at (1,0) {};
        		\node[above left, font=\scriptsize] at (1,0) {$E$};
        \fi\fi\fi\fi\fi\fi\fi\fi
        
        \end{scope}
    }
}
\end{tikzpicture}
\caption{All $2^3 = 8$ possible $3$-reversed-stair paths starting from $(i,j)$.}
\label{fig:f3_reversed}
\end{figure}
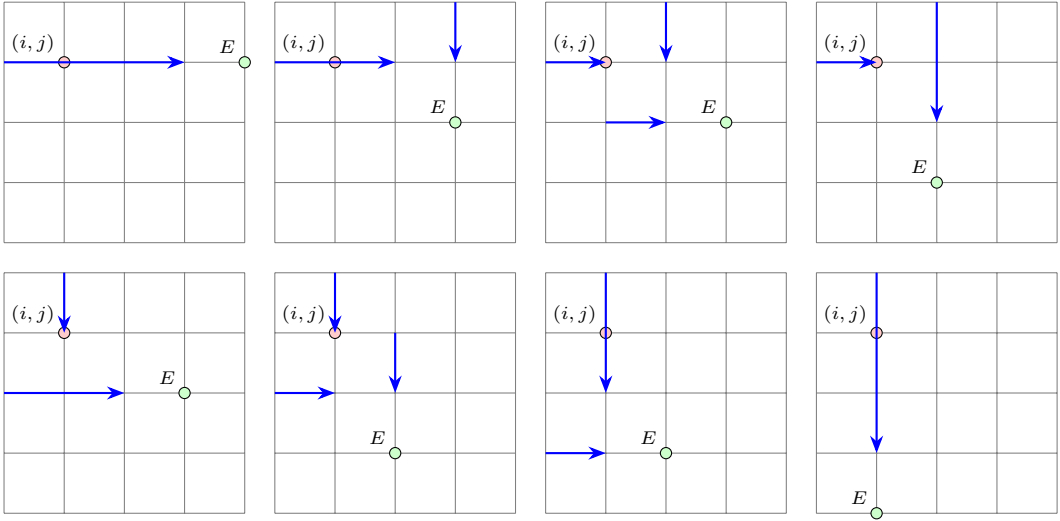
Following the idea of the proof of Proposition~\ref{prop:stairs_combinatorial}, we can write the formula in the stair path structure:
\begin{Prop}
\label{prop:stair_path_structure_time}
The general recurrence:
\Eq{
    \psi_{t+1,k+2}(x_{i,j}) = \frac{\psi_{t+1,k+1}(x_{i-1,j})\psi_{t,k+1}(x_{i+1,j}) + \psi_{t,k+1}(x_{i,j-1})\psi_{t+1,k+1}(x_{i,j+1})}{\psi_{t,k}(x_{i,j})}
}
leads to the formula:
\begin{equation}
\begin{aligned}
    \psi_{t,k}(x_{i,j}) &= \sum_{\substack{(l_1+1),l_i,u_i,(u_g+1),g \in \mathbb{N}: \\ \sum_{i=1}^{g}(l_i+u_i) = k-t+1}} \left(x_{i-\sum_{s=1}^{g}l_s,j+\sum_{h=1}^{g}u_h} \prod_{v=1}^{g} \frac{\psi_{t-1,k-1-\sum_{s=1}^{v-1}l_s-\sum_{h=1}^{v-1}u_h}(x_{i+1-\sum_{s=1}^{v-1}l_s,j+\sum_{h=1}^{v-1}u_h})}{\psi_{t-1,k-1-\sum_{s=1}^{v}l_s-\sum_{h=1}^{v-1}u_h}(x_{i+1-\sum_{s=1}^{v}l_s,j+\sum_{h=1}^{v-1}u_h})} \right. \\
    &\quad \left. \cdot \frac{\psi_{t-1,k-1-\sum_{s=1}^{v}l_s-\sum_{h=1}^{v-1}u_h}(x_{i-\sum_{s=1}^{v}l_s,j-1+\sum_{h=1}^{v-1}u_h})}{\psi_{t-1,k-1-\sum_{s=1}^{v}l_s-\sum_{h=1}^{v}u_h}(x_{i-\sum_{s=1}^{v}l_s,j-1+\sum_{h=1}^{v}u_h})} \right).
\end{aligned}
\label{eq:general_stair_str}
\end{equation}
\end{Prop}
Define two combinatorial functions:
\begin{equation}
\begin{aligned}
    \cS(x_{i,j},x_{k,l}) &:= \text{sum of weights of all stair paths from } (i,j) \text{ to } (k,l), \\
    \cR(x_{i,j},x_{k,l}) &:= \text{sum of weights of all reversed-stair paths from } (i,j) \text{ to } (k,l).
\end{aligned}
\end{equation}
We can add the notations $\cS(x_{i,j},x_{i,j}) := x_{i,j}$ and $\cS(x_{i,j},x_{k,l}) := 0$ if $k > i$ or $l < j$, and $\cR(x_{i,j},x_{i,j}) := x_{i,j}$ and $\cR(x_{i,j},x_{k,l}) := 0$ if $k < i$ or $l > j$. The number of stair paths from $(i,j)$ to $(k,l)$ is $\binom{(i-k)+(l-j)}{i-k}$ when $k \leq i$ and $l \geq j$, and the number of reversed-stair paths is $\binom{(k-i)+(j-l)}{k-i}$ when $k \geq i$ and $l \leq j$. Similar to the way of parametrizing in the formulas of Proposition~\ref{prop:stairs_combinatorial} and~\ref{prop:reversed_stairs}, we can rewrite both functions as:
\begin{Prop}
\label{prop:cS-cR}
The two combinatorial functions defined above are of the form:
\begin{enumerate}[label=(\arabic*)]
    \item For any integer points $(i,j), (i',j')$ such that $i' \leq i$ and $j' \geq j$, we have:
    \begin{equation}
        \cS(x_{i,j},x_{i',j'}) = x_{i',j'} \cdot \sum_{\substack{(l_1+1),l_i,u_i,(u_k+1),k \in \mathbb{N}: \\ \sum_{s=1}^{k}l_s = i-i';\\ \sum_{t=1}^{k}u_t = j'-j}}  \prod_{v=1}^{k} \frac{x_{i+1-\sum_{s=1}^{v-1}l_s,j+\sum_{t=1}^{v-1}u_t}}{x_{i+1-\sum_{s=1}^{v}l_s,j+\sum_{t=1}^{v-1}u_t}} \frac{x_{i-\sum_{s=1}^{v}l_s,j-1+\sum_{t=1}^{v-1}u_t}}{x_{i-\sum_{s=1}^{v}l_s,j-1+\sum_{t=1}^{v}u_t}}.
    \label{eq:cS}
    \end{equation}
    
    \item For any integer points $(i,j), (i',j')$ such that $i' \geq i$ and $j' \leq j$, we have:
    \begin{equation}
        \cR(x_{i,j},x_{i',j'}) = x_{i',j'} \cdot \sum_{\substack{(r_1+1),r_i,d_i,(d_k+1),k \in \mathbb{N}: \\ \sum_{s=1}^{k}r_s = i'-i;\\ \sum_{t=1}^{k}d_t = j-j'}} \prod_{v=1}^{k} \frac{x_{i-1+\sum_{s=1}^{v-1}r_s,j-\sum_{t=1}^{v-1}d_t}}{x_{i-1+\sum_{s=1}^{v}r_s,j-\sum_{t=1}^{v-1}d_t}} \frac{x_{i+\sum_{s=1}^{v}r_s,j+1-\sum_{t=1}^{v-1}d_t}}{x_{i+\sum_{s=1}^{v}r_s,j+1-\sum_{t=1}^{v}d_t}}.
    \label{eq:cR}
    \end{equation}
\end{enumerate}
\end{Prop}

\begin{Prop}
\label{prop:sum_decomposition}
For all non-negative integers $k,l$ and $i,j \in \mathbb{Z}$:
\begin{equation}
\begin{aligned}
    \psi_{1,k+l}(x_{i,j}) &= \sum_{t=0}^{l} \cS(x_{i,j},x_{i-t,j+l-t}) \psi_{1,k}(x_{i-t,j+l-t}), \\
    \psi_{k+l,k+l}(x_{i,j}) &= \sum_{t=0}^{l} \cR(x_{i,j},x_{i+t,j-l+t}) \psi_{k,k}(x_{i+t,j-l+t}).
\end{aligned}
\end{equation}
\end{Prop}

\begin{proof}
From the definition above on the function $\cS$ and Proposition~\ref{prop:stairs_combinatorial}, we have $\cS(x_{i,j},x_{i-t,j+l-t})$ is the sum of weights of all stair paths from $(i,j)$ to $(i-t,j+l-t)$ and $\psi_{1,k}(x_{i-t,j+l-t})$ is the sum of weights of all $k$-stair paths of $(i-t,j+l-t)$, then as the product of a weight of a stair path from $(i,j)$ to $(i-t,j+l-t)$ and a weight of an $k$-stair path of $(i-t,j+l-t)$ is precisely a weight of an $(k+l)$-stair path of $(i-t,j+l-t)$ that has the point $(i-t,j+l-t)$ as the tail of $(l+1)$-th arrow, then taking the sum, we have $\sum_{t=0}^{l} \cS(x_{i,j},x_{i-t,j+l-t}) \psi_{1,k}(x_{i-t,j+l-t})$ is precisely the sum of weights of all $(k+l)$-stair paths of $(i,j)$, which is precisely $\psi_{1,k+l}(x_{i,j})$.

By similar argument, using the definition above on the function $\cR$ and Proposition~\ref{prop:reversed_stairs} implies the remaining formula of the claim. This concludes the proof.
\end{proof}

\subsection{Punctured surface case}
\label{subsec:punctured_case}
In this section, we give some remarks on the punctured surface case, the fundamental difference being that self-folded triangles appear during flips. Consider the case of a punctured disk with 2 marked points on the boundary labeled as follows.
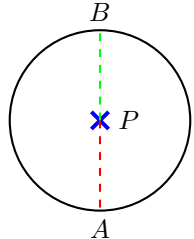
\begin{figure}[H]
\centering
\tikzset{
    circ/.style={thick},
}
\begin{tikzpicture}[scale=0.3,
    mid arrow/.style={
        postaction={decorate},
        decoration={
            markings,
            mark=at position 0.5 with {\arrow{>}}
        }
    }
]
	\draw[circ] (0,0) circle[radius=4];
	\node[cross out, draw=blue, thick, line width=1.5pt, inner sep=2pt, minimum size=5pt] at (0,0) {};
	\node[right] at (0.4,0) {$P$};
	\draw[thick,dashed,green] (0,0) -- (0,4);
	\draw[thick,dashed,red] (0,0) -- (0,-4);
    \node[below] at (0,-4) {$A$};
    \node[above] at (0,4) {$B$};   
\end{tikzpicture}
\caption{A punctured disk with 2 marked points.}
\label{fig:punctured}
\end{figure}
In an ideal triangulation of such disk, two triangles may share 0, 1, or 2 edges. When two triangles share 2 edges (indicating punctures), there are two possible flips, each producing two new triangles with one common edge, one being self-folded. 

The quiver differs from the unpunctured case: the vertex closest to puncture $P$ on edge $AP$ connects to only 2 vertices instead of 4. However, the mutation sequence matches the standard case when treating the quadrilateral as a square with same-colored edges representing identical edges in the original triangulation. For example, in the case $G=\text{SL}_4$, we label vertices with variables $x_i, y_j$ ($i = 1,2,\dots, 12$; $j = 1,2,\dots,6$) as shown in Figure~\ref{fig:punctured_vertex}. We can apply the usual mutation sequence for a flip of diagonal, given by:
\Eq{
    \mu = \{x_1 \rightarrow x_2 \rightarrow x_3 \rightarrow x_7 \rightarrow x_8 \rightarrow x_{10} \rightarrow x_{11} \rightarrow x_9 \rightarrow x_{12} \rightarrow x_2 \}
}
where the order within groups $\{x_1, x_2, x_3\}$, $\{x_7, x_8, x_{10}, x_{11}\}$, and $\{x_9, x_{12}, x_2\}$ can be permuted since the vertices are disconnected within the same group in the quiver during the mutation procedure.
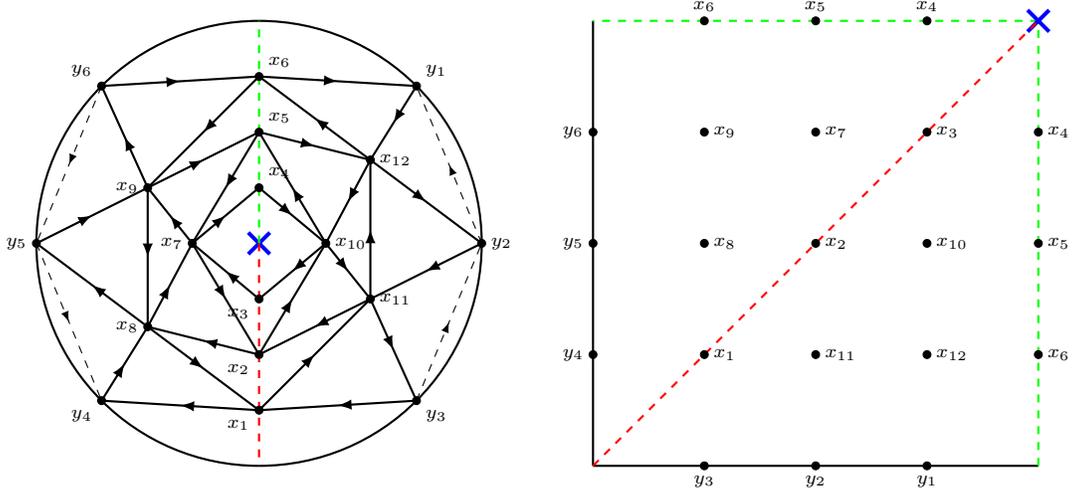
\begin{figure}[H]
\centering
\tikzset{
    circ/.style={thick},
}
\begin{tikzpicture}[scale=0.74,
    mid arrow/.style={
        postaction={decorate},
        decoration={
            markings,
            mark=at position 0.5 with {\arrow{>}}
        }
    }
]
	\draw[circ] (0,0) circle[radius=4];
	\node[cross out, draw=blue, thick, line width=1.5pt] at (0,0) {};
	\draw[thick,dashed,green] (0,0) -- (0,4);
	\draw[thick,dashed,red] (0,0) -- (0,-4);
	\filldraw[black] (0,-3) circle (2pt);
    \node[below left, font=\scriptsize] at (0,-3) {$x_1$};
    \filldraw[black] (0,-2) circle (2pt);
    \node[below left, font=\scriptsize] at (0,-2) {$x_2$};
    \filldraw[black] (0,-1) circle (2pt);
    \node[below left, font=\scriptsize] at (0,-1) {$x_3$};
    \filldraw[black] (0,1) circle (2pt);
    \node[above right, font=\scriptsize] at (0,1) {$x_4$};
    \filldraw[black] (0,2) circle (2pt);
    \node[above right, font=\scriptsize] at (0,2) {$x_5$};
    \filldraw[black] (0,3) circle (2pt);
    \node[above right, font=\scriptsize] at (0,3) {$x_6$};
    
    \coordinate (Y1) at ++(45:4);
    \filldraw[black] ++(45:4) circle (2pt);
    \node[above right, font=\scriptsize] at ++(45:4) {$y_1$};
    \coordinate (Y2) at ++(0:4);
    \filldraw[black] ++(0:4) circle (2pt);
    \node[right, font=\scriptsize] at ++(0:4) {$y_2$};
    \coordinate (Y3) at ++(315:4);
    \filldraw[black] ++(315:4) circle (2pt);
    \node[below right, font=\scriptsize] at ++(315:4) {$y_3$};
    \coordinate (Y6) at ++(135:4);
    \filldraw[black] ++(135:4) circle (2pt);
    \node[above left, font=\scriptsize] at ++(135:4) {$y_6$};
    \coordinate (Y5) at ++(180:4);
    \filldraw[black] ++(180:4) circle (2pt);
    \node[left, font=\scriptsize] at ++(180:4) {$y_5$};
    \coordinate (Y4) at ++(225:4);
    \filldraw[black] ++(225:4) circle (2pt);
    \node[below left, font=\scriptsize] at ++(225:4) {$y_4$};
    
    \filldraw[black] ($(0,2)!0.5!(Y5)$) circle (2pt);
    \node[left, font=\scriptsize] at ($(0,2)!0.5!(Y5)$) {$x_9$};
    \filldraw[black] ($(0,3)!0.5!(Y2)$) circle (2pt);
    \node[right, font=\scriptsize] at ($(0,3)!0.5!(Y2)$) {$x_{12}$};
    \filldraw[black] ($(0,-3)!0.5!(Y5)$) circle (2pt);
    \node[left, font=\scriptsize] at ($(0,-3)!0.5!(Y5)$) {$x_{8}$};
    \filldraw[black] ($(0,-2)!0.5!(Y2)$) circle (2pt);
    \node[right, font=\scriptsize] at ($(0,-2)!0.5!(Y2)$) {$x_{11}$};
    
    \coordinate (X10) at ++(0:1.2);
    \filldraw[black] ++(0:1.2) circle (2pt);
    \node[right, font=\scriptsize] at ++(0:1.2) {$x_{10}$};
    \coordinate (X7) at ++(180:1.2);
    \filldraw[black] ++(180:1.2) circle (2pt);
    \node[left, font=\scriptsize] at ++(180:1.2) {$x_{7}$};
    
    \draw[mid arrow, thick] (0,3) -- (Y1);
    \draw[mid arrow, thick] (Y1) -- ($(0,3)!0.5!(Y2)$);
    \draw[mid arrow, thick] ($(0,3)!0.5!(Y2)$) -- (Y2);
    \draw[mid arrow, thick] (Y2) -- ($(0,-2)!0.5!(Y2)$);
    \draw[mid arrow, thick] ($(0,-2)!0.5!(Y2)$) -- (Y3);
    \draw[mid arrow, thick] (Y3) -- (0,-3);
    
    \draw[mid arrow, thick] (0,-3) -- (Y4);
    \draw[mid arrow, thick] (Y4) -- ($(0,-3)!0.5!(Y5)$);
    \draw[mid arrow, thick] ($(0,-3)!0.5!(Y5)$) -- (Y5);
    \draw[mid arrow, thick] (Y5) -- ($(0,2)!0.5!(Y5)$);
    \draw[mid arrow, thick] ($(0,2)!0.5!(Y5)$) -- (Y6);
    \draw[mid arrow, thick] (Y6) -- (0,3);
    
    \draw[mid arrow, thick] ($(0,3)!0.5!(Y2)$) -- (0,3);
    \draw[mid arrow, thick] (0,3) -- ($(0,2)!0.5!(Y5)$);
    \draw[mid arrow, thick] ($(0,2)!0.5!(Y5)$) -- (0,2);
    \draw[mid arrow, thick] (0,2) -- ($(0,3)!0.5!(Y2)$);
    
    \draw[mid arrow, thick] (X10) -- (0,2);
    \draw[mid arrow, thick] (0,2) -- (X7);
    \draw[mid arrow, thick] (X7) -- (0,1);
    \draw[mid arrow, thick] (0,1) -- (X10);
    
    \draw[mid arrow, thick] (X10) -- (0,-1);
    \draw[mid arrow, thick] (0,-1) -- (X7);
    \draw[mid arrow, thick] (X7) -- (0,-2);
    \draw[mid arrow, thick] (0,-2) -- (X10);
    
    \draw[mid arrow, thick] (X7) -- ($(0,2)!0.5!(Y5)$);
    \draw[mid arrow, thick] ($(0,2)!0.5!(Y5)$) -- ($(0,-3)!0.5!(Y5)$);
    \draw[mid arrow, thick] ($(0,-3)!0.5!(Y5)$) -- (X7);
    
    \draw[mid arrow, thick] (X10) -- ($(0,-2)!0.5!(Y2)$);
    \draw[mid arrow, thick] ($(0,-2)!0.5!(Y2)$) -- ($(0,3)!0.5!(Y2)$);
    \draw[mid arrow, thick] ($(0,3)!0.5!(Y2)$) -- (X10);
    
    \draw[mid arrow, thick] ($(0,-2)!0.5!(Y2)$) -- (0,-2);
    \draw[mid arrow, thick] (0,-2) -- ($(0,-3)!0.5!(Y5)$);
    \draw[mid arrow, thick] ($(0,-3)!0.5!(Y5)$) -- (0,-3);
    \draw[mid arrow, thick] (0,-3) -- ($(0,-2)!0.5!(Y2)$);
    
    \draw[mid arrow,thin,dashed] (Y3) -- (Y2);
    \draw[mid arrow,thin,dashed] (Y2) -- (Y1);
    
    \draw[mid arrow,thin,dashed] (Y6) -- (Y5);
    \draw[mid arrow,thin,dashed] (Y5) -- (Y4);
    
    \draw[thick] (6,4) -- (6,-4);
    \draw[thick] (6,-4) -- (14,-4);
    \draw[thick,dashed,green] (14,-4) -- (14,4);
    \draw[thick,dashed,green] (14,4) -- (6,4);
    \node[cross out, draw=blue, thick, line width=1.5pt] at (14,4) {};
    \draw[thick,dashed,red] (6,-4) -- (14,4);
    \filldraw[black] (8,-2) circle (2pt);
    \node[right, font=\scriptsize] at (8,-2) {$x_1$};
    \filldraw[black] (10,0) circle (2pt);
    \node[right, font=\scriptsize] at (10,0) {$x_2$};
    \filldraw[black] (12,2) circle (2pt);
    \node[right, font=\scriptsize] at (12,2) {$x_3$};
    \filldraw[black] (8,-4) circle (2pt);
    \node[below, font=\scriptsize] at (8,-4) {$y_3$};
    \filldraw[black] (10,-4) circle (2pt);
    \node[below, font=\scriptsize] at (10,-4) {$y_2$};
    \filldraw[black] (12,-4) circle (2pt);
    \node[below, font=\scriptsize] at (12,-4) {$y_1$};
    \filldraw[black] (6,-2) circle (2pt);
    \node[left, font=\scriptsize] at (6,-2) {$y_4$};
    \filldraw[black] (10,-2) circle (2pt);
    \node[right, font=\scriptsize] at (10,-2) {$x_{11}$};
    \filldraw[black] (12,-2) circle (2pt);
    \node[right, font=\scriptsize] at (12,-2) {$x_{12}$};
    \filldraw[black] (14,-2) circle (2pt);
    \node[right, font=\scriptsize] at (14,-2) {$x_{6}$};
    \filldraw[black] (6,0) circle (2pt);
    \node[left, font=\scriptsize] at (6,0) {$y_5$};
    \filldraw[black] (8,0) circle (2pt);
    \node[right, font=\scriptsize] at (8,0) {$x_{8}$};
    \filldraw[black] (12,0) circle (2pt);
    \node[right, font=\scriptsize] at (12,0) {$x_{10}$};
    \filldraw[black] (14,0) circle (2pt);
    \node[right, font=\scriptsize] at (14,0) {$x_{5}$};
    \filldraw[black] (6,2) circle (2pt);
    \node[left, font=\scriptsize] at (6,2) {$y_6$};
    \filldraw[black] (8,2) circle (2pt);
    \node[right, font=\scriptsize] at (8,2) {$x_{9}$};
    \filldraw[black] (10,2) circle (2pt);
    \node[right, font=\scriptsize] at (10,2) {$x_{7}$};
    \filldraw[black] (14,2) circle (2pt);
    \node[right, font=\scriptsize] at (14,2) {$x_{4}$};
    \filldraw[black] (8,4) circle (2pt);
    \node[above, font=\scriptsize] at (8,4) {$x_{6}$};
    \filldraw[black] (10,4) circle (2pt);
    \node[above, font=\scriptsize] at (10,4) {$x_{5}$};
    \filldraw[black] (12,4) circle (2pt);
    \node[above, font=\scriptsize] at (12,4) {$x_{4}$};
    
\end{tikzpicture}
\caption{Vertex labeling for punctured case (left) and its corresponding square lattice (right).}
\label{fig:punctured_vertex}
\end{figure}
The new variables are computed to be:
{\small
\begin{align*}
    x'_{1} &= \frac{y_{3}x_{8} + y_{4}x_{11}}{x_1}; \quad x'_{2} = \frac{x_{7}x_{11} + x_{8}x_{10}}{x_2}; \quad x'_{3} = \frac{x_{7} + x_{10}}{x_3}; \\
    x'_{7} &= \frac{x'_{3}x_{4}x_{9} + x'_{2}x_{5}}{x_7} = \frac{x_{4} x_{9}}{x_{3}}+\frac{x_{4} x_{9} x_{10}}{x_{3} x_{7}}+\frac{x_{5} x_{11}}{x_{2}}+\frac{x_{5} x_{8} x_{10}}{x_{2} x_{7}}; \\
    x'_{8} &= \frac{y_{5}x'_{2} + x_{9}x'_{1}}{x_8} = \frac{x_{7} x_{11} y_{5}}{x_{2} x_{8}}+\frac{x_{10} y_{5}}{x_{2}}+\frac{x_{9} y_{3}}{x_{1}}+\frac{x_{9} x_{11} y_{4}}{x_{1} x_{8}};\\
    x'_{10} &= \frac{x'_{3}x_{4}x_{12} + x'_{2}x_{5}}{x_{10}} = \frac{x_{4} x_{12} x_{7}}{x_{3} x_{10}}+\frac{x_{4} x_{12}}{x_{3}}+\frac{x_{5} x_{7} x_{11}}{x_{2} x_{10}}+\frac{x_{5} x_{8}}{x_{2}};\\
    x'_{11} &= \frac{y_{2}x'_{2} + x_{12}x'_{1}}{x_{11}} = \frac{x_{7} y_{2}}{x_{2}}+\frac{x_{8} x_{10} y_{2}}{x_{2} x_{11}}+\frac{x_{8} x_{12} y_{3}}{x_{1} x_{11}}+\frac{x_{12} y_{4}}{x_{1}};\\    
    x'_{9} &= \frac{x'_{8}x_{6} + x'_{7}y_{6}}{x_9} \\
        &= \frac{x_{4} y_{6}}{x_{3}}+\frac{x_{4} x_{10} y_{6}}{x_{3} x_{7}}+\frac{x_{5} x_{11} y_{6}}{x_{2} x_{9}}+\frac{x_{8} x_{5} x_{10} y_{6}}{x_{2} x_{7} x_{9}}+\frac{x_{7} x_{6} x_{11} y_{5}}{x_{2} x_{8} x_{9}}+\frac{x_{6} x_{10} y_{5}}{x_{2} x_{9}}+\frac{x_{6} y_{3}}{x_{1}}+\frac{x_{6} x_{11} y_{4}}{x_{1} x_{8}}; \\
    x'_{12} &= \frac{x'_{11}x_{6} + x'_{10}y_{1}}{x_{12}} \\
        &= \frac{x_{4} x_{7} y_{1}}{x_{3} x_{10}}+\frac{x_{4} y_{1}}{x_{3}}+\frac{x_{11} x_{5} x_{7} y_{1}}{x_{2} x_{10} x_{12}}+\frac{x_{5} x_{8} y_{1}}{x_{2} x_{12}}+\frac{x_{6} x_{7} y_{2}}{x_{2} x_{12}}+\frac{x_{10} x_{6} x_{8} y_{2}}{x_{2} x_{11} x_{12}}+\frac{x_{6} x_{8} y_{3}}{x_{1} x_{11}}+\frac{x_{6} y_{4}}{x_{1}};\\       
    x''_{2} &= \frac{x'_{7}x'_{11} + x'_{8}x'_{10}}{x'_2} \\
        &= \frac{x_{7} x_{4} x_{12} y_{5}}{x_{3} x_{8} x_{10}}+\frac{x_{4} x_{9} y_{2}}{x_{3} x_{11}}+\frac{x_{4} x_{12} y_{5}}{x_{3} x_{8}}+\frac{x_{10} x_{4} x_{9} y_{2}}{x_{3} x_{7} x_{11}}+\frac{x_{7} x_{11} x_{5} y_{5}}{x_{2} x_{8} x_{10}}+\frac{x_{5} y_{2}}{x_{2}}+\frac{x_{5} y_{5}}{x_{2}}+\frac{x_{8} x_{10} x_{5} y_{2}}{x_{2} x_{7} x_{11}} \\
        &\quad +\frac{x_{2} x_{4} x_{9} x_{12} y_{3}}{x_{3} x_{1} x_{11} x_{10}}+\frac{x_{2} x_{4} x_{9} x_{12} y_{4}}{x_{3} x_{1} x_{8} x_{10}}+\frac{x_{2} x_{4} x_{9} x_{12} y_{3}}{x_{3} x_{7} x_{1} x_{11}}+\frac{x_{2} x_{4} x_{9} x_{12} y_{4}}{x_{3} x_{7} x_{1} x_{8}}+\frac{x_{5} x_{9} y_{3}}{x_{1} x_{10}} \\
        &\quad +\frac{x_{11} x_{5} x_{9} y_{4}}{x_{1} x_{8} x_{10}}+\frac{x_{8} x_{5} x_{12} y_{3}}{x_{7} x_{1} x_{11}}+\frac{x_{5} x_{12} y_{4}}{x_{7} x_{1}}.
\end{align*}
}
With the exception of $x_3$, which follows a modified square structure with $x_4$ replaced by $1$, all variables adhere to the same pattern as in the unpunctured case. Through analogous computations, we find that all variables maintain the unpunctured surface structure up to the flipped diagonal vertex nearest the puncture. Furthermore, for the flipped diagonal vertex nearest the puncture, the number of monomials (up to multiplicities) of the expansion formula of the variable of this vertex remains $2$ (in the case of Figure~\ref{fig:punctured_vertex} we have $x'_3$ having $2$ monomials). Therefore, for the remainder of this paper, where coefficient analysis is primary, it suffices to consider only the unpunctured surface case.

\subsection{Solving recurrences}\label{subsec:solverecur}
 In this section we attempt to solve (case-by-case) the recurrences (\ref{eq:crfs}) and connect them with famous combinatorial structures (\cite{Hen}, \cite{HKW})
\begin{Problem} Find the general formula of $\psi_{t,k}(x_{i_0,j_0})$ in terms of the variables $x_{i,j}$?
\end{Problem}
We present our current progress in deriving a general formula for the expansion structure, beginning with the above observations (Propositions~\ref{prop:stairs_combinatorial}, \ref{prop:reversed_stairs}, and (\ref{eq:general_stair_str})). We shall make connections of the recurrence by defining some auxiliary functions.

\begin{Prop}\label{4DHM} Denote the function $f: (\mathbb{Z}_{\geq 0}^2 \setminus \{(0,0)\}) \times \mathbb{Z}^2 \rightarrow \mathbb{Z}[x^{\pm}_{i,j}]$ such that $$ f(t,k+1-t,i,j):=\psi_{t,k}(x_{i,j}).$$

Then the function $f$ can be extended to $\mathbb{Z}_{\geq 0}^2 \times \mathbb{Z}^2$, and satisfy the \emph{discrete 4D Hirota--Miwa equation} from~\cite{HKW}, which is the 4-dimensional version of the octahedron recurrence~\cite{Hen}.
\end{Prop}
\begin{proof}
The function $f$ satisfies the recurrence:
\Eqn{
	f(0,b,i,j) &= f(a,0,i,j) = x_{i,j}; \\
	 f(1,1,i,j) &= \frac{x_{i-1,j}x_{i+1,j}+x_{i,j-1}x_{i,j+1}}{x_{i,j}} \quad\text{ for all } a,b \in \mathbb{N}, i,j \in \mathbb{Z};\\
	f(a+1,b+1,i,j) &= \frac{f(a+1,b,i-1,j)f(a,b+1,i+1,j)+f(a,b+1,i,j-1)f(a+1,b,i,j+1)}{f(a,b,i,j)}
	\\ &\text{ for all } (a,b) \in \mathbb{Z}_{\geq 0}^2 \setminus \{(0,0)\}, i,j \in \mathbb{Z}.
}
From the recurrence, we can extend to another value $f(0,0,i,j) = x_{i,j}$, such that the recurrence becomes $f: \mathbb{Z}_{\geq 0}^2 \times \mathbb{Z}^2 \rightarrow \mathbb{Z}[x^{\pm}_{i,j}]$ where:
\Eqn{
	f(0,b,i,j) &= f(a,0,i,j) = x_{i,j} \quad\text{ for all } a,b \in \mathbb{Z}_{\geq 0}, i,j \in \mathbb{Z};\\
	f(a+1,b+1,i,j) &= \frac{f(a+1,b,i-1,j)f(a,b+1,i+1,j)+f(a,b+1,i,j-1)f(a+1,b,i,j+1)}{f(a,b,i,j)}
	\\ &\text{ for all } a,b \in \mathbb{Z}_{\geq 0}, i,j \in \mathbb{Z}.
}
This is exactly the \emph{discrete 4D Hirota--Miwa equation}.
\end{proof}
Now define the product of variables:
\begin{gather*}
X^{a,b}_{i,j} := 
\begin{cases}
	\prod_{\substack{p,q \in \mathbb{Z}; \\ 1-a \leq \min \{ p,q,p+q \} \leq \max \{ p,q,p+q \} \leq b-1}} x_{i-p,j+q} \quad \ \text{ if } (a,b) \neq (0,0), \\
	 x_{i,j} \quad \quad \quad \quad \quad \quad \quad \quad \quad \quad \quad \quad \quad \quad \quad \quad \quad \quad \quad \text{ if } (a,b) = (0,0)
\end{cases}
\quad \text{for all } a,b \in \mathbb{Z}_{\geq 0}, i,j \in \mathbb{Z}
\end{gather*}
where in case there are no such values of $(a,b)$ satisfying the condition, we can let the product be $1$. See Figure~\ref{fig:X_var_lattice_example_46} for an example.
\begin{figure}[H]
\centering
\begin{tikzpicture}[scale=0.75,
    mid arrow/.style={
        postaction={decorate},
        decoration={
            markings,
            mark=at position 0.5 with {\arrow{>}}
        }
    }
]
	\filldraw[blue!20] (0,-3)--(3,0)--(3,5)--(0,5)--(-5,0)--(-5,-3)--cycle;
	
	\foreach \i in {0,...,5} {
		\filldraw[black] (3,\i) circle (2pt);
		}
		\foreach \i in {0,...,6} {
		\filldraw[black] (2,\i-1) circle (2pt);
		}
	
		\foreach \i in {0,...,7} {
		\filldraw[black] (1,\i-2) circle (2pt);
		}
	
		\foreach \i in {0,...,8} {
		\filldraw[black] (0,\i-3) circle (2pt);
		}
	
		\foreach \i in {0,...,7} {
		\filldraw[black] (-1,\i-3) circle (2pt);
		}
	
		\foreach \i in {0,...,6} {
		\filldraw[black] (-2,\i-3) circle (2pt);
		}
	
		\foreach \i in {0,...,5} {
		\filldraw[black] (-3,\i-3) circle (2pt);
		}
	
		\foreach \i in {0,...,4} {
		\filldraw[black] (-4,\i-3) circle (2pt);
		}
	
		\foreach \i in {0,...,3} {
		\filldraw[black] (-5,\i-3) circle (2pt);
		}
	\filldraw[red] (0,0) circle (2pt);
	\node[above, font=\scriptsize\bfseries] at (0,0) {$x_{i,j}$};
\end{tikzpicture}
\caption{Visualization for all points appearing in the product $X_{i,j}^{4,6}$}
\label{fig:X_var_lattice_example_46}
\end{figure}
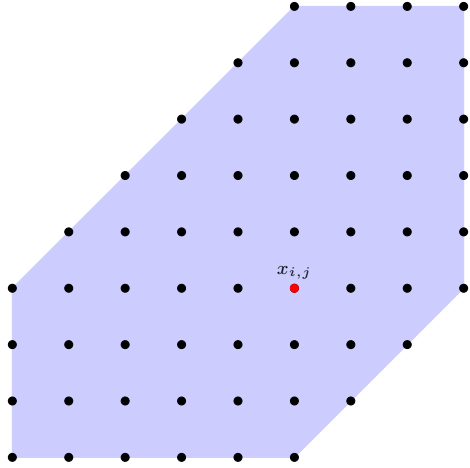
Then we also have another observation:
\begin{Prop}
\label{prop:g_func}
Consider another function $g: \mathbb{Z}_{\geq 0}^2 \times \mathbb{Z}^2 \rightarrow \mathbb{Z}[x^{\pm}_{i,j}]$ such that $g(a,b,i,j) := X^{a,b}_{i,j}\cdot f(a,b,i,j)$. We can compute the base cases:
\begin{gather*}
	g(0,0,i,j) = x_{i,j}^2; \quad g(0,b,i,j) = \prod_{(k,l) \in \mathcal{L}(0,b)} x_{i+k,j+l}; \quad g(a,0,i,j) = \prod_{(k,l) \in \mathcal{L}(a,0)} x_{i+k,j+l}
\end{gather*}
for all $a,b \geq 1$. Then $g$ satisfies the following equation:
\Eqn{
	&g(a+1,b+1,i,j)g(a,b,i,j) \\
  &= x_{i,j-a}x_{i,j+b}\cdot g(a+1,b,i-1,j)g(a,b+1,i+1,j)+x_{i+a,j}x_{i-b,j}\cdot g(a,b+1,i,j-1)g(a+1,b,i,j+1) 
}
for all $a,b \in \mathbb{Z}_{\geq 0}, i,j \in \mathbb{Z}$, with $g(a,b,i,j)$ a polynomial in finitely many variables $x_{i',j'}$ for any $a,b \in \mathbb{Z}_{\geq 0}$ and $i,j \in \mathbb{Z}$. Then we can have a minor observation:
\Eq{
    g(a,b,i,j) = 2^{ab}\cdot x^{1+ab+\frac{(a+b-1)(a+b-2)}{2}}
}
in case all variables $x_{i',j'} = x$. The function $g$ satisfies a similar equation to~\cite[Corollary 19]{BPW}.
\end{Prop}
Appendix \ref{good_lattice} illustrates further tools for attempting to derive the general formula for the function $g$. Additionally, to computationally verify and explore these recurrences, we implemented the functions in Maple. The Maple code in Appendix \ref{maple} defines the recurrence relation and allows for the computation of specific cases. This implementation uses memorization to improve efficiency and handles the base cases explicitly. The recurrence relation matches the discrete 4D Hirota-Miwa equation discussed above, providing a computational tool to verify theoretical results and explore new patterns in the structure of $\psi_{t,k}$.

\section{Cluster expansion formulas of $n$-triangulated $m$-gon}
\label{sec:ntrimgon}
In this section, we discuss the expansion formula in the general $n$-triangulated $m$-gon configuration, i.e., the cluster expansion associated with the moduli space $\mathcal{P}_{\text{SL}_{n+1},\bbS}$ where $\bbS$ is a disk with $m$-marked points with $m\geq 4$.

\subsection{Notation and labeling conventions}
\label{subsec:notation_labeling}
We introduce two labeling schemes for the vertices of an $n$-triangulated polygon. 
\begin{Def}[Grid labeling]\label{gridlabeling}
Consider an $n$-triangulated quadrilateral as a square with its vertices forming a lattice, assign all rows ($n$ rows) and columns ($n$ columns) consisting of all vertices, with rows enumerated from top to bottom and columns from left to right. We label the vertex at row $i$ and column $j$ by $[i,j]$. 
\end{Def}
We shall use the following auxiliary construction in the proof of the well-triangulated theorem below.
\begin{Def}
\label{def:mirror}For a vertex $[i,j]$, its \emph{envelope} is the smallest square containing $[i,j]$ (excluding boundaries) whose edges consist of vertices from the initial square and whose main diagonal is contained in the main diagonal of the initial square. 
\end{Def}
See Figure~\ref{fig:mirror_example} for the envelope corresponding to vertices $[2,2]$, $[5,3]$, and $[2,4]$ with $n=5$.
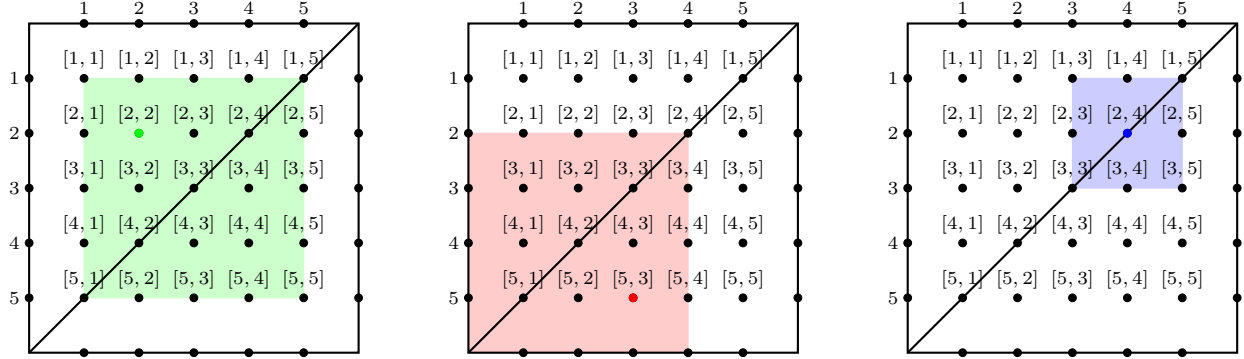
\begin{figure}[H]
\centering
\begin{tikzpicture}[scale=0.73,
    mid arrow/.style={
        postaction={decorate},
        decoration={
            markings,
            mark=at position 0.5 with {\arrow{>}}
        }
    }
]
	\filldraw[green!20] (1,1)--(5,1)--(5,5)--(1,5)--cycle;
	\filldraw[red!20] (8,0)--(12,0)--(12,4)--(8,4)--cycle;
	\filldraw[blue!20] (19,3)--(19,5)--(21,5)--(21,3)--cycle;
	\draw[thick] (0,0)--(6,0)--(6,6)--(0,6)--cycle;
	\draw[thick] (8,0)--(14,0)--(14,6)--(8,6)--cycle;
	\draw[thick] (16,0)--(22,0)--(22,6)--(16,6)--cycle;
	\draw[thick] (0,0)--(6,6);
	\draw[thick] (8,0)--(14,6);
	\draw[thick] (16,0)--(22,6);
	
	\foreach \i[evaluate=\i as \idif using {int(6-\i)}] in {1,...,5} {
		\node[left, font=\scriptsize\bfseries] at (0,\idif) {$\i$};
		\node[left, font=\scriptsize\bfseries] at (8,\idif) {$\i$};
		\node[left, font=\scriptsize\bfseries] at (16,\idif) {$\i$};
	}
	
	\foreach \j in {1,...,5} {
		\node[above, font=\scriptsize\bfseries] at (\j,6) {$\j$};
		\node[above, font=\scriptsize\bfseries] at (\j+8,6) {$\j$};
		\node[above, font=\scriptsize\bfseries] at (\j+16,6) {$\j$};
	}
	
	\foreach \i in {1,...,5} {
		\foreach \j[evaluate=\j as \jdif using {int(6-\j)}] in {1,...,5} {
			\filldraw[black] (\i,\j) circle (2pt);
			\node[above, font=\scriptsize\bfseries] at (\i,\j) {$[\jdif, \i]$};
			\filldraw[black] (8+\i,\j) circle (2pt);
			\node[above, font=\scriptsize\bfseries] at (8+\i,\j) {$[\jdif, \i]$};
			\filldraw[black] (16+\i,\j) circle (2pt);
			\node[above, font=\scriptsize\bfseries] at (16+\i,\j) {$[\jdif, \i]$};
		}
	}
	\filldraw[black] (0,1) circle (2pt);
	\filldraw[black] (0,2) circle (2pt);
	\filldraw[black] (0,3) circle (2pt);
	\filldraw[black] (0,4) circle (2pt);
	\filldraw[black] (0,5) circle (2pt);
	\filldraw[black] (6,1) circle (2pt);
	\filldraw[black] (6,2) circle (2pt);
	\filldraw[black] (6,3) circle (2pt);
	\filldraw[black] (6,4) circle (2pt);
	\filldraw[black] (6,5) circle (2pt);
	\filldraw[black] (8,1) circle (2pt);
	\filldraw[black] (8,2) circle (2pt);
	\filldraw[black] (8,3) circle (2pt);
	\filldraw[black] (8,4) circle (2pt);
	\filldraw[black] (8,5) circle (2pt);
	\filldraw[black] (14,1) circle (2pt);
	\filldraw[black] (14,2) circle (2pt);
	\filldraw[black] (14,3) circle (2pt);
	\filldraw[black] (14,4) circle (2pt);
	\filldraw[black] (14,5) circle (2pt);
	\filldraw[black] (16,1) circle (2pt);
	\filldraw[black] (16,2) circle (2pt);
	\filldraw[black] (16,3) circle (2pt);
	\filldraw[black] (16,4) circle (2pt);
	\filldraw[black] (16,5) circle (2pt);
	\filldraw[black] (22,1) circle (2pt);
	\filldraw[black] (22,2) circle (2pt);
	\filldraw[black] (22,3) circle (2pt);
	\filldraw[black] (22,4) circle (2pt);
	\filldraw[black] (22,5) circle (2pt);
	
	\filldraw[black] (1,0) circle (2pt);
	\filldraw[black] (2,0) circle (2pt);
	\filldraw[black] (3,0) circle (2pt);
	\filldraw[black] (4,0) circle (2pt);
	\filldraw[black] (5,0) circle (2pt);
	\filldraw[black] (1,6) circle (2pt);
	\filldraw[black] (2,6) circle (2pt);
	\filldraw[black] (3,6) circle (2pt);
	\filldraw[black] (4,6) circle (2pt);
	\filldraw[black] (5,6) circle (2pt);
	\filldraw[black] (9,0) circle (2pt);
	\filldraw[black] (10,0) circle (2pt);
	\filldraw[black] (11,0) circle (2pt);
	\filldraw[black] (12,0) circle (2pt);
	\filldraw[black] (13,0) circle (2pt);
	\filldraw[black] (9,6) circle (2pt);
	\filldraw[black] (10,6) circle (2pt);
	\filldraw[black] (11,6) circle (2pt);
	\filldraw[black] (12,6) circle (2pt);
	\filldraw[black] (13,6) circle (2pt);
	\filldraw[black] (17,0) circle (2pt);
	\filldraw[black] (18,0) circle (2pt);
	\filldraw[black] (19,0) circle (2pt);
	\filldraw[black] (20,0) circle (2pt);
	\filldraw[black] (21,0) circle (2pt);
	\filldraw[black] (17,6) circle (2pt);
	\filldraw[black] (18,6) circle (2pt);
	\filldraw[black] (19,6) circle (2pt);
	\filldraw[black] (20,6) circle (2pt);
	\filldraw[black] (21,6) circle (2pt);
	
	\filldraw[green] (2,4) circle (2pt);
	\filldraw[red] (11,1) circle (2pt);
	\filldraw[blue] (20,4) circle (2pt);	
\end{tikzpicture}
\caption{The mirror of $[2,2]$ (left), $[5,3]$ (middle), and $[2,4]$ (right) in case $n = 5$}
\label{fig:mirror_example}
\end{figure}
Next we introduce another labeling for the internal cluster variables. For any $m\geq 2$, define 
\Eq{\Gamma_m := \{(a,b,c)\in \Z_{\geq 0}^3 \mid a+b+c = m\}.}

\begin{Def}[$\G$-labeling]
For $m=4$, consider an $n$-triangulated quadrilateral $ABCD$ with diagonal $AC$. Let us assign variables $x_{(a,b,c)}^{(1)}$ to the vertices of $\triangle ABC$ and $x_{(a,b,c)}^{(2)}$ for $\triangle ACD$ where $(a,b,c) \in \Gamma_{n+1}$. The vertices on the diagonal $AC$ are assigned variables $x_{(t,0,n+1-t)}^{(12)}$ for $t=1,\ldots,n$ (Figure~\ref{fig:triangulated_quad} shows $n=3$). 
\end{Def}
\begin{Ex}
\label{ex:m4}
After flipping the diagonal, the new diagonal $BD$ contains vertices with variables
\begin{equation}
\label{eq:flipped_diagonal}
x_{(1,n-1,1)}^{(1)}, x_{(2,n-3,2)}^{(1)}, \dots, x_{(k,n+1-2k,k)}^{(1)}, x_{(k,n+1-2k,k)}^{(2)}, \dots, x_{(2,n-3,2)}^{(2)}, x_{(1,n-1,1)}^{(2)}
\end{equation}
where $k = \left\lfloor\frac{n+1}{2}\right\rfloor$. When $n$ is odd, the two middle vertices (index $k$) coincide, i.e. $x_{(k,n+1-2k,k)}^{(12)} = x_{(k,n+1-2k,k)}^{(1)} = x_{(k,n+1-2k,k)}^{(2)}$.
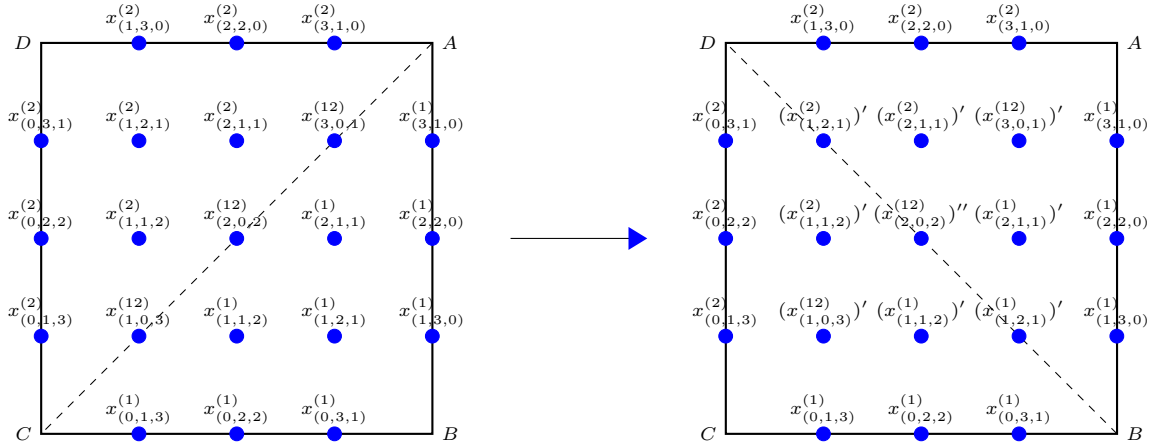
\begin{figure}[H]
\centering
\begin{tikzpicture}[scale=1.3]
  \coordinate (V1) at (4,4);
  \coordinate (V2) at (4,0);
  \coordinate (V3) at (0,0);
  \coordinate (V4) at (0,4);
  \coordinate (V'1) at (11,4);
  \coordinate (V'2) at (11,0);
  \coordinate (V'3) at (7,0);
  \coordinate (V'4) at (7,4); 
  
  \draw[thick] (V1) -- (V2) -- (V3) -- (V4) -- cycle;
  \draw[thick] (V'1) -- (V'2) -- (V'3) -- (V'4) -- cycle;
  
  \node[right, font=\scriptsize] at (V1) {$A$};
  \node[right, font=\scriptsize] at (V2) {$B$};
  \node[left, font=\scriptsize] at (V3) {$C$};
  \node[left, font=\scriptsize] at (V4) {$D$};
  \draw[-{Triangle[blue,scale=2,line width=1.5pt]}] (4.8,2) -- (6.2,2);
  \node[right, font=\scriptsize] at (V'1) {$A$};
  \node[right, font=\scriptsize] at (V'2) {$B$};
  \node[left, font=\scriptsize] at (V'3) {$C$};
  \node[left, font=\scriptsize] at (V'4) {$D$};
  
  \draw[thin, dashed] (V1) -- (V3);
  \draw[thin, dashed] (V'2) -- (V'4);  
  
  \foreach \x[evaluate=\x as \ystart using int(\x+1)] in {1,...,3} {
    \foreach \y[evaluate=\y as \ydiff using int(\y-\x),
                evaluate=\y as \ycomp using int(4-\y)] in {\ystart,...,4} {
      \filldraw[blue] (\x,\y) circle (2pt);
      \node[above, font=\scriptsize] at (\x,\y) {$x^{(2)}_{(\x,\ydiff,\ycomp)}$};
    }
  }
  
  \foreach \y[evaluate=\y as \ycomp using int(4-\y)] in {1,...,3} {
    \filldraw[blue] (0,\y) circle (2pt);
    \node[above, font=\scriptsize] at (0,\y) {$x^{(2)}_{(0,\y,\ycomp)}$};
    }
    
  \foreach \y[evaluate=\y as \ycomp using int(4-\y)] in {1,...,3} {
    \filldraw[blue] (\y,\y) circle (2pt);
    \node[above, font=\scriptsize] at (\y,\y) {$x^{(12)}_{(\y,0,\ycomp)}$};
    }
    
  \foreach \x[evaluate=\x as \ystart using int(\x+1)] in {1,...,3} {
    \foreach \y[evaluate=\y as \ydiff using int(\y-\x),
                evaluate=\y as \ycomp using int(4-\y)] in {\ystart,...,4} {
      \filldraw[blue] (\y,\x) circle (2pt);
      \node[above, font=\scriptsize] at (\y,\x) {$x^{(1)}_{(\x,\ydiff,\ycomp)}$};
    }
  } 
  
  \foreach \y[evaluate=\y as \ycomp using int(4-\y)] in {1,...,3} {
    \filldraw[blue] (\y,0) circle (2pt);
    \node[above, font=\scriptsize] at (\y,0) {$x^{(1)}_{(0,\y,\ycomp)}$};
    }
   
   \filldraw[blue] (7,1) circle (2pt);
   \node[above, font=\scriptsize] at (7,1) {$x^{(2)}_{(0,1,3)}$};
   \filldraw[blue] (7,2) circle (2pt);
   \node[above, font=\scriptsize] at (7,2) {$x^{(2)}_{(0,2,2)}$}; 
   \filldraw[blue] (7,3) circle (2pt);
   \node[above, font=\scriptsize] at (7,3) {$x^{(2)}_{(0,3,1)}$};
   \filldraw[blue] (8,0) circle (2pt);
   \node[above, font=\scriptsize] at (8,0) {$x^{(1)}_{(0,1,3)}$};  
   \filldraw[blue] (8,1) circle (2pt);
   \node[above, font=\scriptsize] at (8,1) {$(x^{(12)}_{(1,0,3)})'$};
   \filldraw[blue] (8,2) circle (2pt);
   \node[above, font=\scriptsize] at (8,2) {$(x^{(2)}_{(1,1,2)})'$}; 
   \filldraw[blue] (8,3) circle (2pt);
   \node[above, font=\scriptsize] at (8,3) {$(x^{(2)}_{(1,2,1)})'$};
   \filldraw[blue] (8,4) circle (2pt);
   \node[above, font=\scriptsize] at (8,4) {$x^{(2)}_{(1,3,0)}$};
   \filldraw[blue] (9,0) circle (2pt);
   \node[above, font=\scriptsize] at (9,0) {$x^{(1)}_{(0,2,2)}$};  
   \filldraw[blue] (9,1) circle (2pt);
   \node[above, font=\scriptsize] at (9,1) {$(x^{(1)}_{(1,1,2)})'$};
   \filldraw[blue] (9,2) circle (2pt);
   \node[above, font=\scriptsize] at (9,2) {$(x^{(12)}_{(2,0,2)})''$}; 
   \filldraw[blue] (9,3) circle (2pt);
   \node[above, font=\scriptsize] at (9,3) {$(x^{(2)}_{(2,1,1)})'$};
   \filldraw[blue] (9,4) circle (2pt);
   \node[above, font=\scriptsize] at (9,4) {$x^{(2)}_{(2,2,0)}$};
   \filldraw[blue] (10,0) circle (2pt);
   \node[above, font=\scriptsize] at (10,0) {$x^{(1)}_{(0,3,1)}$};  
   \filldraw[blue] (10,1) circle (2pt);
   \node[above, font=\scriptsize] at (10,1) {$(x^{(1)}_{(1,2,1)})'$};
   \filldraw[blue] (10,2) circle (2pt);
   \node[above, font=\scriptsize] at (10,2) {$(x^{(1)}_{(2,1,1)})'$}; 
   \filldraw[blue] (10,3) circle (2pt);
   \node[above, font=\scriptsize] at (10,3) {$(x^{(12)}_{(3,0,1)})'$};
   \filldraw[blue] (10,4) circle (2pt);
   \node[above, font=\scriptsize] at (10,4) {$x^{(2)}_{(3,1,0)}$};
   \filldraw[blue] (11,1) circle (2pt);
   \node[above, font=\scriptsize] at (11,1) {$x^{(1)}_{(1,3,0)}$};
   \filldraw[blue] (11,2) circle (2pt);
   \node[above, font=\scriptsize] at (11,2) {$x^{(1)}_{(2,2,0)}$}; 
   \filldraw[blue] (11,3) circle (2pt);
   \node[above, font=\scriptsize] at (11,3) {$x^{(1)}_{(3,1,0)}$};     
\end{tikzpicture}
\caption{Labeling all vertices of $3$-triangulated quadrilateral.}
\label{fig:triangulated_quad}
\end{figure}
Applying results from Propositions~\ref{prop:stairs_combinatorial} and~\ref{prop:reversed_stairs}, after a sequence of cluster mutations flipping the diagonal from $AC$ to $BD$, we calculate the expansion formula at vertices $x^{(r)}_{(1,n-1,1)}$ as follows:
\begin{align}
(x^{(1)}_{(1,n-1,1)})' 
&= \sum_{\substack{(l_1+1),l_i,u_i,(u_k+1), \\ k \in \mathbb{N}: \\ \sum_{i=1}^{k}(l_i+u_i) = n}} 
   \left( x^{(2)}_{(\sum_{t=1}^{k}u_t,1,\sum_{s=1}^{k}l_s)} \prod_{v=1}^{k} \frac{x^{(1)}_{(1+\sum_{t=1}^{v-1}u_t,n-\sum_{s=1}^{v-1}l_s-\sum_{t=1}^{v-1}u_t,\sum_{s=1}^{v-1}l_s)}}{x^{(1)}_{(1+\sum_{t=1}^{v-1}u_t,n-\sum_{s=1}^{v}l_s-\sum_{t=1}^{v-1}u_t,\sum_{s=1}^{v}l_s)}} \right. \nonumber \\
&\quad \left. \cdot \frac{x^{(1)}_{(\sum_{t=1}^{v-1}u_t,n-\sum_{s=1}^{v}l_s-\sum_{t=1}^{v-1}u_t,1+\sum_{s=1}^{v}l_s)}}{x^{(1)}_{(\sum_{t=1}^{v}u_t,n-\sum_{s=1}^{v}l_s-\sum_{t=1}^{v}u_t,1+\sum_{s=1}^{v}l_s)}} \right); \label{eq:flip_formula1}
\end{align}

\begin{align}
(x^{(2)}_{(1,n-1,1)})' 
&= \sum_{\substack{(r_1+1),r_i,d_i,(d_k+1), \\ k \in \mathbb{N}: \\ \sum_{i=1}^{k}(r_i+d_i) = n}} 
   \left( x^{(1)}_{(\sum_{t=1}^{k}r_t,1,\sum_{s=1}^{k}d_s)} \prod_{v=1}^{k} \frac{x^{(2)}_{(\sum_{t=1}^{v-1}r_t,n-\sum_{t=1}^{v-1}r_t-\sum_{s=1}^{v-1}d_s,1+\sum_{s=1}^{v-1}d_s)}}{x^{(2)}_{(\sum_{t=1}^{v}r_t,n-\sum_{t=1}^{v}r_t-\sum_{s=1}^{v-1}d_s,1+\sum_{s=1}^{v-1}d_s)}} \right. \nonumber \\
&\quad \left. \cdot \frac{x^{(2)}_{(1+\sum_{t=1}^{v}r_t,n-\sum_{t=1}^{v}r_t-\sum_{s=1}^{v-1}d_s,\sum_{s=1}^{v-1}d_s)}}{x^{(2)}_{(1+\sum_{t=1}^{v}r_t,n-\sum_{t=1}^{v}r_t-\sum_{s=1}^{v}d_s,\sum_{s=1}^{v}d_s)}} \right). \label{eq:flip_formula2}
\end{align}
\end{Ex}

The $\G$-labeling can be extended to general $m$-gon:
\begin{Ex}
\label{ex:m5}
Consider $m = 5$ and the $n$-triangulated pentagon $ABCDE$ with two main diagonals $AC$ and $AD$. Points (except the 5 vertices $A,B,C,D,E$) are assigned variables $x_{(a,b,c)}^{(1)}$ in $\triangle ABC$, $x_{(a,b,c)}^{(2)}$ in $\triangle ACD$, and $x_{(a,b,c)}^{(3)}$ in $\triangle ADE$ for all $(a,b,c) \in \Gamma_{n+1}$. 

Returning to the $2$-triangulated pentagon (see Figure~\ref{fig:triangulated_pent_SS}), we compute $BE = (f_1, f_3)$, where
{\footnotesize
\begin{align*}
&f_1 = \frac{x^{(1)}_{(1,2,0)}x^{(3)}_{(0,2,1)}}{x^{(23)}_{(1,2,0)}} 
      + \frac{x^{(1)}_{(0,2,1)}x^{(3)}_{(2,0,1)}}{x^{(12)}_{(2,0,1)}} 
      + \frac{x^{(1)}_{(1,2,0)}x^{(2)}_{(0,1,2)}x^{(3)}_{(2,0,1)}}{x^{(12)}_{(1,0,2)}x^{(23)}_{(2,1,0)}} 
      + \frac{x^{(1)}_{(0,2,1)}x^{(1)}_{(2,1,0)}x^{(3)}_{(1,1,1)}}{x^{(1)}_{(1,1,1)}x^{(23)}_{(2,1,0)}}+ \frac{x^{(1)}_{(1,2,0)}x^{(2)}_{(0,2,1)}x^{(3)}_{(1,1,1)}}{x^{(2)}_{(1,1,1)}x^{(23)}_{(1,2,0)}} \\
      &+ \frac{x^{(1)}_{(0,2,1)}x^{(1)}_{(2,1,0)}x^{(2)}_{(1,1,1)}x^{(3)}_{(2,0,1)}}{x^{(1)}_{(1,1,1)}x^{(12)}_{(2,0,1)}x^{(23)}_{(2,1,0)}} + \frac{x^{(1)}_{(1,2,0)}x^{(2)}_{(0,1,2)}x^{(12)}_{(2,0,1)}x^{(3)}_{(1,1,1)}}{x^{(12)}_{(1,0,2)}x^{(2)}_{(1,1,1)}x^{(23)}_{(2,1,0)}} + \frac{x^{(1)}_{(0,1,2)}x^{(1)}_{(1,2,0)}x^{(2)}_{(1,1,1)}x^{(3)}_{(2,0,1)}}{x^{(1)}_{(1,1,1)}x^{(12)}_{(1,0,2)}x^{(23)}_{(2,1,0)}}
      + \frac{x^{(1)}_{(0,1,2)}x^{(1)}_{(1,2,0)}x^{(12)}_{(2,0,1)}x^{(3)}_{(1,1,1)}}{x^{(1)}_{(1,1,1)}x^{(12)}_{(1,0,2)}x^{(23)}_{(2,1,0)}};\\
&f_3 = \frac{x^{(1)}_{(0,1,2)}x^{(3)}_{(1,0,2)}}{x^{(12)}_{(1,0,2)}} 
      + \frac{x^{(1)}_{(2,1,0)}x^{(3)}_{(0,1,2)}}{x^{(23)}_{(2,1,0)}}
      + \frac{x^{(1)}_{(2,1,0)}x^{(2)}_{(0,2,1)}x^{(3)}_{(1,0,2)}}{x^{(12)}_{(2,0,1)}x^{(23)}_{(1,2,0)}} 
      + \frac{x^{(1)}_{(1,1,1)}x^{(3)}_{(0,1,2)}x^{(3)}_{(2,0,1)}}{x^{(12)}_{(2,0,1)}x^{(3)}_{(1,1,1)}}+ \frac{x^{(1)}_{(1,1,1)}x^{(2)}_{(0,1,2)}x^{(3)}_{(1,0,2)}}{x^{(12)}_{(1,0,2)}x^{(2)}_{(1,1,1)}} \\
      &+ \frac{x^{(1)}_{(2,1,0)}x^{(2)}_{(1,1,1)}x^{(3)}_{(0,1,2)}x^{(3)}_{(2,0,1)}}{x^{(12)}_{(2,0,1)}x^{(23)}_{(2,1,0)}x^{(3)}_{(1,1,1)}} + \frac{x^{(1)}_{(1,1,1)}x^{(2)}_{(0,2,1)}x^{(23)}_{(2,1,0)}x^{(3)}_{(1,0,2)}}{x^{(12)}_{(2,0,1)}x^{(2)}_{(1,1,1)}x^{(23)}_{(1,2,0)}}+ \frac{x^{(1)}_{(2,1,0)}x^{(2)}_{(1,1,1)}x^{(3)}_{(0,2,1)}x^{(3)}_{(1,0,2)}}{x^{(12)}_{(2,0,1)}x^{(23)}_{(1,2,0)}x^{(3)}_{(1,1,1)}}
      + \frac{x^{(1)}_{(1,1,1)}x^{(23)}_{(2,1,0)}x^{(3)}_{(0,2,1)}x^{(3)}_{(1,0,2)}}{x^{(12)}_{(2,0,1)}x^{(23)}_{(1,2,0)}x^{(3)}_{(1,1,1)}}.
\end{align*}
}

\begin{figure}[H]
\centering
\begin{tikzpicture}[scale=1.5,
    mid arrow/.style={
        postaction={decorate},
        decoration={
            markings,
            mark=at position 0.5 with {\arrow{>}}
        }
    }
]
  \coordinate (A) at (0, 3.078);
  \coordinate (B) at (1.618, 1.902);
  \coordinate (C) at (1, 0);
  \coordinate (D) at (-1, 0);
  \coordinate (E) at (-1.618, 1.902);
  \coordinate (X1) at ($(A)!0.333!(B)$);
  \coordinate (X2) at ($(A)!0.667!(B)$);
  \coordinate (X3) at ($(B)!0.333!(C)$);
  \coordinate (X4) at ($(B)!0.667!(C)$);
  \coordinate (X5) at ($(C)!0.333!(D)$);
  \coordinate (X6) at ($(C)!0.667!(D)$);
  \coordinate (X7) at ($(D)!0.333!(E)$);
  \coordinate (X8) at ($(D)!0.667!(E)$);
  \coordinate (X9) at ($(E)!0.333!(A)$);
  \coordinate (X10) at ($(E)!0.667!(A)$);
  \coordinate (Y1) at ($(A)!0.333!(D)$);
  \coordinate (Y2) at ($(A)!0.333!(C)$);
  \coordinate (Y3) at ($(X8)!0.5!(Y1)$);
  \coordinate (Y4) at ($(X3)!0.5!(Y2)$);
  \coordinate (Y5) at ($(X5)!0.5!(Y1)$);
  \coordinate (Y6) at ($(A)!0.667!(D)$);
  \coordinate (Y7) at ($(A)!0.667!(C)$);
  
  \node[above, font=\scriptsize] at (A) {$A$};
  \node[right, font=\scriptsize] at (B) {$B$};
  \node[left, font=\scriptsize] at (E) {$E$};  
  \node[right, font=\scriptsize] at (C) {$C$};
  \node[left, font=\scriptsize] at (D) {$D$};
        
  \draw[thick] (A) -- (B) -- (C) -- (D) -- (E) -- cycle;
  \draw[thick] (A) -- (C);
  \draw[thick] (A) -- (D);
  
  \foreach \i in {1,...,10} {
  	\filldraw[blue] (X\i) circle (1pt);  
  }
  
  \foreach \i in {1,...,7} {
  	\filldraw[blue] (Y\i) circle (1pt);  
  }
  
  \node[above right] at (X1) {\scalebox{0.8}{$x^{(1)}_{(2,1,0)}$}};
  \node[above right] at (X2) {\scalebox{0.8}{$x^{(1)}_{(1,2,0)}$}};
  \node[right] at (X3) {\scalebox{0.8}{$x^{(1)}_{(0,2,1)}$}};
  \node[right] at (X4) {\scalebox{0.8}{$x^{(1)}_{(0,1,2)}$}};
  \node[below] at (X5) {\scalebox{0.8}{$x^{(2)}_{(0,1,2)}$}};
  \node[below] at (X6) {\scalebox{0.8}{$x^{(2)}_{(0,2,1)}$}};
  \node[left] at (X7) {\scalebox{0.8}{$x^{(3)}_{(0,2,1)}$}};
  \node[left] at (X8) {\scalebox{0.8}{$x^{(3)}_{(0,1,2)}$}};
  \node[above left] at (X9) {\scalebox{0.8}{$x^{(3)}_{(1,0,2)}$}};
  \node[above left] at (X10) {\scalebox{0.8}{$x^{(3)}_{(2,0,1)}$}};
  \node[above] at (Y1) {\scalebox{0.8}{$x^{(23)}_{(2,1,0)}$}};
  \node[above] at (Y2) {\scalebox{0.8}{$x^{(12)}_{(2,0,1)}$}};
  \node[above] at (Y3) {\scalebox{0.8}{$x^{(3)}_{(1,1,1)}$}};
  \node[above] at (Y4) {\scalebox{0.8}{$x^{(1)}_{(1,1,1)}$}};
  \node[above] at (Y5) {\scalebox{0.8}{$x^{(2)}_{(1,1,1)}$}};
  \node[above] at (Y6) {\scalebox{0.8}{$x^{(23)}_{(1,2,0)}$}};    
  \node[above] at (Y7) {\scalebox{0.8}{$x^{(12)}_{(1,0,2)}$}};            
\end{tikzpicture}
\caption{Labeling vertices in $2$-triangulated pentagon.}
\label{fig:triangulated_pent_SS}
\end{figure}
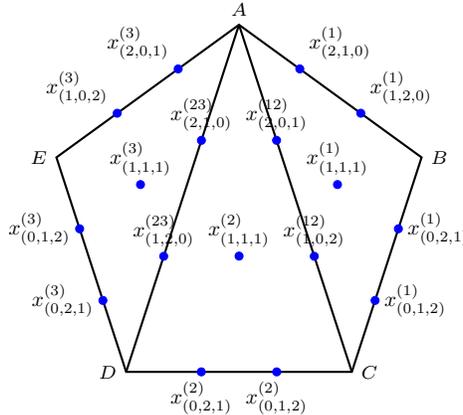
\end{Ex}
\subsection{Well-triangulated Preservation and Number of Monomials Theorem}
\label{subsec:proof_monomials}
In this section, we compute the cluster expansion formula for the variables lying on a \emph{non $\cT$-diagonal} $\c$ of an $m$-gon with respect to the initial triangulation $\cT$, meaning that $\c$ intersects each diagonal of $\cT$ exactly once. By the Laurent phenomenon and the Positivity Theorem (Theorem~\ref{thm:pos_thm}), their expansions with respect to the initial cluster are Laurent polynomials with all coefficients being positive integers. We study the number of monomials, up to multiplicity, in these Laurent polynomials. Here, by multiplicity, we mean counting the sum of all coefficients of those polynomials.

We introduce the \emph{well-triangulated} property, and prove the two main theorems on the combinatorics of the cluster expansion, namely the \emph{Well-triangulated Preservation Theorem} (Theorem~\ref{thm:well_triangulated_preservation}) and the \emph{Number of Monomials Theorem} (Theorem~\ref{thm:monomial_count}).

Recall the $\G$-labeling $x^{(k)}_{(a,b,c)}$ of the cluster variables in a triangle.
\begin{Def}
\label{def:well_triangulated}Consider an $n$-triangulated polygon with triangulation $\cT$. A triangle $T$ of $\cT$ together with an assignment of scalar parameters to its cluster variables is called \emph{well-triangulated} if there exist parameters $x,y,z$ such that for every $(a,b,c) \in \Gamma_{n+1}$, the value assigned to $x^{(k)}_{(a,b,c)}$ is $x^{bc}y^{ca}z^{ab}$ (see the left of Figure~\ref{fig:well_triangulated} for an example). The polygon with values assigned to all the cluster variables is called \emph{well-triangulated} if all triangles of $\cT$ are well-triangulated (see the right of Figure~\ref{fig:well_triangulated} for an example).
\end{Def}
\begin{figure}[H]
\centering
\begin{tikzpicture}[scale=1.5,
    mid arrow/.style={
        postaction={decorate},
        decoration={
            markings,
            mark=at position 0.5 with {\arrow{>}}
        }
    }
]
  \draw[thick] (0,0) -- (1,3) -- (4,0) -- cycle;
  
  \filldraw[blue] (0.333,1) circle (1pt);
  \node[left, font=\scriptsize] at (0.333,1) {$z^2$};
  \filldraw[blue] (0.667,2) circle (1pt);
  \node[left, font=\scriptsize] at (0.667,2) {$z^2$};  
  \filldraw[blue] (2.667,0) circle (1pt);
  \node[below, font=\scriptsize] at (2.667,0) {$y^2$};  
  \filldraw[blue] (2,2) circle (1pt);
  \node[above right, font=\scriptsize] at (2,2) {$x^2$};  
  \filldraw[blue] (3,1) circle (1pt);
  \node[above right, font=\scriptsize] at (3,1) {$x^2$};
  \filldraw[blue] (1.333,0) circle (1pt);
  \node[below, font=\scriptsize] at (1.333,0) {$y^2$};  
  \filldraw[blue] (1.667,1) circle (1pt); 
  \node[above, font=\scriptsize] at (1.667,1) {$xyz$};  
  
  \coordinate (A) at (0+6, 3.078);
  \coordinate (B) at (1.618+6, 1.902);
  \coordinate (C) at (1+6, 0);
  \coordinate (D) at (-1+6, 0);
  \coordinate (E) at (-1.618+6, 1.902);
  \coordinate (X1) at ($(A)!0.333!(B)$);
  \coordinate (X2) at ($(A)!0.667!(B)$);
  \coordinate (X3) at ($(B)!0.333!(C)$);
  \coordinate (X4) at ($(B)!0.667!(C)$);
  \coordinate (X5) at ($(C)!0.333!(D)$);
  \coordinate (X6) at ($(C)!0.667!(D)$);
  \coordinate (X7) at ($(D)!0.333!(E)$);
  \coordinate (X8) at ($(D)!0.667!(E)$);
  \coordinate (X9) at ($(E)!0.333!(A)$);
  \coordinate (X10) at ($(E)!0.667!(A)$);
  \coordinate (Y1) at ($(A)!0.333!(D)$);
  \coordinate (Y2) at ($(A)!0.333!(C)$);
  \coordinate (Y3) at ($(X8)!0.5!(Y1)$);
  \coordinate (Y4) at ($(X3)!0.5!(Y2)$);
  \coordinate (Y5) at ($(X5)!0.5!(Y1)$);
  \coordinate (Y6) at ($(A)!0.667!(D)$);
  \coordinate (Y7) at ($(A)!0.667!(C)$);
        
  \draw[thick] (A) -- (B) -- (C) -- (D) -- (E) -- cycle;
  \draw[thick] (A) -- (C);
  \draw[thick] (A) -- (D);
  
  \foreach \i in {1,...,10} {
  	\filldraw[blue] (X\i) circle (1pt);  
  }
  
  \foreach \i in {1,...,7} {
  	\filldraw[blue] (Y\i) circle (1pt);  
  }
  
  \node[above right] at (X1) {\scalebox{0.8}{$r_1^2$}};
  \node[above right] at (X2) {\scalebox{0.8}{$r_1^2$}};
  \node[right] at (X3) {\scalebox{0.8}{$r_2^2$}};
  \node[right] at (X4) {\scalebox{0.8}{$r_2^2$}};
  \node[below] at (X5) {\scalebox{0.8}{$r_3^2$}};
  \node[below] at (X6) {\scalebox{0.8}{$r_3^2$}};
  \node[left] at (X7) {\scalebox{0.8}{$r_4^2$}};
  \node[left] at (X8) {\scalebox{0.8}{$r_4^2$}};
  \node[above left] at (X9) {\scalebox{0.8}{$r_5^2$}};
  \node[above left] at (X10) {\scalebox{0.8}{$r_5^2$}};
  \node[left] at (Y1) {\scalebox{0.8}{$r_7^2$}};
  \node[right] at (Y2) {\scalebox{0.8}{$r_6^2$}};
  \node[above] at (Y3) {\scalebox{0.8}{$r_4r_5r_7$}};
  \node[above] at (Y4) {\scalebox{0.8}{$r_1r_2r_6$}};
  \node[above] at (Y5) {\scalebox{0.8}{$r_3r_6r_7$}};
  \node[left] at (Y6) {\scalebox{0.8}{$r_7^2$}};    
  \node[right] at (Y7) {\scalebox{0.8}{$r_6^2$}};  
\end{tikzpicture}
\caption{A well-triangulated $2$-triangulated triangle (left) and a well-triangulated $2$-triangulated pentagon where $r_1,r_2,...,r_7$ are parameters (right).}
\label{fig:well_triangulated}
\end{figure}
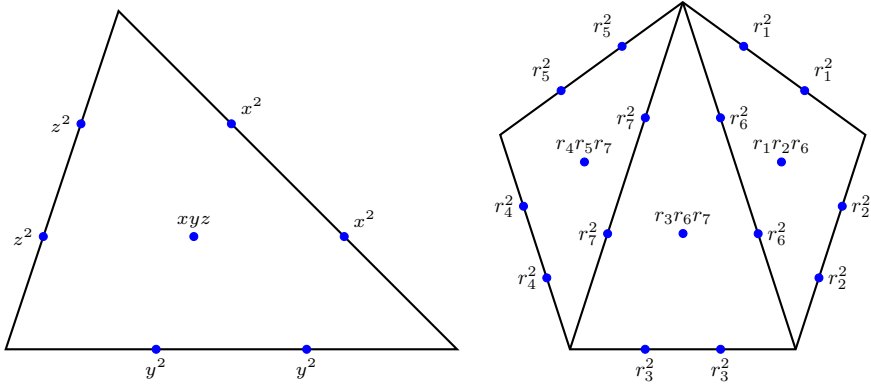
We shall verify the following theorem:
\begin{Thm}[Well-triangulated Preservation Theorem]
\label{thm:well_triangulated_preservation}
For a triangulation $\cT$ of a well-triangulated polygon $\mathcal{P}$ with the cluster realization of type $A_n$, after any sequence of flips, the resulting triangulation is also well-triangulated.
\end{Thm}  
In other words, any sequence of flips preserves the well-triangulated property of all triangles of the new triangulation when all triangles of the initial triangulation are well-triangulated.

\begin{proof}
It suffices to check for a single flip, then the result follows by induction on the number of \emph{crossing points} of a non $\cT$-diagonal, where for a \emph{crossing point} of a non $\cT$-diagonal $\gamma$ we mean the intersection of $\gamma$ and a $\cT$-diagonal, note that by definition, $\gamma$ intersects every $\cT$-diagonal exactly once.

Consider the flip of a well-triangulated quadrilateral $ABCD$ with diagonal $AC$, where $x^{(1)}_{(i,j,k)}$ are assigned to vertices inside $\triangle ABC$ and $x^{(2)}_{(i,j,k)}$ are assigned to vertices inside $\triangle ACD$. This means there exists parameters $a,b,x$ for $\triangle ABC$ and two other parameters $c,d$ with $x$ for $\triangle ACD$ such that each vertex $x^{(1)}_{(i,j,k)}$ has value $a^{ij}b^{jk}x^{ik}$ and each vertex $x^{(2)}_{(i,j,k)}$ has value $d^{ij}c^{jk}x^{ik}$. Denote $l := \dfrac{ac+bd}{x}$. 

Consider $ABCD$ and the inner vertices with the grid labeling as shown in the Definition~\ref{def:mirror}. Let $H_{[i,j]}$ be the initial value of $x_{[i,j]}$ at $[i,j]$ and $H'_{[i,j]}$ as the value after flip. Then:
\begin{align*}
x_{[i,j]} = \begin{cases}
    x^{(2)}_{(j,n+1-i-j,i)} & \text{if } i+j \leq n, \\
    x^{(1)}_{(n+1-i,i+j-n-1,n+1-j)} & \text{if } i+j \geq n+2, \\
    x^{(12)}_{(j,0,i)} & \text{if } i+j = n+1.
\end{cases}
\end{align*}
Hence, we compute:
\begin{align*}
H_{[i,j]} = \begin{cases}
    c^{i(n+1-i-j)}d^{j(n+1-i-j)}x^{ij} & \text{if } i+j \leq n, \\
    a^{(n+1-i)(i+j-n-1)}b^{(n+1-j)(i+j-n-1)}x^{(n+1-i)(n+1-j)} & \text{if } i+j \geq n+2 ,\\
    x^{ij} & \text{if } i+j = n+1.
\end{cases}
\end{align*}
After the flip, each vertex $[i,j]$ is mutated exactly $k$ times if and only if $(i-k)(j-k)(i-n-1+k)(j-n-1+k) = 0$ and $k \leq i,j \leq n+1-k$. Then:
\begin{align*}
H'_{[1,n]} &= \frac{d^n H[2,n] + a^n H[1,n-1]}{x^n} = \frac{d^n bx^{n-1}a^{n-1} + a^n cx^{n-1}d^{n-1}}{x^n} = ld^{n-1}a^{n-1};\\
H'_{[n,1]} &= \frac{b^n H[n-1,1] + c^n H[n,2]}{x^n} = \frac{b^n c^{n-1}x^{n-1}d + c^n b^{n-1}x^{n-1}a}{x^n} = lc^{n-1}b^{n-1}.
\end{align*}
Note that for $(l_1+1),l_i,u_i,(u_k+1), k \in \mathbb{N}$ such that $\sum_{i=1}^{k}(l_i+u_i) = n$, we have the relations:
{\small
\begin{align*}
&(l_1 + l_2 + \dots + l_k)(u_1 + u_2 + \dots + u_k) - \sum_{v=1}^{k} \biggl( l_v \biggl( 1 + \sum_{t=1}^{v-1} u_t \biggr) + u_v \biggl( 1 + \sum_{t=1}^{v} l_s \biggr) \biggr) \\
&\quad = - (l_1 + l_2 + \dots + l_k) - (u_1 + u_2 + \dots + u_k) = -n;\\
&\sum_{v=1}^{k} \biggl( l_v^2 + l_v \biggl( \sum_{s=1}^{v-1} l_s \biggr) - l_v \biggl( n - \sum_{s=1}^{v-1} l_s - \sum_{t=1}^{v-1} u_t \biggr) + u_v \biggl( 1 + \sum_{t=1}^{v} l_s \biggr) \biggr) \\
&\quad = l_1^2 + l_2^2 + \dots + l_k^2 + 2 \sum_{1 \leq i < j \leq k} l_i l_j - n (l_1 + l_2 + \dots + l_k) \\
&\quad \quad + (l_1 + l_2 + \dots + l_k)(u_1 + u_2 + \dots + u_k) + (u_1 + u_2 + \dots + u_k) \\
&\quad = (l_1 + l_2 + \dots + l_k)^2 + (l_1 + l_2 + \dots + l_k)(u_1 + u_2 + \dots + u_k) - n (l_1 + l_2 + \dots + l_k) \\
&\quad \quad + (u_1 + u_2 + \dots + u_k) = (u_1 + u_2 + \dots + u_k);\\
&\sum_{v=1}^{k} \biggl( l_v \biggl( 1 + \sum_{t=1}^{v-1} u_t \biggr) + u_v^2 + u_v \biggl( \sum_{t=1}^{v-1} u_t \biggr) - u_v \biggl( n - \sum_{s=1}^{v} l_s - \sum_{t=1}^{v-1} u_t \biggr) \biggr) \\
&\quad = u_1^2 + u_2^2 + \dots + u_k^2 + 2 \sum_{1 \leq i < j \leq k} u_i u_j - n (u_1 + u_2 + \dots + u_k) \\
&\quad \quad + (l_1 + l_2 + \dots + l_k)(u_1 + u_2 + \dots + u_k) + (l_1 + l_2 + \dots + l_k) \\
&\quad = (u_1 + u_2 + \dots + u_k)^2 + (l_1 + l_2 + \dots + l_k)(u_1 + u_2 + \dots + u_k) - n (u_1 + u_2 + \dots + u_k) \\
&\quad \quad + (l_1 + l_2 + \dots + l_k) = (l_1 + l_2 + \dots + l_k).
\end{align*}
}
Similarly, for $(r_1+1),r_i,d_i,(d_k+1), k \in \mathbb{N}$ such that $\sum_{i=1}^{k}(r_i+d_i) = n$, we can also compute:
{\small
\begin{align*}
    &\sum_{v=1}^{k} \biggl( r_v \biggl( 1 + \sum_{s=1}^{v-1} d_s \biggr) + d_v^2 + d_v \biggl( \sum_{s=1}^{v} d_s \biggr) - d_v \biggl( n - \sum_{t=1}^{v} r_t - \sum_{s=1}^{v-1} d_s \biggr) \biggr) = r_1 + r_2 + \dots + r_k;\\
    &(r_1 + r_2 + \dots + r_k)(d_1 + d_2 + \dots + d_k) - \sum_{v=1}^{k} \biggl( l_v \biggl( 1 + \sum_{t=1}^{v-1} u_t \biggr) + u_v \biggl( 1 + \sum_{t=1}^{v} l_s \biggr) \biggr) = -n;\\
    &\sum_{v=1}^{k} \biggl( r_v^2 + r_v \biggl( \sum_{t=1}^{v-1} r_t \biggr) - r_v \biggl( n - \sum_{t=1}^{v-1} r_t - \sum_{s=1}^{v-1} d_s \biggr) + d_v \biggl( 1 + \sum_{t=1}^{v} r_t \biggr) \biggr) = d_1 + d_2 + \dots + d_k.
\end{align*}
}
Applying formulas \eqref{eq:flip_formula1} and \eqref{eq:flip_formula2} for the envelope of $[n,n]$ and $[1,1]$ yields: 
\begin{align*}
    H'_{[n,n]}
    &= \sum_{\substack{(l_1+1),l_i,u_i,(u_k+1), k \in \mathbb{N}: \\ \sum_{i=1}^{k}(l_i+u_i) = n}} 
       \left(c^{\sum_{s=1}^{k}l_s}x^{(\sum_{s=1}^{k}l_s)(\sum_{t=1}^{k}u_t)}d^{\sum_{t=1}^{k}u_t} \right. \\
       &\quad \left. \cdot \prod_{v=1}^{k} \frac{b^{(n-\sum_{s=1}^{v-1}l_s-\sum_{t=1}^{v-1}u_t)(\sum_{s=1}^{v-1}l_s)}x^{(\sum_{s=1}^{v-1}l_s)(1+\sum_{t=1}^{v-1}u_t)}a^{(1+\sum_{t=1}^{v-1}u_t)(n-\sum_{s=1}^{v-1}l_s-\sum_{t=1}^{v-1}u_t)}}{b^{(n-\sum_{s=1}^{v}l_s-\sum_{t=1}^{v-1}u_t)(\sum_{s=1}^{v}l_s)}x^{(\sum_{s=1}^{v}l_s)(1+\sum_{t=1}^{v-1}u_t)}a^{(1+\sum_{t=1}^{v-1}u_t)(n-\sum_{s=1}^{v}l_s-\sum_{t=1}^{v-1}u_t)}} \right. \\
    &\quad \left. \cdot \frac{b^{(n-\sum_{s=1}^{v}l_s-\sum_{t=1}^{v-1}u_t)(1+\sum_{s=1}^{v}l_s)}x^{(1+\sum_{s=1}^{v}l_s)(\sum_{t=1}^{v-1}u_t)}a^{(\sum_{t=1}^{v-1}u_t)(n-\sum_{s=1}^{v}l_s-\sum_{t=1}^{v-1}u_t)}}{b^{(n-\sum_{s=1}^{v}l_s-\sum_{t=1}^{v}u_t)(1+\sum_{s=1}^{v}l_s)}x^{(1+\sum_{s=1}^{v}l_s)(\sum_{t=1}^{v}u_t)}a^{(\sum_{t=1}^{v}u_t)(n-\sum_{s=1}^{v}l_s-\sum_{t=1}^{v}u_t)}} \right) \\
    &\quad = \sum_{\substack{(l_1+1),l_i,u_i,(u_k+1), k \in \mathbb{N}: \\ \sum_{i=1}^{k}(l_i+u_i) = n}} 
       \left(c^{\sum_{s=1}^{k}l_s}x^{(\sum_{s=1}^{k}l_s)(\sum_{t=1}^{k}u_t)}d^{\sum_{t=1}^{k}u_t} \right. \\
       &\quad \left. \cdot \prod_{v=1}^{k} b^{l_v^2 + l_v(\sum_{s=1}^{v-1}l_s) - l_v(n-\sum_{s=1}^{v-1}l_s-\sum_{t=1}^{v-1}u_t)}x^{-l_v(1+\sum_{t=1}^{v-1}u_t)}a^{l_v(1+\sum_{t=1}^{v-1}u_t)} \right. \\
    &\quad \left. \cdot b^{u_v(1+\sum_{t=1}^{v}l_s)}x^{-u_v(1+\sum_{t=1}^{v}l_s)}a^{u_v^2 + u_v(\sum_{t=1}^{v-1}u_t) - u_v(n-\sum_{s=1}^{v}l_s-\sum_{t=1}^{v-1}u_t)} \right) \\
    &\quad = \frac{1}{x^n} \sum_{p=0}^{n} \binom{n}{p} (ac)^p (bd)^{n-p} = l^n;
\end{align*}

\begin{align*}
    H'_{[1,1]}
    &= \sum_{\substack{(r_1+1),r_i,d_i,(d_k+1), k \in \mathbb{N}: \\ \sum_{i=1}^{k}(r_i+d_i) = n}} 
       \left(b^{\sum_{s=1}^{k}d_s}x^{(\sum_{s=1}^{k}d_s)(\sum_{t=1}^{k}r_t)}a^{\sum_{t=1}^{k}r_t}\right. \\
       &\quad \left. \cdot \prod_{v=1}^{k} \frac{c^{(n-\sum_{t=1}^{v-1}r_t-\sum_{s=1}^{v-1}d_s)(1+\sum_{s=1}^{v-1}d_s)}x^{(1+\sum_{s=1}^{v-1}d_s)(\sum_{t=1}^{v-1}r_t)}d^{(\sum_{t=1}^{v-1}r_t)(n-\sum_{t=1}^{v-1}r_t-\sum_{s=1}^{v-1}d_s)}}{c^{(n-\sum_{t=1}^{v}r_t-\sum_{s=1}^{v-1}d_s)(1+\sum_{s=1}^{v-1}d_s)}x^{(1+\sum_{s=1}^{v-1}d_s)(\sum_{t=1}^{v}r_t)}d^{(\sum_{t=1}^{v}r_t)(n-\sum_{t=1}^{v}r_t-\sum_{s=1}^{v-1}d_s)}} \right. \\
    &\quad \left. \cdot \frac{c^{(n-\sum_{t=1}^{v}r_t-\sum_{s=1}^{v-1}d_s)(\sum_{s=1}^{v-1}d_s)}x^{(\sum_{s=1}^{v-1}d_s)(1+\sum_{t=1}^{v}r_t)}d^{(1+\sum_{t=1}^{v}r_t)(n-\sum_{t=1}^{v}r_t-\sum_{s=1}^{v-1}d_s)}}{c^{(n-\sum_{t=1}^{v}r_t-\sum_{s=1}^{v}d_s)(\sum_{s=1}^{v}d_s)}x^{(\sum_{s=1}^{v}d_s)(1+\sum_{t=1}^{v}r_t)}d^{(1+\sum_{t=1}^{v}r_t)(n-\sum_{t=1}^{v}r_t-\sum_{s=1}^{v}d_s)}} \right) \\
    &\quad = \sum_{\substack{(l_1+1),l_i,u_i,(u_k+1), k \in \mathbb{N}: \\ \sum_{i=1}^{k}(l_i+u_i) = n}} 
       \left(b^{\sum_{s=1}^{k}d_s}x^{(\sum_{s=1}^{k}d_s)(\sum_{t=1}^{k}r_t)}a^{\sum_{t=1}^{k}r_t} \right. \\
       &\quad \left. \cdot \prod_{v=1}^{k} c^{r_v(1+\sum_{s=1}^{v-1}d_s)}x^{-r_v(1+\sum_{s=1}^{v-1}d_s)}d^{r_v^2+r_v(\sum_{t=1}^{v-1}r_t)-r_v(n-\sum_{t=1}^{v-1}r_t-\sum_{s=1}^{v-1}d_s)} \right. \\
    &\quad \left. \cdot c^{d_v^2+d_v(\sum_{s=1}^{v}d_s)-d_v(n-\sum_{t=1}^{v}r_t-\sum_{s=1}^{v-1}d_s)}x^{-d_v(1+\sum_{t=1}^{v}r_t)}d^{d_v(1+\sum_{t=1}^{v}r_t)} \right) \\
    &\quad = \frac{1}{x^n} \sum_{p=0}^{n} \binom{n}{p} (ac)^p (bd)^{n-p} = l^n.
\end{align*}
Now we compute all values $H'_{[g,h]}$ for $(g-1)(h-1)(g-n)(h-n) = 0$ and $(g,h) \notin \{(1,1), (1,n), (n,1), (n,n) \}$. Note that by symmetry from the center of the square, it suffices to compute $H'_{[g,h]}$ for $(g-n)(h-n) = 0$ by applying \eqref{eq:flip_formula1}, then the remaining cases follow similarly by applying \eqref{eq:flip_formula2}. 

For $g = n$, then $h \neq n$, computing $H'_{[g,h]} = H'_{[n,h]}$ becomes considering its envelope (a square of side length $h+1$, see Definition \ref{def:mirror}) and applying formula \eqref{eq:flip_formula1} for that envelope. Now assigning the coordinates $(p,q,r) \in \Gamma_{h+1}$ for that smaller square, we check that all vertices with coordinate $(p,q,r)$ inside that smaller square (only excluding the 4 corners) have coordinate $(p, q, r+n-h)$ from the initial square (big square), and the corresponding initial value of each vertex is either of the form $b^{q(r+n-h)}x^{p(r+n-h)}a^{pq}$ (lower-right part with respect to the main diagonal) or $c^{q(r+n-h)}x^{p(r+n-h)}d^{pq}$. Hence, we compute:
\begin{align*}
	H'_{[n,h]} 
	&= \sum_{\substack{(l_1+1),l_i,u_i,(u_k+1), \\ k \in \mathbb{N}: \\ \sum_{i=1}^{k}(l_i+u_i) = h}} 
       \left(c^{n-h+\sum_{s=1}^{k}l_s}x^{(n-h+\sum_{s=1}^{k}l_s)(\sum_{t=1}^{k}u_t)}d^{\sum_{t=1}^{k}u_t} \right. \\
       &\quad \left. \cdot \prod_{v=1}^{k} \frac{b^{(h-\sum_{s=1}^{v-1}l_s-\sum_{t=1}^{v-1}u_t)(n-h+\sum_{s=1}^{v-1}l_s)}x^{(n-h+\sum_{s=1}^{v-1}l_s)(1+\sum_{t=1}^{v-1}u_t)}a^{(1+\sum_{t=1}^{v-1}u_t)(h-\sum_{s=1}^{v-1}l_s-\sum_{t=1}^{v-1}u_t)}}{b^{(h-\sum_{s=1}^{v}l_s-\sum_{t=1}^{v-1}u_t)(n-h+\sum_{s=1}^{v}l_s)}x^{(n-h+\sum_{s=1}^{v}l_s)(1+\sum_{t=1}^{v-1}u_t)}a^{(1+\sum_{t=1}^{v-1}u_t)(h-\sum_{s=1}^{v}l_s-\sum_{t=1}^{v-1}u_t)}} \right. \\
    &\quad \left. \cdot \frac{b^{(h-\sum_{s=1}^{v}l_s-\sum_{t=1}^{v-1}u_t)(n-h+1+\sum_{s=1}^{v}l_s)}x^{(n-h+1+\sum_{s=1}^{v}l_s)(\sum_{t=1}^{v-1}u_t)}a^{(\sum_{t=1}^{v-1}u_t)(h-\sum_{s=1}^{v}l_s-\sum_{t=1}^{v-1}u_t)}}{b^{(h-\sum_{s=1}^{v}l_s-\sum_{t=1}^{v}u_t)(n-h+1+\sum_{s=1}^{v}l_s)}x^{(n-h+1+\sum_{s=1}^{v}l_s)(\sum_{t=1}^{v}u_t)}a^{(\sum_{t=1}^{v}u_t)(h-\sum_{s=1}^{v}l_s-\sum_{t=1}^{v}u_t)}} \right) \\
    &\quad = \sum_{\substack{(l_1+1),l_i,u_i,(u_k+1), \\ k \in \mathbb{N}: \\ \sum_{i=1}^{k}(l_i+u_i) = n}} 
       \left(c^{n-h}c^{\sum_{s=1}^{k}l_s}x^{-h}d^{\sum_{t=1}^{k}u_t}b^{h(n-h)}b^{\sum_{t=1}^{k}u_t}a^{\sum_{s=1}^{k}l_s} \right) \\
    &\quad = b^{h(n-h)}c^{n-h}l^h.
\end{align*}
Similarly, by reflection with respect to the resulting diagonal (after flip), we conclude all computations:
\Eq{
    H'_{[n,h]} &= b^{h(n-h)}c^{n-h}l^h,&H'_{[g,n]} &= a^{g(n-g)}d^{n-g}l^g, \\
    H'_{[1,h]} &= a^{h-1}d^{(h-1)(n+1-h)}l^{n+1-h}, &H'_{[g,1]} &= b^{g-1}c^{(g-1)(n+1-g)}l^{n+1-g}.
}
\label{eq:bound_1}
Recall the layers definition as mentioned in the beginning of Section~\ref{subsec:unpunctured_case}. Consider a point $[i,j]$ lying in the $k$-th layer, or equivalently, $(i-k)(j-k)(i-n-1+k)(j-n-1+k) = 0$ and $k \leq i,j \leq n+1-k$. For each $l = 1,2,...,k$, let $H^{[k]}_{[i,j]}$ be the value of the vertex $[i,j]$ after the $k$-th mutation in the mutations sequence order to get the flip. From \eqref{eq:bound_1}, we have already computed all values $H^{[1]}_{[i,j]}$ for all first layer vertices $[i,j]$. Now we consider 3 possible cases:

\emph{Case 1.1:} $i+j = n+1$. Rewrite $j = n+1-i$, then $H^{[1]}_{[i,j]}$ can be computed:
		\begin{align*}
			H^{[1]}_{[i,j]} &= \frac{H_{[i-1,j]}H_{[i+1,j]}+H_{[i,j-1]}H_{[i,j+1]}}{H_{[i,j]}} \\
			&= \frac{c^{(i-1)(n+2-i-j)}d^{j(n+2-i-j)}x^{(i-1)j} \cdot a^{(n-i)(i+j-n)}b^{(n+1-j)(i+j-n)}x^{(n-i)(n+1-j)}}{x^{ij}} \\
			&\quad + \frac{c^{i(n+2-i-j)}d^{(j-1)(n+2-i-j)}x^{i(j-1)} \cdot a^{(n+1-i)(i+j-n)}b^{(n-j)(i+j-n)}x^{(n+1-i)(n-j)}}{x^{ij}} \\
			&= (ad)^{j-1}(bc)^{i-1}x^{(i-1)(j-1)}l.
		\end{align*}

\emph{Case 1.2:} $i+j \geq n+2$. Consider the envelope of $[i,j]$, which is a square with side length $(i+j)-n+1$. If we have the coordinate $(p,q,r) \in \Gamma_{(i+j)-n+1}$ of the envelope (either in the upper or lower part of the main diagonal), then the coordinate for the initial square is $(p+n-i,q,r+n-j) \in \Gamma_{n+1}$. Apply \eqref{eq:flip_formula1}, $H^{[1]}_{[i,j]}$ can be computed to be

\begin{align*}
	H^{[1]}_{[i,j]} 
	&= \sum_{\substack{(l_1+1),l_w,u_w,(u_k+1), \\ k \in \mathbb{N}: \\ \sum_{w=1}^{k}(l_i+u_i) = i+j-n}} 
       \left(c^{n-j+\sum_{s=1}^{k}l_s}x^{(n-j+\sum_{s=1}^{k}l_s)(n-i+\sum_{t=1}^{k}u_t)}d^{n-i+\sum_{t=1}^{k}u_t} \right. \\
       &\quad \left. \cdot \prod_{v=1}^{k} \frac{b^{(i+j-n-\sum_{s=1}^{v-1}l_s-\sum_{t=1}^{v-1}u_t)(n-j+\sum_{s=1}^{v-1}l_s)}}{b^{(i+j-n-\sum_{s=1}^{v}l_s-\sum_{t=1}^{v-1}u_t)(n-j+\sum_{s=1}^{v}l_s)}}\frac{x^{(n-j+\sum_{s=1}^{v-1}l_s)(1+n-i+\sum_{t=1}^{v-1}u_t)}}{x^{(n-j+\sum_{s=1}^{v}l_s)(1+n-i+\sum_{t=1}^{v-1}u_t)}} \right. \\
    &\quad \left. \cdot \frac{a^{(1+n-i+\sum_{t=1}^{v-1}u_t)(i+j-n-\sum_{s=1}^{v-1}l_s-\sum_{t=1}^{v-1}u_t)}}{a^{(1+n-i+\sum_{t=1}^{v-1}u_t)(i+j-n-\sum_{s=1}^{v}l_s-\sum_{t=1}^{v-1}u_t)}}\frac{b^{(i+j-n-\sum_{s=1}^{v}l_s-\sum_{t=1}^{v-1}u_t)(n-j+1+\sum_{s=1}^{v}l_s)}}{b^{(i+j-n-\sum_{s=1}^{v}l_s-\sum_{t=1}^{v}u_t)(n-j+1+\sum_{s=1}^{v}l_s)}} \right. \\
    &\quad \left. \cdot \frac{x^{(n-j+1+\sum_{s=1}^{v}l_s)(n-i+\sum_{t=1}^{v-1}u_t)}}{x^{(n-j+1+\sum_{s=1}^{v}l_s)(n-i+\sum_{t=1}^{v}u_t)}}\frac{a^{(n-i+\sum_{t=1}^{v-1}u_t)(i+j-n-\sum_{s=1}^{v}l_s-\sum_{t=1}^{v-1}u_t)}}{a^{(n-i+\sum_{t=1}^{v}u_t)(i+j-n-\sum_{s=1}^{v}l_s-\sum_{t=1}^{v}u_t)}} \right) \\
    &\quad = a^{(n-i)(i+j-n)}b^{(n-j)(i+j-n)}c^{n-j}d^{n-i}x^{(n-i)(n-j)-(i+j-n)} \cdot \sum_{p=0}^{i+j-n} \binom{i+j-n}{p} (ac)^p (bd)^{i+j-n-p} \\
    &\quad = a^{(n-i)(i+j-n)}b^{(n-j)(i+j-n)}c^{n-j}d^{n-i}x^{(n-i)(n-j)}l^{i+j-n}.
\end{align*}

\emph{Case 1.3:} $i+j \leq n$. Consider the envelope of $[i,j]$, which is a square with side length $n+3-(i+j)$. If we have the coordinate $(p,q,r) \in \Gamma_{n+3-(i+j)}$ of the envelope (either in upper or lower part of the main diagonal), then the coordinate for the initial square is $(p+j-1,q,r+i-1) \in \Gamma_{n+1}$. Apply \eqref{eq:flip_formula2}, $H^{[1]}_{[i,j]}$ can be computed to be
\begin{align*}
	H^{[1]}_{[i,j]} 
	&= \sum_{\substack{(r_1+1),r_i,d_i,(d_k+1), \\ k \in \mathbb{N}: \\ \sum_{i=1}^{k}(r_i+d_i) = n+2-(i+j)}} 
       \left(b^{i-1+\sum_{s=1}^{k}d_s}x^{(i-1+\sum_{s=1}^{k}d_s)(j-1+\sum_{t=1}^{k}r_t)}a^{j-1+\sum_{t=1}^{k}r_t}\right. \\
       &\quad \left. \cdot \prod_{v=1}^{k} \frac{c^{(n+2-(i+j)-\sum_{t=1}^{v-1}r_t-\sum_{s=1}^{v-1}d_s)(1+i-1+\sum_{s=1}^{v-1}d_s)}}{c^{(n+2-(i+j)-\sum_{t=1}^{v}r_t-\sum_{s=1}^{v-1}d_s)(1+i-1+\sum_{s=1}^{v-1}d_s)}}\frac{x^{(1+i-1+\sum_{s=1}^{v-1}d_s)(j-1+\sum_{t=1}^{v-1}r_t)}}{x^{(1+i-1+\sum_{s=1}^{v-1}d_s)(j-1+\sum_{t=1}^{v}r_t)}} \right. \\
       &\quad \left. \cdot \frac{d^{(j-1+\sum_{t=1}^{v-1}r_t)(n+2-(i+j)-\sum_{t=1}^{v-1}r_t-\sum_{s=1}^{v-1}d_s)}}{d^{(j-1+\sum_{t=1}^{v}r_t)(n+2-(i+j)-\sum_{t=1}^{v}r_t-\sum_{s=1}^{v-1}d_s)}}\frac{c^{(n+2-(i+j)-\sum_{t=1}^{v}r_t-\sum_{s=1}^{v-1}d_s)(i-1+\sum_{s=1}^{v-1}d_s)}}{c^{(n+2-(i+j)-\sum_{t=1}^{v}r_t-\sum_{s=1}^{v}d_s)(i-1+\sum_{s=1}^{v}d_s)}} \right. \\
    &\quad \left. \cdot \frac{x^{(i-1+\sum_{s=1}^{v-1}d_s)(1+j-1+\sum_{t=1}^{v}r_t)}}{x^{(i-1+\sum_{s=1}^{v}d_s)(1+j-1+\sum_{t=1}^{v}r_t)}}\frac{d^{(1+j-1+\sum_{t=1}^{v}r_t)(n+2-(i+j)-\sum_{t=1}^{v}r_t-\sum_{s=1}^{v-1}d_s)}}{d^{(1+j-1+\sum_{t=1}^{v}r_t)(n+2-(i+j)-\sum_{t=1}^{v}r_t-\sum_{s=1}^{v}d_s)}} \right) \\
    &\quad = a^{j-1}b^{i-1}c^{(i-1)(n+2-i-j)}d^{(j-1)(n+2-i-j)}x^{(i-1)(j-1)-(n+2-i-j)}\\
    &\quad \cdot \sum_{p=0}^{n+2-i-j} \binom{n+2-i-j}{p} (ac)^p (bd)^{n+2-i-j-p} \\
    &\quad = a^{j-1}b^{i-1}c^{(i-1)(n+2-i-j)}d^{(j-1)(n+2-i-j)}x^{(i-1)(j-1)}l^{n+2-i-j}.
\end{align*}	
Checking the formulas, we see that (\ref{eq:bound_1}) is satisfied. Therefore, we conclude that:
\begin{align}
	H^{[1]}_{[i,j]} = \begin{cases}
		a^{j-1}b^{i-1}c^{(i-1)(n+2-i-j)}d^{(j-1)(n+2-i-j)}x^{(i-1)(j-1)}l^{n+2-i-j} & \text{if } i+j \leq n, \\
		a^{(n-i)(i+j-n)}b^{(n-j)(i+j-n)}c^{n-j}d^{n-i}x^{(n-i)(n-j)}l^{i+j-n} & \text{if } i+j \geq n+2, \\
		(ad)^{j-1}(bc)^{i-1}x^{(i-1)(j-1)}l & \text{if } i+j = n+1.
	\end{cases}
\end{align}
Now consider the square with $4$ edges precisely containing all the vertices of layer 1; it remains to consider only this square. Thanks to the stair path structure following from Proposition~\ref{prop:stair_path_structure_time} with formulas \eqref{eq:flip_formula1} and \eqref{eq:flip_formula2}, we can compute $H^{[2]}_{[i,j]}$ for $2 \leq i,j \leq n-1$ by using the resulting square, since the method satisfies the order of the mutation sequence of the flip. Consider the envelope of those vertices, then again, we consider 3 possible cases:

\emph{Case 2.1:} $i+j = n+1$. Rewrite $j = n+1-i$, then $H^{[1]}_{[i,j]}$ can be computed to be:
		\begin{align*}
			H^{[2]}_{[i,j]} = \frac{H^{[1]}_{[i-1,j]}H^{[1]}_{[i+1,j]}+H^{[1]}_{[i,j-1]}H^{[1]}_{[i,j+1]}}{H^{[1]}_{[i,j]}} = (ad)^{2(j-2)}(bc)^{2(i-2)}x^{(i-2)(j-2)}l^4.
		\end{align*}
		
\emph{Case 2.2:} $i+j \geq n+2$. Consider the envelope of $[i,j]$ which is a square with side length $(i+j)-n+1$, then $H^{[2]}_{[i,j]}$ can be computed to be:
\begin{align*}
	H^{[2]}_{[i,j]} 
	&= \sum_{\substack{(l_1+1),l_w,u_w,(u_k+1), \\ k \in \mathbb{N}: \\ \sum_{w=1}^{k}(l_i+u_i) = i+j-n}} 
       \left(H^{[1]}_{[i-\sum_{t=1}^{k}u_t,j-\sum_{s=1}^{k}l_s]} \cdot \prod_{v=1}^{k} \frac{H^{[1]}_{[i-\sum_{t=1}^{v-1}u_t,j+1-\sum_{s=1}^{v-1}l_s]}}{H^{[1]}_{[i-\sum_{t=1}^{v-1}u_t,j+1-\sum_{s=1}^{v}l_s]}} \cdot \frac{H^{[1]}_{[i+1-\sum_{t=1}^{v-1}u_t,j-\sum_{s=1}^{v}l_s]}}{H^{[1]}_{[i+1-\sum_{t=1}^{v}u_t,j-\sum_{s=1}^{v}l_s]}} \right) \\
    &\quad = a^{(n-2-i)(i+j-n)+(j-1)}b^{(n-2-j)(i+j-n)+(i-1)}c^{2(n-j-1)}d^{2(n-i-1)}x^{(i-1)(j-1)-(n-2)(i+j-n)}l^{2(i+j-n+1)} \\
    &\quad = a^{(n-1-i)(i+j-n+1)}b^{(n-1-j)(i+j-n+1)}c^{2(n-j-1)}d^{2(n-i-1)}x^{(n-1-i)(n-1-j)}l^{2(i+j-n+1)}.
\end{align*}

\emph{Case 2.3:} $i+j \leq n$. Consider the envelope of $[i,j]$ which is a square with side length $n+3-(i+j)$, then $H^{[2]}_{[i,j]}$ can be computed to be:
\begin{align*}
	H^{[2]}_{[i,j]} 
	&= \sum_{\substack{(r_1+1),r_w,d_w,(d_k+1), \\ k \in \mathbb{N}: \\ \sum_{w=1}^{k}(l_i+u_i) = n+2-i-j}} 
       \left(H^{[1]}_{[i+\sum_{t=1}^{k}d_t,j+\sum_{s=1}^{k}r_s]} \cdot \prod_{v=1}^{k} \frac{H^{[1]}_{[i+\sum_{t=1}^{v-1}d_t,j-1+\sum_{s=1}^{v-1}r_s]}}{H^{[1]}_{[i+\sum_{t=1}^{v-1}d_t,j-1+\sum_{s=1}^{v}r_s]}} \cdot \frac{H^{[1]}_{[i-1+\sum_{t=1}^{v-1}d_t,j+\sum_{s=1}^{v}r_s]}}{H^{[1]}_{[i-1+\sum_{t=1}^{v}d_t,j+\sum_{s=1}^{v}r_s]}} \right) \\
    &\quad = a^{2(j-2)}b^{2(i-2)}c^{(i-2)(n+3-i-j)}d^{(j-2)(n+3-i-j)}x^{(i-2)(j-2)}l^{2(n+3-i-j)}.
\end{align*}
Therefore, we conclude that:
\begin{align}
	H^{[2]}_{[i,j]} = \begin{cases}
		a^{2(j-2)}b^{2(i-2)}c^{(i-2)(n+3-i-j)}d^{(j-2)(n+3-i-j)}x^{(i-2)(j-2)}l^{2(n+3-i-j)} & \text{if } i+j \leq n, \\
		a^{(n-1-i)(i+j-n+1)}b^{(n-1-j)(i+j-n+1)}c^{2(n-j-1)}d^{2(n-i-1)}x^{(n-1-i)(n-1-j)}l^{2(i+j-n+1)} & \text{if } i+j \geq n+2, \\
		(ad)^{2(j-2)}(bc)^{2(i-2)}x^{(i-2)(j-2)}l^4 & \text{if } i+j = n+1.
	\end{cases}
\end{align}
Repeating the process, we continue with all remaining vertices inside the square formed by layer 2. Then inductively, we can compute for general $1 \leq k \leq \lfloor \frac{n+1}{2} \rfloor$:
\begin{align}
	H^{[k]}_{[i,j]} =\case{
		a^{k(j-k)}b^{k(i-k)}c^{(i-k)(n+k+1-i-j)}d^{(j-k)(n+k+1-i-j)}x^{(i-k)(j-k)}l^{k(n+k+1-i-j)} & \text{ if } i+j \leq n, \\\\
		a^{(n-k+1-i)(i+j-n+k-1)}b^{(n-k+1-j)(i+j-n+k-1)}c^{k(n-k+1-j)}d^{k(n-k+1-i)}\cdot& \text{ if } i+j \geq n+2,\\
		\quad\cdot x^{(n-k+1-i)(n-k+1-j)}l^{k(i+j-n+k-1)} \\\\
		(ad)^{k(j-k)}(bc)^{k(i-k)}x^{(i-k)(j-k)}l^{k^2} &\text{ if } i+j = n+1.
		}
\end{align}
Then for any vertex $[i,j]$ on layer $k$, or equivalently, $(i-k)(j-k)(i-n-1+k)(j-n-1+k) = 0$ and $k \leq i,j \leq n+1-k$, we get the desired formula:
\begin{align}
	H'_{[i,j]} = H^{[k]}_{[i,j]} =
	\begin{cases}
		H^{[k]}_{[k,j]} = a^{k(j-k)}d^{(j-k)(n+1-j)}l^{k(n+1-j)} & \text{ for } k \leq j \leq n+1-k, \\
		H^{[k]}_{[i,k]} = b^{k(i-k)}c^{(i-k)(n+1-i)}l^{k(n+1-i)} & \text{ for } k \leq i \leq n+1-k, \\
		H^{[k]}_{[n+1-k,j]} = b^{(n-k+1-j)j}c^{k(n-k+1-j)}l^{kj}  & \text{ for } k \leq j \leq n+1-k, \\
		H^{[k]}_{[i,n+1-k]} = a^{i(n-k+1-i)}d^{k(n-k+1-i)}l^{ki}  & \text{ for } k \leq i \leq n+1-k \\		
	\end{cases}
\end{align}
which verifies the well triangulated property of both triangles after the flip of the initial square. Therefore, from a given ideal triangulation such that all triangles are well-triangulated, a flip does not change the well-triangulated property for any triangles from the resulting triangulation. Hence, we can apply the property to any sequence of flips. This concludes the proof of the theorem.
\end{proof}
See Figures~\ref{fig:flip_values_Hat} and~\ref{fig:flip_values_GS19} for examples of values before and after flips. 
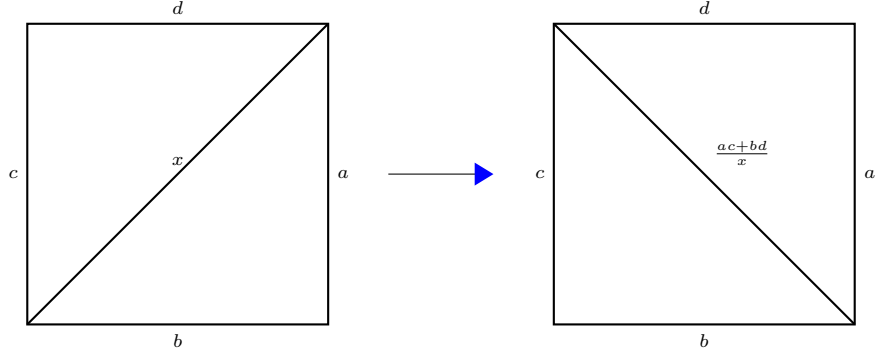
\begin{figure}[H]
\centering
\begin{tikzpicture}[scale=1,
    mid arrow/.style={
        postaction={decorate},
        decoration={
            markings,
            mark=at position 0.5 with {\arrow{>}}
        }
    }
]
	\draw[thick] (0,0) -- (4,0) -- (4,4) -- (0,4) -- cycle;
	\draw[thick] (7,0) -- (11,0) -- (11,4) -- (7,4) -- cycle;
	\draw[thick] (0,0) -- (4,4);
	\draw[thick] (11,0) -- (7,4);
	\draw[-{Triangle[blue,scale=2,line width=1.5pt]}] (4.8,2) -- (6.2,2);
	\node[above, font=\scriptsize] at (2,4) {$d$};
	\node[right, font=\scriptsize] at (4,2) {$a$};
	\node[below, font=\scriptsize] at (2,0) {$b$};
	\node[left, font=\scriptsize] at (0,2) {$c$};
	\node[above, font=\scriptsize] at (2,2) {$x$};
	\node[above, font=\scriptsize] at (9,4) {$d$};
	\node[right, font=\scriptsize] at (11,2) {$a$};
	\node[below, font=\scriptsize] at (9,0) {$b$};
	\node[left, font=\scriptsize] at (7,2) {$c$};
	\node[above right, font=\scriptsize] at (9,2) {$\frac{ac+bd}{x}$};
\end{tikzpicture}
\caption{Values of vertices before and after flip in case $n=1$.}
\label{fig:flip_values_Hat}
\end{figure}

\begin{figure}[H]
\centering
\begin{tikzpicture}[scale=1,
    mid arrow/.style={
        postaction={decorate},
        decoration={
            markings,
            mark=at position 0.5 with {\arrow{>}}
        }
    }
]
	\draw[thick] (0,0) -- (4,0) -- (4,4) -- (0,4) -- cycle;
	\draw[thick] (7,0) -- (11,0) -- (11,4) -- (7,4) -- cycle;
	\draw[thick] (0,0) -- (4,4);
	\draw[thick] (11,0) -- (7,4);
	\draw[-{Triangle[blue,scale=2,line width=1.5pt]}] (4.8,2) -- (6.2,2);
	\node[below] at (5.5,1.8) {$l = \frac{ac+bd}{x}$};
	
\foreach \i in {0,1} {	
	\filldraw[blue] (7*\i+1,4) circle (2pt); 
  	\node[above, font=\scriptsize] at (7*\i+1,4) {$d^3$};
	\filldraw[blue] (7*\i+2,4) circle (2pt); 
  	\node[above, font=\scriptsize] at (7*\i+2,4) {$d^4$};
  	\filldraw[blue] (7*\i+3,4) circle (2pt); 
  	\node[above, font=\scriptsize] at (7*\i+3,4) {$d^3$};
  	\filldraw[blue] (7*\i+4,3) circle (2pt); 
  	\node[right, font=\scriptsize] at (7*\i+4,3) {$a^3$};
	\filldraw[blue] (7*\i+4,2) circle (2pt); 
  	\node[right, font=\scriptsize] at (7*\i+4,2) {$a^4$};
  	\filldraw[blue] (7*\i+4,1) circle (2pt); 
  	\node[right, font=\scriptsize] at (7*\i+4,1) {$a^3$};
  	\filldraw[blue] (7*\i+3,0) circle (2pt); 
  	\node[below, font=\scriptsize] at (7*\i+3,0) {$b^3$};
	\filldraw[blue] (7*\i+2,0) circle (2pt); 
  	\node[below, font=\scriptsize] at (7*\i+2,0) {$b^4$};
  	\filldraw[blue] (7*\i+1,0) circle (2pt); 
  	\node[below, font=\scriptsize] at (7*\i+1,0) {$b^3$};
  	\filldraw[blue] (7*\i+0,1) circle (2pt); 
  	\node[left, font=\scriptsize] at (7*\i+0,1) {$c^3$};
	\filldraw[blue] (7*\i+0,2) circle (2pt); 
  	\node[left, font=\scriptsize] at (7*\i+0,2) {$c^4$};
  	\filldraw[blue] (7*\i+0,3) circle (2pt); 
  	\node[left, font=\scriptsize] at (7*\i+0,3) {$c^3$};
}
  	\filldraw[blue] (1,1) circle (2pt); 
  	\node[above, font=\scriptsize] at (1,1) {$x^3$};
  	\filldraw[blue] (2,2) circle (2pt); 
  	\node[above, font=\scriptsize] at (2,2) {$x^4$};
  	\filldraw[blue] (3,3) circle (2pt); 
  	\node[above, font=\scriptsize] at (3,3) {$x^3$};
  	\filldraw[blue] (2,1) circle (2pt); 
  	\node[above, font=\scriptsize] at (2,1) {$ab^2x^2$};
  	\filldraw[blue] (3,1) circle (2pt); 
  	\node[above, font=\scriptsize] at (3,1) {$a^2b^2x$};
  	\filldraw[blue] (1,2) circle (2pt); 
  	\node[above, font=\scriptsize] at (1,2) {$c^2dx^2$};
  	\filldraw[blue] (3,2) circle (2pt); 
  	\node[above, font=\scriptsize] at (3,2) {$a^2bx^2$};
  	\filldraw[blue] (1,3) circle (2pt); 
  	\node[above, font=\scriptsize] at (1,3) {$c^2d^2x$};
  	\filldraw[blue] (2,3) circle (2pt); 
  	\node[above, font=\scriptsize] at (2,3) {$cd^2x^2$};
  	
  	\filldraw[blue] (8,1) circle (2pt); 
  	\node[above, font=\scriptsize] at (8,1) {$b^2c^2l$};
  	\filldraw[blue] (9,2) circle (2pt); 
  	\node[above, font=\scriptsize] at (9,2) {$l^4$};
  	\filldraw[blue] (10,3) circle (2pt); 
  	\node[above, font=\scriptsize] at (10,3) {$a^2d^2l$};
  	\filldraw[blue] (9,1) circle (2pt); 
  	\node[above, font=\scriptsize] at (9,1) {$b^2cl^2$};
  	\filldraw[blue] (10,1) circle (2pt); 
  	\node[above, font=\scriptsize] at (10,1) {$l^3$};
  	\filldraw[blue] (8,2) circle (2pt); 
  	\node[above, font=\scriptsize] at (8,2) {$bc^2l^2$};
  	\filldraw[blue] (10,2) circle (2pt); 
  	\node[above, font=\scriptsize] at (10,2) {$a^2dl^2$};
  	\filldraw[blue] (8,3) circle (2pt); 
  	\node[above, font=\scriptsize] at (8,3) {$l^3$};
  	\filldraw[blue] (9,3) circle (2pt); 
  	\node[above, font=\scriptsize] at (9,3) {$ad^2l^2$};
\end{tikzpicture}
\caption{Values of vertices before and after flip in case $n=3$.}
\label{fig:flip_values_GS19}
\end{figure}
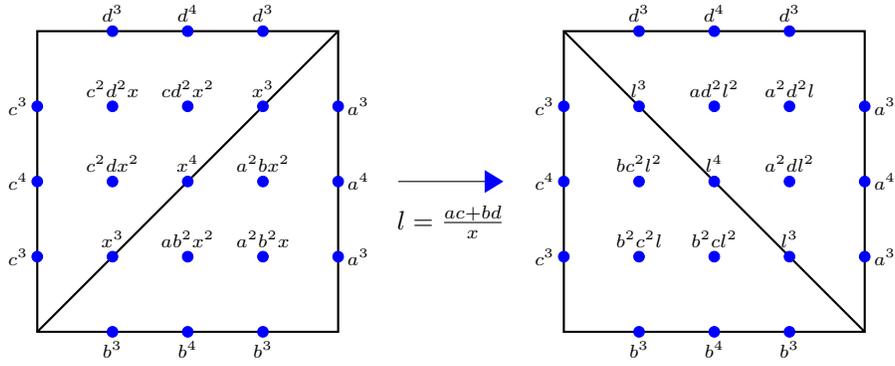

Now consider the labels of all vertices $x^{(k)}_{(a,b,c)}$ as described in Section~\ref{subsec:notation_labeling}. Instead of assigning variables or Laurent polynomials from each vertex after a sequence of flips, we set all initial variables to $1$, whence each vertex is assigned an integral value. Following from the sequence of flips from the initial triangulation of a surface for the base case, i.e. each triangle is $1$-triangulated, we can apply the theorem above to compute the number of monomials (counting multiplicities) in the case we use the same sequence of flips for the general case $n$-triangulated. Therefore, we conclude with another result:
\begin{Thm}[Number of Monomials Theorem]
\label{thm:monomial_count} Consider the cluster realization of $\mathscr{A}_{\text{SL}_{n+1},\bbS}$ for a triangulation $\cT$ of a surface $\bbS$ with boundary and without punctures, and fix any non $\cT$-diagonal $\gamma$. If the number of monomials (counting multiplicities) in the Laurent polynomial of the expansion formula of $\gamma$ for the case $n=1$ is $K$, then the number of monomials (counting multiplicities) in the Laurent polynomial in each coordinate of the expansion formula of $\gamma$ for general $n$ is of the form 
\Eq{(K^n, K^{2n-2}, \ldots, K^{t(n+1-t)}, \ldots, K^{2n-2}, K^n).}
\end{Thm}
\subsection{General recursive formulas for particular values of \( n \)}
Now turning to the setting of general polygon. Let \( \mathcal{P} \) be an $m$-gon with $m\geq 3$ and fix a diagonal $AB$. Finally let $m_1,n_1,m_2,n_2,...,m_k,n_k$ be positive integers (except $m_1,n_k$ which can be zero) satisfying $$m=3+m_1+\cdots+m_k+n_1+\cdots+n_k.$$
\begin{Def}(Vertex labeling)
We recursively label the boundary vertices of $\cP$ as follows:
\begin{enumerate}[label=(\arabic*)]
    \item Start with vertices \( A \) and \( B \);
    \item At level 1, append \( m_1 \) vertices \( V_1, \ldots, V_{m_1} \) to the left of $AB$, followed by \( n_1 \) vertices \( V_{m_1,1}, \ldots, V_{m_1,n_1} \) to the right of $AB$;
    \item At level 2, append \( m_2 \) vertices \( V_{m_1,n_1,1}, \ldots, V_{m_1,n_1,m_2} \) to the left, followed by \( n_2 \) vertices \( V_{m_1,n_1,m_2,1}, \ldots,\) \( V_{m_1,n_1,m_2,n_2} \) to the right;
    \item Continue recursively: at level \( i \), append \( m_i \) vertices labeled:
        \[
        V_{m_1,n_1,\ldots,m_{i-1},n_{i-1},1}, \ldots, V_{m_1,n_1,\ldots,m_{i-1},n_{i-1},m_i},
        \]
        to the left, followed by \( n_i \) vertices labeled:
        \[
        V_{m_1,n_1,\ldots,m_i,1}, \ldots, V_{m_1,n_1,\ldots,m_i,n_i}
        \]
        to the right.
        
    \item For the final vertex $V$, we can split into 2 different cases:
    		\begin{enumerate}
    			\item[(i)] If $n_k > 0$, then the vertex \( V:=V_{m_1,n_1,\ldots,m_k,n_k} \) is next to $B$, then end with the remaining vertex \( V:=V_{m_1,n_1,\ldots,m_k,n_k,1} \) next to $B$.
    			
    			\item[(ii)] If $n_k = 0$, then the vertex \( V:=V_{m_1,n_1,\ldots,m_k} \) is next to $B$, then end with the remaining vertex \( V:=V_{m_1,n_1,\ldots,m_k,1} \) next to $B$.
    		\end{enumerate}
In both cases, $V$ is always on the right of $AB$.  
\end{enumerate}
\end{Def}
\begin{Def}\label{def:set-up-poly} A triangulation $\cT$ of $\cP$ of type $\cP(m_1,n_1,m_2,n_2,...,m_k,n_k)$ is produced by joining vertices of $\cP$ in the following way:
\begin{itemize}
\item For $i=1,...,k$, join all vertices  $V_{m_1,n_1,m_2,n_2,...n_{i-1},x}$ with $V_{m_1,n_1,m_2,n_2,...,m_i,1}$,$(x=1,...,m_i)$;
\item For $i=1,...,k$, join all vertices  $V_{m_1,n_1,m_2,n_2,...,m_i,y}$ with $V_{m_1,n_1,m_2,n_2,...,m_i,n_i,1}$, $(y=1,...,n_i)$.
\end{itemize}
By construction the diagonal $AB$ is a \emph{non $\cT$-diagonal}, meaning that $AB$ intersects every triangle of $T$.
\end{Def}
See Figure \ref{fig:coord} for an example of the setup.

\begin{Def}
    For a polygon \( \mathcal{P}(p_1, p_2, \ldots, p_K) \) with \( p_i > 0 \), its \emph{reflected polygon} is
    \[
    \overline{\mathcal{P}}(p_1, p_2, \ldots, p_K) := \mathcal{P}(p_K, p_{K-1}, \ldots, p_1),
    \]
    with the same fixed diagonal $AB$.
\label{Def:ref_poly}
\end{Def}
By reflecting the sequence, we may assume \( m_1 > 0 \). For notation convenience we may omit trailing zeros: e.g., \( \mathcal{P}(3,5,12,0) = \mathcal{P}(3,5,12) \).

For the rest of the section, we shall fix the triangulation $\cT$ of an $m$-gon of type $\mathcal{P}(m_1,n_1,\ldots,m_k,n_k)$ with a designated non $\cT$-diagonal $\c=AB$, where we assume $ m_1, n_1, m_2, n_2, \ldots, m_k, (n_k+1) \in \mathbb{N}$ such that $(m_k, n_k) \notin \{(m_k,1), (1,0)\}$ (unless $k=1$, then $(m_k,n_k)=(1,0)$ is allowed).

\begin{Def}(Standard sequence of flips)
\label{def:standard seq_flips}
The diagonal \( \gamma = AB \) can be obtained from $\cT$ through a sequence of flips, producing diagonals \( d_1, d_2, \ldots, d_{M_k + N_k} \). Let \( r_u \in \{1, \ldots, m_u - 1\} \), \( s_u \in \{1, \ldots, n_u - 1\} \). The diagonals are constructed as follows:
\begin{align*}
    d_1 &= AV_2,\ d_2 = AV_3,\ \ldots,\ d_{m_1 - 1} = AV_{m_1},\ d_{m_1} = AV_{m_1,n_1,1}; \\
    d_{m_1 + 1} &= AV_{m_1,2},\ \ldots,\ d_{m_1 + n_1 - 1} = AV_{m_1,n_1},\ d_{m_1 + n_1} = AV_{m_1,n_1,m_2,1}; \\
    &\vdots \\
    d_{M_{u-1} + N_{u-1}} &= AV_{m_1,n_1,\ldots,m_{u-1},n_{u-1},m_u,1}, \\
    d_{M_{u-1} + N_{u-1} + r_u} &= AV_{m_1,n_1,\ldots,m_{u-1},n_{u-1},r_u+1}; \\
    d_{M_u + N_{u-1}} &= AV_{m_1,n_1,\ldots,m_{u-1},n_{u-1},m_u,n_u,1}, \\
   d_{M_u + N_{u-1} + s_u} &= AV_{m_1,n_1,\ldots,m_{u-1},n_{u-1},m_u,s_u+1}
\end{align*}
and for \( n_k = 0 \):
\begin{align*}
    d_{M_{k-1} + N_{k-1}} &= AV, \\
    d_{M_{k-1} + N_{k-1} + r_k} &= AV_{m_1,n_1,\ldots,m_{k-1},n_{k-1},r_k+1}; \\
    d_{M_k + N_k} &= AB = \gamma
\end{align*}
while for \( n_k > 0 \):
\begin{align*}
    d_{M_k + N_{k-1}} &= AV, \\
    d_{M_k + N_{k-1} + s_k} &= AV_{m_1,n_1,\ldots,m_{k-1},n_{k-1},m_k,s_k+1}; \\
    d_{M_k + N_k} &= AB = \gamma.
\end{align*}
We shall call the procedure above as the \emph{standard sequence of flips of} $\gamma$.
\end{Def}

\begin{Def}(Diagonal and inner vertices notations)
	Consider an $n$-triangulated $m$-gon with triangulation $\cT$ of type $\mathcal{P}(m_1, n_1, m_2, n_2, \ldots, m_k, n_k)$ as in Definition \ref{def:set-up-poly}. Denote the vertices of the diagonal \( \gamma = AB \) by the vector:
\Eq{
(D^{[m_1,n_1,\ldots,m_k,n_k]}_1, \ldots, D^{[m_1,n_1,\ldots,m_k,n_k]}_n),
}
and let the inner vertices of \( \triangle AVB \) be denoted by:
\Eq{
I^{[m_1,n_1,\ldots,m_k,n_k]}_{(i,j)}, \quad i, j \in \mathbb{N},\ i + j \leq n.
}
For \( n = 1 \), define \( D^{[m_1,n_1,\ldots,m_k,n_k]} := D^{[m_1,n_1,\ldots,m_k,n_k]}_1 \); for \( n = 2 \), define \( I^{[m_1,n_1,\ldots,m_k,n_k]} := I^{[m_1,n_1,\ldots,m_k,n_k]}_{(1,1)} \). See Figure~\ref{fig:inner3} for the case \( n = 3 \), where, in case $n_k > 0$, we denote $[X] := [m_1,n_1,m_2,n_2,...,m_k]$ and $[X,**] := [m_1,n_1,m_2,n_2,...,m_k,n_k]$, and in case $n_k = 0$, we denote $[X] := [m_1,n_1,m_2,n_2,...,n_{k-1}]$ and $[X,**] := [m_1,n_1,m_2,n_2,...,m_k]$.
\end{Def}
\begin{figure}[H]
\centering
\begin{tikzpicture}[scale=0.9,
    mid arrow/.style={
        postaction={decorate},
        decoration={
            markings,
            mark=at position 0.5 with {\arrow{>}}
        }
    }
]
  \draw[thick] (0,0) -- (1,4) -- (5,0) -- cycle;
  \filldraw[black] (0,0) circle (2pt);
  \node[left, font=\scriptsize] at (0,0) {$B$};
  \filldraw[black] (1,4) circle (2pt);
  \node[above, font=\scriptsize] at (1,4) {$V$};
  \filldraw[black] (5,0) circle (2pt);
  \node[right, font=\scriptsize] at (5,0) {$A$};
  
  \filldraw[blue] (0.25,1) circle (2pt);
  \filldraw[blue] (0.5,2) circle (2pt);
  \filldraw[blue] (0.75,3) circle (2pt);
  
  \filldraw[red] (2,3) circle (2pt);
  \node[above right, font=\scriptsize] at (2,3) {$D^{[X]}_3$};
  \filldraw[red] (3,2) circle (2pt);
  \node[above right, font=\scriptsize] at (3,2) {$D^{[X]}_2$};
  \filldraw[red] (4,1) circle (2pt);
  \node[above right, font=\scriptsize] at (4,1) {$D^{[X]}_1$};
  
  \filldraw[red] (1.25,0) circle (2pt);
  \node[below, font=\scriptsize] at (1.25,0) {$D^{[X,**]}_3$};
  \filldraw[red] (2.5,0) circle (2pt);
  \node[below, font=\scriptsize] at (2.5,0) {$D^{[X,**]}_2$};
  \filldraw[red] (3.75,0) circle (2pt);
  \node[below, font=\scriptsize] at (3.75,0) {$D^{[X,**]}_1$};
  
  \filldraw[red] (1.75,2) circle (2pt);
  \node[right, font=\scriptsize] at (1.75,2) {$I^{[X,**]}_{(1,2)}$};
  \filldraw[red] (1.5,1) circle (2pt);
  \node[right, font=\scriptsize] at (1.5,1) {$I^{[X,**]}_{(2,1)}$};
  \filldraw[red] (2.75,1) circle (2pt);
  \node[right, font=\scriptsize] at (2.75,1) {$I^{[X,**]}_{(1,1)}$};
\end{tikzpicture}
\caption{The resulting triangle after the sequence of flips for \( n = 3 \)}
\label{fig:inner3}
\end{figure}
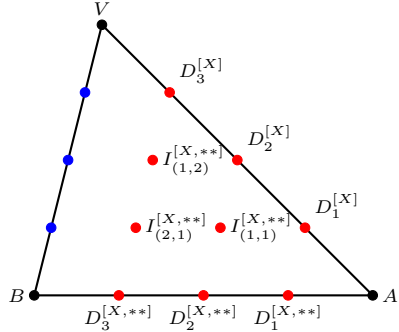
\begin{Def}
	Edges in the triangulation are colored (first = red, second = green, third = blue) to assign coordinates \( (a, b, c) \in \Gamma_{n+1} \). Starting from \( \triangle AV_{m_1,1}V_1 \) as an \emph{odd triangle}, all triangles are classified as either \emph{odd} or \emph{even} depending on the clockwise or counter-clockwise direction of the colors order {red, green, blue} on the edges ($v_1$ represents red, $v_2$ represents green, $v_3$ represents blue), with adjacent triangles having opposite parity (see Figure~\ref{fig:oddeven}).
\end{Def}
\begin{figure}[H]
\centering
\begin{tikzpicture}[scale=0.9,
    mid arrow/.style={
        postaction={decorate},
        decoration={
            markings,
            mark=at position 0.5 with {\arrow{>}}
        }
    }
]
  \draw[-{Latex[length=5mm]}, red] (0,0) -- (1,3);
  \draw[-{Latex[length=5mm]}, green] (1,3) -- (4,0);
  \draw[-{Latex[length=5mm]}, blue] (4,0) -- (0,0);
  
  \draw[-{Latex[length=5mm]}, red] (6,3) -- (5,0);
  \draw[-{Latex[length=5mm]}, green] (5,0) -- (9,0);
  \draw[-{Latex[length=5mm]}, blue] (9,0) -- (6,3);
  
  \node[left, font=\scriptsize] at (0.5,1.5) {$v_1$};
  \node[right, font=\scriptsize] at (2.5,1.5) {$v_2$};
  \node[below, font=\scriptsize] at (2,0) {$v_3$};    

  \node[left, font=\scriptsize] at (5.5,1.5) {$v_1$};
  \node[right, font=\scriptsize] at (7.5,1.5) {$v_3$};
  \node[below, font=\scriptsize] at (7,0) {$v_2$}; 
  
  \filldraw[black] (0,0) circle (2pt);
  \filldraw[black] (1,3) circle (2pt);
  \filldraw[black] (4,0) circle (2pt);
  \filldraw[black] (5,0) circle (2pt);
  \filldraw[black] (6,3) circle (2pt);
  \filldraw[black] (9,0) circle (2pt);
\end{tikzpicture}
\caption{Odd triangle (left) and even triangle (right)}
\label{fig:oddeven}
\end{figure}
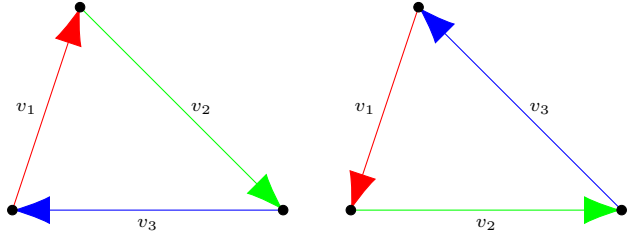
This colouring and parity assignment allow us to identify all coordinates in the general \( n \)-triangulated \( m \)-gon, as illustrated in Figure~\ref{fig:coord} for \( n = 2 \).
\begin{figure}[H]
    \centering
    \begin{tikzpicture}[scale=0.95]
        \def\n{15}
        \def\side{1.5}
        \def\extAngle{24} 
        
        \coordinate (V0) at (0,0);
        
        \foreach \i in {1,...,\n} {
            \coordinate (V\i) at ($(V\the\numexpr\i-1) + (\the\numexpr(\i-1)*\extAngle:\side)$);
        }
        
        \node[right, font=\scriptsize] at (V5) {$A$};
        \node[left, font=\scriptsize] at (V11) {$B$};
        \node[right, font=\scriptsize] at (V4) {$V_1$};
        \node[right, font=\scriptsize] at (V3) {$V_2$};
        \node[right, font=\scriptsize] at (V6) {$V_{2,1}$};
        \node[right, font=\scriptsize] at (V7) {$V_{2,2}$};
        \node[above, font=\scriptsize] at (V8) {$V_{2,3}$};
        \node[right, font=\scriptsize] at (V2) {$V_{2,3,1}$};
        \node[below, font=\scriptsize] at (V1) {$V_{2,3,2}$};
        \node[below, font=\scriptsize] at (V0) {$V_{2,3,3}$};
        \node[left, font=\scriptsize] at (V9) {$V_{2,3,3,1}$};
        \node[left, font=\scriptsize] at (V14) {$V_{2,3,3,1,1}$};
        \node[left, font=\scriptsize] at (V13) {$V_{2,3,3,1,2}$};
        \node[left, font=\scriptsize] at (V12) {$V_{2,3,3,1,3}$};
        \node[left, font=\scriptsize] at (V10) {$V$};
        
        \draw[thin,dashed] (V5) -- (V11);
        
        \draw[thick,red] (V0) -- (V1);
        \draw[thick,red] (V1) -- (V2);
        \draw[thick,blue] (V2) -- (V3);
        \draw[thick,blue] (V3) -- (V4);
        \draw[thick,blue] (V4) -- (V5);
        \draw[thick,green] (V5) -- (V6);
        \draw[thick,green] (V6) -- (V7);
        \draw[thick,green] (V7) -- (V8);
        \draw[thick,green] (V8) -- (V9);
        \draw[thick,blue] (V9) -- (V10);
        \draw[thick,blue] (V10) -- (V11);
        \draw[thick,green] (V11) -- (V12);
        \draw[thick,green] (V12) -- (V13);
        \draw[thick,green] (V13) -- (V14);
        \draw[thick,red] (V14) -- (V0);
        
        \draw[thick,red] (V6) -- (V4);
        \filldraw[blue] ($(V6)!0.333!(V4)$) circle (2pt);
        \filldraw[blue] ($(V6)!0.667!(V4)$) circle (2pt);
        \draw[thick,green] (V6) -- (V3);
        \filldraw[blue] ($(V6)!0.333!(V3)$) circle (2pt);
        \filldraw[blue] ($(V6)!0.667!(V3)$) circle (2pt);
        \draw[thick,red] (V6) -- (V2);
        \filldraw[blue] ($(V6)!0.333!(V2)$) circle (2pt);
        \filldraw[blue] ($(V6)!0.667!(V2)$) circle (2pt);
        \draw[thick,blue] (V2) -- (V7);
        \filldraw[blue] ($(V2)!0.333!(V7)$) circle (2pt);
        \filldraw[blue] ($(V2)!0.667!(V7)$) circle (2pt);
        \draw[thick,red] (V2) -- (V8);
        \filldraw[blue] ($(V2)!0.333!(V8)$) circle (2pt);
        \filldraw[blue] ($(V2)!0.667!(V8)$) circle (2pt);
        \draw[thick,blue] (V2) -- (V9);
        \filldraw[blue] ($(V2)!0.333!(V9)$) circle (2pt);
        \filldraw[blue] ($(V2)!0.667!(V9)$) circle (2pt);
        \draw[thick,green] (V9) -- (V1);
        \filldraw[blue] ($(V9)!0.333!(V1)$) circle (2pt);
        \filldraw[blue] ($(V9)!0.667!(V1)$) circle (2pt);
        \draw[thick,blue] (V9) -- (V0);
        \filldraw[blue] ($(V9)!0.333!(V0)$) circle (2pt);
        \filldraw[blue] ($(V9)!0.667!(V0)$) circle (2pt);
        \draw[thick,green] (V9) -- (V14);  
        \filldraw[blue] ($(V9)!0.333!(V14)$) circle (2pt);
        \filldraw[blue] ($(V9)!0.667!(V14)$) circle (2pt);     
        \draw[thick,red] (V10) -- (V14);
        \filldraw[blue] ($(V10)!0.333!(V14)$) circle (2pt);
        \filldraw[blue] ($(V10)!0.667!(V14)$) circle (2pt);
        \draw[thick,blue] (V10) -- (V13);
        \filldraw[blue] ($(V10)!0.333!(V13)$) circle (2pt);
        \filldraw[blue] ($(V10)!0.667!(V13)$) circle (2pt);
        \draw[thick,red] (V10) -- (V12); 
        \filldraw[blue] ($(V10)!0.333!(V12)$) circle (2pt);
        \filldraw[blue] ($(V10)!0.667!(V12)$) circle (2pt);      
         
        \foreach \i in {0,...,\n} {
            \filldraw[black] (V\i) circle (2pt);
        }
        
        \foreach \i[evaluate=\i as \isum using int(1+\i)] in {0,...,14} {
            \filldraw[blue] ($(V\i)!0.333!(V\isum)$) circle (2pt);
            \filldraw[blue] ($(V\i)!0.667!(V\isum)$) circle (2pt);
        }
        
        \coordinate (C1) at (barycentric cs:V0=1,V1=1,V9=1);
        \filldraw[blue] (C1) circle (2pt);
        \coordinate (C2) at (barycentric cs:V1=1,V2=1,V9=1);
        \filldraw[blue] (C2) circle (2pt);
        \coordinate (C3) at (barycentric cs:V14=1,V0=1,V9=1);
        \filldraw[blue] (C3) circle (2pt);
        \coordinate (C4) at (barycentric cs:V2=1,V8=1,V9=1);
        \filldraw[blue] (C4) circle (2pt);
        \coordinate (C5) at (barycentric cs:V2=1,V7=1,V8=1);
        \filldraw[blue] (C5) circle (2pt);
        \coordinate (C6) at (barycentric cs:V2=1,V6=1,V7=1);
        \filldraw[blue] (C6) circle (2pt);
        \coordinate (C7) at (barycentric cs:V2=1,V3=1,V6=1);
        \filldraw[blue] (C7) circle (2pt);
        \coordinate (C8) at (barycentric cs:V3=1,V4=1,V6=1);
        \filldraw[blue] (C8) circle (2pt);
        \coordinate (C9) at (barycentric cs:V4=1,V5=1,V6=1);
        \filldraw[blue] (C9) circle (2pt);
        \coordinate (C10) at (barycentric cs:V9=1,V10=1,V14=1);
        \filldraw[blue] (C10) circle (2pt);
        \coordinate (C11) at (barycentric cs:V10=1,V13=1,V14=1);
        \filldraw[blue] (C11) circle (2pt);
        \coordinate (C12) at (barycentric cs:V10=1,V12=1,V13=1);
        \filldraw[blue] (C12) circle (2pt);
        \coordinate (C13) at (barycentric cs:V10=1,V11=1,V12=1);
        \filldraw[blue] (C13) circle (2pt);
    \end{tikzpicture}
    \caption{Coordinates in the $2$-triangulated $15$-gon \( \mathcal{P}(2,3,3,1,3,0) \)}
    \label{fig:coord}
\end{figure}
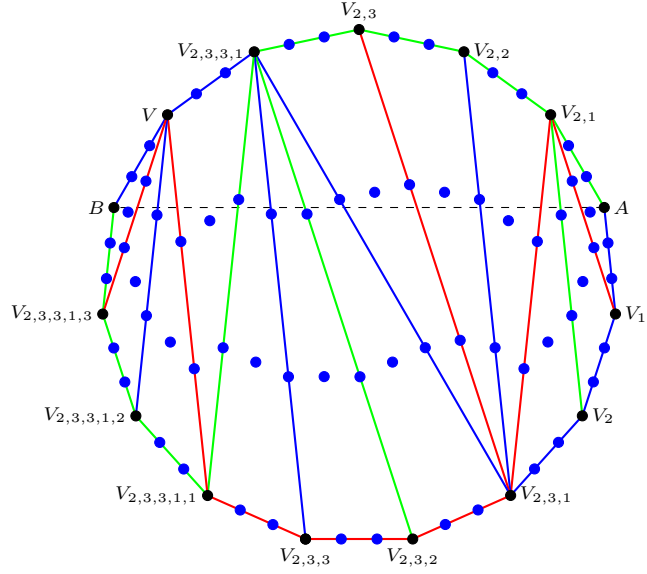

We are now ready to set up recurrence relations to calculate the general formulas of the functions $D^{[m_1,n_1,\ldots,m_k,n_k]}_l$ (with $l = 1,2,...,n$) and $I^{[m_1,n_1,\ldots,m_k,n_k]}_{(i,j)}$ (with $i,j \in \mathbb{N}$ such that $i+j \leq n$) with respect to $\gamma = AB$ based on the standard sequence of flips mentioned in the Definition~\ref{def:standard seq_flips}.

We first introduce some short hand notations concerning the parties of $m_i$ and $n_j$. Define the partial sums \Eq{M_i := m_1 + m_2 + \cdots + m_i\And N_j := n_1 + n_2 + \cdots + n_j.}
\begin{Def}
Consider the permutation actions $\upsilon, \tau$ on $\Gamma_m$ defined by $$\upsilon \cdot (w_1,w_2,w_3) = (w_3,w_1,w_2)\And\tau \cdot (w_1,w_2,w_3) = (w_1,w_3,w_2),\tab (w_1,w_2,w_3)\in\G_m.$$ Clearly, $\upsilon^3=\tau^2=e$ and $\upsilon \tau \upsilon = \tau$. We define the following sums:
\begin{align}
	\sigma &:= (-1)^{m_1} + (-1)^{n_1} + (-1)^{m_2} + (-1)^{n_2} + \cdots + (-1)^{m_{k-1}} + (-1)^{n_{k-1}} + (-1)^{m_k};\\
	\sigma' &:= (-1)^{m_1} + (-1)^{n_1} + (-1)^{m_1} + (-1)^{n_2} + \cdots + (-1)^{m_{k-1}} + (-1)^{n_{k-1}}. 
\end{align}
Finally we let $p \in \{0,1,2\}, q \in \{0,1\}$ be such that $\sigma \equiv p \text{ (mod 3)}$ and $M_k+N_{k-1} \equiv q \text{ (mod 2)}$, and denote $\omega := \upsilon^p \tau^q$. Similarly, let $p' \in \{0,1,2\}, q' \in \{0,1\}$ such that $\sigma' \equiv p' \text{ (mod 3)}$ and $M_{k-1}+N_{k-1} \equiv q' \text{ (mod 2)}$, and denote $\omega' := \upsilon^{p'} \tau^{q'}$. 
\end{Def}

The color of all edges and $\cT$-diagonals of the polygon can be identified uniquely (see Figure~\ref{fig:coord} for an example). Refer to Definition~\ref{def:standard seq_flips} again, we start from the last level, and continue following the standard sequence of flips recursively.
\begin{enumerate}
	\item[(I)] For \(n_k > 0\), we introduce extra notations $V_{m_1,n_1,m_2,n_2,...,m_k,n_k+1} := B$ and $S_{r} := (m_1,n_1,m_2,n_2,...,$ $m_{r},n_{r})$ for all $r \in \mathbb{N}$. When we have reached the diagonal $d_{M_{k-1}+N_k}$, then the color of all edges $V_{S_{k-1},m_k,j}V_{S_{k-1},m_k,j+1}$ for $j=1,2,...,n_k$ are the same. That common color and the color of the diagonal $VV_{S_{k-1},m_k,1}$ is identified uniquely by $p$ and $q$. Hence, we consider the 6 following cases, where the figure shows how the vertices are being labeled after previous flips.

\begin{figure}[H]
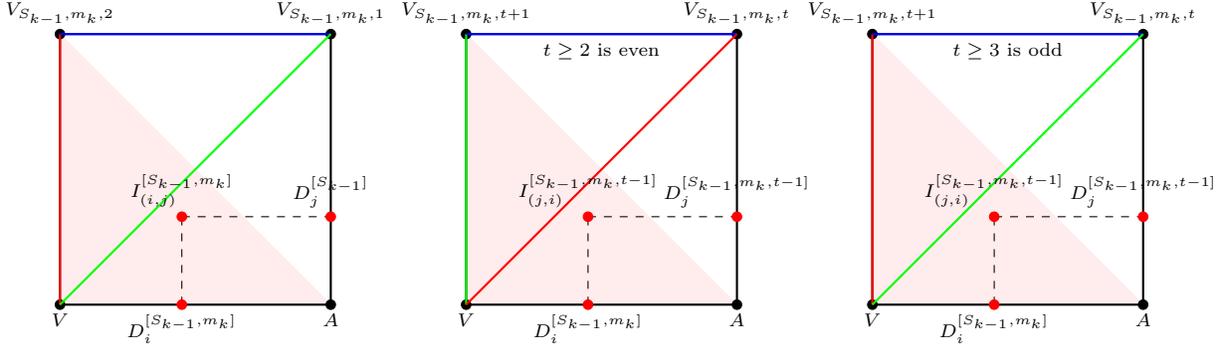

\centering

\caption{\textbf{Case 1.1}: $(p,q) = (0,0)$}
\label{fig:case1-1}
\end{figure}

\begin{figure}[H]
\centering
%
\caption{
\textbf{Case 1.2}: $(p,q) = (0,1)$}
\label{fig:case1-2}
\end{figure}
\begin{figure}[H]
\centering
%
\caption{
\textbf{Case 1.3}: $(p,q) = (1,0)$}
\label{fig:case1-3}
\end{figure}

\begin{figure}[H]
\centering
%
\caption{
\textbf{Case 1.4}: $(p,q) = (1,1)$}
\label{fig:case1-4}
\end{figure}

\begin{figure}[H]
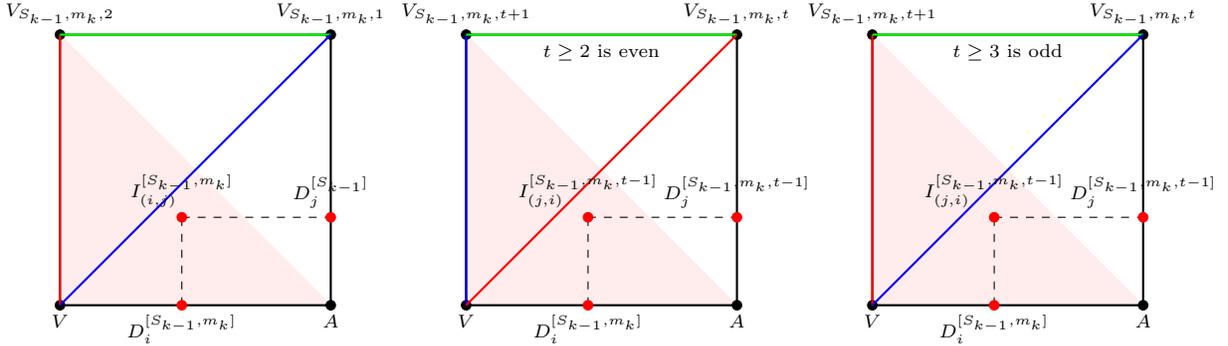

\centering
%
\caption{
\textbf{Case 1.6}: $(p,q) = (2,1)$}
\label{fig:case1-6}
\end{figure}
	
	\item[(II)] For \(n_k = 0\), redefine the notation $V_{m_1,n_1,m_2,n_2,...,n_{k-1},m_k+1} := B$ and keep the notation $S_{r} := (m_1,n_1,m_2,n_2,...,$ $m_{r},n_{r}$) for all $r \in \mathbb{N}$. When we have reached the diagonal $d_{M_{k-1}+N_{k-1}}$, then the color of all edges $V_{S_{k-1},j}V_{S_{k-1},j+1}$ for $j=1,2,...,m_k$ are the same. That common color and the color of the diagonal $VV_{S_{k-1},1}$ is identified uniquely by $p'$ and $q'$. Hence, we consider the 6 following cases, again the figure shows how the vertices are being labeled after previous flips.

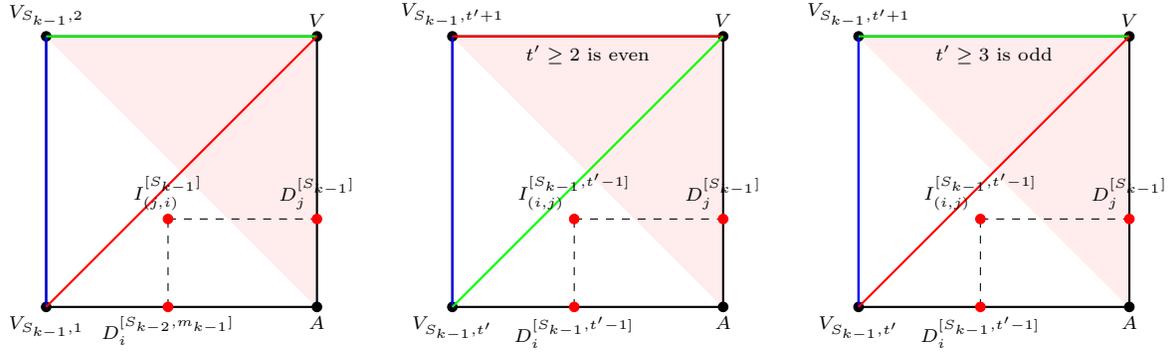
\begin{figure}[H]
\centering
\begin{tikzpicture}[scale=0.9,
    mid arrow/.style={
        postaction={decorate},
        decoration={
            markings,
            mark=at position 0.5 with {\arrow{>}}
        }
    }
]
  \filldraw[pink!30] (4,4) -- (4,0) -- (0,4) -- cycle;
  \filldraw[pink!30] (10,4) -- (10,0) -- (6,4) -- cycle;
  \filldraw[pink!30] (16,4) -- (16,0) -- (12,4) -- cycle;
  
  \draw[thick] (0,0) -- (4,0) -- (4,4) -- (0,4) -- cycle;
  \draw[thick] (6,0) -- (10,0) -- (10,4) -- (6,4) -- cycle;
  \draw[thick] (12,0) -- (16,0) -- (16,4) -- (12,4) -- cycle;
  
  \foreach \i in {0,1,2} {
  		\filldraw[black] (6*\i+0,0) circle (2pt);
  		\filldraw[black] (6*\i+4,0) circle (2pt);
  		\filldraw[black] (6*\i+4,4) circle (2pt);
  		\filldraw[black] (6*\i+0,4) circle (2pt);
  }
  
  \node[above, font=\scriptsize] at (4,4) {$V$};
  \node[above, font=\scriptsize] at (10,4) {$V$};
  \node[above, font=\scriptsize] at (16,4) {$V$};
  
  \node[below, font=\scriptsize] at (4,0) {$A$};
  \node[below, font=\scriptsize] at (10,0) {$A$};
  \node[below, font=\scriptsize] at (16,0) {$A$};
  
  \node[above, font=\scriptsize] at (0,4) {$V_{S_{k-1},2}$};
  \node[above, font=\scriptsize] at (6,4) {$V_{S_{k-1},t'+1}$};
  \node[above, font=\scriptsize] at (12,4) {$V_{S_{k-1},t'+1}$};
  
  \node[below, font=\scriptsize] at (0,0) {$V_{S_{k-1},1}$};
  \node[below, font=\scriptsize] at (6,0) {$V_{S_{k-1},t'}$};
  \node[below, font=\scriptsize] at (12,0) {$V_{S_{k-1},t'}$};
  
  \node[below, font=\scriptsize] at (8,4) {$t' \geq 2 \text{ is even}$};
  \node[below, font=\scriptsize] at (14,4) {$t' \geq 3 \text{ is odd}$};
  
  \draw[thin,dashed] (1.8,0) -- (1.8,1.3);
  \draw[thin,dashed] (1.8,1.3) -- (4,1.3);
  \draw[thin,dashed] (7.8,0) -- (7.8,1.3);
  \draw[thin,dashed] (7.8,1.3) -- (10,1.3);
  \draw[thin,dashed] (13.8,0) -- (13.8,1.3);
  \draw[thin,dashed] (13.8,1.3) -- (16,1.3);
  
  \draw[thick,blue] (0,0) -- (0,4);
  \draw[thick,green] (0,4) -- (4,4);
  \draw[thick,red] (4,4) -- (0,0);
  \draw[thick,blue] (6,0) -- (6,4);
  \draw[thick,red] (6,4) -- (10,4);
  \draw[thick,green] (10,4) -- (6,0);
  \draw[thick,blue] (12,0) -- (12,4);
  \draw[thick,green] (12,4) -- (16,4);
  \draw[thick,red] (16,4) -- (12,0);
  
  \filldraw[red] (1.8,0) circle (2pt);
  \node[below, font=\scriptsize] at (1.8,0) {$D^{[S_{k-2},m_{k-1}]}_i$};
  \filldraw[red] (7.8,0) circle (2pt);
  \node[below, font=\scriptsize] at (7.8,0) {$D^{[S_{k-1},t'-1]}_i$};
  \filldraw[red] (13.8,0) circle (2pt);
  \node[below, font=\scriptsize] at (13.8,0) {$D^{[S_{k-1},t'-1]}_i$};
  
  \filldraw[red] (1.8,1.3) circle (2pt);
  \node[above, font=\scriptsize] at (1.8,1.3) {$I^{[S_{k-1}]}_{(j,i)}$};
  \filldraw[red] (7.8,1.3) circle (2pt);
  \node[above, font=\scriptsize] at (7.8,1.3) {$I^{[S_{k-1},t'-1]}_{(i,j)}$};
  \filldraw[red] (13.8,1.3) circle (2pt);
  \node[above, font=\scriptsize] at (13.8,1.3) {$I^{[S_{k-1},t'-1]}_{(i,j)}$};
  
  \filldraw[red] (4,1.3) circle (2pt);
  \node[above, font=\scriptsize] at (4,1.3) {$D^{[S_{k-1}]}_j$};
  \filldraw[red] (10,1.3) circle (2pt);
  \node[above, font=\scriptsize] at (10,1.3) {$D^{[S_{k-1}]}_j$};
  \filldraw[red] (16,1.3) circle (2pt);
  \node[above, font=\scriptsize] at (16,1.3) {$D^{[S_{k-1}]}_j$};
\end{tikzpicture}
\caption{
\textbf{Case 2.1}: $(p',q') = (0,0)$}
\label{fig:case2-1}
\end{figure}

\begin{figure}[H]
\centering
\begin{tikzpicture}[scale=0.9,
    mid arrow/.style={
        postaction={decorate},
        decoration={
            markings,
            mark=at position 0.5 with {\arrow{>}}
        }
    }
]
  \filldraw[pink!30] (4,4) -- (4,0) -- (0,4) -- cycle;
  \filldraw[pink!30] (10,4) -- (10,0) -- (6,4) -- cycle;
  \filldraw[pink!30] (16,4) -- (16,0) -- (12,4) -- cycle;
  
  \draw[thick] (0,0) -- (4,0) -- (4,4) -- (0,4) -- cycle;
  \draw[thick] (6,0) -- (10,0) -- (10,4) -- (6,4) -- cycle;
  \draw[thick] (12,0) -- (16,0) -- (16,4) -- (12,4) -- cycle;
  
  \foreach \i in {0,1,2} {
  		\filldraw[black] (6*\i+0,0) circle (2pt);
  		\filldraw[black] (6*\i+4,0) circle (2pt);
  		\filldraw[black] (6*\i+4,4) circle (2pt);
  		\filldraw[black] (6*\i+0,4) circle (2pt);
  }
  
  \node[above, font=\scriptsize] at (4,4) {$V$};
  \node[above, font=\scriptsize] at (10,4) {$V$};
  \node[above, font=\scriptsize] at (16,4) {$V$};
  
  \node[below, font=\scriptsize] at (4,0) {$A$};
  \node[below, font=\scriptsize] at (10,0) {$A$};
  \node[below, font=\scriptsize] at (16,0) {$A$};
  
  \node[above, font=\scriptsize] at (0,4) {$V_{S_{k-1},2}$};
  \node[above, font=\scriptsize] at (6,4) {$V_{S_{k-1},t'+1}$};
  \node[above, font=\scriptsize] at (12,4) {$V_{S_{k-1},t'+1}$};
  
  \node[below, font=\scriptsize] at (0,0) {$V_{S_{k-1},1}$};
  \node[below, font=\scriptsize] at (6,0) {$V_{S_{k-1},t'}$};
  \node[below, font=\scriptsize] at (12,0) {$V_{S_{k-1},t'}$};
  
  \node[below, font=\scriptsize] at (8,4) {$t' \geq 2 \text{ is even}$};
  \node[below, font=\scriptsize] at (14,4) {$t' \geq 3 \text{ is odd}$};
  
  \draw[thin,dashed] (1.8,0) -- (1.8,1.3);
  \draw[thin,dashed] (1.8,1.3) -- (4,1.3);
  \draw[thin,dashed] (7.8,0) -- (7.8,1.3);
  \draw[thin,dashed] (7.8,1.3) -- (10,1.3);
  \draw[thin,dashed] (13.8,0) -- (13.8,1.3);
  \draw[thin,dashed] (13.8,1.3) -- (16,1.3);
  
  \draw[thick,blue] (0,0) -- (0,4);
  \draw[thick,red] (0,4) -- (4,4);
  \draw[thick,green] (4,4) -- (0,0);
  \draw[thick,blue] (6,0) -- (6,4);
  \draw[thick,green] (6,4) -- (10,4);
  \draw[thick,red] (10,4) -- (6,0);
  \draw[thick,blue] (12,0) -- (12,4);
  \draw[thick,red] (12,4) -- (16,4);
  \draw[thick,green] (16,4) -- (12,0);
  
  \filldraw[red] (1.8,0) circle (2pt);
  \node[below, font=\scriptsize] at (1.8,0) {$D^{[S_{k-2},m_{k-1}]}_i$};
  \filldraw[red] (7.8,0) circle (2pt);
  \node[below, font=\scriptsize] at (7.8,0) {$D^{[S_{k-1},t'-1]}_i$};
  \filldraw[red] (13.8,0) circle (2pt);
  \node[below, font=\scriptsize] at (13.8,0) {$D^{[S_{k-1},t'-1]}_i$};
  
  \filldraw[red] (1.8,1.3) circle (2pt);
  \node[above, font=\scriptsize] at (1.8,1.3) {$I^{[S_{k-1}]}_{(j,i)}$};
  \filldraw[red] (7.8,1.3) circle (2pt);
  \node[above, font=\scriptsize] at (7.8,1.3) {$I^{[S_{k-1},t'-1]}_{(i,j)}$};
  \filldraw[red] (13.8,1.3) circle (2pt);
  \node[above, font=\scriptsize] at (13.8,1.3) {$I^{[S_{k-1},t'-1]}_{(i,j)}$};
  
  \filldraw[red] (4,1.3) circle (2pt);
  \node[above, font=\scriptsize] at (4,1.3) {$D^{[S_{k-1}]}_j$};
  \filldraw[red] (10,1.3) circle (2pt);
  \node[above, font=\scriptsize] at (10,1.3) {$D^{[S_{k-1}]}_j$};
  \filldraw[red] (16,1.3) circle (2pt);
  \node[above, font=\scriptsize] at (16,1.3) {$D^{[S_{k-1}]}_j$};
\end{tikzpicture}
\caption{
\textbf{Case 2.2}: $(p',q') = (0,1)$}
\label{fig:case2-2}
\end{figure}

\begin{figure}[H]
\centering
\begin{tikzpicture}[scale=0.9,
    mid arrow/.style={
        postaction={decorate},
        decoration={
            markings,
            mark=at position 0.5 with {\arrow{>}}
        }
    }
]
  \filldraw[pink!30] (4,4) -- (4,0) -- (0,4) -- cycle;
  \filldraw[pink!30] (10,4) -- (10,0) -- (6,4) -- cycle;
  \filldraw[pink!30] (16,4) -- (16,0) -- (12,4) -- cycle;
  
  \draw[thick] (0,0) -- (4,0) -- (4,4) -- (0,4) -- cycle;
  \draw[thick] (6,0) -- (10,0) -- (10,4) -- (6,4) -- cycle;
  \draw[thick] (12,0) -- (16,0) -- (16,4) -- (12,4) -- cycle;
  
  \foreach \i in {0,1,2} {
  		\filldraw[black] (6*\i+0,0) circle (2pt);
  		\filldraw[black] (6*\i+4,0) circle (2pt);
  		\filldraw[black] (6*\i+4,4) circle (2pt);
  		\filldraw[black] (6*\i+0,4) circle (2pt);
  }
  
  \node[above, font=\scriptsize] at (4,4) {$V$};
  \node[above, font=\scriptsize] at (10,4) {$V$};
  \node[above, font=\scriptsize] at (16,4) {$V$};
  
  \node[below, font=\scriptsize] at (4,0) {$A$};
  \node[below, font=\scriptsize] at (10,0) {$A$};
  \node[below, font=\scriptsize] at (16,0) {$A$};
  
  \node[above, font=\scriptsize] at (0,4) {$V_{S_{k-1},2}$};
  \node[above, font=\scriptsize] at (6,4) {$V_{S_{k-1},t'+1}$};
  \node[above, font=\scriptsize] at (12,4) {$V_{S_{k-1},t'+1}$};
  
  \node[below, font=\scriptsize] at (0,0) {$V_{S_{k-1},1}$};
  \node[below, font=\scriptsize] at (6,0) {$V_{S_{k-1},t'}$};
  \node[below, font=\scriptsize] at (12,0) {$V_{S_{k-1},t'}$};
  
  \node[below, font=\scriptsize] at (8,4) {$t' \geq 2 \text{ is even}$};
  \node[below, font=\scriptsize] at (14,4) {$t' \geq 3 \text{ is odd}$};
  
  \draw[thin,dashed] (1.8,0) -- (1.8,1.3);
  \draw[thin,dashed] (1.8,1.3) -- (4,1.3);
  \draw[thin,dashed] (7.8,0) -- (7.8,1.3);
  \draw[thin,dashed] (7.8,1.3) -- (10,1.3);
  \draw[thin,dashed] (13.8,0) -- (13.8,1.3);
  \draw[thin,dashed] (13.8,1.3) -- (16,1.3);
  
  \draw[thick,red] (0,0) -- (0,4);
  \draw[thick,blue] (0,4) -- (4,4);
  \draw[thick,green] (4,4) -- (0,0);
  \draw[thick,red] (6,0) -- (6,4);
  \draw[thick,green] (6,4) -- (10,4);
  \draw[thick,blue] (10,4) -- (6,0);
  \draw[thick,red] (12,0) -- (12,4);
  \draw[thick,blue] (12,4) -- (16,4);
  \draw[thick,green] (16,4) -- (12,0);
  
  \filldraw[red] (1.8,0) circle (2pt);
  \node[below, font=\scriptsize] at (1.8,0) {$D^{[S_{k-2},m_{k-1}]}_i$};
  \filldraw[red] (7.8,0) circle (2pt);
  \node[below, font=\scriptsize] at (7.8,0) {$D^{[S_{k-1},t'-1]}_i$};
  \filldraw[red] (13.8,0) circle (2pt);
  \node[below, font=\scriptsize] at (13.8,0) {$D^{[S_{k-1},t'-1]}_i$};
  
  \filldraw[red] (1.8,1.3) circle (2pt);
  \node[above, font=\scriptsize] at (1.8,1.3) {$I^{[S_{k-1}]}_{(j,i)}$};
  \filldraw[red] (7.8,1.3) circle (2pt);
  \node[above, font=\scriptsize] at (7.8,1.3) {$I^{[S_{k-1},t'-1]}_{(i,j)}$};
  \filldraw[red] (13.8,1.3) circle (2pt);
  \node[above, font=\scriptsize] at (13.8,1.3) {$I^{[S_{k-1},t'-1]}_{(i,j)}$};
  
  \filldraw[red] (4,1.3) circle (2pt);
  \node[above, font=\scriptsize] at (4,1.3) {$D^{[S_{k-1}]}_j$};
  \filldraw[red] (10,1.3) circle (2pt);
  \node[above, font=\scriptsize] at (10,1.3) {$D^{[S_{k-1}]}_j$};
  \filldraw[red] (16,1.3) circle (2pt);
  \node[above, font=\scriptsize] at (16,1.3) {$D^{[S_{k-1}]}_j$};
\end{tikzpicture}
\caption{
\textbf{Case 2.3}: $(p',q') = (1,0)$}
\label{fig:case2-3}
\end{figure}

\begin{figure}[H]
\centering
\begin{tikzpicture}[scale=0.9,
    mid arrow/.style={
        postaction={decorate},
        decoration={
            markings,
            mark=at position 0.5 with {\arrow{>}}
        }
    }
]
  \filldraw[pink!30] (4,4) -- (4,0) -- (0,4) -- cycle;
  \filldraw[pink!30] (10,4) -- (10,0) -- (6,4) -- cycle;
  \filldraw[pink!30] (16,4) -- (16,0) -- (12,4) -- cycle;
  
  \draw[thick] (0,0) -- (4,0) -- (4,4) -- (0,4) -- cycle;
  \draw[thick] (6,0) -- (10,0) -- (10,4) -- (6,4) -- cycle;
  \draw[thick] (12,0) -- (16,0) -- (16,4) -- (12,4) -- cycle;
  
  \foreach \i in {0,1,2} {
  		\filldraw[black] (6*\i+0,0) circle (2pt);
  		\filldraw[black] (6*\i+4,0) circle (2pt);
  		\filldraw[black] (6*\i+4,4) circle (2pt);
  		\filldraw[black] (6*\i+0,4) circle (2pt);
  }
  
  \node[above, font=\scriptsize] at (4,4) {$V$};
  \node[above, font=\scriptsize] at (10,4) {$V$};
  \node[above, font=\scriptsize] at (16,4) {$V$};
  
  \node[below, font=\scriptsize] at (4,0) {$A$};
  \node[below, font=\scriptsize] at (10,0) {$A$};
  \node[below, font=\scriptsize] at (16,0) {$A$};
  
  \node[above, font=\scriptsize] at (0,4) {$V_{S_{k-1},2}$};
  \node[above, font=\scriptsize] at (6,4) {$V_{S_{k-1},t'+1}$};
  \node[above, font=\scriptsize] at (12,4) {$V_{S_{k-1},t'+1}$};
  
  \node[below, font=\scriptsize] at (0,0) {$V_{S_{k-1},1}$};
  \node[below, font=\scriptsize] at (6,0) {$V_{S_{k-1},t'}$};
  \node[below, font=\scriptsize] at (12,0) {$V_{S_{k-1},t'}$};
  
  \node[below, font=\scriptsize] at (8,4) {$t' \geq 2 \text{ is even}$};
  \node[below, font=\scriptsize] at (14,4) {$t' \geq 3 \text{ is odd}$};
  
  \draw[thin,dashed] (1.8,0) -- (1.8,1.3);
  \draw[thin,dashed] (1.8,1.3) -- (4,1.3);
  \draw[thin,dashed] (7.8,0) -- (7.8,1.3);
  \draw[thin,dashed] (7.8,1.3) -- (10,1.3);
  \draw[thin,dashed] (13.8,0) -- (13.8,1.3);
  \draw[thin,dashed] (13.8,1.3) -- (16,1.3);
  
  \draw[thick,red] (0,0) -- (0,4);
  \draw[thick,green] (0,4) -- (4,4);
  \draw[thick,blue] (4,4) -- (0,0);
  \draw[thick,red] (6,0) -- (6,4);
  \draw[thick,blue] (6,4) -- (10,4);
  \draw[thick,green] (10,4) -- (6,0);
  \draw[thick,red] (12,0) -- (12,4);
  \draw[thick,green] (12,4) -- (16,4);
  \draw[thick,blue] (16,4) -- (12,0);
  
  \filldraw[red] (1.8,0) circle (2pt);
  \node[below, font=\scriptsize] at (1.8,0) {$D^{[S_{k-2},m_{k-1}]}_i$};
  \filldraw[red] (7.8,0) circle (2pt);
  \node[below, font=\scriptsize] at (7.8,0) {$D^{[S_{k-1},t'-1]}_i$};
  \filldraw[red] (13.8,0) circle (2pt);
  \node[below, font=\scriptsize] at (13.8,0) {$D^{[S_{k-1},t'-1]}_i$};
  
  \filldraw[red] (1.8,1.3) circle (2pt);
  \node[above, font=\scriptsize] at (1.8,1.3) {$I^{[S_{k-1}]}_{(j,i)}$};
  \filldraw[red] (7.8,1.3) circle (2pt);
  \node[above, font=\scriptsize] at (7.8,1.3) {$I^{[S_{k-1},t'-1]}_{(i,j)}$};
  \filldraw[red] (13.8,1.3) circle (2pt);
  \node[above, font=\scriptsize] at (13.8,1.3) {$I^{[S_{k-1},t'-1]}_{(i,j)}$};
  
  \filldraw[red] (4,1.3) circle (2pt);
  \node[above, font=\scriptsize] at (4,1.3) {$D^{[S_{k-1}]}_j$};
  \filldraw[red] (10,1.3) circle (2pt);
  \node[above, font=\scriptsize] at (10,1.3) {$D^{[S_{k-1}]}_j$};
  \filldraw[red] (16,1.3) circle (2pt);
  \node[above, font=\scriptsize] at (16,1.3) {$D^{[S_{k-1}]}_j$};
\end{tikzpicture}
\caption{
\textbf{Case 2.4}: $(p',q') = (1,1)$}
\label{fig:case2-4}
\end{figure}

\begin{figure}[H]
\centering
\begin{tikzpicture}[scale=0.9,
    mid arrow/.style={
        postaction={decorate},
        decoration={
            markings,
            mark=at position 0.5 with {\arrow{>}}
        }
    }
]
  \filldraw[pink!30] (4,4) -- (4,0) -- (0,4) -- cycle;
  \filldraw[pink!30] (10,4) -- (10,0) -- (6,4) -- cycle;
  \filldraw[pink!30] (16,4) -- (16,0) -- (12,4) -- cycle;
  
  \draw[thick] (0,0) -- (4,0) -- (4,4) -- (0,4) -- cycle;
  \draw[thick] (6,0) -- (10,0) -- (10,4) -- (6,4) -- cycle;
  \draw[thick] (12,0) -- (16,0) -- (16,4) -- (12,4) -- cycle;
  
  \foreach \i in {0,1,2} {
  		\filldraw[black] (6*\i+0,0) circle (2pt);
  		\filldraw[black] (6*\i+4,0) circle (2pt);
  		\filldraw[black] (6*\i+4,4) circle (2pt);
  		\filldraw[black] (6*\i+0,4) circle (2pt);
  }
  
  \node[above, font=\scriptsize] at (4,4) {$V$};
  \node[above, font=\scriptsize] at (10,4) {$V$};
  \node[above, font=\scriptsize] at (16,4) {$V$};
  
  \node[below, font=\scriptsize] at (4,0) {$A$};
  \node[below, font=\scriptsize] at (10,0) {$A$};
  \node[below, font=\scriptsize] at (16,0) {$A$};
  
  \node[above, font=\scriptsize] at (0,4) {$V_{S_{k-1},2}$};
  \node[above, font=\scriptsize] at (6,4) {$V_{S_{k-1},t'+1}$};
  \node[above, font=\scriptsize] at (12,4) {$V_{S_{k-1},t'+1}$};
  
  \node[below, font=\scriptsize] at (0,0) {$V_{S_{k-1},1}$};
  \node[below, font=\scriptsize] at (6,0) {$V_{S_{k-1},t'}$};
  \node[below, font=\scriptsize] at (12,0) {$V_{S_{k-1},t'}$};
  
  \node[below, font=\scriptsize] at (8,4) {$t' \geq 2 \text{ is even}$};
  \node[below, font=\scriptsize] at (14,4) {$t' \geq 3 \text{ is odd}$};
  
  \draw[thin,dashed] (1.8,0) -- (1.8,1.3);
  \draw[thin,dashed] (1.8,1.3) -- (4,1.3);
  \draw[thin,dashed] (7.8,0) -- (7.8,1.3);
  \draw[thin,dashed] (7.8,1.3) -- (10,1.3);
  \draw[thin,dashed] (13.8,0) -- (13.8,1.3);
  \draw[thin,dashed] (13.8,1.3) -- (16,1.3);
  
  \draw[thick,green] (0,0) -- (0,4);
  \draw[thick,red] (0,4) -- (4,4);
  \draw[thick,blue] (4,4) -- (0,0);
  \draw[thick,green] (6,0) -- (6,4);
  \draw[thick,blue] (6,4) -- (10,4);
  \draw[thick,red] (10,4) -- (6,0);
  \draw[thick,green] (12,0) -- (12,4);
  \draw[thick,red] (12,4) -- (16,4);
  \draw[thick,blue] (16,4) -- (12,0);
  
  \filldraw[red] (1.8,0) circle (2pt);
  \node[below, font=\scriptsize] at (1.8,0) {$D^{[S_{k-2},m_{k-1}]}_i$};
  \filldraw[red] (7.8,0) circle (2pt);
  \node[below, font=\scriptsize] at (7.8,0) {$D^{[S_{k-1},t'-1]}_i$};
  \filldraw[red] (13.8,0) circle (2pt);
  \node[below, font=\scriptsize] at (13.8,0) {$D^{[S_{k-1},t'-1]}_i$};
  
  \filldraw[red] (1.8,1.3) circle (2pt);
  \node[above, font=\scriptsize] at (1.8,1.3) {$I^{[S_{k-1}]}_{(j,i)}$};
  \filldraw[red] (7.8,1.3) circle (2pt);
  \node[above, font=\scriptsize] at (7.8,1.3) {$I^{[S_{k-1},t'-1]}_{(i,j)}$};
  \filldraw[red] (13.8,1.3) circle (2pt);
  \node[above, font=\scriptsize] at (13.8,1.3) {$I^{[S_{k-1},t'-1]}_{(i,j)}$};
  
  \filldraw[red] (4,1.3) circle (2pt);
  \node[above, font=\scriptsize] at (4,1.3) {$D^{[S_{k-1}]}_j$};
  \filldraw[red] (10,1.3) circle (2pt);
  \node[above, font=\scriptsize] at (10,1.3) {$D^{[S_{k-1}]}_j$};
  \filldraw[red] (16,1.3) circle (2pt);
  \node[above, font=\scriptsize] at (16,1.3) {$D^{[S_{k-1}]}_j$};
\end{tikzpicture}
\caption{
\textbf{Case 2.5}: $(p',q') = (2,0)$}
\label{fig:case2-5}
\end{figure}

\begin{figure}[H]
\centering
\begin{tikzpicture}[scale=0.9,
    mid arrow/.style={
        postaction={decorate},
        decoration={
            markings,
            mark=at position 0.5 with {\arrow{>}}
        }
    }
]
  \filldraw[pink!30] (4,4) -- (4,0) -- (0,4) -- cycle;
  \filldraw[pink!30] (10,4) -- (10,0) -- (6,4) -- cycle;
  \filldraw[pink!30] (16,4) -- (16,0) -- (12,4) -- cycle;
  
  \draw[thick] (0,0) -- (4,0) -- (4,4) -- (0,4) -- cycle;
  \draw[thick] (6,0) -- (10,0) -- (10,4) -- (6,4) -- cycle;
  \draw[thick] (12,0) -- (16,0) -- (16,4) -- (12,4) -- cycle;
  
  \foreach \i in {0,1,2} {
  		\filldraw[black] (6*\i+0,0) circle (2pt);
  		\filldraw[black] (6*\i+4,0) circle (2pt);
  		\filldraw[black] (6*\i+4,4) circle (2pt);
  		\filldraw[black] (6*\i+0,4) circle (2pt);
  }
  
  \node[above, font=\scriptsize] at (4,4) {$V$};
  \node[above, font=\scriptsize] at (10,4) {$V$};
  \node[above, font=\scriptsize] at (16,4) {$V$};
  
  \node[below, font=\scriptsize] at (4,0) {$A$};
  \node[below, font=\scriptsize] at (10,0) {$A$};
  \node[below, font=\scriptsize] at (16,0) {$A$};
  
  \node[above, font=\scriptsize] at (0,4) {$V_{S_{k-1},2}$};
  \node[above, font=\scriptsize] at (6,4) {$V_{S_{k-1},t'+1}$};
  \node[above, font=\scriptsize] at (12,4) {$V_{S_{k-1},t'+1}$};
  
  \node[below, font=\scriptsize] at (0,0) {$V_{S_{k-1},1}$};
  \node[below, font=\scriptsize] at (6,0) {$V_{S_{k-1},t'}$};
  \node[below, font=\scriptsize] at (12,0) {$V_{S_{k-1},t'}$};
  
  \node[below, font=\scriptsize] at (8,4) {$t' \geq 2 \text{ is even}$};
  \node[below, font=\scriptsize] at (14,4) {$t' \geq 3 \text{ is odd}$};
  
  \draw[thin,dashed] (1.8,0) -- (1.8,1.3);
  \draw[thin,dashed] (1.8,1.3) -- (4,1.3);
  \draw[thin,dashed] (7.8,0) -- (7.8,1.3);
  \draw[thin,dashed] (7.8,1.3) -- (10,1.3);
  \draw[thin,dashed] (13.8,0) -- (13.8,1.3);
  \draw[thin,dashed] (13.8,1.3) -- (16,1.3);
  
  \draw[thick,green] (0,0) -- (0,4);
  \draw[thick,blue] (0,4) -- (4,4);
  \draw[thick,red] (4,4) -- (0,0);
  \draw[thick,green] (6,0) -- (6,4);
  \draw[thick,red] (6,4) -- (10,4);
  \draw[thick,blue] (10,4) -- (6,0);
  \draw[thick,green] (12,0) -- (12,4);
  \draw[thick,blue] (12,4) -- (16,4);
  \draw[thick,red] (16,4) -- (12,0);
  
  \filldraw[red] (1.8,0) circle (2pt);
  \node[below, font=\scriptsize] at (1.8,0) {$D^{[S_{k-2},m_{k-1}]}_i$};
  \filldraw[red] (7.8,0) circle (2pt);
  \node[below, font=\scriptsize] at (7.8,0) {$D^{[S_{k-1},t'-1]}_i$};
  \filldraw[red] (13.8,0) circle (2pt);
  \node[below, font=\scriptsize] at (13.8,0) {$D^{[S_{k-1},t'-1]}_i$};
  
  \filldraw[red] (1.8,1.3) circle (2pt);
  \node[above, font=\scriptsize] at (1.8,1.3) {$I^{[S_{k-1}]}_{(j,i)}$};
  \filldraw[red] (7.8,1.3) circle (2pt);
  \node[above, font=\scriptsize] at (7.8,1.3) {$I^{[S_{k-1},t'-1]}_{(i,j)}$};
  \filldraw[red] (13.8,1.3) circle (2pt);
  \node[above, font=\scriptsize] at (13.8,1.3) {$I^{[S_{k-1},t'-1]}_{(i,j)}$};
  
  \filldraw[red] (4,1.3) circle (2pt);
  \node[above, font=\scriptsize] at (4,1.3) {$D^{[S_{k-1}]}_j$};
  \filldraw[red] (10,1.3) circle (2pt);
  \node[above, font=\scriptsize] at (10,1.3) {$D^{[S_{k-1}]}_j$};
  \filldraw[red] (16,1.3) circle (2pt);
  \node[above, font=\scriptsize] at (16,1.3) {$D^{[S_{k-1}]}_j$};
\end{tikzpicture}
\caption{
\textbf{Case 2.6}: $(p',q') = (2,1)$}
\label{fig:case2-6}
\end{figure}
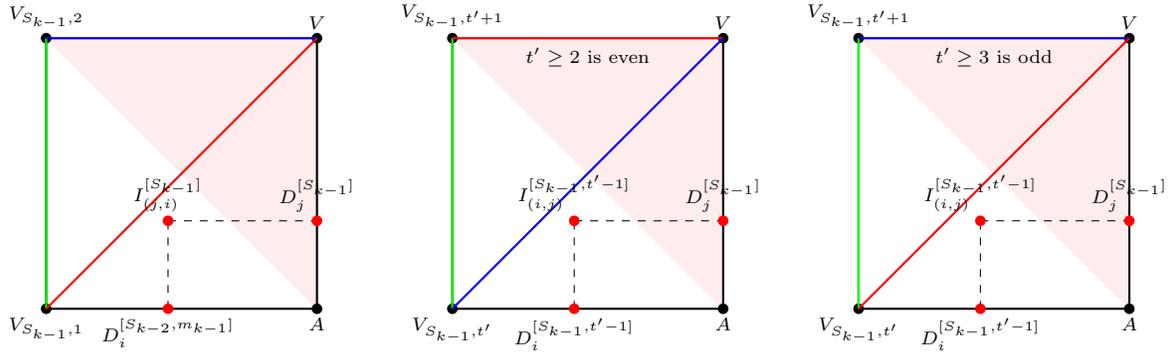
\end{enumerate}
Definition~\ref{Def:ref_poly} simplifies calculating expansion formulas by reducing the number of coordinates to check.

\subsection*{The case of $1$-triangulated $m$-gon}
After the previous reduction, our problem becomes just solving the general formula $D^{[m_1,n_1,m_2,n_2,...,m_k,n_k]}$. Actually, this is just an application of perfect matchings on snake graph \cite{MSW08, Sch07}. This part will elaborate the general formula in the recurrence formula format which is different from the previous combinatorial method, but maintain the same idea about the standard sequence of flips. By the calculation from each figure, we can rewrite all recurrences in the statement below:

\begin{Thm}[Recurrence formula for a general $1$-triangulated polygon]
\label{thm:1tri}
Consider a general $1$-triangulated $m$-gon with $\cT$-triangulation of type $\mathcal{P}(m_1,n_1,\ldots,m_k,n_k)$. Let $D^{[m_1,n_1,m_2,n_2,\ldots,m_k,n_k]}$ be the expansion formula with respect to the non-$\cT$ diagonal $\gamma = AB$. Then the functions satisfy the following recurrence relations:

    \emph{I.} (Type $\mathcal{P}(m_1)$.) For any $k \in \mathbb{N}$,
{\small\begin{align*}
D^{[2k-1]} 
&= x^{(1)}_{(1,1,0)} \cdot x^{(2k)(2k+1)}_{(1,1,0)}\cdot \sum_{j=1}^{k} \Biggl( 
\frac{x^{(2j)}_{(0,1,1)}}{x^{(2j)(2j+1)}_{(1,1,0)} x^{(2j-1)(2j)}_{(1,0,1)}} 
+ \frac{x^{(2j-1)}_{(0,1,1)}}{x^{(2j-1)(2j)}_{(1,0,1)} x^{(2j-2)(2j-1)}_{(1,1,0)}} 
\Biggr);\\
D^{[2k]} 
&= x^{(1)}_{(1,1,0)} \cdot x^{(2k+1)(2k+2)}_{(1,0,1)} \\
&\quad \cdot \Biggl[ 
\frac{x^{(2k+1)}_{(0,1,1)}}{x^{(2k+1)(2k+2)}_{(1,0,1)} x^{(2k)(2k+1)}_{(1,1,0)}} 
+ \sum_{j=1}^{k} \Biggl( 
\frac{x^{(2j)}_{(0,1,1)}}{x^{(2j)(2j+1)}_{(1,1,0)} x^{(2j-1)(2j)}_{(1,0,1)}} 
+ \frac{x^{(2j-1)}_{(0,1,1)}}{x^{(2j-1)(2j)}_{(1,0,1)} x^{(2j-2)(2j-1)}_{(1,1,0)}} 
\Biggr) \Biggr].
\end{align*}
    }
    \emph{II.} (Type $\mathcal{P}(m_1,n_1)$.) For any $k,l \in \mathbb{N}$, then:
{\small
\begin{align*}
D^{[2k-1,2l-1]} 
&= \frac{x^{((2k+2l-1)(2k+2l))}_{(1,0,1)}}{x^{((2k)(2k+1))}_{(0,1,1)}} \cdot x^{(1)}_{(1,1,0)} + x^{((2k+2l-1)(2k+2l))}_{(1,0,1)} \cdot \Biggl[ 
\frac{x^{(2k+2l-1)}_{(1,1,0)}}{x^{((2k+2l-1)(2k+2l))}_{(1,0,1)} x^{((2k+2l-2)(2k+2l-1))}_{(0,1,1)}} \\
&\quad+ \sum_{j=1}^{l-1} \Biggl( 
\frac{x^{(2k+2j)}_{(1,1,0)}}{x^{((2k+2j)(2k+2j+1))}_{(0,1,1)} x^{((2k+2j-1)(2k+2j))}_{(1,0,1)}} 
+ \frac{x^{(2k+2j-1)}_{(1,1,0)}}{x^{((2k+2j-1)(2k+2j))}_{(1,0,1)} x^{((2k+2j-2)(2k+2j-1))}_{(0,1,1)}} 
\Biggr) \Biggr]\\
&\quad  \cdot \Biggl[ 
x^{(1)}_{(1,1,0)} \cdot x^{(2k)(2k+1)}_{(1,1,0)} 
\cdot \sum_{j=1}^{k} \Biggl( 
\frac{x^{(2j)}_{(0,1,1)}}{x^{(2j)(2j+1)}_{(1,1,0)} x^{(2j-1)(2j)}_{(1,0,1)}} 
+ \frac{x^{(2j-1)}_{(0,1,1)}}{x^{(2j-1)(2j)}_{(1,0,1)} x^{(2j-2)(2j-1)}_{(1,1,0)}} 
\Biggr) \Biggr];\\
D^{[2k-1,2l]} 
&= \frac{x^{((2k+2l)(2k+2l+1))}_{(0,1,1)}}{x^{((2k)(2k+1))}_{(0,1,1)}} \cdot x^{(1)}_{(1,1,0)} \\
&\quad + x^{((2k+2l)(2k+2l+1))}_{(0,1,1)} \cdot x^{(1)}_{(1,1,0)} \\
&\quad \cdot \Biggl[ \sum_{j=1}^{l} \Biggl( 
\frac{x^{(2k+2j)}_{(1,1,0)}}{x^{((2k+2j)(2k+2j+1))}_{(0,1,1)} x^{((2k+2j-1)(2k+2j))}_{(1,0,1)}} 
+ \frac{x^{(2k+2j-1)}_{(1,1,0)}}{x^{((2k+2j-1)(2k+2j))}_{(1,0,1)} x^{((2k+2j-2)(2k+2j-1))}_{(0,1,1)}} 
\Biggr) \Biggr] \\
&\quad \cdot \Biggl[ 
x^{(2k)(2k+1)}_{(1,1,0)} 
\cdot \sum_{j=1}^{k} \Biggl( 
\frac{x^{(2j)}_{(0,1,1)}}{x^{(2j)(2j+1)}_{(1,1,0)} x^{(2j-1)(2j)}_{(1,0,1)}} 
+ \frac{x^{(2j-1)}_{(0,1,1)}}{x^{(2j-1)(2j)}_{(1,0,1)} x^{(2j-2)(2j-1)}_{(1,1,0)}} 
\Biggr) \Biggr];\\
D^{[2k,2l-1]} 
&= \frac{x^{((2k+2l)(2k+2l+1))}_{(1,1,0)}}{x^{((2k+1)(2k+2))}_{(0,1,1)}} \cdot x^{(1)}_{(1,1,0)} + x^{((2k+2l)(2k+2l+1))}_{(1,1,0)}\cdot \Biggl[ 
\frac{x^{(2k+2l)}_{(1,0,1)}}{x^{((2k+2l)(2k+2l+1))}_{(1,1,0)} x^{((2k+2l-1)(2k+2l))}_{(0,1,1)}}\\
&\quad + \sum_{j=1}^{l-1} \Biggl( 
\frac{x^{(2k+2j+1)}_{(1,0,1)}}{x^{((2k+2j+1)(2k+2j+2))}_{(0,1,1)} x^{((2k+2j)(2k+2j+1))}_{(1,1,0)}} + \frac{x^{(2k+2j)}_{(1,0,1)}}{x^{((2k+2j)(2k+2j+1))}_{(1,1,0)} x^{((2k+2j-1)(2k+2j))}_{(0,1,1)}} 
\Biggr) \Biggr] \\
&\quad \cdot \Biggl[ 
x^{(1)}_{(1,1,0)} \cdot x^{(2k+1)(2k+2)}_{(1,0,1)} 
\cdot \Biggl[ 
\frac{x^{(2k+1)}_{(0,1,1)}}{x^{(2k+1)(2k+2)}_{(1,0,1)} x^{(2k)(2k+1)}_{(1,1,0)}} + \sum_{j=1}^{k} \Biggl( 
\frac{x^{(2j)}_{(0,1,1)}}{x^{(2j)(2j+1)}_{(1,1,0)} x^{(2j-1)(2j)}_{(1,0,1)}} 
+ \frac{x^{(2j-1)}_{(0,1,1)}}{x^{(2j-1)(2j)}_{(1,0,1)} x^{(2j-2)(2j-1)}_{(1,1,0)}} 
\Biggr) \Biggr] \Biggr];\\
D^{[2k,2l]} 
&= \frac{x^{((2k+2l+1)(2k+2l+2))}_{(0,1,1)}}{x^{((2k+1)(2k+2))}_{(0,1,1)}} \cdot x^{(1)}_{(1,1,0)} \\
&\quad + x^{((2k+2l+1)(2k+2l+2))}_{(0,1,1)} \\
&\quad \cdot \Biggl[ \sum_{j=1}^{l} \Biggl( 
\frac{x^{(2k+2j+1)}_{(1,0,1)}}{x^{((2k+2j+1)(2k+2j+2))}_{(0,1,1)} x^{((2k+2j)(2k+2j+1))}_{(1,1,0)}} 
+ \frac{x^{(2k+2j)}_{(1,0,1)}}{x^{((2k+2j)(2k+2j+1))}_{(1,0,1)} x^{((2k+2j-1)(2k+2j))}_{(0,1,1)}} 
\Biggr) \Biggr] \\
&\quad \cdot \Biggl[ 
x^{(1)}_{(1,1,0)} \cdot x^{(2k+1)(2k+2)}_{(1,0,1)} 
\cdot \Biggl[ 
\frac{x^{(2k+1)}_{(0,1,1)}}{x^{(2k+1)(2k+2)}_{(1,0,1)} x^{(2k)(2k+1)}_{(1,1,0)}} + \sum_{j=1}^{k} \Biggl( 
\frac{x^{(2j)}_{(0,1,1)}}{x^{(2j)(2j+1)}_{(1,1,0)} x^{(2j-1)(2j)}_{(1,0,1)}} 
+ \frac{x^{(2j-1)}_{(0,1,1)}}{x^{(2j-1)(2j)}_{(1,0,1)} x^{(2j-2)(2j-1)}_{(1,1,0)}} 
\Biggr) \Biggr] \Biggr].
\end{align*}
}

    \emph{III-1.} (For $n_k \geq 1$) For any $k, m_1, n_1, m_2, n_2, \ldots, m_k, (n_k-1) \in \mathbb{N}$, the function satisfies the following relations: for any $n_k = 2l \geq 2$ even:
{\small
\begin{align*}
    D^{[m_1,n_1,\ldots,m_k,n_k]} 
    &= \frac{x^{((M_k+N_{k-1}+2l+1)(M_k+N_{k-1}+2l+2))}_{\omega \cdot (1,1,0)}}{x^{((M_k+N_{k-1}+1)(M_k+N_{k-1}+2))}_{\omega \cdot (1,1,0)}} D^{[m_1,n_1,\ldots,n_{k-1}]}+ x^{((M_k+N_{k-1}+2l+1)(M_k+N_{k-1}+2l+2))}_{\omega \cdot (1,1,0)}\cdot \\
    &\quad \cdot \Biggl[ \sum_{j=1}^{l} \Biggl( 
    \frac{x^{(M_k+N_{k-1}+2j)}_{\omega \cdot (0,1,1)}}{x^{((M_k+N_{k-1}+2j)(M_k+N_{k-1}+2j+1))}_{\omega \cdot (1,0,1)} x^{((M_k+N_{k-1}+2j-1)(M_k+N_{k-1}+2j))}_{\omega \cdot (1,1,0)}} 
    \\
    &\quad + \frac{x^{(M_k+N_{k-1}+2j+1)}_{\omega \cdot (0,1,1)}}{x^{((M_k+N_{k-1}+2j+1)(M_k+N_{k-1}+2j+2))}_{\omega \cdot (1,1,0)} x^{((M_k+N_{k-1}+2j)(M_k+N_{k-1}+2j+1))}_{\omega \cdot (1,0,1)}} 
    \Biggr) \Biggr] \cdot D^{[m_1,n_1,\ldots,m_k]}.
\end{align*}
For any $n_k = 2l+1\geq 1$ odd:
\begin{align*}
    D^{[m_1,n_1,\ldots,m_k,n_k]} 
    &= \frac{x^{((M_k+N_{k-1}+2l+2)(M_k+N_{k-1}+2l+3))}_{\omega \cdot (1,0,1)}}{x^{((M_k+N_{k-1}+1)(M_k+N_{k-1}+2))}_{\omega \cdot (1,1,0)}} D^{[m_1,n_1,\ldots,n_{k-1}]} + x^{((M_k+N_{k-1}+2l+2)(M_k+N_{k-1}+2l+3))}_{\omega \cdot (1,0,1)}\\
    &\quad \cdot \Biggl[ 
    \frac{x^{(M_k+N_{k-1}+2l+2)}_{\omega \cdot (0,1,1)}}{x^{((M_k+N_{k-1}+2l+2)(M_k+N_{k-1}+2l+3))}_{\omega \cdot (1,0,1)} x^{((M_k+N_{k-1}+2l+1)(M_k+N_{k-1}+2l+2))}_{\omega \cdot (1,1,0)}} 
    \\
    &\quad + \sum_{j=1}^{l} \Biggl( 
    \frac{x^{(M_k+N_{k-1}+2j)}_{\omega \cdot (0,1,1)}}{x^{((M_k+N_{k-1}+2j)(M_k+N_{k-1}+2j+1))}_{\omega \cdot (1,0,1)} x^{((M_k+N_{k-1}+2j-1)(M_k+N_{k-1}+2j))}_{\omega \cdot (1,1,0)}} 
    \\
    &\quad + \frac{x^{(M_k+N_{k-1}+2j+1)}_{\omega \cdot (0,1,1)}}{x^{((M_k+N_{k-1}+2j+1)(M_k+N_{k-1}+2j+2))}_{\omega \cdot (1,1,0)} x^{((M_k+N_{k-1}+2j)(M_k+N_{k-1}+2j+1))}_{\omega \cdot (1,0,1)}} 
    \Biggr) \Biggr]\cdot D^{[m_1,n_1,\ldots,m_k]}.
\end{align*}
}

    \emph{III-2.} (For $n_k=0$) For any $(k-1), m_1, n_1, m_2, n_2, \ldots, m_{k-1}, n_{k-1}, (m_k-1) \in \mathbb{N}$, the function satisfies the following relations: For any $m_k = 2l\geq 2$ even:
{\small\begin{align*}
    D^{[m_1,n_1,\ldots,m_k]} 
    &= \frac{x^{((M_{k-1}+N_{k-1}+2l+1)(M_{k-1}+N_{k-1}+2l+2))}_{\omega' \cdot (1,0,1)}}{x^{((M_{k-1}+N_{k-1}+1)(M_{k-1}+N_{k-1}+2))}_{\omega' \cdot (1,0,1)}} D^{[m_1,n_1,\ldots,m_{k-1}]} + x^{((M_{k-1}+N_{k-1}+2l+1)(M_{k-1}+N_{k-1}+2l+2))}_{\omega' \cdot (1,0,1)} \\
    &\quad \cdot \Biggl[ \sum_{j=1}^{l} \Biggl( 
    \frac{x^{(M_{k-1}+N_{k-1}+2j)}_{\omega' \cdot (0,1,1)}}{x^{((M_{k-1}+N_{k-1}+2j)(M_{k-1}+N_{k-1}+2j+1))}_{\omega' \cdot (1,1,0)} x^{((M_{k-1}+N_{k-1}+2j-1)(M_{k-1}+N_{k-1}+2j))}_{\omega' \cdot (1,0,1)}} 
    \\
    &\quad + \frac{x^{(M_{k-1}+N_{k-1}+2j+1)}_{\omega' \cdot (0,1,1)}}{x^{((M_{k-1}+N_{k-1}+2j+1)(M_{k-1}+N_{k-1}+2j+2))}_{\omega' \cdot (1,0,1)} x^{((M_{k-1}+N_{k-1}+2j)(M_{k-1}+N_{k-1}+2j+1))}_{\omega' \cdot (1,1,0)}} 
    \Biggr) \Biggr] \cdot D^{[m_1,n_1,\ldots,n_{k-1}]}.
  \end{align*}
For any $m_k = 2l+1\geq 1$ odd:
    \begin{align*}
    D^{[m_1,n_1,\ldots,m_k]} 
    &= \frac{x^{((M_{k-1}+N_{k-1}+2l+2)(M_{k-1}+N_{k-1}+2l+3))}_{\omega' \cdot (1,1,0)}}{x^{((M_{k-1}+N_{k-1}+1)(M_{k-1}+N_{k-1}+2))}_{\omega' \cdot (1,0,1)}} D^{[m_1,n_1,\ldots,m_{k-1}]} + x^{((M_{k-1}+N_{k-1}+2l+2)(M_{k-1}+N_{k-1}+2l+3))}_{\omega' \cdot (1,1,0)} \\
    &\quad \cdot \Biggl[ 
    \frac{x^{(M_{k-1}+N_{k-1}+2l+2)}_{\omega' \cdot (0,1,1)}}{x^{((M_{k-1}+N_{k-1}+2l+2)(M_{k-1}+N_{k-1}+2l+3))}_{\omega' \cdot (1,1,0)} x^{((M_{k-1}+N_{k-1}+2l+1)(M_{k-1}+N_{k-1}+2l+2))}_{\omega' \cdot (1,0,1)}} 
    \\
    &\quad + \sum_{j=1}^{l} \Biggl( 
    \frac{x^{(M_{k-1}+N_{k-1}+2j)}_{\omega' \cdot (0,1,1)}}{x^{((M_{k-1}+N_{k-1}+2j)(M_{k-1}+N_{k-1}+2j+1))}_{\omega' \cdot (1,1,0)} x^{((M_{k-1}+N_{k-1}+2j-1)(M_{k-1}+N_{k-1}+2j))}_{\omega' \cdot (1,0,1)}} 
    \\
    &\quad + \frac{x^{(M_{k-1}+N_{k-1}+2j+1)}_{\omega' \cdot (0,1,1)}}{x^{((M_{k-1}+N_{k-1}+2j+1)(M_{k-1}+N_{k-1}+2j+2))}_{\omega' \cdot (1,0,1)} x^{((M_{k-1}+N_{k-1}+2j)(M_{k-1}+N_{k-1}+2j+1))}_{\omega' \cdot (1,1,0)}} 
    \Biggr) \Biggr] \cdot D^{[m_1,n_1,\ldots,n_{k-1}]}.
\end{align*}}
\end{Thm}

\subsection*{The case of $2$-triangulated $m$-gon}
In this case our problem becomes solving the general formula $D^{[m_1,n_1,m_2,n_2,...,m_k,n_k]}_i$ (for $i = 1,2$) and the inner vertex $I^{[m_1,n_1,m_2,n_2,...,m_k,n_k]} = I^{[m_1,n_1,m_2,n_2,...,m_k,n_k]}_{(1,1)}$. By doing part-by-part after each flip, we can draw quadrilaterals to calculate the desired expansion formula.
\begin{Thm}[Recurrence formula for a general $2$-triangulated polygon]
\label{thm:2tri}
    In the general $2$-triangulated $m$-gon with $\cT$-triangulation in type $\mathcal{P}(m_1,n_1,...,m_k,n_k)$, denote the following functions:
\begin{enumerate}
	\item[(i)] $(D^{[m_1,n_1,m_2,n_2,...,m_k,n_k]}_1,D^{[m_1,n_1,m_2,n_2,...,m_k,n_k]}_2)$ be the expansion formula of the non $T$-diagonal $d=AB$;
	\item[(ii)] $I^{[m_1,n_1,...,m_k,n_k]}$ be the expansion formula of the inner vertex of $\triangle AVB$ formed by taking the flips in the right-to-left order (i.e. from $A$ to $B$).
\end{enumerate}
Then the functions satisfy the following recurrence relations:

    \emph{I.} ($\mathcal{P}(m_1)$) For any $k \in \mathbb{N}$, then:
{\small
\begin{align*}
    D^{[1]}_2 &= \frac{x^{(1)}_{(0,2,1)}x^{(2)}_{(2,1,0)}}{x^{(12)}_{(2,0,1)}} + \frac{x^{(1)}_{(1,2,0)}x^{(2)}_{(0,1,2)}}{x^{(12)}_{(1,0,2)}} + \frac{x^{(1)}_{(1,1,1)}x^{(2)}_{(0,1,2)}x^{(2)}_{(1,2,0)}}{x^{(12)}_{(1,0,2)}x^{(2)}_{(1,1,1)}} + \frac{x^{(1)}_{(1,1,1)}x^{(2)}_{(0,2,1)}x^{(2)}_{(2,1,0)}}{x^{(12)}_{(2,0,1)}x^{(2)}_{(1,1,1)}};\\
    D^{[2k]}_2 &= \frac{x^{(2k+1)}_{(0,1,2)}x^{(1)}_{(2,1,0)}}{x^{((2k)(2k+1))}_{(2,1,0)}} 
    + \Biggl[ x^{(1)}_{(2,1,0)} \cdot \sum_{j=0}^{k-1} 
    \Biggl( \frac{x^{(2k)}_{(2,1,0)} x^{(2j+2)}_{(1,1,1)}}{x^{((2j+2)(2j+3))}_{(2,1,0)} x^{((2j+1)(2j+2))}_{(2,0,1)}} 
    + \frac{x^{(2k)}_{(2,1,0)} x^{(2j+1)}_{(1,1,1)}}{x^{((2j+1)(2j+2))}_{(2,0,1)} x^{((2j)(2j+1))}_{(2,1,0)}} \Biggr) \Biggr] \\
    &\quad \cdot \Biggl( \frac{x^{(2k+1)}_{(0,1,2)}}{x^{(2k+1)}_{(1,1,1)}} \cdot \frac{x^{(2k+1)}_{(2,0,1)}}{x^{((2k)(2k+1))}_{(2,1,0)}} 
    + \frac{x^{(2k+1)}_{(1,0,2)}}{x^{(2k+1)}_{(1,1,1)}} \cdot \frac{x^{(2k+1)}_{(0,2,1)}}{x^{((2k)(2k+1))}_{(1,2,0)}} \Biggr) 
    \frac{x^{(2k+1)}_{(1,0,2)}}{x^{((2k)(2k+1))}_{(1,2,0)}} \cdot D^{[2k-1]}_2 \\
    &= \frac{x^{(2k+1)}_{(1,0,2)}x^{(2k)}_{(0,2,1)}x^{(1)}_{(2,1,0)}}{x^{((2k)(2k+1))}_{(1,2,0)} x^{((2k-1)(2k))}_{(2,0,1)}} 
    + \frac{x^{(2k+1)}_{(0,1,2)}x^{(1)}_{(2,1,0)}}{x^{((2k)(2k+1))}_{(2,1,0)}} 
    + \frac{x^{(2k+1)}_{(1,0,2)}x^{(1)}_{(2,1,0)}x^{(2k-1)}_{(1,1,1)}}{x^{((2k)(2k+1))}_{(1,2,0)} x^{((2k-2)(2k-1))}_{(2,1,0)}} \\
    &\quad \cdot \Biggl( \frac{x^{(2k)}_{(0,2,1)}}{x^{(2k)}_{(1,1,1)}} \cdot \frac{x^{(2k)}_{(2,1,0)}}{x^{((2k-1)(2k))}_{(2,0,1)}} 
    + \frac{x^{(2k)}_{(1,2,0)}}{x^{(2k)}_{(1,1,1)}} \cdot \frac{x^{(2k)}_{(0,1,2)}}{x^{((2k-1)(2k))}_{(1,0,2)}} \Biggr) \\
    &\quad + \Biggl[ x^{(1)}_{(2,1,0)} \cdot \sum_{j=0}^{k-1} 
    \Biggl( \frac{x^{(2k)}_{(2,1,0)} x^{(2j+2)}_{(1,1,1)}}{x^{((2j+2)(2j+3))}_{(2,1,0)} x^{((2j+1)(2j+2))}_{(2,0,1)}} 
    + \frac{x^{(2k)}_{(2,1,0)} x^{(2j+1)}_{(1,1,1)}}{x^{((2j+1)(2j+2))}_{(2,0,1)} x^{((2j)(2j+1))}_{(2,1,0)}} \Biggr) \Biggr] \\
    &\quad \cdot \Biggl( \frac{x^{(2k+1)}_{(0,1,2)}}{x^{(2k+1)}_{(1,1,1)}} \cdot \frac{x^{(2k+1)}_{(2,0,1)}}{x^{((2k)(2k+1))}_{(2,1,0)}} 
    + \frac{x^{(2k+1)}_{(1,0,2)}}{x^{(2k+1)}_{(1,1,1)}} \cdot \frac{x^{(2k+1)}_{(0,2,1)}}{x^{((2k)(2k+1))}_{(1,2,0)}} \Biggr) \\
    &\quad + \frac{x^{(2k+1)}_{(1,0,2)}}{x^{((2k)(2k+1))}_{(1,2,0)}} \cdot \Biggl[ x^{(1)}_{(2,1,0)} \cdot  \sum_{j=0}^{k-2} 
    \Biggl( \frac{x^{(2k-1)}_{(2,0,1)} x^{(2j+2)}_{(1,1,1)}}{x^{((2j+2)(2j+3))}_{(2,1,0)} x^{((2j+1)(2j+2))}_{(2,0,1)}} 
    + \frac{x^{(2k-1)}_{(2,0,1)} x^{(2j+1)}_{(1,1,1)}}{x^{((2j+1)(2j+2))}_{(2,0,1)} x^{((2j)(2j+1))}_{(2,1,0)}} \Biggr) \Biggr]  \\
    &\quad \cdot \Biggl( \frac{x^{(2k)}_{(0,2,1)}}{x^{(2k)}_{(1,1,1)}} \cdot \frac{x^{(2k)}_{(2,1,0)}}{x^{((2k-1)(2k))}_{(2,0,1)}} 
    + \frac{x^{(2k)}_{(1,2,0)}}{x^{(2k)}_{(1,1,1)}} \cdot \frac{x^{(2k)}_{(0,1,2)}}{x^{((2k-1)(2k))}_{(1,0,2)}} \Biggr) + \frac{x^{(2k+1)}_{(1,0,2)}}{x^{((2k-1)(2k))}_{(1,0,2)}} \cdot D^{[2k-2]}_2;
\end{align*}

\begin{align*}
    D^{[2k+1]}_2 &= \frac{x^{(2k+2)}_{(0,2,1)}x^{(1)}_{(2,1,0)}}{x^{((2k+1)(2k+2))}_{(2,0,1)}} + \Biggl[ \frac{x^{(1)}_{(2,1,0)}x^{(2k+1)}_{(1,1,1)}}{x^{((2k)(2k+1))}_{(2,1,0)}} + x^{(1)}_{(2,1,0)} \cdot \sum_{j=0}^{k-1} \Biggl( \frac{x^{(2k+1)}_{(2,0,1)} x^{(2j+2)}_{(1,1,1)}}{x^{((2j+2)(2j+3))}_{(2,1,0)} x^{((2j+1)(2j+2))}_{(2,0,1)}} + \frac{x^{(2k+1)}_{(2,0,1)} x^{(2j+1)}_{(1,1,1)}}{x^{((2j+1)(2j+2))}_{(2,0,1)} x^{((2j)(2j+1))}_{(2,1,0)}} \Biggr) \Biggr]\\
    &\quad \cdot \Biggl( \frac{x^{(2k+2)}_{(0,2,1)}}{x^{(2k+2)}_{(1,1,1)}} \cdot \frac{x^{(2k+2)}_{(2,1,0)}}{x^{((2k+1)(2k+2))}_{(2,0,1)}}+  \frac{x^{(2k+2)}_{(1,2,0)}}{x^{(2k+2)}_{(1,1,1)}} \cdot \frac{x^{(2k+2)}_{(0,1,2)}}{x^{((2k+1)(2k+2))}_{(1,0,2)}} \Biggr) + \frac{x^{(2k+2)}_{(1,2,0)}}{x^{((2k+1)(2k+2))}_{(1,0,2)}} \cdot D^{[2k]}_2 \\
    &= \frac{x^{(2k+2)}_{(1,2,0)}x^{(2k+1)}_{(0,1,2)}x^{(1)}_{(2,1,0)}}{x^{((2k+1)(2k+2))}_{(1,0,2)} x^{((2k)(2k+1))}_{(2,1,0)}} + \frac{x^{(2k+2)}_{(0,2,1)}x^{(1)}_{(2,1,0)}}{x^{((2k+1)(2k+2))}_{(2,0,1)}} + \frac{x^{(1)}_{(2,1,0)}x^{(2k+1)}_{(1,1,1)}}{x^{((2k)(2k+1))}_{(2,1,0)}} \\
    &\quad \cdot \Biggl( \frac{x^{(2k+2)}_{(0,2,1)}}{x^{(2k+2)}_{(1,1,1)}} \cdot \frac{x^{(2k+2)}_{(2,1,0)}}{x^{((2k+1)(2k+2))}_{(2,0,1)}}  + \frac{x^{(2k+2)}_{(1,2,0)}}{x^{(2k+2)}_{(1,1,1)}} \cdot \frac{x^{(2k+2)}_{(0,1,2)}}{x^{((2k+1)(2k+2))}_{(1,0,2)}} \Biggr) \\
    &\quad+ \Biggl[ x^{(1)}_{(2,1,0)} \cdot  \sum_{j=0}^{k-1} \Biggl( \frac{x^{(2k+1)}_{(2,0,1)}  x^{(2j+2)}_{(1,1,1)}}{x^{((2j+2)(2j+3))}_{(2,1,0)} x^{((2j+1)(2j+2))}_{(2,0,1)}} +  \frac{x^{(2k+1)}_{(2,0,1)} x^{(2j+1)}_{(1,1,1)}}{x^{((2j+1)(2j+2))}_{(2,0,1)} x^{((2j)(2j+1))}_{(2,1,0)}} \Biggr) \Biggr]\\
    &\quad \cdot \Biggl( \frac{x^{(2k+2)}_{(0,2,1)}}{x^{(2k+2)}_{(1,1,1)}} \cdot \frac{x^{(2k+2)}_{(2,1,0)}}{x^{((2k+1)(2k+2))}_{(2,0,1)}} + \frac{x^{(2k+2)}_{(1,2,0)}}{x^{(2k+2)}_{(1,1,1)}} \cdot \frac{x^{(2k+2)}_{(0,1,2)}}{x^{((2k+1)(2k+2))}_{(1,0,2)}}\Biggr)  \\
    &\quad + \frac{x^{(2k+2)}_{(1,2,0)}}{x^{((2k+1)(2k+2))}_{(1,0,2)}} \cdot \Biggl[ x^{(1)}_{(2,1,0)} \cdot \sum_{j=0}^{k-1} \Biggl( \frac{x^{(2k)}_{(2,1,0)} x^{(2j+2)}_{(1,1,1)}}{x^{((2j+2)(2j+3))}_{(2,1,0)} x^{((2j+1)(2j+2))}_{(2,0,1)}} + \frac{x^{(2k)}_{(2,1,0)} x^{(2j+1)}_{(1,1,1)}}{x^{((2j+1)(2j+2))}_{(2,0,1)} x^{((2j)(2j+1))}_{(2,1,0)}} \Biggr) \Biggr] \\
    &\quad \cdot \Biggl( \frac{x^{(2k+1)}_{(0,1,2)}}{x^{(2k+1)}_{(1,1,1)}} \cdot \frac{x^{(2k+1)}_{(2,0,1)}}{x^{((2k)(2k+1))}_{(2,1,0)}} + \frac{x^{(2k+1)}_{(1,0,2)}}{x^{(2k+1)}_{(1,1,1)}} \cdot \frac{x^{(2k+1)}_{(0,2,1)}}{x^{((2k)(2k+1))}_{(1,2,0)}} \Biggr) + \frac{x^{(2k+2)}_{(1,2,0)}}{x^{((2k)(2k+1))}_{(1,2,0)}} \cdot D^{[2k-1]}_2;
\end{align*}

\begin{align*}
    I^{[1]} &= \frac{x^{(1)}_{(1,1,1)}x^{(2)}_{(2,1,0)}}{x^{(12)}_{(2,0,1)}} + \frac{x^{(1)}_{(2,1,0)}x^{(2)}_{(1,1,1)}}{x^{(12)}_{(2,0,1)}}; \\
    I^{[2]} &= \frac{x^{(1)}_{(2,1,0)}x^{(3)}_{(1,1,1)}}{x^{(23)}_{(2,1,0)}} + \frac{x^{(1)}_{(1,1,1)}x^{(3)}_{(2,0,1)}}{x^{(12)}_{(2,0,1)}} + \frac{x^{(1)}_{(2,1,0)}x^{(2)}_{(1,1,1)}x^{(3)}_{(2,0,1)}}{x^{(12)}_{(2,0,1)}x^{(23)}_{(2,1,0)}}; \\
    I^{[2k-1]} &= x^{(1)}_{(2,1,0)} \cdot \sum_{j=0}^{k-1} \Biggl( \frac{x^{(2k)}_{(2,1,0)}  x^{(2j+2)}_{(1,1,1)}}{x^{((2j+2)(2j+3))}_{(2,1,0)} x^{((2j+1)(2j+2))}_{(2,0,1)}} + \frac{x^{(2k)}_{(2,1,0)} x^{(2j+1)}_{(1,1,1)}}{x^{((2j+1)(2j+2))}_{(2,0,1)} x^{((2j)(2j+1))}_{(2,1,0)}} \Biggr); \\
    I^{[2k]} &= \frac{x^{(1)}_{(2,1,0)}x^{(2k+1)}_{(1,1,1)}}{x^{((2k)(2k+1))}_{(2,1,0)}} + x^{(1)}_{(2,1,0)} \cdot \sum_{j=0}^{k-1} \Biggl( \frac{x^{(2k+1)}_{(2,0,1)} x^{(2j+2)}_{(1,1,1)}}{x^{((2j+2)(2j+3))}_{(2,1,0)} x^{((2j+1)(2j+2))}_{(2,0,1)}} + \frac{x^{(2k+1)}_{(2,0,1)} x^{(2j+1)}_{(1,1,1)}}{x^{((2j+1)(2j+2))}_{(2,0,1)} x^{((2j)(2j+1))}_{(2,1,0)}} \Biggr).
\end{align*}
}
	\emph{II.} ($\mathcal{P}(m_1,n_1)$) For any $k,(l+1) \in \mathbb{N}$, then:
{	\small
\begin{align*}
D^{[2k-1,2l]}_2 
&= \frac{x^{((2k+2l)(2k+2l+1))}_{(0,2,1)} x^{(1)}_{(2,1,0)}}{x^{((2k)(2k+1))}_{(0,2,1)}}+ x^{((2k+2l)(2k+2l+1))}_{(0,2,1)} \cdot\\
&\quad\cdot \left[ \sum_{j=1}^l \left( 
\frac{x^{(2k+2j)}_{(1,2,0)}}{x^{((2k+2j)(2k+2j+1))}_{(0,2,1)} x^{((2k+2j-1)(2k+2j))}_{(1,0,2)}} 
+ \frac{x^{(2k+2j-1)}_{(2,1,0)}}{x^{((2k+2j-1)(2k+2j))}_{(2,0,1)} x^{((2k+2j-2)(2k+2j-1))}_{(0,1,2)}} \right) \right] \cdot D^{[2k-1]}_2 \\
&\quad + x^{((2k+2l)(2k+2l+1))}_{(0,2,1)} \\
&\quad \cdot \sum_{j=1}^l \Biggl[ 
\frac{1}{x^{(2k+2j)}_{(1,1,1)}} \cdot \Biggl( 
\frac{x^{(2k+2j)}_{(1,2,0)} x^{((2k+2j)(2k+2j+1))}_{(0,1,2)}}{x^{((2k+2j)(2k+2j+1))}_{(0,2,1)} x^{((2k+2j-1)(2k+2j))}_{(1,0,2)}} + \frac{x^{(2k+2j)}_{(2,1,0)}}{x^{((2k+2j-1)(2k+2j))}_{(2,0,1)}} \Biggr) 
\cdot I^{[2k-1,2j-1]} \\
&\quad + \frac{1}{x^{(2k+2j-1)}_{(1,1,1)}} \cdot \Biggl( 
\frac{x^{(2k+2j-1)}_{(2,1,0)} x^{((2k+2j-1)(2k+2j))}_{(1,0,2)}}{x^{((2k+2j-1)(2k+2j))}_{(2,0,1)} x^{((2k+2j-2)(2k+2j-1))}_{(0,1,2)}}+ \frac{x^{(2k+2j-1)}_{(1,2,0)}}{x^{((2k+2j-2)(2k+2j-1))}_{(0,2,1)}} \Biggr) \cdot I^{[2k-1,2j-2]} \Biggr];
\end{align*}

\begin{align*}
D^{[2k-1,2l+1]}_2 
&= \frac{x^{((2k+2l+1)(2k+2l+2))}_{(2,0,1)} x^{(1)}_{(2,1,0)}}{x^{((2k)(2k+1))}_{(0,2,1)}} 
+ x^{((2k+2l+1)(2k+2l+2))}_{(2,0,1)} \cdot \Biggl[ 
\frac{x^{(2k+2l+1)}_{(2,1,0)}}{x^{((2k+2l+1)(2k+2l+2))}_{(2,0,1)} x^{((2k+2l)(2k+2l+1))}_{(0,1,2)}} \\
&\quad + \sum_{j=1}^l \Biggl( \frac{x^{(2k+2j)}_{(1,2,0)}}{x^{((2k+2j)(2k+2j+1))}_{(0,2,1)} x^{((2k+2j-1)(2k+2j))}_{(1,0,2)}} 
+ \frac{x^{(2k+2j-1)}_{(2,1,0)}}{x^{((2k+2j-1)(2k+2j))}_{(2,0,1)} x^{((2k+2j-2)(2k+2j-1))}_{(0,1,2)}} \Biggr) \Biggr] \cdot D^{[2k-1]}_2 \\
&\quad + x^{((2k+2l+1)(2k+2l+2))}_{(2,0,1)} \\
&\quad \cdot \Biggl[ 
\frac{1}{x^{(2k+2l+1)}_{(1,1,1)}} \cdot \Biggl( 
\frac{x^{(2k+2l+1)}_{(2,1,0)} x^{((2k+2l+1)(2k+2l+2))}_{(1,0,2)}}{x^{((2k+2l+1)(2k+2l+2))}_{(2,0,1)} x^{((2k+2l)(2k+2l+1))}_{(0,1,2)}} + \frac{x^{(2k+2l+1)}_{(1,2,0)}}{x^{((2k+2l)(2k+2l+1))}_{(0,2,1)}} \Biggr) 
\cdot I^{[2k-1,2l]} \\
&\quad + \sum_{j=1}^l \Biggl[ \frac{1}{x^{(2k+2j)}_{(1,1,1)}} \cdot \Biggl( 
\frac{x^{(2k+2j)}_{(1,2,0)} x^{((2k+2j)(2k+2j+1))}_{(0,1,2)}}{x^{((2k+2j)(2k+2j+1))}_{(0,2,1)} x^{((2k+2j-1)(2k+2j))}_{(1,0,2)}} 
+ \frac{x^{(2k+2j)}_{(2,1,0)}}{x^{((2k+2j-1)(2k+2j))}_{(2,0,1)}} \Biggr) 
\cdot I^{[2k-1,2j-1]} \\
&\quad + \frac{1}{x^{(2k+2j-1)}_{(1,1,1)}} \cdot \Biggl( 
\frac{x^{(2k+2j-1)}_{(2,1,0)} x^{((2k+2j-1)(2k+2j))}_{(1,0,2)}}{x^{((2k+2j-1)(2k+2j))}_{(2,0,1)} x^{((2k+2j-2)(2k+2j-1))}_{(0,1,2)}}  + \frac{x^{(2k+2j-1)}_{(1,2,0)}}{x^{((2k+2j-2)(2k+2j-1))}_{(0,2,1)}} \Biggr) 
\cdot I^{[2k-1,2j-2]} \Biggr] \Biggr];
\end{align*}

\begin{align*}
D^{[2k,2l]}_2 
&= \frac{x^{((2k+2l+1)(2k+2l+2))}_{(0,1,2)} x^{(1)}_{(2,1,0)}}{x^{((2k+1)(2k+2))}_{(0,1,2)}} 
+ x^{((2k+2l+1)(2k+2l+2))}_{(0,1,2)}\\
&\quad \cdot \Biggl[ \sum_{j=1}^l \Biggl( 
\frac{x^{(2k+2j+1)}_{(1,0,2)}}{x^{((2k+2j+1)(2k+2j+2))}_{(0,1,2)} x^{((2k+2j)(2k+2j+1))}_{(1,2,0)}} + \frac{x^{(2k+2j)}_{(2,0,1)}}{x^{((2k+2j)(2k+2j+1))}_{(2,1,0)} x^{((2k+2j-1)(2k+2j))}_{(0,2,1)}} \Biggr) \Biggr] 
\cdot D^{[2k]}_2 \\
&\quad + x^{((2k+2l+1)(2k+2l+2))}_{(0,1,2)} \\
&\quad \cdot \sum_{j=1}^l \Biggl[ 
\frac{1}{x^{(2k+2j+1)}_{(1,1,1)}} \cdot \Biggl( 
\frac{x^{(2k+2j+1)}_{(1,0,2)} x^{((2k+2j+1)(2k+2j+2))}_{(0,2,1)}}{x^{((2k+2j+1)(2k+2j+2))}_{(0,1,2)} x^{((2k+2j)(2k+2j+1))}_{(1,2,0)}} + \frac{x^{(2k+2j+1)}_{(2,0,1)}}{x^{((2k+2j)(2k+2j+1))}_{(2,1,0)}} \Biggr) 
\cdot I^{[2k,2j-1]} \\
&\quad + \frac{1}{x^{(2k+2j)}_{(1,1,1)}} \cdot \Biggl( 
\frac{x^{(2k+2j)}_{(2,0,1)} x^{((2k+2j)(2k+2j+1))}_{(1,2,0)}}{x^{((2k+2j)(2k+2j+1))}_{(2,1,0)} x^{((2k+2j-1)(2k+2j))}_{(0,2,1)}} + \frac{x^{(2k+2j)}_{(1,0,2)}}{x^{((2k+2j-1)(2k+2j))}_{(0,1,2)}} \Biggr) 
\cdot I^{[2k,2j-2]} \Biggr];
\end{align*}

\begin{align*}
D^{[2k,2l+1]}_2 
&= \frac{x^{((2k+2l+2)(2k+2l+3))}_{(2,1,0)} x^{(1)}_{(2,1,0)}}{x^{((2k+1)(2k+2))}_{(0,1,2)}} 
+ x^{((2k+2l+2)(2k+2l+3))}_{(2,1,0)} \cdot \Biggl[ 
\frac{x^{(2k+2l+2)}_{(2,0,1)}}{x^{((2k+2l+2)(2k+2l+3))}_{(2,1,0)} x^{((2k+2l+1)(2k+2l+2))}_{(0,2,1)}} \\
&\quad + \sum_{j=1}^l \Biggl( \frac{x^{(2k+2j+1)}_{(1,0,2)}}{x^{((2k+2j+1)(2k+2j+2))}_{(0,1,2)} x^{((2k+2j)(2k+2j+1))}_{(1,2,0)}} 
+ \frac{x^{(2k+2j)}_{(2,0,1)}}{x^{((2k+2j)(2k+2j+1))}_{(2,1,0)} x^{((2k+2j-1)(2k+2j))}_{(0,2,1)}} \Biggr) \Biggr] \cdot D^{[2k]}_2 \\
&\quad + x^{((2k+2l+2)(2k+2l+3))}_{(2,1,0)} \\
&\quad \cdot \Biggl[ 
\frac{1}{x^{(2k+2l+2)}_{(1,1,1)}} \cdot \Biggl( 
\frac{x^{(2k+2l+2)}_{(2,0,1)} x^{((2k+2l+2)(2k+2l+3))}_{(1,2,0)}}{x^{((2k+2l+2)(2k+2l+3))}_{(2,1,0)} x^{((2k+2l+1)(2k+2l+2))}_{(0,2,1)}}+ \frac{x^{(2k+2l+2)}_{(1,0,2)}}{x^{((2k+2l+1)(2k+2l+2))}_{(0,1,2)}} \Biggr) 
\cdot I^{[2k,2l]} \\
&\quad + \sum_{j=1}^l \Biggl[ \frac{1}{x^{(2k+2j+1)}_{(1,1,1)}} \cdot \Biggl( 
\frac{x^{(2k+2j+1)}_{(1,0,2)} x^{((2k+2j+1)(2k+2j+2))}_{(0,2,1)}}{x^{((2k+2j+1)(2k+2j+2))}_{(0,1,2)} x^{((2k+2j)(2k+2j+1))}_{(1,2,0)}} 
+ \frac{x^{(2k+2j+1)}_{(2,0,1)}}{x^{((2k+2j)(2k+2j+1))}_{(2,1,0)}} \Biggr) 
\cdot I^{[2k,2j-1]} \\
&\quad + \frac{1}{x^{(2k+2j)}_{(1,1,1)}} \cdot \Biggl( 
\frac{x^{(2k+2j)}_{(2,0,1)} x^{((2k+2j)(2k+2j+1))}_{(1,2,0)}}{x^{((2k+2j)(2k+2j+1))}_{(2,1,0)} x^{((2k+2j-1)(2k+2j))}_{(0,2,1)}} + \frac{x^{(2k+2j)}_{(1,0,2)}}{x^{((2k+2j-1)(2k+2j))}_{(0,1,2)}} \Biggr) 
\cdot I^{[2k,2j-2]} \Biggr] \Biggr];
\end{align*}

\begin{align*}
I^{[2k-1,2l]} &= \frac{x^{((2k+2l)(2k+2l+1))}_{(0,1,2)}}{x^{((2k)(2k+1))}_{(0,1,2)}} \cdot I^{[2k-1]} 
+ x^{((2k+2l)(2k+2l+1))}_{(0,1,2)} \cdot \Biggl[ \sum_{j=1}^l \Biggl( \frac{x^{(2k+2j)}_{(1,1,1)}}{x^{((2k+2j)(2k+2j+1))}_{(0,1,2)} x^{((2k+2j-1)(2k+2j))}_{(1,0,2)}} \\
&\quad + \frac{x^{(2k+2j-1)}_{(1,1,1)}}{x^{((2k+2j-1)(2k+2j))}_{(1,0,2)} x^{((2k+2j-2)(2k+2j-1))}_{(0,1,2)}} \Biggr) \Biggr] \cdot D^{[2k-1]}_2; \\
I^{[2k-1,2l+1]} &= \frac{x^{((2k+2l+1)(2k+2l+2))}_{(1,0,2)}}{x^{((2k)(2k+1))}_{(0,1,2)}} \cdot I^{[2k-1]} 
+ x^{((2k+2l+1)(2k+2l+2))}_{(1,0,2)} \cdot \Biggl[ \frac{x^{(2k+2l+1)}_{(1,1,1)}}{x^{((2k+2l+1)(2k+2l+2))}_{(1,0,2)} x^{((2k+2l)(2k+2l+1))}_{(0,1,2)}} \\
&\quad + \sum_{j=1}^l \Biggl( \frac{x^{(2k+2j)}_{(1,1,1)}}{x^{((2k+2j)(2k+2j+1))}_{(0,1,2)} x^{((2k+2j-1)(2k+2j))}_{(1,0,2)}} 
+ \frac{x^{(2k+2j-1)}_{(1,1,1)}}{x^{((2k+2j-1)(2k+2j))}_{(1,0,2)} x^{((2k+2j-2)(2k+2j-1))}_{(0,1,2)}} \Biggr) \Biggr] \cdot D^{[2k-1]}_2; \\
I^{[2k,2l]} &= \frac{x^{((2k+2l+1)(2k+2l+2))}_{(0,2,1)}}{x^{((2k+1)(2k+2))}_{(0,2,1)}} \cdot I^{[2k]} 
+ x^{((2k+2l+1)(2k+2l+2))}_{(0,2,1)} \cdot \Biggl[ \sum_{j=1}^l \Biggl( \frac{x^{(2k+2j+1)}_{(1,1,1)}}{x^{((2k+2j+1)(2k+2j+2))}_{(0,2,1)} x^{((2k+2j)(2k+2j+1))}_{(1,2,0)}} \\
&\quad + \frac{x^{(2k+2j)}_{(1,1,1)}}{x^{((2k+2j)(2k+2j+1))}_{(1,2,0)} x^{((2k+2j-1)(2k+2j))}_{(0,2,1)}} \Biggr) \Biggr] \cdot D^{[2k]}_2; \\
I^{[2k,2l+1]} &= \frac{x^{((2k+2l+2)(2k+2l+3))}_{(1,2,0)}}{x^{((2k+1)(2k+2))}_{(0,2,1)}} \cdot I^{[2k]} 
+ x^{((2k+2l+2)(2k+2l+3))}_{(1,2,0)} \cdot \Biggl[ \frac{x^{(2k+2l+2)}_{(1,1,1)}}{x^{((2k+2l+2)(2k+2l+3))}_{(1,2,0)} x^{((2k+2l+1)(2k+2l+2))}_{(0,2,1)}} \\
&\quad + \sum_{j=1}^l \Biggl( \frac{x^{(2k+2j+1)}_{(1,1,1)}}{x^{((2k+2j+1)(2k+2j+2))}_{(0,2,1)} x^{((2k+2j)(2k+2j+1))}_{(1,2,0)}} 
+ \frac{x^{(2k+2j)}_{(1,1,1)}}{x^{((2k+2j)(2k+2j+1))}_{(1,2,0)} x^{((2k+2j-1)(2k+2j))}_{(0,2,1)}} \Biggr) \Biggr] \cdot D^{[2k]}_2.
\end{align*}
}
    \emph{III-1.} (For $n_k \geq 1$) For any $k,m_1,n_1,m_2,n_2,...,m_k,(n_k-1) \in \mathbb{N}$, then the functions satisfy the following relations:  
    
    If $n_k = 2l \geq 2$ is even:
  {\small  \begin{align*}
    &I^{[m_1,n_1,\ldots,m_{k-1},n_{k-1},m_k,n_k]} \\
    &= \frac{x^{((M_k+N_k+1)(M_k+N_k+2))}_{\omega \cdot (2,1,0)}}{x^{((M_k+N_{k-1}+1)(M_k+N_{k-1}+2))}_{\omega \cdot (2,1,0)}} \cdot I^{[m_1,n_1,\ldots,m_{k-1},n_{k-1},m_k]}  + x^{((M_k+N_k+1)(M_k+N_k+2))}_{\omega \cdot (2,1,0)} \\
    &\quad \cdot \Biggl[ \sum_{j=1}^l \Biggl( \frac{x^{(M_k+N_{k-1}+2j+1)}_{(1,1,1)}}{x^{((M_k+N_{k-1}+2j+1)(M_k+N_{k-1}+2j+2))}_{\omega \cdot (2,1,0)} x^{((M_k+N_{k-1}+2j)(M_k+N_{k-1}+2j+1))}_{\omega \cdot (2,0,1)}} \\
    &\quad + \frac{x^{(M_k+N_{k-1}+2j)}_{(1,1,1)}}{x^{((M_k+N_{k-1}+2j)(M_k+N_{k-1}+2j+1))}_{\omega \cdot (2,0,1)} x^{((M_k+N_{k-1}+2j-1)(M_k+N_{k-1}+2j))}_{\omega \cdot (2,1,0)}} \Biggr) \Biggr] \cdot D^{[m_1,n_1,\ldots,m_{k-1},n_{k-1},m_k]}_{2};
\end{align*}

\begin{align*}
&D^{[m_1,n_1,\ldots,m_{k-1},n_{k-1},m_k,n_k]}_{2} \\
    &= \frac{x^{((M_k+N_k+1)(M_k+N_k+2))}_{\omega \cdot (1,2,0)}}{x^{((M_k+N_{k-1}+1)(M_k+N_{k-1}+2))}_{\omega \cdot (1,2,0)}} \cdot D^{[m_1,n_1,\ldots,m_{k-1},n_{k-1}]}_{2} \\
    &\quad + x^{((M_k+N_k+1)(M_k+N_k+2))}_{\omega \cdot (1,2,0)} \cdot \Biggl[ \sum_{j=1}^l \Biggl( \frac{x^{(M_k+N_{k-1}+2j+1)}_{\omega \cdot (0,2,1)}}{x^{((M_k+N_{k-1}+2j+1)(M_k+N_{k-1}+2j+2))}_{\omega \cdot (1,2,0)} x^{((M_k+N_{k-1}+2j)(M_k+N_{k-1}+2j+1))}_{\omega \cdot (2,0,1)}} \\
    &\quad + \frac{x^{(M_k+N_{k-1}+2j)}_{\omega \cdot (0,1,2)}}{x^{((M_k+N_{k-1}+2j)(M_k+N_{k-1}+2j+1))}_{\omega \cdot (1,0,2)} x^{((M_k+N_{k-1}+2j-1)(M_k+N_{k-1}+2j))}_{\omega \cdot (2,1,0)}} \Biggr) \Biggr] \cdot D^{[m_1,n_1,\ldots,m_{k-1},n_{k-1},m_k]}_{2} \\
    &\quad + x^{((M_k+N_k+1)(M_k+N_k+2))}_{\omega \cdot (1,2,0)} \\
    &\quad \cdot \sum_{j=1}^l \Biggl[ \frac{1}{x^{(M_k+N_{k-1}+2j+1)}_{(1,1,1)}} \cdot \Biggl( \frac{x^{(M_k+N_{k-1}+2j+1)}_{\omega \cdot (0,2,1)} x^{((M_k+N_{k-1}+2j+1)(M_k+N_{k-1}+2j+2))}_{\omega \cdot (2,1,0)}}{x^{((M_k+N_{k-1}+2j+1)(M_k+N_{k-1}+2j+2))}_{\omega \cdot (1,2,0)} x^{((M_k+N_{k-1}+2j)(M_k+N_{k-1}+2j+1))}_{\omega \cdot (2,0,1)}} \\
    &\quad +  \frac{x^{(M_k+N_{k-1}+2j+1)}_{\omega \cdot (0,1,2)}}{x^{((M_k+N_{k-1}+2j)(M_k+N_{k-1}+2j+1))}_{\omega \cdot (1,0,2)}} \Biggr) \cdot I^{[m_1,n_1,\ldots,m_{k-1},n_{k-1},m_k,2j-1]} \\
    &\quad + \frac{1}{x^{(M_k+N_{k-1}+2j)}_{(1,1,1)}} \cdot \Biggl( \frac{x^{(M_k+N_{k-1}+2j)}_{\omega \cdot (0,1,2)} x^{((M_k+N_{k-1}+2j)(M_k+N_{k-1}+2j+1))}_{\omega \cdot (2,0,1)}}{x^{((M_k+N_{k-1}+2j)(M_k+N_{k-1}+2j+1))}_{\omega \cdot (1,0,2)} x^{((M_k+N_{k-1}+2j-1)(M_k+N_{k-1}+2j))}_{\omega \cdot (2,1,0)}} \\
    &\quad + \frac{x^{(M_k+N_{k-1}+2j)}_{\omega \cdot (0,2,1)}}{x^{((M_k+N_{k-1}+2j-1)(M_k+N_{k-1}+2j))}_{\omega \cdot (1,2,0)}} \Biggr) \cdot I^{[m_1,n_1,\ldots,m_{k-1},n_{k-1},m_k,2j-2]} \Biggr].
\end{align*}

If $n_k = 2l+1 \geq 1$ is odd:
\begin{align*}
    &I^{[m_1,n_1,\ldots,m_{k-1},n_{k-1},m_k,n_k]} \\
    &= \frac{x^{((M_k+N_k+1)(M_k+N_k+2))}_{\omega \cdot (2,0,1)}}{x^{((M_k+N_{k-1}+1)(M_k+N_{k-1}+2))}_{\omega \cdot (2,1,0)}} \cdot I^{[m_1,n_1,\ldots,m_{k-1},n_{k-1},m_k]} \\
    &\quad + x^{((M_k+N_k+1)(M_k+N_k+2))}_{\omega \cdot (2,0,1)} \cdot \Biggl[ \frac{x^{(M_k+N_k+1)}_{(1,1,1)}}{x^{((M_k+N_k+1)(M_k+N_k+2))}_{\omega \cdot (2,0,1)} x^{((M_k+N_k)(M_k+N_k+1))}_{\omega \cdot (2,1,0)}} \\
    &\quad + \sum_{j=1}^l \Biggl( \frac{x^{(M_k+N_{k-1}+2j+1)}_{(1,1,1)}}{x^{((M_k+N_{k-1}+2j+1)(M_k+N_{k-1}+2j+2))}_{\omega \cdot (2,1,0)} x^{((M_k+N_{k-1}+2j)(M_k+N_{k-1}+2j+1))}_{\omega \cdot (2,0,1)}} \\
    &\quad + \frac{x^{(M_k+N_{k-1}+2j)}_{(1,1,1)}}{x^{((M_k+N_{k-1}+2j)(M_k+N_{k-1}+2j+1))}_{\omega \cdot (2,0,1)} x^{((M_k+N_{k-1}+2j-1)(M_k+N_{k-1}+2j))}_{\omega \cdot (2,1,0)}} \Biggr) \Biggr] \cdot D^{[m_1,n_1,\ldots,m_{k-1},n_{k-1},m_k]}_{2}; 
\end{align*}

\begin{align*}
    &D^{[m_1,n_1,\ldots,m_{k-1},n_{k-1},m_k,n_k]}_{2} \\
    &= \frac{x^{((M_k+N_k+1)(M_k+N_k+2))}_{\omega \cdot (1,0,2)}}{x^{((M_k+N_{k-1}+1)(M_k+N_{k-1}+2))}_{\omega \cdot (1,2,0)}} \cdot D^{[m_1,n_1,\ldots,m_{k-1},n_{k-1}]}_{2} \\
    &\quad + x^{((M_k+N_k+1)(M_k+N_k+2))}_{\omega \cdot (1,0,2)} \cdot \Biggl[ \frac{x^{(M_k+N_k+1)}_{\omega \cdot (0,1,2)}}{x^{((M_k+N_k+1)(M_k+N_k+2))}_{\omega \cdot (1,0,2)} x^{((M_k+N_k)(M_k+N_k+1))}_{\omega \cdot (2,1,0)}} \\
    &\quad + \sum_{j=1}^l \Biggl( \frac{x^{(M_k+N_{k-1}+2j+1)}_{\omega \cdot (0,2,1)}}{x^{((M_k+N_{k-1}+2j+1)(M_k+N_{k-1}+2j+2))}_{\omega \cdot (1,2,0)} x^{((M_k+N_{k-1}+2j)(M_k+N_{k-1}+2j+1))}_{\omega \cdot (2,0,1)}} \\
    &\quad + \frac{x^{(M_k+N_{k-1}+2j)}_{\omega \cdot (0,1,2)}}{x^{((M_k+N_{k-1}+2j)(M_k+N_{k-1}+2j+1))}_{\omega \cdot (1,0,2)} x^{((M_k+N_{k-1}+2j-1)(M_k+N_{k-1}+2j))}_{\omega \cdot (2,1,0)}} \Biggr) \Biggr] \cdot D^{[m_1,n_1,\ldots,m_{k-1},n_{k-1},m_k]}_{2} \\
    &\quad + x^{((M_k+N_k+1)(M_k+N_k+2))}_{\omega \cdot (1,0,2)} \cdot \Biggl[ \frac{1}{x^{(M_k+N_k+1)}_{(1,1,1)}} \cdot \Biggl( \frac{x^{(M_k+N_k+1)}_{\omega \cdot (0,1,2)} x^{((M_k+N_k+1)(M_k+N_k+2))}_{\omega \cdot (2,0,1)}}{x^{((M_k+N_k+1)(M_k+N_k+2))}_{\omega \cdot (1,0,2)} x^{((M_k+N_k)(M_k+N_k+1))}_{\omega \cdot (2,1,0)}} \\
    &\quad + \frac{x^{(M_k+N_k+1)}_{\omega \cdot (0,2,1)}}{x^{((M_k+N_k)(M_k+N_k+1))}_{\omega \cdot (1,2,0)}} \Biggr) \cdot I^{[m_1,n_1,\ldots,m_{k-1},n_{k-1},m_k,2l]} \\
    &\quad + \sum_{j=1}^l \Biggl[ \frac{1}{x^{(M_k+N_{k-1}+2j+1)}_{(1,1,1)}} \cdot \Biggl( \frac{x^{(M_k+N_{k-1}+2j+1)}_{\omega \cdot (0,2,1)} x^{((M_k+N_{k-1}+2j+1)(M_k+N_{k-1}+2j+2))}_{\omega \cdot (2,1,0)}}{x^{((M_k+N_{k-1}+2j+1)(M_k+N_{k-1}+2j+2))}_{\omega \cdot (1,2,0)} x^{((M_k+N_{k-1}+2j)(M_k+N_{k-1}+2j+1))}_{\omega \cdot (2,0,1)}} \\
    &\quad + \frac{x^{(M_k+N_{k-1}+2j+1)}_{\omega \cdot (0,1,2)}}{x^{((M_k+N_{k-1}+2j)(M_k+N_{k-1}+2j+1))}_{\omega \cdot (1,0,2)}} \Biggr) \cdot I^{[m_1,n_1,\ldots,m_{k-1},n_{k-1},m_k,2j-1]} \\
    &\quad + \frac{1}{x^{(M_k+N_{k-1}+2j)}_{(1,1,1)}} \cdot \Biggl( \frac{x^{(M_k+N_{k-1}+2j)}_{\omega \cdot (0,1,2)} x^{((M_k+N_{k-1}+2j)(M_k+N_{k-1}+2j+1))}_{\omega \cdot (2,0,1)}}{x^{((M_k+N_{k-1}+2j)(M_k+N_{k-1}+2j+1))}_{\omega \cdot (1,0,2)} x^{((M_k+N_{k-1}+2j-1)(M_k+N_{k-1}+2j))}_{\omega \cdot (2,1,0)}} \\
    &\quad + \frac{x^{(M_k+N_{k-1}+2j)}_{\omega \cdot (0,2,1)}}{x^{((M_k+N_{k-1}+2j-1)(M_k+N_{k-1}+2j))}_{\omega \cdot (1,2,0)}} \Biggr) \cdot I^{[m_1,n_1,\ldots,m_{k-1},n_{k-1},m_k,2j-2]} \Biggr] \Biggr].
\end{align*}
}
    \emph{III-2.} (For $n_k=0$) For any $(k-1),m_1,n_1,m_2,n_2,...,m_{k-1},n_{k-1},(m_k-1) \in \mathbb{N}$, then the functions satisfy the following relations: 
    
    If $m_k=2l\geq 2$ is even:
{\small\begin{align*}
&I^{[m_1,n_1,m_2,n_2,...,m_{k-1},n_{k-1},m_k]} \\
&= \frac{x^{((M_k+N_{k-1}+1)(M_k+N_{k-1}+2))}_{\omega' \cdot (2,0,1)}}{x^{((M_{k-1}+N_{k-1}+1)(M_{k-1}+N_{k-1}+2))}_{\omega' \cdot (2,0,1)}} \cdot I^{[m_1,n_1,m_2,n_2,...,m_{k-1},n_{k-1}]} + x^{((M_k+N_{k-1}+1)(M_k+N_{k-1}+2))}_{\omega' \cdot (2,0,1)} \\
&\quad \cdot \Biggl[ \sum_{j=1}^l \Biggl( \frac{x^{(M_{k-1}+N_{k-1}+2j+1)}_{(1,1,1)}}{x^{((M_{k-1}+N_{k-1}+2j+1)(M_{k-1}+N_{k-1}+2j+2))}_{\omega' \cdot (2,0,1)} x^{((M_{k-1}+N_{k-1}+2j)(M_{k-1}+N_{k-1}+2j+1))}_{\omega' \cdot (2,1,0)}} \\
&\quad + \frac{x^{(M_{k-1}+N_{k-1}+2j)}_{(1,1,1)}}{x^{((M_{k-1}+N_{k-1}+2j)(M_{k-1}+N_{k-1}+2j+1))}_{\omega' \cdot (2,1,0)} x^{((M_{k-1}+N_{k-1}+2j-1)(M_{k-1}+N_{k-1}+2j))}_{\omega' \cdot (2,0,1)}} \Biggr) \Biggr] \cdot D^{[m_1,n_1,...,m_{k-1},n_{k-1}]}_{2};
\end{align*}

\begin{align*}
&D^{[m_1,n_1,...,m_{k-1},n_{k-1},m_k]}_{2} \\
&= \frac{x^{((M_k+N_{k-1}+1)(M_k+N_{k-1}+2))}_{\omega' \cdot (1,0,2)}}{x^{((M_{k-1}+N_{k-1}+1)(M_{k-1}+N_{k-1}+2))}_{\omega' \cdot (1,0,2)}} \cdot D^{[m_1,n_1,...,m_{k-2},n_{k-2},m_{k-1}]}_{2} \\
&\quad + x^{((M_k+N_{k-1}+1)(M_k+N_{k-1}+2))}_{\omega' \cdot (1,0,2)} \cdot \Biggl[ \sum_{j=1}^l \Biggl( \frac{x^{(M_{k-1}+N_{k-1}+2j+1)}_{\omega' \cdot (0,1,2)}}{x^{((M_{k-1}+N_{k-1}+2j+1)(M_{k-1}+N_{k-1}+2j+2))}_{\omega' \cdot (1,0,2)} x^{((M_{k-1}+N_{k-1}+2j)(M_{k-1}+N_{k-1}+2j+1))}_{\omega' \cdot (2,1,0)}} \\
&\quad + \frac{x^{(M_{k-1}+N_{k-1}+2j)}_{\omega' \cdot (0,2,1)}}{x^{((M_{k-1}+N_{k-1}+2j)(M_{k-1}+N_{k-1}+2j+1))}_{\omega' \cdot (1,2,0)} x^{((M_{k-1}+N_{k-1}+2j-1)(M_{k-1}+N_{k-1}+2j))}_{\omega' \cdot (2,0,1)}} \Biggr) \Biggr] \cdot D^{[m_1,n_1,...,m_{k-1},n_{k-1}]}_{2} \\
&\quad + x^{((M_k+N_{k-1}+1)(M_k+N_{k-1}+2))}_{\omega' \cdot (1,0,2)} \cdot \sum_{j=1}^l \Biggl[ \frac{1}{x^{(M_{k-1}+N_{k-1}+2j+1)}_{(1,1,1)}} \\
&\quad \cdot \Biggl( \frac{x^{(M_{k-1}+N_{k-1}+2j+1)}_{\omega' \cdot (0,1,2)} x^{((M_{k-1}+N_{k-1}+2j+1)(M_{k-1}+N_{k-1}+2j+2))}_{\omega' \cdot (2,0,1)}}{x^{((M_{k-1}+N_{k-1}+2j+1)(M_{k-1}+N_{k-1}+2j+2))}_{\omega' \cdot (1,0,2)} x^{((M_{k-1}+N_{k-1}+2j)(M_{k-1}+N_{k-1}+2j+1))}_{\omega' \cdot (2,1,0)}} \\
&\quad + \frac{x^{(M_{k-1}+N_{k-1}+2j+1)}_{\omega' \cdot (0,2,1)}}{x^{((M_{k-1}+N_{k-1}+2j)(M_{k-1}+N_{k-1}+2j+1))}_{\omega' \cdot (1,2,0)}} \Biggr) \cdot I^{[m_1,n_1,...,m_{k-1},n_{k-1},2j-1]} \\
&\quad + \frac{1}{x^{(M_{k-1}+N_{k-1}+2j)}_{(1,1,1)}} \cdot \Biggl( \frac{x^{(M_{k-1}+N_{k-1}+2j)}_{\omega' \cdot (0,2,1)} x^{((M_{k-1}+N_{k-1}+2j)(M_{k-1}+N_{k-1}+2j+1))}_{\omega' \cdot (2,1,0)}}{x^{((M_{k-1}+N_{k-1}+2j)(M_{k-1}+N_{k-1}+2j+1))}_{\omega' \cdot (1,2,0)} x^{((M_{k-1}+N_{k-1}+2j-1)(M_{k-1}+N_{k-1}+2j))}_{\omega' \cdot (2,0,1)}} \\
&\quad + \frac{x^{(M_{k-1}+N_{k-1}+2j)}_{\omega' \cdot (0,1,2)}}{x^{((M_{k-1}+N_{k-1}+2j-1)(M_{k-1}+N_{k-1}+2j))}_{\omega' \cdot (1,0,2)}} \Biggr) \cdot I^{[m_1,n_1,...,m_{k-1},n_{k-1},2j-2]} \Biggr].
\end{align*}

If $m_k=2l+1\geq 1$ is odd:
\begin{align*}
&I^{[m_1,n_1,m_2,n_2,...,m_{k-1},n_{k-1},m_k]} \\
&= \frac{x^{((M_k+N_{k-1}+1)(M_k+N_{k-1}+2))}_{\omega' \cdot (2,1,0)}}{x^{((M_{k-1}+N_{k-1}+1)(M_{k-1}+N_{k-1}+2))}_{\omega' \cdot (2,0,1)}} \cdot I^{[m_1,n_1,m_2,n_2,...,m_{k-1},n_{k-1}]} \\
&\quad + x^{((M_k+N_{k-1}+1)(M_k+N_{k-1}+2))}_{\omega' \cdot (2,1,0)}\cdot \Biggl[ \frac{x^{(M_k+N_{k-1}+1)}_{(1,1,1)}}{x^{((M_k+N_{k-1}+1)(M_k+N_{k-1}+2))}_{\omega' \cdot (2,1,0)} x^{((M_k+N_{k-1})(M_k+N_{k-1}+1))}_{\omega' \cdot (2,0,1)}} \\
&\quad + \sum_{j=1}^l \Biggl( \frac{x^{(M_{k-1}+N_{k-1}+2j+1)}_{(1,1,1)}}{x^{((M_{k-1}+N_{k-1}+2j+1)(M_{k-1}+N_{k-1}+2j+2))}_{\omega' \cdot (2,0,1)} x^{((M_{k-1}+N_{k-1}+2j)(M_{k-1}+N_{k-1}+2j+1))}_{\omega' \cdot (2,1,0)}} \\
&\quad + \frac{x^{(M_{k-1}+N_{k-1}+2j)}_{(1,1,1)}}{x^{((M_{k-1}+N_{k-1}+2j)(M_{k-1}+N_{k-1}+2j+1))}_{\omega' \cdot (2,1,0)} x^{((M_{k-1}+N_{k-1}+2j-1)(M_{k-1}+N_{k-1}+2j))}_{\omega' \cdot (2,0,1)}} \Biggr) \Biggr] \cdot D^{[m_1,n_1,...,m_{k-1},n_{k-1}]}_{2};
\end{align*}

\begin{align*}
&D^{[m_1,n_1,...,m_{k-1},n_{k-1},m_k]}_{2} \\
&= \frac{x^{((M_k+N_{k-1}+1)(M_k+N_{k-1}+2))}_{\omega' \cdot (1,2,0)}}{x^{((M_{k-1}+N_{k-1}+1)(M_{k-1}+N_{k-1}+2))}_{\omega' \cdot (1,0,2)}} \cdot D^{[m_1,n_1,...,m_{k-2},n_{k-2},m_{k-1}]}_{2} \\
&\quad + x^{((M_k+N_{k-1}+1)(M_k+N_{k-1}+2))}_{\omega' \cdot (1,2,0)} \cdot \Biggl[ \frac{x^{(M_k+N_{k-1}+1)}_{\omega' \cdot (0,2,1)}}{x^{((M_k+N_{k-1}+1)(M_k+N_{k-1}+2))}_{\omega' \cdot (1,2,0)} x^{((M_k+N_{k-1})(M_k+N_{k-1}+1))}_{\omega' \cdot (2,0,1)}} \\
&\quad + \sum_{j=1}^l \Biggl( \frac{x^{(M_{k-1}+N_{k-1}+2j+1)}_{\omega' \cdot (0,1,2)}}{x^{((M_{k-1}+N_{k-1}+2j+1)(M_{k-1}+N_{k-1}+2j+2))}_{\omega' \cdot (1,0,2)} x^{((M_{k-1}+N_{k-1}+2j)(M_{k-1}+N_{k-1}+2j+1))}_{\omega' \cdot (2,1,0)}} \\
&\quad + \frac{x^{(M_{k-1}+N_{k-1}+2j)}_{\omega' \cdot (0,2,1)}}{x^{((M_{k-1}+N_{k-1}+2j)(M_{k-1}+N_{k-1}+2j+1))}_{\omega' \cdot (1,2,0)} x^{((M_{k-1}+N_{k-1}+2j-1)(M_{k-1}+N_{k-1}+2j))}_{\omega' \cdot (2,0,1)}} \Biggr) \Biggr] \cdot D^{[m_1,n_1,...,m_{k-1},n_{k-1}]}_{2} \\
&\quad + x^{((M_k+N_{k-1}+1)(M_k+N_{k-1}+2))}_{\omega' \cdot (1,2,0)} \cdot \Biggl[ \frac{1}{x^{(M_k+N_{k-1}+1)}_{(1,1,1)}} \cdot \Biggl( \frac{x^{(M_k+N_{k-1}+1)}_{\omega' \cdot (0,2,1)} x^{((M_k+N_{k-1}+1)(M_k+N_{k-1}+2))}_{\omega' \cdot (2,1,0)}}{x^{((M_k+N_{k-1}+1)(M_k+N_{k-1}+2))}_{\omega' \cdot (1,2,0)} x^{((M_k+N_{k-1})(M_k+N_{k-1}+1))}_{\omega' \cdot (2,0,1)}} \\
&\quad + \frac{x^{(M_k+N_{k-1}+1)}_{\omega' \cdot (0,1,2)}}{x^{((M_k+N_{k-1})(M_k+N_{k-1}+1))}_{\omega' \cdot (1,0,2)}} \Biggr) \cdot I^{[m_1,n_1,...,m_{k-1},n_{k-1},2l]} \\
&\quad + \sum_{j=1}^l \Biggl[ \frac{1}{x^{(M_{k-1}+N_{k-1}+2j+1)}_{(1,1,1)}} \cdot \Biggl( \frac{x^{(M_{k-1}+N_{k-1}+2j+1)}_{\omega' \cdot (0,1,2)} x^{((M_{k-1}+N_{k-1}+2j+1)(M_{k-1}+N_{k-1}+2j+2))}_{\omega' \cdot (2,0,1)}}{x^{((M_{k-1}+N_{k-1}+2j+1)(M_{k-1}+N_{k-1}+2j+2))}_{\omega' \cdot (1,0,2)} x^{((M_{k-1}+N_{k-1}+2j)(M_{k-1}+N_{k-1}+2j+1))}_{\omega' \cdot (2,1,0)}} \\
&\quad + \frac{x^{(M_{k-1}+N_{k-1}+2j+1)}_{\omega' \cdot (0,2,1)}}{x^{((M_{k-1}+N_{k-1}+2j)(M_{k-1}+N_{k-1}+2j+1))}_{\omega' \cdot (1,2,0)}} \Biggr) \cdot I^{[m_1,n_1,...,m_{k-1},n_{k-1},2j-1]} \\
&\quad + \frac{1}{x^{(M_{k-1}+N_{k-1}+2j)}_{(1,1,1)}} \cdot \Biggl( \frac{x^{(M_{k-1}+N_{k-1}+2j)}_{\omega' \cdot (0,2,1)} x^{((M_{k-1}+N_{k-1}+2j)(M_{k-1}+N_{k-1}+2j+1))}_{\omega' \cdot (2,1,0)}}{x^{((M_{k-1}+N_{k-1}+2j)(M_{k-1}+N_{k-1}+2j+1))}_{\omega' \cdot (1,2,0)} x^{((M_{k-1}+N_{k-1}+2j-1)(M_{k-1}+N_{k-1}+2j))}_{\omega' \cdot (2,0,1)}} \\
&\quad + \frac{x^{(M_{k-1}+N_{k-1}+2j)}_{\omega' \cdot (0,1,2)}}{x^{((M_{k-1}+N_{k-1}+2j-1)(M_{k-1}+N_{k-1}+2j))}_{\omega' \cdot (1,0,2)}} \Biggr) \cdot I^{[m_1,n_1,...,m_{k-1},n_{k-1},2j-2]} \Biggr] \Biggr].
\end{align*}
}
\end{Thm}

\subsection{General number of terms property and examples}
\label{subsec:general-number-terms}Consider the $n$-triangulated $\cT$-triangulation of polygon of type $\mathcal{P}(p_1,p_2,\ldots,p_N)$ for $p_i>0$. Let $a_l(p_1,p_2,\ldots,p_N)$ be the value of $D^{[p_1,p_2,\ldots,p_N]}_l$ when all variables $x^{(r)}_{(w_1,w_2,w_3)} = 1$. In other words $a_l(p_1,p_2,\ldots,p_N)$ is also the number of terms (counting multiplicities) in the Laurent expansion of $D^{[p_1,p_2,\ldots,p_N]}_l$. 

For $n = 1$, we simply write $a(p_1,p_2,\ldots,p_N) := a_1(p_1,p_2,\ldots,p_N)$. Let $a(\emptyset) := 1$ (i.e.\ in case $N = 0$). Based on all possible cases above, we deduce the following recurrence:
\begin{align*}
    a(p_1) &= p_1+1; \\
    a(p_1,p_2) &= p_1p_2 + p_2 + 1; \\
    a(p_1,p_2,\ldots,p_{N+2}) &= p_{N+2} \cdot a(p_1,p_2,\ldots,p_{N+1}) + a(p_1,p_2,\ldots,p_N)
\end{align*}
for all $N,p_1,p_2,\ldots,p_{N+2} \in \mathbb{N}$. By Proposition~\ref{prop:fun_con}, we conclude that $a(p_1,p_2,\ldots,p_N)$ is precisely the numerator of the lowest form of the continued fraction $[1; p_1, p_2, \ldots, p_N]$. For general $n$, apply Theorem~\ref{thm:monomial_count}, then thinking of $a(p_1,p_2,\ldots,p_N)$ as the numerator of the lowest form of the continued fraction $[1; p_1, p_2, \ldots, p_N]$, we can deduce that $a_l(p_1,p_2,\ldots,p_N) = a(p_1,p_2,\ldots,p_N)^{l(n+1-l)}$. Therefore, counting multiplicities when counting terms, we obtain the following corollary:
\begin{Cor}
\label{cor:diagonal-terms}
    For any $n$-triangulated polygon $\mathcal{P}(p_1,p_2,\ldots,p_N)$, the general formula $D^{[p_1,p_2,\ldots,p_N]}_l$ ($l = 1,2,\ldots,n$) of the diagonal has precisely $a(p_1,p_2,\ldots,p_N)^{l(n+1-l)}$ terms, where $a(p_1,p_2,\ldots,p_N)$ is the numerator of the lowest form of the continued fraction $[1; p_1, p_2, \ldots, p_N]$.
\end{Cor}
Furthermore, applying Theorem~\ref{thm:well_triangulated_preservation}, as any triangle in Figure~\ref{fig:inner3} is well-triangulated, we get the following consequence:

\begin{Cor}
\label{cor:inner-terms}
    For any $n$-triangulated polygon $\mathcal{P}(p_1,p_2,\ldots,p_N)$, the general formula $I^{[p_1,p_2,\ldots,p_N]}_{(i,j)}$ ($i, j \in \mathbb{N}$, $i+j \leq n$) has precisely $a(p_1,p_2,\ldots,p_N)^{i(n+1-i-j)}a(p_1,p_2,\ldots,p_{N-1})^{j(n+1-i-j)}$ terms, where $a(p_1,p_2,\ldots,p_N)$ is the numerator of the lowest form of the continued fraction $[1; p_1, p_2, \ldots, p_N]$.
\end{Cor}
Returning to Theorem~\ref{thm:1tri} and Theorem~\ref{thm:2tri}, we see that the functions can be computed inductively by the linear combination of the functions that are computed for smaller cases.
\begin{Ex}We shall return to the example in Figure~\ref{fig:coord}. Recall the polygon $\mathcal{P}(2,3,3,1,3,0)$, or equivalently, $\mathcal{P}(2,3,3,1,3)$, the target is to calculate the specific values $D^{[2,3,3,1,3]}_i$ with $i = 1,2,\ldots,n$ for both cases $n = 1$ and $n = 2$. We shall apply Theorem~\ref{thm:1tri} and Theorem~\ref{thm:2tri} above. Indeed:
\begin{enumerate}[label=(\arabic*)]
    \item For $n = 1$, we only need to compute $D^{[2,3,3,1,3]}$. Applying Theorem~\ref{thm:1tri}, we get:
{\small    \begin{align*}
        D^{[2]} &= \frac{x^{(1)}_{(1,1,0)}x^{(3)}_{(0,1,1)}}{x^{(2)(3)}_{(1,1,0)}} + \frac{x^{(1)}_{(1,1,0)}x^{(2)}_{(0,1,1)}x^{(3)(4)}_{(1,0,1)}}{x^{(1)(2)}_{(1,0,1)}x^{(2)(3)}_{(1,1,0)}} + \frac{x^{(1)}_{(0,1,1)}x^{(3)(4)}_{(1,0,1)}}{x^{(1)(2)}_{(1,0,1)}}; \\
        D^{[2,3]} &= \frac{x^{(1)}_{(1,1,0)}x^{(6)(7)}_{(1,1,0)}}{x^{(3)(4)}_{(0,1,1)}} + \left[\frac{x^{(6)}_{(1,0,1)}}{x^{(5)(6)}_{(1,1,0)}} + \frac{x^{(5)}_{(1,0,1)}x^{(6)(7)}_{(1,1,0)}}{x^{(4)(5)}_{(1,1,0)}x^{(5)(6)}_{(0,1,1)}} + \frac{x^{(4)}_{(1,0,1)}x^{(6)(7)}_{(1,1,0)}}{x^{(3)(4)}_{(0,1,1)}x^{(4)(5)}_{(1,1,0)}}\right] \cdot D^{[2]}; \\
        D^{[2,3,3]} &= \frac{x^{(9)(10)}_{(1,0,1)}}{x^{(6)(7)}_{(1,1,0)}} \cdot D^{[2]} + x^{(9)(10)}_{(1,0,1)} \cdot \left[\frac{x^{(9)}_{(0,1,1)}}{x^{(9)(10)}_{(1,0,1)}x^{(8)(9)}_{(1,1,0)}} + \frac{x^{(8)}_{(0,1,1)}}{x^{(8)(9)}_{(1,1,0)}x^{(7)(8)}_{(1,0,1)}} + \frac{x^{(7)}_{(0,1,1)}}{x^{(7)(8)}_{(1,0,1)}x^{(6)(7)}_{(1,1,0)}}\right] \cdot D^{[2,3]}; \\
        D^{[2,3,3,1]} &= \frac{x^{(10)(11)}_{(0,1,1)}}{x^{(9)(10)}_{(1,0,1)}} \cdot D^{[2,3]} + \frac{x^{(10)}_{(1,1,0)}}{x^{(9)(10)}_{(1,0,1)}} \cdot D^{[2,3,3]}; \\
        D^{[2,3,3,1,3]} &= \frac{x^{(13)(14)}_{(1,1,0)}}{x^{(10)(11)}_{(0,1,1)}} \cdot D^{[2,3,3]} + x^{(13)(14)}_{(1,1,0)} \cdot \left[\frac{x^{(13)}_{(1,0,1)}}{x^{(13)(14)}_{(1,1,0)}x^{(12)(13)}_{(0,1,1)}} + \frac{x^{(12)}_{(1,0,1)}}{x^{(12)(13)}_{(0,1,1)}x^{(11)(12)}_{(1,1,0)}} + \frac{x^{(11)}_{(1,0,1)}}{x^{(11)(12)}_{(1,1,0)}x^{(10)(11)}_{(0,1,1)}}\right] \cdot D^{[2,3,3,1]}
    \end{align*}
    }
whence we can directly imply the formula. As the formula is long, we shall omit the detail result. For example $D^{[2,3,3,1,3]}$ consists of $162$ Laurent monomials (up to multiplicities) with every coefficient being $1$.

    \item For $n = 2$, we need to compute $D^{[2,3,3,1,3]}_1$, $D^{[2,3,3,1,3]}_2$, and $I^{[2,3,3,1,3]}$. We shall only discuss how to apply Theorem~\ref{thm:2tri} since the directed computation is very large and inappropriate to present in details. The theorem expresses $D^{[2,3,3,1,3]}_2$ and $I^{[2,3,3,1,3]}$ as linear combinations of $D^{[2,3,3]}_2$, $D^{[2,3,3,1]}_2$, and all $I^{[2,3,3,1,t]}$ for $t = 0,1,2$ (note $I^{[2,3,3,1,0]} = I^{[2,3,3,1]}$). This application is applied by the congruences $(-1)^2 + (-1)^3 + (-1)^3 + (-1)^1 \equiv 1 \pmod{3}$ and $2+3+3+1 = 9 \equiv 1 \pmod{2}$.

    Iterating this process, we express these quantities in terms of $D^{[2,3]}_2$, $D^{[2,3,3]}_2$, $I^{[2,3,3]}$, and $I^{[2,3,3,1]}$. Continuing recursively (noting that the applicable cases must be verified at each step) eventually yields closed forms for $D^{[2,3,3,1,3]}_2$ and $I^{[2,3,3,1,3]}$.

    For $D^{[2,3,3,1,3]}_1$, we consider the reflected polygon $\overline{\mathcal{P}}(3,1,3,3,2) = \mathcal{P}(2,3,3,1,3)$. After re-enumerating vertices for the reflected polygon, we apply Theorem~\ref{thm:2tri} to $D^{[3,1,3,3,2]}_2$, and recover the original variable assignments to obtain the closed form for $D^{[2,3,3,1,3]}_1$. Each $D^{[2,3,3,1,3]}_1$ and $D^{[2,3,3,1,3]}_2$ consists of $162^2 = 26,244$ Laurent monomials (up to multiplicities) with every coefficient being $1$, while the same holds for $I^{[2,3,3,1,3]}$ which consists of $6,966$ Laurent monomials (up to multiplicities).
\end{enumerate}
\end{Ex}

\begin{Ex}
Consider the following $3$-triangulated $5$-gon:
\begin{figure}[H]
\centering
\scalebox{0.43}{%
\definecolor{ffxfqq}{rgb}{1,0.4980392156862745,0}
\definecolor{qqqqff}{rgb}{0,0,1}
\definecolor{wwwwww}{rgb}{0.4,0.4,0.4}
\begin{tikzpicture}[line cap=round,line join=round,>=triangle 45,x=1cm,y=1cm]
\clip(-15.530702428433898,-1.426937934016519) rectangle (3.6906060060912185,13.686436789352616);
\draw [line width=2pt] (-6.75164,12.705785607532832)-- (-0.07181320220448306,7.85260735922452);
\draw [line width=2pt] (-0.07181320220448306,7.85260735922452)-- (-2.62328,0);
\draw [line width=2pt] (-2.62328,0)-- (-10.88,0);
\draw [line width=2pt] (-10.88,0)-- (-13.431466797795515,7.852607359224523);
\draw [line width=2pt] (-13.431466797795515,7.852607359224523)-- (-6.75164,12.705785607532832);
\draw [line width=2pt] (-6.75164,12.705785607532832)-- (-10.88,0);
\draw [line width=2pt] (-6.75164,12.705785607532832)-- (-2.62328,0);
\draw [line width=2pt,dash pattern=on 1pt off 1pt on 1pt off 4pt] (-13.431466797795515,7.852607359224523)-- (-0.07181320220448306,7.85260735922452);
\draw [line width=2pt,dash pattern=on 1pt off 1pt on 1pt off 4pt,color=ffxfqq] (-13.431466797795515,7.852607359224523)-- (-2.62328,0);
\draw [line width=2pt,dash pattern=on 1pt off 1pt on 1pt off 4pt,color=ffxfqq] (-0.07181320220448306,7.85260735922452)-- (-10.88,0);
\begin{scriptsize}
\draw [fill=wwwwww] (-10.88,0) circle (2pt);
\draw[color=wwwwww] (-11.497955354070253,0.23472220041804344) node {\huge $D$}; 
\draw [fill=wwwwww] (-2.62328,0) circle (2pt);
\draw[color=wwwwww] (-2.366207884310831,0.29204220568998706) node {\huge $C$}; 
\draw [fill=wwwwww] (-0.07181320220448306,7.85260735922452) circle (2pt);
\draw[color=wwwwww] (0.14123233711080379,8.450589622729963) node {\huge $B$};
\draw [fill=wwwwww] (-6.75164,12.705785607532832) circle (2pt);
\draw[color=wwwwww] (-6.760741621826513,13.322790070845172) node {\huge $A$};
\draw [fill=wwwwww] (-13.431466797795515,7.852607359224523) circle (2pt);
\draw[color=wwwwww] (-13.872888933427625,8.374162949034038) node {\huge $E$};
\draw [fill=qqqqff] (-1.7417699016533619,9.065901921301599) circle (3.5pt);
\draw[color=qqqqff] (-1.525514473655657,9.530736407136704) node {\Large $y_3$};
\draw [fill=qqqqff] (-5.08168330055112,11.492491045455754) circle (3.5pt);
\draw[color=qqqqff] (-4.869181447852371,11.957283296982318) node {\Large $y_1$};
\draw [fill=qqqqff] (-3.4117266011022407,10.279196483378676) circle (3.5pt);
\draw[color=qqqqff] (-3.206901294966005,10.753563186271501) node {\Large $y_2$};
\draw [fill=qqqqff] (-3.65537,3.176446401883208) circle (3.5pt);
\draw[color=qqqqff] (-3.4935013213257236,3.6267758641265107) node {\Large $x_{6}$};
\draw [fill=qqqqff] (-5.71955,9.529339205649624) circle (3.5pt);
\draw[color=qqqqff] (-5.5188081742677335,9.989296449312253) node {\Large $x_{4}$};
\draw [fill=qqqqff] (-4.68746,6.352892803766416) circle (3.5pt);
\draw[color=qqqqff] (-4.52526141622071,6.77937615408341) node {\Large $x_{5}$};
\draw [fill=qqqqff] (-4.04959330055112,8.316044643572546) circle (3.5pt);
\draw[color=qqqqff] (-3.7801013476854424,8.785576338601437) node {\Large $x_{11}$};
\draw [fill=qqqqff] (-7.78373,9.529339205649624) circle (3.5pt);
\draw[color=qqqqff] (-8.059995074657238,10.027509786160216) node {\Large $x_{1}$};
\draw [fill=qqqqff] (-8.81582,6.352892803766416) circle (3.5pt);
\draw[color=qqqqff] (-9.091755169552224,6.8558028277793355) node {\Large $x_{2}$};
\draw [fill=qqqqff] (-9.84791,3.1764464018832075) circle (3.5pt);
\draw[color=qqqqff] (-10.14262193287119,3.664989200974473) node {\Large $x_{3}$};
\draw [fill=qqqqff] (-10.091553398897757,10.279196483378676) circle (3.5pt);
\draw[color=qqqqff] (-10.295475280263041,10.868203196815388) node {\Large $y_{14}$};
\draw [fill=qqqqff] (-11.761510098346637,9.065901921301599) circle (3.5pt);
\draw[color=qqqqff] (-11.976862101573388,9.62626974925661) node {\Large $y_{13}$};
\draw [fill=qqqqff] (-9.453686699448877,8.316044643572546) circle (3.5pt);
\draw[color=qqqqff] (-9.244608516944073,8.785576338601437) node {\Large $x_{9}$};
\draw [fill=qqqqff] (-8.421596699448878,11.492491045455754) circle (3.5pt);
\draw[color=qqqqff] (-8.594981790528712,12.071923307526205) node {\Large $y_{15}$};
\draw [fill=qqqqff] (-3.0175033005511196,5.139598241689338) circle (3.5pt);
\draw[color=qqqqff] (-2.8056612580623996,5.594762711796576) node {\Large $x_{12}$};
\draw [fill=qqqqff] (-2.3796366011022396,7.1027500814954685) circle (3.5pt);
\draw[color=qqqqff] (-2.1178211947990753,7.562749559466639) node {\Large $x_{10}$};
\draw [fill=qqqqff] (-1.3475466011022394,3.9263036796122606) circle (3.5pt);
\draw[color=qqqqff] (-0.8669544314801076,4.076402590541871) node {\Large $y_{5}$}; 
\draw [fill=qqqqff] (-0.7096799016533595,5.889455519418391) circle (3.5pt);
\draw[color=qqqqff] (-0.2173277050647457,5.9870694329399925) node {\Large $y_{4}$}; 
\draw [fill=qqqqff] (-1.9854133005511194,1.9631518398061303) circle (3.5pt);
\draw[color=qqqqff] (-1.4592611526235257,2.0510957375998637) node {\Large $y_{6}$}; 
\draw [fill=qqqqff] (-11.123643398897755,7.1027500814954685) circle (3.5pt);
\draw[color=qqqqff] (-10.90688866983044,7.562749559466639) node {\Large $x_{7}$};
\draw [fill=qqqqff] (-12.793600098346634,5.889455519418391) circle (3.5pt);
\draw[color=qqqqff] (-13.361475543860226,5.996002748700181) node {\Large $y_{12}$};
\draw [fill=qqqqff] (-10.485776699448873,5.139598241689338) circle (3.5pt);
\draw[color=qqqqff] (-10.811355327710533,5.594762711796576) node {\Large $x_{8}$};
\draw [fill=qqqqff] (-12.155733398897748,3.9263036796122606) circle (3.5pt);
\draw[color=qqqqff] (-12.711848817444863,4.047122569454097) node {\Large $y_{11}$};
\draw [fill=qqqqff] (-11.517866699448863,1.9631518398061303) circle (3.5pt);
\draw[color=qqqqff] (-12.119542096301445,2.117349058631995) node {\Large $y_{10}$};
\draw [fill=qqqqff] (-6.75164,6.352892803766416) circle (3.5pt);
\draw[color=qqqqff] (-6.493248263890776,6.817589490931373) node {\Large $x_{13}$};
\draw [fill=qqqqff] (-7.78373,3.176446401883208) circle (3.5pt);
\draw[color=qqqqff] (-8.155528416777143,3.664989200974473) node {\Large $x_{14}$};
\draw [fill=qqqqff] (-8.81582,0) circle (3.5pt);
\draw[color=qqqqff] (-8.862475148464448,-0.4417044732778814) node {\Large $y_{9}$}; 
\draw [fill=qqqqff] (-5.71955,3.1764464018832084) circle (3.5pt);
\draw[color=qqqqff] (-5.5188081742677335,3.6267758641265107) node {\Large $x_{15}$};
\draw [fill=qqqqff] (-4.68746,0) circle (3.5pt);
\draw[color=qqqqff] (-4.67811476361256,-0.4225978048539002) node {\Large $y_{7}$}; 
\draw [fill=qqqqff] (-6.75164,0) circle (3.5pt);
\draw[color=qqqqff] (-6.798954958674476,-0.4608111417018626) node {\Large $y_{8}$}; 
\end{scriptsize}
\end{tikzpicture}%
}
\caption{\label{pic:poy_3_5}$m = 5$, $n = 3$}
\end{figure}
In order to reach the diagonal $BE$, we can use any flip sequence from Figure~\ref{fig:both_flip}. Then we can compute:
\begin{gather*}
	V_2V_4=(x'_{10}, x''_5, x'_{14}); \quad V_3V_5=(x'_{15}, x''_2, x'_7); \quad V_2V_5=(X_{2,5}^1, X_{2,5}^2, X_{2,5}^3);
\end{gather*}
{\small
\begin{align*}
	x'_{10} =& \frac{y_6y_{10}}{x_6}+\frac{y_7x_1}{x_4}+\frac{y_4y_7x_{13}}{x_4x_{11}}+\frac{y_6y_9x_{15}}{x_6x_{12}}+\frac{y_5y_7x_{15}}{x_5x_{10}}+\frac{y_6y_8x_{13}}{x_5x_{10}}+\frac{y_6y_8x_{11}x_{15}}{x_5x_{10}x_{12}}\ +\frac{y_5y_7x_{12}x_{13}}{x_5x_{10}x_{11}};\\
	x''_5 =& \frac{y_5y_{11}}{x_5}+\frac{y_8x_2}{x_5}+\frac{y_5y_9x_{14}}{x_6x_{11}}+\frac{y_4y_8x_{14}}{x_4x_{12}}+\frac{y_{11}x_1x_{10}}{x_4x_{15}}+\frac{y_{10}x_2x_{10}}{x_6x_{13}}+\frac{y_5y_{11}x_{12}x_{13}}{x_5x_{11}x_{15}}+\frac{y_4y_{11}x_{10}x_{13}}{x_4x_{11}x_{15}}+\frac{y_5y_{10}x_{12}x_{14}}{x_6x_{11}x_{15}}\\
	&+\frac{y_8x_2x_{11}x_{15}}{x_5x_{12}x_{13}}+\frac{y_8x_1x_{11}x_{14}}{x_4x_{12}x_{13}}+\frac{y_9x_2x_{10}x_{15}}{x_6x_{12}x_{13}}\ +\frac{y_4y_{10}x_5x_{10}x_{14}}{x_4x_6x_{11}x_{15}}+\frac{y_4y_9x_5x_{10}x_{14}}{x_4x_6x_{11}x_{12}}+\frac{y_{10}x_1x_5x_{10}x_{14}}{x_4x_6x_{13}x_{15}}+\frac{y_9x_1x_5x_{10}x_{14}}{x_4x_6x_{12}x_{13}};\\
	x'_{14} =& \frac{y_4y_{12}}{x_4}+\frac{y_9x_3}{x_6}+\frac{y_{12}x_1x_{11}}{x_4x_{13}}+\frac{y_{10}x_3x_{12}}{x_6x_{15}}+\frac{y_{11}x_3x_{11}}{x_5x_{14}}+\frac{y_{12}x_2x_{12}}{x_5x_{14}}+\frac{y_{11}x_3x_{12}x_{13}}{x_5x_{14}x_{15}}+\frac{y_{12}x_2x_{11}x_{15}}{x_5x_{13}x_{14}};
\end{align*}

\begin{align*}
	x'_{15} =&\frac{y_3y_{10}}{x_1}+\frac{y_{13}x_6}{x_3}+\frac{y_{10}x_4x_9}{x_1x_{13}}+\frac{y_{12}x_6x_8}{x_3x_{14}}+\frac{y_{10}x_5x_8}{x_2x_{15}}+\frac{y_{11}x_6x_9}{x_2x_{15}}+\frac{y_{11}x_6x_8x_{13}}{x_2x_{14}x_{15}}+\frac{y_{10}x_5x_9x_{14}}{x_2x_{13}x_{15}};\\
	x''_2 =&\frac{y_2y_{11}}{x_2}+\frac{y_{14}x_5}{x_2}+\frac{y_3y_{14}x_{15}}{x_1x_8}+\frac{y_2y_{13}x_{15}}{x_3x_9}+\frac{y_{12}x_5x_7}{x_3x_{13}}+\frac{y_{11}x_4x_7}{x_1x_{14}}+\frac{y_2y_{11}x_8x_{13}}{x_2x_9x_{14}}+\frac{y_3y_{11}x_7x_{13}}{x_1x_9x_{14}}+\frac{y_2y_{12}x_8x_{15}}{x_3x_9x_{14}}\\
	&+\frac{y_{14}x_5x_9x_{14}}{x_2x_8x_{13}}+\frac{y_{14}x_4x_9x_{15}}{x_1x_8x_{13}}+\frac{y_{13}x_5x_7x_{14}}{x_3x_8x_{13}}+\frac{y_3y_{13}x_2x_7x_{15}}{x_1x_3x_8x_9}+\frac{y_3y_{12}x_2x_7x_{15}}{x_1x_3x_9x_{14}}+\frac{y_{13}x_2x_4x_7x_{15}}{x_1x_3x_8x_{13}}+\frac{y_{12}x_2x_4x_7x_{15}}{x_1x_3x_{13}x_{14}};\\
	x'_7 =&\frac{y_1y_{12}}{x_3}+\frac{y_{15}x_4}{x_1}+\frac{y_3y_{15}x_{13}}{x_1x_9}+\frac{y_1y_{13}x_{14}}{x_3x_8}+\frac{y_1y_{14}x_{13}}{x_2x_7}+\frac{y_2y_{15}x_{14}}{x_2x_7}+\frac{y_1y_{14}x_9x_{14}}{x_2x_7x_8}+\frac{y_2y_{15}x_8x_{13}}{x_2x_7x_9};
\end{align*}

\begin{align*}
X_{2,5}^1=&  \frac{y_3y_7}{x_4} + \frac{y_6y_{13}}{x_3} + \frac{y_3y_6y_{10}}{x_1x_6} + \frac{y_6y_{12}x_8}{x_3x_{14}} + \frac{y_6y_{11}x_9}{x_2x_{15}}+ \frac{y_4y_7x_9}{x_1x_{11}} + \frac{y_5y_7x_8}{x_2x_{10}} + \frac{y_3y_4y_7x_{13}}{x_1x_4x_{11}} + \frac{y_3y_6y_9x_{15}}{x_1x_6x_{12}}\\
            & + \frac{y_3y_6y_8x_{13}}{x_1x_5x_{10}} + \frac{y_6y_{10}x_4x_9}{x_1x_6x_{13}} + \frac{y_6y_{10}x_5x_8}{x_2x_6x_{15}}+ \frac{y_6y_{11}x_8x_{13}}{x_2x_{14}x_{15}} + \frac{y_6y_9x_5x_8}{x_2x_6x_{12}} + \frac{y_5y_7x_9x_{14}}{x_2x_{10}x_{13}} + \frac{y_6y_8x_4x_9}{x_1x_5x_{10}} \\
            &+ \frac{y_3y_5y_7x_{15}}{x_1x_5x_{10}} + \frac{y_6y_8x_8x_{11}}{x_2x_{10}x_{12}} + \frac{y_3y_6y_8x_{11}x_{15}}{x_1x_5x_{10}x_{12}} + \frac{y_3y_5y_7x_{12}x_{13}}{x_1x_5x_{10}x_{11}} + \frac{y_6y_{10}x_5x_9x_{14}}{x_2x_6x_{13}x_{15}} + \frac{y_6y_9x_5x_9x_{14}}{x_2x_6x_{12}x_{13}}\\
            & + \frac{y_6y_9x_4x_9x_{15}}{x_1x_6x_{12}x_{13}} + \frac{y_5y_7x_4x_9x_{15}}{x_1x_5x_{10}x_{13}} + \frac{y_6y_8x_9x_{11}x_{14}}{x_2x_{10}x_{12}x_{13}} + \frac{y_5y_7x_4x_9x_{12}}{x_1x_5x_{10}x_{11}} + \frac{y_6y_8x_4x_9x_{11}x_{15}}{x_1x_5x_{10}x_{12}x_{13}};
\end{align*}

\begin{align*}
X_{2,5}^2 =& \frac{y_5y_{14}}{x_2} + \frac{y_2y_8}{x_5} + \frac{y_2y_5y_{11}}{x_2x_5} + \frac{y_5y_{12}x_7}{x_3x_{13}} + \frac{y_2y_{10}x_{10}}{x_6x_{13}} + \frac{y_3y_{14}x_{10}}{x_4x_8} + \frac{y_4y_8x_7}{x_1x_{12}} + \frac{y_2y_5y_9x_{14}}{x_2x_6x_{11}} + \frac{y_2y_5y_{13}x_{15}}{x_3x_5x_9} \\
            & + \frac{y_2y_4y_8x_{14}}{x_2x_4x_{12}} + \frac{y_3y_5y_{14}x_{15}}{x_1x_5x_8} + \frac{y_2y_{11}x_1x_{10}}{x_2x_4x_{15}} + \frac{y_5y_{11}x_4x_7}{x_1x_5x_{14}} + \frac{y_2y_{13}x_1x_{10}}{x_3x_4x_9} + \frac{y_5y_9x_4x_7}{x_1x_6x_{11}} + \frac{y_2y_9x_{10}x_{15}}{x_6x_{12}x_{13}} \\
            & + \frac{y_5y_{13}x_7x_{14}}{x_3x_8x_{13}} + \frac{y_2y_8x_{11}x_{15}}{x_5x_{12}x_{13}} + \frac{y_5y_{14}x_9x_{14}}{x_2x_8x_{13}} + \frac{y_3y_8x_7x_{11}}{x_4x_9x_{12}} + \frac{y_4y_{14}x_9x_{10}}{x_1x_8x_{11}} + \frac{y_2y_5y_{10}x_{12}x_{14}}{x_2x_6x_{11}x_{15}} + \frac{y_2y_5y_{12}x_8x_{15}}{x_3x_5x_9x_{14}} \\
            & + \frac{y_2y_5y_{11}x_8x_{13}}{x_2x_5x_9x_{14}} + \frac{y_2y_5y_{11}x_{12}x_{13}}{x_2x_5x_{11}x_{15}} + \frac{y_2y_5y_9x_8x_{13}}{x_2x_6x_9x_{11}} + \frac{y_2y_5y_{13}x_{12}x_{13}}{x_3x_5x_9x_{11}} + \frac{y_2y_4y_8x_8x_{13}}{x_2x_4x_9x_{12}} + \frac{y_3y_5y_{14}x_{12}x_{13}}{x_1x_5x_8x_{11}} \\
            & + \frac{y_2y_4y_{11}x_{10}x_{13}}{x_2x_4x_{11}x_{15}} + \frac{y_3y_5y_{11}x_7x_{13}}{x_1x_5x_9x_{14}} + \frac{y_2y_4y_{13}x_{10}x_{13}}{x_3x_4x_9x_{11}} + \frac{y_3y_5y_9x_7x_{13}}{x_1x_6x_9x_{11}} + \frac{y_3y_4y_8x_7x_{13}}{x_1x_4x_9x_{12}} + \frac{y_3y_4y_{14}x_{10}x_{13}}{x_1x_4x_8x_{11}} \\
            & + \frac{y_2y_8x_1x_8x_{11}}{x_2x_4x_9x_{12}} + \frac{y_5y_{14}x_4x_9x_{12}}{x_1x_5x_8x_{11}} + \frac{y_2y_{12}x_1x_8x_{10}}{x_3x_4x_9x_{14}} + \frac{y_5y_{10}x_4x_7x_{12}}{x_1x_6x_{11}x_{15}} + \frac{y_2y_8x_1x_{11}x_{14}}{x_2x_4x_{12}x_{13}} + \frac{y_5y_{14}x_4x_9x_{15}}{x_1x_5x_8x_{13}} \\
            & + \frac{y_3y_9x_5x_7x_{10}}{x_4x_6x_9x_{12}} + \frac{y_4y_{13}x_2x_7x_{10}}{x_1x_3x_8x_{11}} + \frac{y_3y_{10}x_5x_7x_{10}}{x_4x_6x_9x_{15}} + \frac{y_4y_{12}x_2x_7x_{10}}{x_1x_3x_{11}x_{14}} + \frac{y_3y_{11}x_7x_{10}x_{13}}{x_4x_9x_{14}x_{15}} + \frac{y_4y_{11}x_7x_{10}x_{13}}{x_1x_{11}x_{14}x_{15}} \\
            & + \frac{y_3y_{12}x_2x_7x_{10}}{x_3x_4x_9x_{14}} + \frac{y_4y_{10}x_5x_7x_{10}}{x_1x_6x_{11}x_{15}} + \frac{y_3y_{13}x_2x_7x_{10}}{x_3x_4x_8x_9} + \frac{y_4y_9x_5x_7x_{10}}{x_1x_6x_{11}x_{12}} + \frac{y_2y_5y_{10}x_8x_{12}x_{13}}{x_2x_6x_9x_{11}x_{15}} + \frac{y_2y_5y_{12}x_8x_{12}x_{13}}{x_3x_5x_9x_{11}x_{14}} \\
            & + \frac{y_2y_4y_9x_5x_{10}x_{14}}{x_2x_4x_6x_{11}x_{12}} + \frac{y_3y_5y_{13}x_2x_7x_{15}}{x_1x_3x_5x_8x_9} + \frac{y_2y_4y_{12}x_8x_{10}x_{13}}{x_3x_4x_9x_{11}x_{14}} + \frac{y_3y_5y_{10}x_7x_{12}x_{13}}{x_1x_6x_9x_{11}x_{15}} + \frac{y_2y_4y_{10}x_5x_{10}x_{14}}{x_2x_4x_6x_{11}x_{15}} \\
            &+ \frac{y_3y_5y_{12}x_2x_7x_{15}}{x_1x_3x_5x_9x_{14}} + \frac{y_2y_9x_1x_5x_8x_{10}}{x_2x_4x_6x_9x_{12}} + \frac{y_5y_{13}x_2x_4x_7x_{12}}{x_1x_3x_5x_8x_{11}} + \frac{y_2y_9x_1x_5x_{10}x_{14}}{x_2x_4x_6x_{12}x_{13}} + \frac{y_5y_{13}x_2x_4x_7x_{15}}{x_1x_3x_5x_8x_{13}} \\
            &+ \frac{y_2y_{10}x_1x_5x_8x_{10}}{x_2x_4x_6x_9x_{15}} + \frac{y_5y_{12}x_2x_4x_7x_{12}}{x_1x_3x_5x_{11}x_{14}} + \frac{y_2y_{10}x_1x_5x_{10}x_{14}}{x_2x_4x_6x_{13}x_{15}} + \frac{y_5y_{12}x_2x_4x_7x_{15}}{x_1x_3x_5x_{13}x_{14}} + \frac{y_2y_{11}x_1x_8x_{10}x_{13}}{x_2x_4x_9x_{14}x_{15}} \\
            &+ \frac{y_5y_{11}x_4x_7x_{12}x_{13}}{x_1x_5x_{11}x_{14}x_{15}} + \frac{y_2y_4y_9x_5x_8x_{10}x_{13}}{x_2x_4x_6x_9x_{11}x_{12}} + \frac{y_3y_5y_{13}x_2x_7x_{12}x_{13}}{x_1x_3x_5x_8x_9x_{11}} + \frac{y_2y_4y_{10}x_5x_8x_{10}x_{13}}{x_2x_4x_6x_9x_{11}x_{15}}+ \frac{y_3y_5y_{12}x_2x_7x_{12}x_{13}}{x_1x_3x_5x_9x_{11}x_{14}} \\
            & + \frac{y_3y_4y_9x_5x_7x_{10}x_{13}}{x_1x_4x_6x_9x_{11}x_{12}} + \frac{y_3y_4y_{13}x_2x_7x_{10}x_{13}}{x_1x_3x_4x_8x_9x_{11}} + \frac{y_3y_4y_{10}x_5x_7x_{10}x_{13}}{x_1x_4x_6x_9x_{11}x_{15}} + \frac{y_3y_4y_{12}x_2x_7x_{10}x_{13}}{x_1x_3x_4x_9x_{11}x_{14}} \\
            & + \frac{y_2y_5y_{11}x_8x_{12}x_{13}^2}{x_2x_5x_9x_{11}x_{14}x_{15}} + \frac{y_2y_4y_{11}x_8x_{10}x_{13}^2}{x_2x_4x_9x_{11}x_{14}x_{15}} + \frac{y_3y_5y_{11}x_7x_{12}x_{13}^2}{x_1x_5x_9x_{11}x_{14}x_{15}} + \frac{y_3y_4y_{11}x_7x_{10}x_{13}^2}{x_1x_4x_9x_{11}x_{14}x_{15}};
\end{align*}

\begin{align*}
	X_{2,5}^3 =&\frac{y_4y_{15}}{x_1}+\frac{y_1y_9}{x_6}+\frac{y_1y_4y_{12}}{x_3x_4}+\frac{y_1y_{10}x_{12}}{x_6x_{15}}+\frac{y_1y_{11}x_{11}}{x_5x_{14}}+\frac{y_3y_{15}x_{11}}{x_4x_9}+\frac{y_2y_{15}x_{12}}{x_5x_7}+\frac{y_3y_4y_{15}x_{13}}{x_1x_4x_9}+\frac{y_1y_4y_{13}x_{14}}{x_3x_4x_8}\\
	&+\frac{y_1y_4y_{14}x_{13}}{x_2x_4x_7}+\frac{y_2y_4y_{15}x_{14}}{x_2x_4x_7}+\frac{y_1y_{12}x_1x_{11}}{x_3x_4x_{13}}+\frac{y_1y_{12}x_2x_{12}}{x_3x_5x_{14}}+\frac{y_1y_{11}x_{12}x_{13}}{x_5x_{14}x_{15}}+\frac{y_1y_{13}x_2x_{12}}{x_3x_5x_8}+\frac{y_2y_{15}x_{11}x_{15}}{x_5x_7x_{13}}\\
	&+\frac{y_1y_{14}x_1x_{11}}{x_2x_4x_7}+\frac{y_1y_{14}x_9x_{12}}{x_5x_7x_8}+\frac{y_1y_4y_{14}x_9x_{14}}{x_2x_4x_7x_8}+\frac{y_2y_4y_{15}x_8x_{13}}{x_2x_4x_7x_9}+\frac{y_1y_{12}x_2x_{11}x_{15}}{x_3x_5x_{13}x_{14}}+\frac{y_1y_{13}x_2x_{11}x_{15}}{x_3x_5x_8x_{13}}\\
	&+\frac{y_1y_{13}x_1x_{11}x_{14}}{x_3x_4x_8x_{13}}+\frac{y_2y_{15}x_1x_{11}x_{14}}{x_2x_4x_7x_{13}}+\frac{y_1y_{14}x_9x_{11}x_{15}}{x_5x_7x_8x_{13}}+\frac{y_2y_{15}x_1x_8x_{11}}{x_2x_4x_7x_9}+\frac{y_1y_{14}x_1x_9x_{11}x_{14}}{x_2x_4x_7x_8x_{13}}.
\end{align*}
}
This result verifies that the number of monomials in the Laurent polynomial in each coordinate of the expansion formula of diagonal $V_2V_5$ is $(27,81,27) = (3^3, 3^4, 3^3)$, which satisfies the Number of Monomials Theorem~\ref{thm:monomial_count} since in the base case (cluster realization of type $A_1$), the number of monomials (counting multiplicities) of diagonal $BE$ is $3$. Clearly, each coefficient of every component is $1$.
\end{Ex}
We shall end the section with a conjecture and the hope about finding general recursive formulas for general $n$. Based on the observation above, we have that the conjecture holds for $n = 1,2$.
\begin{Con}[Multiplicity free] Fix a non-$\cT$ diagonal $\gamma = AB$ in a general $n$-triangulated $m$-gon with $\cT$-triangulation of type $\mathcal{P}(m_1,n_1,\ldots,m_k,n_k)$. Then for $l = 1,2,...,n$ and $i,j \in \mathbb{N}$ such that $i+j \leq n$, every function $D^{[m_1,n_1,\ldots,m_k,n_k]}_l$ and $I^{[m_1,n_1,\ldots,m_k,n_k]}_{(i,j)}$ is a Laurent polynomial of variables $x^{(r)}_{(a,b,c)}$ where $r = 1,2,...,(m-2)$ and $(a,b,c) \in \Gamma_{n+1}$ with each coefficient being $1$.
\end{Con}

\section{General expansion formulas of cluster realization of type $G_2$ in $4$-gon}
\label{sec:g2-cluster-realization}
 In this section, we are going to illustrate how the basic quiver is associated with a surface in the cluster realization where $G$ is of type $G_2$~\cite{Ip}, and compute the cluster expansion formula related to the flip of the diagonal over a quadrilateral.
 \subsection{Cluster realization of type $G_2$ local systems}
\begin{Def}
\label{def:g2-quiver}
For any triangle of a triangulation of a surface $\bbS$, the \emph{type $G_2$ basic quiver} can be drawn as either one of the two ways below (Figure~\ref{fig:g2-quiver-triangle-21} and Figure~\ref{fig:g2-quiver-triangle-12}), corresponding to the longest reduced expression being $w_0=s_2s_1s_2s_1s_2s_1$ or $w_0=s_1s_2s_1s_2s_1s_2$ respectively, where the index $1$ belongs to the long root while $2$ belongs to the short root.
\begin{figure}[H]
\centering
\begin{tikzpicture}[scale=1.5,
    mid arrow/.style={
        postaction={decorate},
        decoration={
            markings,
            mark=at position 0.5 with {\arrow{>}}
        }
    }
]
  \coordinate (S1) at (0,1);
  \coordinate (S2) at (1,1);
  \coordinate (S3) at (2,1);
  \coordinate (S4) at (3,1);
  \coordinate (S5) at (1,0);
  
  \coordinate (L1) at (0,2);
  \coordinate (L2) at (1,2);
  \coordinate (L3) at (2,2);
  \coordinate (L4) at (3,2);
  \coordinate (L5) at (2,3);
  
  \draw[mid arrow, line width=2pt, dash pattern=on 8pt off 4pt, red] (L1) -- (S1);
  \draw[mid arrow, line width=2pt, dash pattern=on 8pt off 4pt, red] (L4) -- (S4);
  \draw[mid arrow, line width=2pt, dash pattern=on 8pt off 4pt, red] (S5) -- (L5);
  
  \draw[mid arrow, line width=2pt, black] (L1) -- (L2);
  \draw[mid arrow, line width=2pt, black] (L2) -- (L3);
  \draw[mid arrow, line width=2pt, black] (L3) -- (L4);
  \draw[mid arrow, line width=2pt, black] (L4) -- (L5);
  \draw[mid arrow, line width=2pt, black] (L5) -- (L3);
  \draw[mid arrow, line width=2pt, black] (S4) -- (L3);
  \draw[mid arrow, line width=2pt, black] (L3) -- (S3);
  \draw[mid arrow, line width=2pt, black] (S3) -- (L2);
  \draw[mid arrow, line width=2pt, black] (L2) -- (S2);
  \draw[mid arrow, line width=2pt, black] (S2) -- (L1);
  
  \draw[mid arrow, black] (S2) -- (S5);
  \draw[mid arrow, black] (S5) -- (S1);
  \draw[mid arrow, black] (S1) -- (S2);
  \draw[mid arrow, black] (S2) -- (S3);
  \draw[mid arrow, black] (S3) -- (S4);

  \filldraw[black] (L1) circle (2pt);
  \filldraw[black] (L2) circle (2pt);
  \filldraw[black] (L3) circle (2pt);
  \filldraw[black] (L4) circle (2pt);
  \filldraw[black] (L5) circle (2pt);
  
  \filldraw[blue] (S1) circle (2pt);
  \filldraw[blue] (S2) circle (2pt);
  \filldraw[blue] (S3) circle (2pt);
  \filldraw[blue] (S4) circle (2pt);
  \filldraw[blue] (S5) circle (2pt);
  
  \node[left, font=\scriptsize] at (L1) {$L_1$};
  \node[above, font=\scriptsize] at (L2) {$L_2$};
  \node[above right, font=\scriptsize] at (L3) {$L_3$};
  \node[right, font=\scriptsize] at (L4) {$L_4$};
  \node[right, font=\scriptsize] at (L5) {$L_5$};
  \node[left, font=\scriptsize] at (S1) {$S_1$};
  \node[below left, font=\scriptsize] at (S2) {$S_2$};
  \node[below, font=\scriptsize] at (S3) {$S_3$};
  \node[below, font=\scriptsize] at (S4) {$S_4$};
  \node[right, font=\scriptsize] at (S5) {$S_5$};
  
  \draw[thick] (4,0) -- (6,3) -- (8,0) -- cycle;
  
  \coordinate (X1) at (6,1.6);
  \coordinate (X4) at (5.5,1.2);
  \coordinate (X2) at (6.5,1.2);
  \coordinate (X3) at (6,0.8);
  
  \draw[mid arrow, line width=2pt, dash pattern=on 8pt off 4pt, red] (6.667,2) -- (7.333,1);
  \draw[mid arrow, line width=2pt, dash pattern=on 8pt off 4pt, red] (5.333,0) -- (6.667,0);
  \draw[mid arrow, line width=2pt, dash pattern=on 8pt off 4pt, red] (4.667,1) -- (5.333,2);
  
  \draw[mid arrow, line width=2pt, black] (5.333,2) -- (X1);
  \draw[mid arrow, line width=2pt, black] (X1) -- (6.667,2);
  \draw[mid arrow, line width=2pt, black] (X1) -- (X2);
  \draw[mid arrow, line width=2pt, black] (6.667,2) -- (6.333-1,2);
  \draw[mid arrow, line width=2pt, black] (8.333-1,1) -- (X1);
  \draw[mid arrow, line width=2pt, black] (X4) -- (X1);
  \draw[mid arrow, line width=2pt, black] (X2) -- (X4);
  \draw[mid arrow, line width=2pt, black] (X4) -- (X3);
  \draw[mid arrow, line width=2pt, black] (6.333-1,0) -- (X4);
  \draw[mid arrow, line width=2pt, black] (X3) -- (6.333-1,0);
  
  \draw[mid arrow, black] (X2) -- (8.333-1,1);
  \draw[mid arrow, black] (5.667-1,1) -- (7.667-1,0);
  \draw[mid arrow, black] (7.667-1,0) -- (X3);
  \draw[mid arrow, black] (X3) -- (5.667-1,1);
  \draw[mid arrow, black] (X3) -- (X2);
  
  \filldraw[blue] (5.667-1,1) circle (2pt);
  \filldraw[black] (6.333-1,2) circle (2pt);
  \filldraw[black] (6.333-1,0) circle (2pt);
  \filldraw[blue] (7.667-1,0) circle (2pt);
  \filldraw[black] (7.667-1,2) circle (2pt);
  \filldraw[blue] (8.333-1,1) circle (2pt);
  
  \filldraw[black] (X1) circle (2pt);
  \filldraw[blue] (X2) circle (2pt);
  \filldraw[blue] (X3) circle (2pt);
  \filldraw[black] (X4) circle (2pt);
  
  \node[left, font=\scriptsize] at (5.667-1,1) {$S_5$};
  \node[left, font=\scriptsize] at (6.333-1,2) {$L_5$};
  \node[above right, font=\scriptsize] at (7.667-1,0) {$S_1$};
  \node[right, font=\scriptsize] at (X3) {$S_2$};
  \node[below, font=\scriptsize] at (X2) {$S_3$};
  \node[right, font=\scriptsize] at (8.333-1,1) {$S_4$};
  \node[above left, font=\scriptsize] at (6.333-1,0) {$L_1$};
  \node[left, font=\scriptsize] at (X4) {$L_2$};
  \node[above, font=\scriptsize] at (X1) {$L_3$};
  \node[right, font=\scriptsize] at (7.667-1,2) {$L_4$};
\end{tikzpicture}
\caption{The type $G_2$ quiver on each triangle for reduced word $w_0 = s_2s_1s_2s_1s_2s_1$}
\label{fig:g2-quiver-triangle-21}
\end{figure}
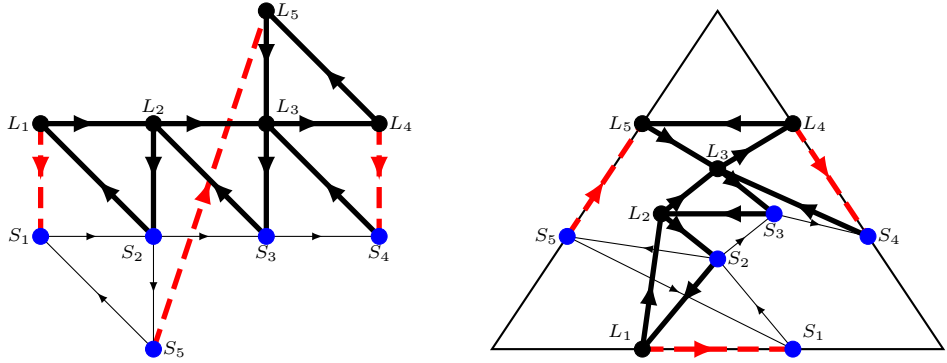

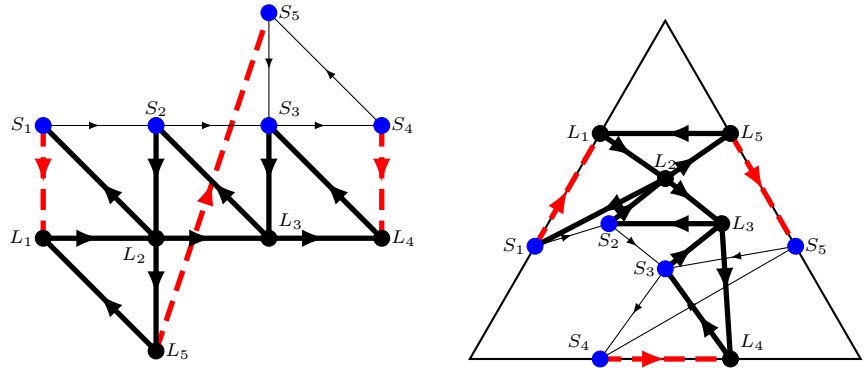
\begin{figure}[H]
\centering
\begin{tikzpicture}[scale=1.5,
    mid arrow/.style={
        postaction={decorate},
        decoration={
            markings,
            mark=at position 0.5 with {\arrow{>}}
        }
    }
]
  \coordinate (L1) at (0,0+1);
  \coordinate (L2) at (1,0+1);
  \coordinate (L3) at (2,0+1);
  \coordinate (L4) at (3,0+1);
  \coordinate (L5) at (1,-1+1);
  
  \coordinate (S1) at (0,1+1);
  \coordinate (S2) at (1,1+1);
  \coordinate (S3) at (2,1+1);
  \coordinate (S4) at (3,1+1);
  \coordinate (S5) at (2,2+1);
  
  \draw[mid arrow, line width=2pt, dash pattern=on 8pt off 4pt, red] (S1) -- (L1);
  \draw[mid arrow, line width=2pt, dash pattern=on 8pt off 4pt, red] (S4) -- (L4);
  \draw[mid arrow, line width=2pt, dash pattern=on 8pt off 4pt, red] (L5) -- (S5);
  
  \draw[mid arrow, line width=2pt, black] (L2) -- (L5);
  \draw[mid arrow, line width=2pt, black] (L5) -- (L1);
  \draw[mid arrow, line width=2pt, black] (L1) -- (L2);
  \draw[mid arrow, line width=2pt, black] (L2) -- (L3);
  \draw[mid arrow, line width=2pt, black] (L3) -- (L4);
  \draw[mid arrow, line width=2pt, black] (L4) -- (S3);
  \draw[mid arrow, line width=2pt, black] (S3) -- (L3);
  \draw[mid arrow, line width=2pt, black] (L3) -- (S2);
  \draw[mid arrow, line width=2pt, black] (S2) -- (L2);
  \draw[mid arrow, line width=2pt, black] (L2) -- (S1);
  
  \draw[mid arrow, black] (S1) -- (S2);
  \draw[mid arrow, black] (S2) -- (S3);
  \draw[mid arrow, black] (S3) -- (S4);
  \draw[mid arrow, black] (S4) -- (S5);
  \draw[mid arrow, black] (S5) -- (S3);

  \filldraw[black] (L1) circle (2pt);
  \filldraw[black] (L2) circle (2pt);
  \filldraw[black] (L3) circle (2pt);
  \filldraw[black] (L4) circle (2pt);
  \filldraw[black] (L5) circle (2pt);
  
  \filldraw[blue] (S1) circle (2pt);
  \filldraw[blue] (S2) circle (2pt);
  \filldraw[blue] (S3) circle (2pt);
  \filldraw[blue] (S4) circle (2pt);
  \filldraw[blue] (S5) circle (2pt);
  
  \node[left, font=\scriptsize] at (L1) {$L_1$};
  \node[below left, font=\scriptsize] at (L2) {$L_2$};
  \node[above right, font=\scriptsize] at (L3) {$L_3$};
  \node[right, font=\scriptsize] at (L4) {$L_4$};
  \node[right, font=\scriptsize] at (L5) {$L_5$};
  \node[left, font=\scriptsize] at (S1) {$S_1$};
  \node[above, font=\scriptsize] at (S2) {$S_2$};
  \node[above right, font=\scriptsize] at (S3) {$S_3$};
  \node[right, font=\scriptsize] at (S4) {$S_4$};
  \node[right, font=\scriptsize] at (S5) {$S_5$};
  \end{tikzpicture}\tab
 \begin{tikzpicture}[scale=1.5,
    mid arrow/.style={
        postaction={decorate},
        decoration={
            markings,
            mark=at position 0.5 with {\arrow{>}}
        }
    }
]
  \pgfmathsetmacro{\sqrtThree}{sqrt(3)}
  
  \draw[thick] ({6-\sqrtThree},0) -- (6,3) -- ({6+\sqrtThree},0) -- cycle;
  
  \coordinate (C) at ({6-\sqrtThree},0);
  \coordinate (A) at (6,3);
  \coordinate (B) at ({6+\sqrtThree},0);  
  
  \coordinate (x2) at ({6-2*\sqrtThree/3},1);
  \coordinate (x1) at ({6-\sqrtThree/3},2);
  
  \coordinate (y3) at ({6+\sqrtThree/3},2);
  \coordinate (y4) at ({6+2*\sqrtThree/3},1);
  \coordinate (y5) at ({6+\sqrtThree/3},0);
  \coordinate (y6) at ({6-\sqrtThree/3},0);
  
  \coordinate (x6) at (6,1.6);
  \coordinate (x5) at (5.5,1.2);
  \coordinate (x3) at (6.5,1.2);
  \coordinate (x4) at (6,0.8);
  
  \draw[mid arrow, line width=2pt, dash pattern=on 8pt off 4pt, red] (y3) -- (y4);
  \draw[mid arrow, line width=2pt, dash pattern=on 8pt off 4pt, red] (y6) -- (y5);
  \draw[mid arrow, line width=2pt, dash pattern=on 8pt off 4pt, red] (x2) -- (x1);
  
  \draw[mid arrow, line width=2pt, black] (x1) -- (x6);
  \draw[mid arrow, line width=2pt, black] (x6) -- (y3);
  \draw[mid arrow, line width=2pt, black] (y3) -- (x1);
  \draw[mid arrow, line width=2pt, black] (x3) -- (y5);
  \draw[mid arrow, line width=2pt, black] (y5) -- (x4);
  \draw[mid arrow, line width=2pt, black] (x4) -- (x3);
  \draw[mid arrow, line width=2pt, black] (x3) -- (x5);
  \draw[mid arrow, line width=2pt, black] (x5) -- (x6);
  \draw[mid arrow, line width=2pt, black] (x6) -- (x3);
  \draw[mid arrow, line width=2pt, black] (x6) -- (x2);
  
  \draw[mid arrow, black] (y6) -- (y4);
  \draw[mid arrow, black] (y4) -- (x4);
  \draw[mid arrow, black] (x4) -- (y6);
  \draw[mid arrow, black] (x2) -- (x5);
  \draw[mid arrow, black] (x5) -- (x4);
  
  \filldraw[black] (x1) circle (2pt);
  \filldraw[blue] (x2) circle (2pt);
  
  \filldraw[black] (x3) circle (2pt);
  \filldraw[blue] (x4) circle (2pt);
  \filldraw[blue] (x5) circle (2pt);
  \filldraw[black] (x6) circle (2pt);
  
  \filldraw[black] (y3) circle (2pt);
  \filldraw[blue] (y4) circle (2pt);
  \filldraw[black] (y5) circle (2pt);
  \filldraw[blue] (y6) circle (2pt);
  
  \node[left, font=\scriptsize] at (x2) {$S_1$};
  \node[left, font=\scriptsize] at (x1) {$L_1$};
  \node[above right, font=\scriptsize] at (y5) {$L_4$};
  \node[left, font=\scriptsize] at (x4) {$S_3$};
  \node[right, font=\scriptsize] at (x3) {$L_3$};
  \node[right, font=\scriptsize] at (y4) {$S_5$};
  \node[above left, font=\scriptsize] at (y6) {$S_4$};
  \node[below, font=\scriptsize] at (x5) {$S_2$};
  \node[above, font=\scriptsize] at (x6) {$L_2$};
  \node[right, font=\scriptsize] at (y3) {$L_5$};
\end{tikzpicture}
\caption{The type $G_2$ quiver on each triangle for reduced word $w_0 = s_1s_2s_1s_2s_1s_2$}
\label{fig:g2-quiver-triangle-12}
\end{figure}
Here for the arrows:
\begin{enumerate}
	\item[(i)] Red dotted arrow (dashed, thick) refers to half of an undashed, thick arrow;
	\item[(ii)] Undashed, thin arrow refers to $\frac{1}{3}$ of an undashed, thick arrow.
\end{enumerate}
Moreover, for the vertices (ignoring the color on the figure above):
\begin{enumerate}
	\item[(i)] Black vertex refers to the vertex corresponding to the long root;
	\item[(ii)] Blue vertex refers to the vertex corresponding to the short root.
\end{enumerate}
\end{Def}

\begin{Def}
\label{def:g2-flip-sequences}
The \emph{flip sequence of type $G_2$} (i.e. the sequence of mutations to get the quiver after the flip) can be applied in two ways, where the sequence can be applied in the following order (below, by $... \rightarrow a \rightarrow b \rightarrow ...$ we mean mutating at vertex $a$ first, then $b$, and so on):
\begin{enumerate}[label=(\arabic*)]
	\item The flips sequence $\mu_1$:
		\begin{gather*}
			\mu_1 = \{ x_1 \rightarrow x_7 \rightarrow x_8 \rightarrow x_9 \rightarrow x_7 \rightarrow x_2 \rightarrow x_9 \rightarrow x_7 \rightarrow x_2 \rightarrow x_8 \rightarrow x_7 \rightarrow x_{10} 
			\\ \rightarrow x_6 \rightarrow x_7 \rightarrow x_8 \rightarrow x_2 \rightarrow x_7 \rightarrow x_5 \rightarrow x_2 \rightarrow x_7 \rightarrow x_5 \rightarrow x_8 \rightarrow x_7 \rightarrow x_9 
			\\ \rightarrow x_3 \rightarrow x_7 \rightarrow x_8 \rightarrow x_5 \rightarrow x_7 \rightarrow x_4 \rightarrow x_5 \rightarrow x_7 \rightarrow x_4 \rightarrow x_8 \rightarrow x_7 \rightarrow x_2 \}
		\end{gather*}
		where the long roots (satisfying $d_i = 1$) are $\{x_1, x_3, x_6, x_7, x_8, y_2, y_3, y_5, y_8\}$ and short roots (satisfying $d_i = \frac{1}{3}$) are $\{x_2, x_4, x_5, x_9, x_{10}, y_1, y_4, y_6, y_7\}$.
\begin{figure}[H]
\centering
\begin{tikzpicture}[scale=1.35,
    mid arrow/.style={
        postaction={decorate},
        decoration={
            markings,
            mark=at position 0.5 with {\arrow{>}}
        }
    }
]
  \pgfmathsetmacro{\sqrtThree}{sqrt(3)}

  \draw[-{Triangle[blue,scale=2,line width=1.5pt]}] (8.1,1.5) -- (9.1,1.5);
  \node[above] at (8.6,1.5) {$\mu_1$};
  
  \begin{scope}[shift={(\sqrtThree, 3)}]  
    \begin{scope}[shift={(-2*\sqrtThree, 0)}]  
      \begin{scope}[rotate around={-60:({6-\sqrtThree},0)}]  
      
        \draw[thick] ({6-\sqrtThree},0) -- (6,3) -- ({6+\sqrtThree},0) -- cycle;
        
        \coordinate (D) at ({6-\sqrtThree},0);
        \coordinate (A) at (6,3);
        \coordinate (C) at ({6+\sqrtThree},0);  
        \coordinate (x2) at ({6+2*\sqrtThree/3},1);
        \coordinate (x1) at ({6+\sqrtThree/3},2);
        \coordinate (y1) at ({6-2*\sqrtThree/3},1);
        \coordinate (y2) at ({6-\sqrtThree/3},2);
        \coordinate (y7) at ({6+\sqrtThree/3},0);
        \coordinate (y8) at ({6-\sqrtThree/3},0); 
        \coordinate (x8) at (6,1.6);
        \coordinate (x7) at (5.5,1.2);
        \coordinate (x9) at (6.5,1.2);
        \coordinate (x10) at (6,0.8);  
        
        \draw[mid arrow, line width=2pt, dash pattern=on 8pt off 4pt, red] (y8) -- (y7);
        \draw[mid arrow, line width=2pt, dash pattern=on 8pt off 4pt, red] (y1) -- (y2);
        
        \draw[mid arrow, line width=2pt, black] (x1) -- (y2);
        \draw[mid arrow, line width=2pt, black] (y2) -- (x8);
        \draw[mid arrow, line width=2pt, black] (x8) -- (x1);
        \draw[mid arrow, line width=2pt, black] (x8) -- (x9);
        \draw[mid arrow, line width=2pt, black] (x9) -- (x7);
        \draw[mid arrow, line width=2pt, black] (x7) -- (x8);
        \draw[mid arrow, line width=2pt, black] (x10) -- (y8);
        \draw[mid arrow, line width=2pt, black] (y8) -- (x7);
        \draw[mid arrow, line width=2pt, black] (x7) -- (x10);
        \draw[mid arrow, line width=2pt, black] (x2) -- (x8);
        
        \draw[mid arrow, black] (y1) -- (y7);
        \draw[mid arrow, black] (y7) -- (x10);
        \draw[mid arrow, black] (x10) -- (y1);
        \draw[mid arrow, black] (x10) -- (x9);
        \draw[mid arrow, black] (x9) -- (x2);  
        
        \filldraw[black] (x1) circle (2pt);
        \filldraw[blue] (x2) circle (2pt);
        \filldraw[black] (x7) circle (2pt);
        \filldraw[black] (x8) circle (2pt);
        \filldraw[blue] (x9) circle (2pt);
        \filldraw[blue] (x10) circle (2pt);
        
        \filldraw[blue] (y1) circle (2pt);
        \filldraw[black] (y2) circle (2pt);
        \filldraw[blue] (y7) circle (2pt);
        \filldraw[black] (y8) circle (2pt);
        
        \node[above, font=\scriptsize] at (y1) {$y_1$};
        \node[above, font=\scriptsize] at (y2) {$y_2$};
        \node[left, font=\scriptsize] at (y7) {$y_7$};
        \node[below, font=\scriptsize] at (x10) {$x_{10}$};
        \node[below left, font=\scriptsize] at (x9) {$x_9$};
        \node[left, font=\scriptsize] at (x2) {$x_2$};
        \node[left, font=\scriptsize] at (y8) {$y_8$};
        \node[above, font=\scriptsize] at (x7) {$x_7$};
        \node[above right, font=\scriptsize] at (x8) {$x_8$};
        \node[above right, font=\scriptsize] at (x1) {$x_1$};
      
      \end{scope}
    \end{scope}
  \end{scope}
  
  \draw[thick] (6,3) -- ({6+\sqrtThree},0) -- ({6-\sqrtThree},0);
  
  \coordinate (C) at ({6-\sqrtThree},0);
  \coordinate (A) at (6,3);
  \coordinate (B) at ({6+\sqrtThree},0);  
  
  \coordinate (y3) at ({6+\sqrtThree/3},2);
  \coordinate (y4) at ({6+2*\sqrtThree/3},1);
  \coordinate (y5) at ({6+\sqrtThree/3},0);
  \coordinate (y6) at ({6-\sqrtThree/3},0);
  
  \coordinate (x6) at (6,1.6);
  \coordinate (x5) at (5.5,1.2);
  \coordinate (x3) at (6.5,1.2);
  \coordinate (x4) at (6,0.8);
  
  \draw[mid arrow, line width=2pt, dash pattern=on 8pt off 4pt, red] (y3) -- (y4);
  \draw[mid arrow, line width=2pt, dash pattern=on 8pt off 4pt, red] (y6) -- (y5);
  
  \draw[mid arrow, line width=2pt, black] (x1) -- (x6);
  \draw[mid arrow, line width=2pt, black] (x6) -- (y3);
  \draw[mid arrow, line width=2pt, black] (y3) -- (x1);
  \draw[mid arrow, line width=2pt, black] (x3) -- (y5);
  \draw[mid arrow, line width=2pt, black] (y5) -- (x4);
  \draw[mid arrow, line width=2pt, black] (x4) -- (x3);
  \draw[mid arrow, line width=2pt, black] (x3) -- (x5);
  \draw[mid arrow, line width=2pt, black] (x5) -- (x6);
  \draw[mid arrow, line width=2pt, black] (x6) -- (x3);
  \draw[mid arrow, line width=2pt, black] (x6) -- (x2);
  
  \draw[mid arrow, black] (y6) -- (y4);
  \draw[mid arrow, black] (y4) -- (x4);
  \draw[mid arrow, black] (x4) -- (y6);
  \draw[mid arrow, black] (x2) -- (x5);
  \draw[mid arrow, black] (x5) -- (x4);
  
  
  \filldraw[black] (x3) circle (2pt);
  \filldraw[blue] (x4) circle (2pt);
  \filldraw[blue] (x5) circle (2pt);
  \filldraw[black] (x6) circle (2pt);
  
  \filldraw[black] (y3) circle (2pt);
  \filldraw[blue] (y4) circle (2pt);
  \filldraw[black] (y5) circle (2pt);
  \filldraw[blue] (y6) circle (2pt);

  \node[above right, font=\scriptsize] at (y5) {$y_5$};
  \node[left, font=\scriptsize] at (x4) {$x_4$};
  \node[right, font=\scriptsize] at (x3) {$x_3$};
  \node[right, font=\scriptsize] at (y4) {$y_4$};
  \node[above left, font=\scriptsize] at (y6) {$y_6$};
  \node[below, font=\scriptsize] at (x5) {$x_5$};
  \node[above, font=\scriptsize] at (x6) {$x_6$};
  \node[right, font=\scriptsize] at (y3) {$y_3$};
  
  \node[above right] at (A) {$A$};
  \node[below right] at (B) {$B$};
  \node[below left] at (C) {$C$};
  \node[above left] at (D) {$D$};

  \begin{scope}[shift={(+8.5, 0)}]
  \begin{scope}[shift={(\sqrtThree, 3)}]  
    \begin{scope}[shift={(-2*\sqrtThree, 0)}]  
      \begin{scope}[rotate around={-60:({6-\sqrtThree},0)}]  
      
        \draw[thick] ({6-\sqrtThree},0) -- (6,3) -- ({6+\sqrtThree},0) -- cycle;
        
        \coordinate (D') at ({6-\sqrtThree},0);
        \coordinate (A') at (6,3);
        \coordinate (B') at ({6+\sqrtThree},0);  
        \coordinate (y'4) at ({6+2*\sqrtThree/3},1);
        \coordinate (y'3) at ({6+\sqrtThree/3},2);
        \coordinate (y'1) at ({6-2*\sqrtThree/3},1);
        \coordinate (y'2) at ({6-\sqrtThree/3},2);
        \coordinate (x'3) at ({6+\sqrtThree/3},0);
        \coordinate (x'2) at ({6-\sqrtThree/3},0); 
        \coordinate (x'1) at (6,1.6);
        \coordinate (x'10) at (5.5,1.2);
        \coordinate (x'6) at (6.5,1.2);
        \coordinate (x'9) at (6,0.8);  
        
        \draw[mid arrow, line width=2pt, dash pattern=on 8pt off 4pt, red] (y'1) -- (y'2);
        \draw[mid arrow, line width=2pt, dash pattern=on 8pt off 4pt, red] (y'3) -- (y'4);
        
        \draw[mid arrow, line width=2pt, black] (x'2) -- (x'3);
        \draw[mid arrow, line width=2pt, black] (y'2) -- (x'1);
        \draw[mid arrow, line width=2pt, black] (x'1) -- (y'3);
        \draw[mid arrow, line width=2pt, black] (y'3) -- (y'2);
        \draw[mid arrow, line width=2pt, black] (x'1) -- (x'6);
        \draw[mid arrow, line width=2pt, black] (x'6) -- (x'10);
        \draw[mid arrow, line width=2pt, black] (x'10) -- (x'1);
        \draw[mid arrow, line width=2pt, black] (x'3) -- (x'9);
        \draw[mid arrow, line width=2pt, black] (x'9) -- (x'6);
        \draw[mid arrow, line width=2pt, black] (x'6) -- (x'3);
        \draw[mid arrow, line width=2pt, black] (x'1) -- (y'1);
        
        \draw[mid arrow, black] (x'9) -- (x'2);
        \draw[mid arrow, black] (x'2) -- (y'4);
        \draw[mid arrow, black] (y'4) -- (x'9);
        \draw[mid arrow, black] (y'1) -- (x'10);
        \draw[mid arrow, black] (x'10) -- (x'9);
        
        \filldraw[black] (x'1) circle (2pt);
        \filldraw[blue] (x'2) circle (2pt);
        \filldraw[black] (x'3) circle (2pt);
        \filldraw[black] (x'6) circle (2pt);
        \filldraw[blue] (x'9) circle (2pt);
        \filldraw[blue] (x'10) circle (2pt);
  
        \filldraw[blue] (y'1) circle (2pt);
        \filldraw[black] (y'2) circle (2pt);
        \filldraw[black] (y'3) circle (2pt);
        \filldraw[blue] (y'4) circle (2pt);
        
        \node[above, font=\scriptsize] at (y'1) {$y_1$};
        \node[above, font=\scriptsize] at (y'2) {$y_2$};
        \node[above right, font=\scriptsize] at (x'2) {$x_2^{(5)}$};
        \node[above left, font=\scriptsize] at (x'9) {$x_9^{(3)}$};
        \node[left, font=\scriptsize] at (x'10) {$x_{10}^{(1)}$};
        \node[right, font=\scriptsize] at (y'4) {$y_4$};
        \node[right, font=\scriptsize] at (x'3) {$x_3^{(1)}$};
        \node[right, font=\scriptsize] at (x'6) {$x_{6}^{(1)}$};
        \node[above right, font=\scriptsize] at (x'1) {$x_1^{(1)}$};
        \node[right, font=\scriptsize] at (y'3) {$y_3$};
      
      \end{scope}
    \end{scope}
  \end{scope}
  
  \begin{scope}[shift={(-2*\sqrtThree, 0)}]
\begin{scope}[rotate around={-120:(6,1)}]
  \draw[thick] ({6-\sqrtThree},0) -- (6,3) -- ({6+\sqrtThree},0) -- cycle;
  
  \coordinate (D') at ({6-\sqrtThree},0);
  \coordinate (B') at (6,3);
  \coordinate (C') at ({6+\sqrtThree},0);

  \coordinate (y'6) at ({6+2*\sqrtThree/3},1);
  \coordinate (y'5) at ({6+\sqrtThree/3},2);
  
  \coordinate (x'2) at ({6-2*\sqrtThree/3},1);
  \coordinate (x'3) at ({6-\sqrtThree/3},2);
  \coordinate (y'7) at ({6+\sqrtThree/3},0);
  \coordinate (y'8) at ({6-\sqrtThree/3},0); 
  
  \coordinate (x'8) at (6,1.6);
  \coordinate (x'7) at (5.5,1.2);
  \coordinate (x'4) at (6.5,1.2);
  \coordinate (x'5) at (6,0.8);  
  
  \draw[mid arrow, line width=2pt, dash pattern=on 8pt off 4pt, red] (y'5) -- (y'6);
  \draw[mid arrow, line width=2pt, dash pattern=on 8pt off 4pt, red] (y'8) -- (y'7);

  \draw[mid arrow, line width=2pt, black] (y'8) -- (x'7);
  \draw[mid arrow, line width=2pt, black] (x'7) -- (x'5);
  \draw[mid arrow, line width=2pt, black] (x'5) -- (y'8);
  \draw[mid arrow, line width=2pt, black] (y'5) -- (x'3);
  \draw[mid arrow, line width=2pt, black] (x'3) -- (x'8);
  \draw[mid arrow, line width=2pt, black] (x'8) -- (y'5);
  \draw[mid arrow, line width=2pt, black] (y'6) -- (x'8);
  \draw[mid arrow, line width=2pt, black] (x'8) -- (x'4);
  \draw[mid arrow, line width=2pt, black] (x'4) -- (x'7);
  \draw[mid arrow, line width=2pt, black] (x'7) -- (x'8);
  
  \draw[mid arrow, black] (x'2) -- (y'7);
  \draw[mid arrow, black] (y'7) -- (x'5);
  \draw[mid arrow, black] (x'5) -- (x'2);
  \draw[mid arrow, black] (x'5) -- (x'4);
  \draw[mid arrow, black] (x'4) -- (y'6);
  
  \filldraw[blue] (x'4) circle (2pt);
  \filldraw[blue] (x'5) circle (2pt);
  \filldraw[black] (x'7) circle (2pt);
  \filldraw[black] (x'8) circle (2pt);
  
  \filldraw[black] (y'5) circle (2pt);
  \filldraw[blue] (y'6) circle (2pt);
  \filldraw[blue] (y'7) circle (2pt);
  \filldraw[black] (y'8) circle (2pt);
  
  \node[left, font=\scriptsize] at (y'7) {$y_7$};
  \node[left, font=\scriptsize] at (y'8) {$y_8$};
  \node[below, font=\scriptsize] at (y'5) {$y_5$};
  \node[below left, font=\scriptsize] at (x'5) {$x_5^{(4)}$};
  \node[above, font=\scriptsize] at (x'7) {$x_7^{(12)}$};
  \node[below, font=\scriptsize] at (y'6) {$y_6$};
  \node[left, font=\scriptsize] at (x'4) {$x_4^{(2)}$};
  \node[right, font=\scriptsize] at (x'8) {$x_8^{(6)}$};
  
  \end{scope}
  \end{scope}
  \node[above right] at (A') {$A$};
  \node[below right] at (B') {$B$};
  \node[below left] at (C') {$C$};
  \node[above left] at (D') {$D$}; 
  \end{scope}
\end{tikzpicture}
\caption{Mutation sequence $\mu_1$ for type $G_2$}
\label{fig:g2-mutation-sequence-1}
\end{figure}
			
	\item The flips sequence $\mu_2$:
		\begin{gather*}
			\mu_2 = \{ x_1 \rightarrow x_2 \rightarrow x_7 \rightarrow x_8 \rightarrow x_{10} \rightarrow x_7 \rightarrow x_9 \rightarrow x_{10} \rightarrow x_7 \rightarrow x_9 \rightarrow x_8 \rightarrow x_7
			\\ \rightarrow x_6 \rightarrow x_5 \rightarrow x_7 \rightarrow x_8 \rightarrow x_9 \rightarrow x_7 \rightarrow x_2 \rightarrow x_9 \rightarrow x_7 \rightarrow x_2 \rightarrow x_8 \rightarrow x_7
			\\ \rightarrow x_3 \rightarrow x_4 \rightarrow x_7 \rightarrow x_8 \rightarrow x_2 \rightarrow x_7 \rightarrow x_5 \rightarrow x_2 \rightarrow x_7 \rightarrow x_5 \rightarrow x_8 \rightarrow x_7 \}
		\end{gather*}
		where the long roots (satisfying $d_i = 1$) are $\{x_1, x_3, x_6, x_7, x_8, y_2, y_3, y_5, y_7\}$ and short roots (satisfying $d_i = \frac{1}{3}$) are $\{x_2, x_4, x_5, x_9, x_{10}, y_1, y_4, y_6, y_8\}$.
\begin{figure}[H]
\centering
\begin{tikzpicture}[scale=1.35,
    mid arrow/.style={
        postaction={decorate},
        decoration={
            markings,
            mark=at position 0.5 with {\arrow{>}}
        }
    }
]
  \pgfmathsetmacro{\sqrtThree}{sqrt(3)}

  \draw[-{Triangle[blue,scale=2,line width=1.5pt]}] (8.1,1.5) -- (9.1,1.5);
  \node[above] at (8.6,1.5) {$\mu_2$};
  
  \begin{scope}[shift={(\sqrtThree, 3)}]  
    \begin{scope}[shift={(-2*\sqrtThree, 0)}]  
      \begin{scope}[rotate around={-60:({6-\sqrtThree},0)}]  
      
        \draw[thick] ({6-\sqrtThree},0) -- (6,3) -- ({6+\sqrtThree},0) -- cycle;
        
        \coordinate (D) at ({6-\sqrtThree},0);
        \coordinate (A) at (6,3);
        \coordinate (C) at ({6+\sqrtThree},0);  
        \coordinate (x2) at ({6+2*\sqrtThree/3},1);
        \coordinate (x1) at ({6+\sqrtThree/3},2);
        \coordinate (y1) at ({6-2*\sqrtThree/3},1);
        \coordinate (y2) at ({6-\sqrtThree/3},2);
        \coordinate (y7) at ({6+\sqrtThree/3},0);
        \coordinate (y8) at ({6-\sqrtThree/3},0); 
        \coordinate (x7) at (6,1.6);
        \coordinate (x10) at (5.5,1.2);
        \coordinate (x8) at (6.5,1.2);
        \coordinate (x9) at (6,0.8);  
        
        \draw[mid arrow, line width=2pt, dash pattern=on 8pt off 4pt, red] (y1) -- (y2);
        \draw[mid arrow, line width=2pt, dash pattern=on 8pt off 4pt, red] (y8) -- (y7);
        
        
        \draw[mid arrow, line width=2pt, black] (x1) -- (y2);
        \draw[mid arrow, line width=2pt, black] (y2) -- (x7);
        \draw[mid arrow, line width=2pt, black] (x7) -- (x1);
        \draw[mid arrow, line width=2pt, black] (x9) -- (x8);
        \draw[mid arrow, line width=2pt, black] (x8) -- (x10);
        \draw[mid arrow, line width=2pt, black] (x7) -- (x8);
        \draw[mid arrow, line width=2pt, black] (x7) -- (y1);
        \draw[mid arrow, line width=2pt, black] (y7) -- (x9);
        \draw[mid arrow, line width=2pt, black] (x10) -- (x7);
        \draw[mid arrow, line width=2pt, black] (x8) -- (y7);
        
        \draw[mid arrow, black] (y1) -- (x10);
        \draw[mid arrow, black] (x10) -- (x9);
        \draw[mid arrow, black] (y8) -- (x2);
        \draw[mid arrow, black] (x9) -- (y8);
        \draw[mid arrow, black] (x2) -- (x9);
        
        \filldraw[black] (x1) circle (2pt);
        \filldraw[blue] (x2) circle (2pt);
        \filldraw[black] (x7) circle (2pt);
        \filldraw[black] (x8) circle (2pt);
        \filldraw[blue] (x9) circle (2pt);
        \filldraw[blue] (x10) circle (2pt);
  
        \filldraw[blue] (y1) circle (2pt);
        \filldraw[black] (y2) circle (2pt);
        \filldraw[black] (y7) circle (2pt);
        \filldraw[blue] (y8) circle (2pt);
        
        \node[above, font=\scriptsize] at (y1) {$y_1$};
        \node[above, font=\scriptsize] at (y2) {$y_2$};
        \node[left, font=\scriptsize] at (y7) {$y_7$};
        \node[left, font=\scriptsize] at (x10) {$x_{10}$};
        \node[above left, font=\scriptsize] at (x9) {$x_9$};
        \node[left, font=\scriptsize] at (x2) {$x_2$};
        \node[left, font=\scriptsize] at (y8) {$y_8$};
        \node[above right, font=\scriptsize] at (x7) {$x_7$};
        \node[right, font=\scriptsize] at (x8) {$x_8$};
        \node[above right, font=\scriptsize] at (x1) {$x_1$};
      \end{scope}
    \end{scope}
  \end{scope}
  
  \draw[thick] (6,3) -- ({6+\sqrtThree},0) -- ({6-\sqrtThree},0);
  
  \coordinate (C) at ({6-\sqrtThree},0);
  \coordinate (A) at (6,3);
  \coordinate (B) at ({6+\sqrtThree},0);  
  
  \coordinate (y3) at ({6+\sqrtThree/3},2);
  \coordinate (y4) at ({6+2*\sqrtThree/3},1);
  \coordinate (y5) at ({6+\sqrtThree/3},0);
  \coordinate (y6) at ({6-\sqrtThree/3},0);
  
  \coordinate (x6) at (6,1.6);
  \coordinate (x5) at (5.5,1.2);
  \coordinate (x3) at (6.5,1.2);
  \coordinate (x4) at (6,0.8);
  
  \draw[mid arrow, line width=2pt, dash pattern=on 8pt off 4pt, red] (y3) -- (y4);
  \draw[mid arrow, line width=2pt, dash pattern=on 8pt off 4pt, red] (y6) -- (y5);
  
  \draw[mid arrow, line width=2pt, black] (x1) -- (x6);
  \draw[mid arrow, line width=2pt, black] (x6) -- (y3);
  \draw[mid arrow, line width=2pt, black] (y3) -- (x1);
  \draw[mid arrow, line width=2pt, black] (x3) -- (y5);
  \draw[mid arrow, line width=2pt, black] (y5) -- (x4);
  \draw[mid arrow, line width=2pt, black] (x4) -- (x3);
  \draw[mid arrow, line width=2pt, black] (x3) -- (x5);
  \draw[mid arrow, line width=2pt, black] (x5) -- (x6);
  \draw[mid arrow, line width=2pt, black] (x6) -- (x3);
  \draw[mid arrow, line width=2pt, black] (x6) -- (x2);
  
  \draw[mid arrow, black] (y6) -- (y4);
  \draw[mid arrow, black] (y4) -- (x4);
  \draw[mid arrow, black] (x4) -- (y6);
  \draw[mid arrow, black] (x2) -- (x5);
  \draw[mid arrow, black] (x5) -- (x4);
  
  
  \filldraw[black] (x3) circle (2pt);
  \filldraw[blue] (x4) circle (2pt);
  \filldraw[blue] (x5) circle (2pt);
  \filldraw[black] (x6) circle (2pt);
  
  \filldraw[black] (y3) circle (2pt);
  \filldraw[blue] (y4) circle (2pt);
  \filldraw[black] (y5) circle (2pt);
  \filldraw[blue] (y6) circle (2pt);

  \node[above right, font=\scriptsize] at (y5) {$y_5$};
  \node[left, font=\scriptsize] at (x4) {$x_4$};
  \node[right, font=\scriptsize] at (x3) {$x_3$};
  \node[right, font=\scriptsize] at (y4) {$y_4$};
  \node[above left, font=\scriptsize] at (y6) {$y_6$};
  \node[below, font=\scriptsize] at (x5) {$x_5$};
  \node[above, font=\scriptsize] at (x6) {$x_6$};
  \node[right, font=\scriptsize] at (y3) {$y_3$};
  
  \node[above right] at (A) {$A$};
  \node[below right] at (B) {$B$};
  \node[below left] at (C) {$C$};
  \node[above left] at (D) {$D$};

  \begin{scope}[shift={(+8.5, 0)}]
  \begin{scope}[shift={(\sqrtThree, 3)}]  
    \begin{scope}[shift={(-2*\sqrtThree, 0)}]  
      \begin{scope}[rotate around={-60:({6-\sqrtThree},0)}]  
      
        \draw[thick] ({6-\sqrtThree},0) -- (6,3) -- ({6+\sqrtThree},0) -- cycle;
        
        \coordinate (D') at ({6-\sqrtThree},0);
        \coordinate (A') at (6,3);
        \coordinate (B') at ({6+\sqrtThree},0);  
        \coordinate (y'4) at ({6+2*\sqrtThree/3},1);
        \coordinate (y'3) at ({6+\sqrtThree/3},2);
        \coordinate (y'1) at ({6-2*\sqrtThree/3},1);
        \coordinate (y'2) at ({6-\sqrtThree/3},2);
        \coordinate (x'3) at ({6+\sqrtThree/3},0);
        \coordinate (x'2) at ({6-\sqrtThree/3},0); 
        \coordinate (x'1) at (6,1.6);
        \coordinate (x'10) at (5.5,1.2);
        \coordinate (x'6) at (6.5,1.2);
        \coordinate (x'9) at (6,0.8);  
        
        \draw[mid arrow, line width=2pt, dash pattern=on 8pt off 4pt, red] (y'1) -- (y'2);
        \draw[mid arrow, line width=2pt, dash pattern=on 8pt off 4pt, red] (y'3) -- (y'4);
        
        \draw[mid arrow, line width=2pt, black] (x'2) -- (x'3);
        \draw[mid arrow, line width=2pt, black] (y'2) -- (x'1);
        \draw[mid arrow, line width=2pt, black] (x'1) -- (y'3);
        \draw[mid arrow, line width=2pt, black] (y'3) -- (y'2);
        \draw[mid arrow, line width=2pt, black] (x'1) -- (x'6);
        \draw[mid arrow, line width=2pt, black] (x'6) -- (x'10);
        \draw[mid arrow, line width=2pt, black] (x'10) -- (x'1);
        \draw[mid arrow, line width=2pt, black] (x'3) -- (x'9);
        \draw[mid arrow, line width=2pt, black] (x'9) -- (x'6);
        \draw[mid arrow, line width=2pt, black] (x'6) -- (x'3);
        \draw[mid arrow, line width=2pt, black] (x'1) -- (y'1);
        
        \draw[mid arrow, black] (x'9) -- (x'2);
        \draw[mid arrow, black] (x'2) -- (y'4);
        \draw[mid arrow, black] (y'4) -- (x'9);
        \draw[mid arrow, black] (y'1) -- (x'10);
        \draw[mid arrow, black] (x'10) -- (x'9);
        
        \filldraw[black] (x'1) circle (2pt);
        \filldraw[blue] (x'2) circle (2pt);
        \filldraw[black] (x'3) circle (2pt);
        \filldraw[black] (x'6) circle (2pt);
        \filldraw[blue] (x'9) circle (2pt);
        \filldraw[blue] (x'10) circle (2pt);
  
        \filldraw[blue] (y'1) circle (2pt);
        \filldraw[black] (y'2) circle (2pt);
        \filldraw[black] (y'3) circle (2pt);
        \filldraw[blue] (y'4) circle (2pt);
        
        \node[above, font=\scriptsize] at (y'1) {$y_1$};
        \node[above, font=\scriptsize] at (y'2) {$y_2$};
        \node[above right, font=\scriptsize] at (x'2) {$x_2^{(5)}$};
        \node[above left, font=\scriptsize] at (x'9) {$x_9^{(4)}$};
        \node[left, font=\scriptsize] at (x'10) {$x_{10}^{(1)}$};
        \node[right, font=\scriptsize] at (y'4) {$y_4$};
        \node[right, font=\scriptsize] at (x'3) {$x_3^{(1)}$};
        \node[right, font=\scriptsize] at (x'6) {$x_{6}^{(1)}$};
        \node[above right, font=\scriptsize] at (x'1) {$x_1^{(1)}$};
        \node[right, font=\scriptsize] at (y'3) {$y_3$};
      
      \end{scope}
    \end{scope}
  \end{scope}
  
  \begin{scope}[shift={(-2*\sqrtThree, 0)}]
  \begin{scope}[rotate around={-120:(6,1)}]
  \draw[thick] ({6-\sqrtThree},0) -- (6,3) -- ({6+\sqrtThree},0) -- cycle;
  
  \coordinate (D') at ({6-\sqrtThree},0);
  \coordinate (B') at (6,3);
  \coordinate (C') at ({6+\sqrtThree},0); 
  
  \coordinate (x'2) at ({6-2*\sqrtThree/3},1);
  \coordinate (x'3) at ({6-\sqrtThree/3},2);
  
  \coordinate (y'5) at ({6+\sqrtThree/3},2);
  \coordinate (y'6) at ({6+2*\sqrtThree/3},1);
  \coordinate (y'7) at ({6+\sqrtThree/3},0);
  \coordinate (y'8) at ({6-\sqrtThree/3},0);
  
  \coordinate (x'7) at (6,1.6);
  \coordinate (x'5) at (5.5,1.2);
  \coordinate (x'8) at (6.5,1.2);
  \coordinate (x'4) at (6,0.8);
  
  \draw[mid arrow, line width=2pt, dash pattern=on 8pt off 4pt, red] (y'5) -- (y'6);
  \draw[mid arrow, line width=2pt, dash pattern=on 8pt off 4pt, red] (y'8) -- (y'7);

  \draw[mid arrow, line width=2pt, black] (y'7) -- (x'4);
  \draw[mid arrow, line width=2pt, black] (x'5) -- (x'7);
  \draw[mid arrow, line width=2pt, black] (x'8) -- (x'5);
  \draw[mid arrow, line width=2pt, black] (y'5) -- (x'3);
  \draw[mid arrow, line width=2pt, black] (x'3) -- (x'7);
  \draw[mid arrow, line width=2pt, black] (x'7) -- (y'5);
  \draw[mid arrow, line width=2pt, black] (x'8) -- (y'7);
  \draw[mid arrow, line width=2pt, black] (x'4) -- (x'8);
  \draw[mid arrow, line width=2pt, black] (x'7) -- (x'2);
  \draw[mid arrow, line width=2pt, black] (x'7) -- (x'8);
  
  \draw[mid arrow, black] (x'2) -- (x'5);
  \draw[mid arrow, black] (x'5) -- (x'4);
  \draw[mid arrow, black] (y'8) -- (y'6);
  \draw[mid arrow, black] (x'4) -- (y'8);
  \draw[mid arrow, black] (y'6) -- (x'4);
  
  \filldraw[blue] (x'4) circle (2pt);
  \filldraw[blue] (x'5) circle (2pt);
  \filldraw[black] (x'7) circle (2pt);
  \filldraw[black] (x'8) circle (2pt);
  
  \filldraw[black] (y'5) circle (2pt);
  \filldraw[blue] (y'6) circle (2pt);
  \filldraw[black] (y'7) circle (2pt);
  \filldraw[blue] (y'8) circle (2pt);
  
  \node[left, font=\scriptsize] at (y'7) {$y_7$};
  \node[left, font=\scriptsize] at (y'8) {$y_8$};
  \node[below, font=\scriptsize] at (y'5) {$y_5$};
  \node[above left, font=\scriptsize] at (x'5) {$x_5^{(3)}$};
  \node[right, font=\scriptsize] at (x'7) {$x_7^{(12)}$};
  \node[below, font=\scriptsize] at (y'6) {$y_6$};
  \node[above, font=\scriptsize] at (x'4) {$x_4^{(2)}$};
  \node[below, font=\scriptsize] at (x'8) {$x_8^{(6)}$};
  \end{scope}
  \end{scope}
  \node[above right] at (A') {$A$};
  \node[below right] at (B') {$B$};
  \node[below left] at (C') {$C$};
  \node[above left] at (D') {$D$}; 
  \end{scope}
\end{tikzpicture}
\caption{Mutation sequence $\mu_2$ for type $G_2$}
\label{fig:g2-mutation-sequence-2}
\end{figure}
\end{enumerate} 
\end{Def}
By drawing the triangulation using equilateral triangles as above, an important observation from \cite{Ip} is that after applying the mutation sequence, the resulting quiver over each triangle is still the same as one of the two basic ones from Figure \ref{fig:g2-quiver-triangle-21} and \ref{fig:g2-quiver-triangle-12}, but up to a rotation.

\subsection{Detailed calculations}
\begin{Prop}
\label{prop:useful-calculations} 
In the quiver for type $G_2$ local system, we have the following properties:
	\begin{enumerate}
		\item[(i)] For any arrow of weight $w_{ij}$ connecting vertices $i$ and $j$ (of course $w_{ij} \geq 0$), then the coefficients satisfy $b_{ij} = \frac{w_{ij}}{d_i}$, $b_{ji} = -\frac{w_{ij}}{d_j}$.
		\item[(ii)] If the arrow is undashed and thick, then:
			\begin{gather*}
				b_{ij} = \begin{cases}
    1 & \text{if the vertex } i \text{ is black}\\
    3 & \text{if the vertex } i \text{ is blue}
\end{cases}; \quad b_{ji} = \begin{cases}
    -1 & \text{if the vertex } j \text{ is black}\\
    -3 & \text{if the vertex } j \text{ is blue}
\end{cases}.
			\end{gather*}
	\end{enumerate}

\end{Prop}
For the flip sequence $\mu_1$, we get the new expansion formulas of the vertices of the form:
\begin{gather*}
	(y_1,y_2,...,y_8,x_1,x_2,...,x_{10}) \rightarrow (y_1,y_2,...,y_8,x_1^{(1)},x_{2}^{(5)},x_3^{(1)},x_4^{(2)},x_{5}^{(4)},x_6^{(1)},x_{7}^{(12)},x_8^{(6)},x_9^{(3)},x_{10}^{(1)}).
\end{gather*}
We shall omit the results of remaining vertices whereas the formulas are too long. In the expressions below, each fraction is written in a canonical form with all \(y\)-variables (in increasing index) preceding all \(x\)-variables (also increasing index) in the numerator, and the terms are ordered by the number of distinct variables followed by their exponents.

We finally get the diagonal expansion formula after flip:
\begin{gather*}
	BD = (x_3^{(1)}, x_2^{(5)})
\end{gather*}
where:
{\footnotesize
\begin{align*}
    x_{2}^{(5)} &= \frac{y_{1} y_{6}}{x_{2}} + \frac{y_{4} y_{7}}{x_{2}} + \frac{2 y_{1} y_{4} x_{4} x_{10}}{x_{5} x_{9}} + \frac{y_{1} y_{4} y_{8} x_{5}}{x_{6} x_{10}} + \frac{y_{1} y_{4} y_{5} x_{9}}{x_{4} x_{8}} + \frac{y_{1} y_{4} x_{3} x_{10}^{2}}{x_{5}^{2} x_{7}} + \frac{y_{1} y_{4} x_{4}^{2} x_{7}}{x_{3} x_{9}^{2}} + \frac{y_{4} y_{7} x_{5} x_{7}}{x_{6} x_{9} x_{10}} + \frac{2 y_{4} y_{7} x_{4} x_{9}}{x_{2} x_{5} x_{10}} + \frac{y_{4} y_{7} x_{5} x_{8}}{x_{2} x_{6} x_{9}} \\
    &\quad + \frac{y_{1} y_{6} x_{6} x_{9}}{x_{2} x_{5} x_{8}} + \frac{2 y_{1} y_{6} x_{5} x_{10}}{x_{2} x_{4} x_{9}} + \frac{y_{1} y_{6} x_{3} x_{9}}{x_{4} x_{5} x_{8}} + \frac{y_{2} y_{6} y_{7} x_{6}}{x_{1} x_{5} x_{10}} + \frac{y_{3} y_{6} y_{7} x_{8}}{x_{1} x_{5} x_{10}} + \frac{y_{2} y_{6} y_{7} x_{6}}{x_{1} x_{4} x_{9}} + \frac{y_{3} y_{6} y_{7} x_{8}}{x_{1} x_{4} x_{9}} + \frac{2 y_{6} y_{7} x_{5} x_{9}}{x_{2}^{2} x_{4} x_{10}} + \frac{y_{1} y_{6} x_{5}^{2} x_{7}}{x_{4} x_{6} x_{9}^{2}} \\
    &\quad + \frac{y_{1} y_{6} x_{6} x_{10}^{2}}{x_{2}^{2} x_{4} x_{7}} + \frac{y_{4} y_{7} x_{4}^{2} x_{8}}{x_{2}^{2} x_{3} x_{10}} + \frac{y_{4} y_{7} x_{3} x_{9}^{2}}{x_{5}^{2} x_{8} x_{10}} + \frac{y_{3} y_{6} y_{7} x_{2} x_{7}}{x_{1} x_{4} x_{9} x_{10}} + \frac{y_{1} y_{3} y_{6} y_{8} x_{2}}{x_{1} x_{4} x_{10}} + \frac{y_{2} y_{4} y_{5} y_{7} x_{2}}{x_{1} x_{4} x_{10}} + \frac{y_{6} y_{7} x_{6} x_{9}^{2}}{x_{2}^{2} x_{4} x_{8} x_{10}} + \frac{y_{6} y_{7} x_{5}^{2} x_{8}}{x_{2}^{2} x_{4} x_{6} x_{10}} \\
    &\quad + \frac{2 y_{1} y_{4} x_{4} x_{6} x_{10}^{2}}{x_{2} x_{5}^{2} x_{7}} + \frac{2 y_{1} y_{4} x_{4}^{2} x_{8} x_{10}}{x_{2} x_{3} x_{9}^{2}} + \frac{y_{2} y_{4} y_{7} x_{2}^{2} x_{3}}{x_{1} x_{5}^{2} x_{9}} + \frac{y_{1} y_{3} y_{6} x_{2}^{2} x_{7}}{x_{1} x_{5} x_{9}^{2}} + \frac{y_{2} y_{6} y_{7} x_{2} x_{3}}{x_{1} x_{4} x_{5} x_{10}} + \frac{2 y_{4} y_{7} x_{4} x_{6} x_{9}^{2}}{x_{2} x_{5}^{2} x_{8} x_{10}} + \frac{2 y_{1} y_{6} x_{5}^{2} x_{8} x_{10}}{x_{2} x_{4} x_{6} x_{9}^{2}} \\
    &\quad + \frac{2 y_{1} y_{6} x_{5} x_{8} x_{10}^{2}}{x_{2}^{2} x_{4} x_{7} x_{9}} + \frac{2 y_{4} y_{7} x_{4}^{2} x_{6} x_{9}}{x_{2}^{2} x_{3} x_{5} x_{10}} + \frac{2 y_{1} y_{4} x_{4} x_{8} x_{10}^{2}}{x_{2} x_{5} x_{7} x_{9}} + \frac{2 y_{1} y_{4} x_{4}^{2} x_{6} x_{10}}{x_{2} x_{3} x_{5} x_{9}} + \frac{y_{1} y_{4} x_{4}^{2} x_{8}^{2} x_{10}^{2}}{x_{2}^{2} x_{3} x_{7} x_{9}^{2}} + \frac{y_{1} y_{4} x_{4}^{2} x_{6}^{2} x_{10}^{2}}{x_{2}^{2} x_{3} x_{5}^{2} x_{7}} + \frac{2 y_{2} y_{4} y_{7} x_{2} x_{4} x_{6}}{x_{1} x_{5}^{2} x_{9}} \\
    &\quad + \frac{2 y_{3} y_{4} y_{7} x_{2} x_{4} x_{8}}{x_{1} x_{5}^{2} x_{9}} + \frac{2 y_{1} y_{2} y_{6} x_{2} x_{6} x_{10}}{x_{1} x_{5} x_{9}^{2}} + \frac{2 y_{1} y_{3} y_{6} x_{2} x_{8} x_{10}}{x_{1} x_{5} x_{9}^{2}} + \frac{y_{1} y_{3} y_{6} x_{8}^{2} x_{10}^{2}}{x_{1} x_{5} x_{7} x_{9}^{2}} + \frac{y_{2} y_{4} y_{7} x_{4}^{2} x_{6}^{2}}{x_{1} x_{3} x_{5}^{2} x_{9}} + \frac{2 y_{1} y_{6} y_{8} x_{5} x_{9}^{2}}{x_{2}^{2} x_{4} x_{7} x_{10}} + \frac{2 y_{4} y_{5} y_{7} x_{5}^{2} x_{9}}{x_{2}^{2} x_{3} x_{4} x_{10}} \\
    &\quad + \frac{2 y_{1} y_{4} y_{8} x_{4} x_{9}^{2}}{x_{2} x_{5} x_{7} x_{10}} + \frac{2 y_{1} y_{4} y_{5} x_{5}^{2} x_{10}}{x_{2} x_{3} x_{4} x_{9}} + \frac{2 y_{1} y_{3} y_{4} y_{8} x_{2}^{2} x_{4}}{x_{1} x_{5}^{2} x_{10}} + \frac{2 y_{1} y_{2} y_{4} y_{5} x_{2}^{2} x_{10}}{x_{1} x_{4} x_{9}^{2}} + \frac{y_{1} y_{4} y_{8} x_{3} x_{9}^{3}}{x_{5}^{2} x_{7} x_{8} x_{10}} + \frac{y_{1} y_{4} y_{5} x_{5}^{3} x_{7}}{x_{3} x_{4} x_{6} x_{9}^{2}} + \frac{y_{3} y_{4} y_{5} y_{7} x_{5} x_{8}}{x_{1} x_{3} x_{4} x_{9}} \\
    &\quad + \frac{y_{1} y_{2} y_{6} y_{8} x_{2} x_{6}}{x_{1} x_{4} x_{8} x_{10}} + \frac{y_{3} y_{4} y_{5} y_{7} x_{2} x_{8}}{x_{1} x_{4} x_{6} x_{10}} + \frac{y_{1} y_{2} y_{6} y_{8} x_{6} x_{9}}{x_{1} x_{5} x_{7} x_{10}} + \frac{y_{1} y_{3} y_{6} y_{8} x_{8} x_{9}}{x_{1} x_{5} x_{7} x_{10}} + \frac{y_{2} y_{4} y_{5} y_{7} x_{5} x_{6}}{x_{1} x_{3} x_{4} x_{9}} + \frac{y_{1} y_{6} x_{5}^{2} x_{8}^{2} x_{10}^{2}}{x_{2}^{2} x_{4} x_{6} x_{7} x_{9}^{2}} + \frac{y_{4} y_{7} x_{4}^{2} x_{6}^{2} x_{9}^{2}}{x_{2}^{2} x_{3} x_{5}^{2} x_{8} x_{10}} \\
    &\quad + \frac{y_{1} y_{2} y_{6} x_{6} x_{8} x_{10}^{2}}{x_{1} x_{5} x_{7} x_{9}^{2}} + \frac{y_{3} y_{4} y_{7} x_{4}^{2} x_{6} x_{8}}{x_{1} x_{3} x_{5}^{2} x_{9}} + \frac{y_{1} y_{4} y_{5} x_{5} x_{6} x_{10}^{2}}{x_{2}^{2} x_{3} x_{4} x_{7}} + \frac{y_{3} y_{4} y_{7} x_{2}^{2} x_{3} x_{8}}{x_{1} x_{5}^{2} x_{6} x_{9}} + \frac{2 y_{1} y_{2} y_{6} x_{2}^{2} x_{3} x_{10}}{x_{1} x_{4} x_{5} x_{9}^{2}} + \frac{2 y_{3} y_{4} y_{7} x_{2}^{2} x_{4} x_{7}}{x_{1} x_{5}^{2} x_{9} x_{10}} \\
    &\quad + \frac{y_{1} y_{2} y_{6} x_{2}^{2} x_{6} x_{7}}{x_{1} x_{5} x_{8} x_{9}^{2}} + \frac{y_{1} y_{4} y_{8} x_{4}^{2} x_{8} x_{9}}{x_{2}^{2} x_{3} x_{7} x_{10}} + \frac{y_{1} y_{6} y_{8} x_{6} x_{9}^{3}}{x_{2}^{2} x_{4} x_{7} x_{8} x_{10}} + \frac{y_{4} y_{5} y_{7} x_{5}^{3} x_{8}}{x_{2}^{2} x_{3} x_{4} x_{6} x_{10}} + \frac{y_{1} y_{2} y_{4} y_{8} x_{2}^{3} x_{3}}{x_{1} x_{5}^{2} x_{8} x_{10}} + \frac{y_{1} y_{3} y_{4} y_{8} x_{2}^{3} x_{3}}{x_{1} x_{5}^{2} x_{6} x_{10}} + \frac{y_{1} y_{2} y_{4} y_{5} x_{2}^{3} x_{7}}{x_{1} x_{4} x_{8} x_{9}^{2}} \\
    &\quad + \frac{y_{1} y_{3} y_{4} y_{5} x_{2}^{3} x_{7}}{x_{1} x_{4} x_{6} x_{9}^{2}} + \frac{y_{3} y_{6} y_{7} x_{2} x_{3} x_{8}}{x_{1} x_{4} x_{5} x_{6} x_{10}} + \frac{y_{2} y_{6} y_{7} x_{2} x_{6} x_{7}}{x_{1} x_{4} x_{8} x_{9} x_{10}} + \frac{y_{1} y_{2} y_{4} y_{5} y_{8} x_{2} x_{9}}{x_{1} x_{4} x_{7} x_{10}} + \frac{y_{1} y_{3} y_{4} y_{5} y_{8} x_{2} x_{5}}{x_{1} x_{3} x_{4} x_{10}} + \frac{2 y_{1} y_{4} x_{4}^{2} x_{6} x_{8} x_{10}^{2}}{x_{2}^{2} x_{3} x_{5} x_{7} x_{9}} \\
    &\quad + \frac{y_{1} y_{6} y_{8} x_{5}^{2} x_{8} x_{9}}{x_{2}^{2} x_{4} x_{6} x_{7} x_{10}} + \frac{y_{4} y_{5} y_{7} x_{5} x_{6} x_{9}^{2}}{x_{2}^{2} x_{3} x_{4} x_{8} x_{10}} + \frac{2 y_{1} y_{4} y_{5} x_{5}^{2} x_{8} x_{10}^{2}}{x_{2}^{2} x_{3} x_{4} x_{7} x_{9}} + \frac{2 y_{1} y_{4} y_{8} x_{4}^{2} x_{6} x_{9}^{2}}{x_{2}^{2} x_{3} x_{5} x_{7} x_{10}} + \frac{2 y_{1} y_{2} y_{4} y_{8} x_{2}^{2} x_{4} x_{6}}{x_{1} x_{5}^{2} x_{8} x_{10}} + \frac{2 y_{1} y_{3} y_{4} y_{5} x_{2}^{2} x_{8} x_{10}}{x_{1} x_{4} x_{6} x_{9}^{2}}\\
    &\quad + \frac{y_{1} y_{2} y_{4} y_{5} x_{2} x_{8} x_{10}^{2}}{x_{1} x_{4} x_{7} x_{9}^{2}} + \frac{y_{1} y_{3} y_{4} y_{8} x_{2} x_{4}^{2} x_{6}}{x_{1} x_{3} x_{5}^{2} x_{10}} + \frac{2 y_{1} y_{4} y_{5} y_{8} x_{5}^{2} x_{9}^{2}}{x_{2}^{2} x_{3} x_{4} x_{7} x_{10}} + \frac{y_{2} y_{4} y_{7} x_{2}^{3} x_{3} x_{7}}{x_{1} x_{5}^{2} x_{8} x_{9} x_{10}} + \frac{y_{3} y_{4} y_{7} x_{2}^{3} x_{3} x_{7}}{x_{1} x_{5}^{2} x_{6} x_{9} x_{10}} + \frac{y_{1} y_{2} y_{6} x_{2}^{3} x_{3} x_{7}}{x_{1} x_{4} x_{5} x_{8} x_{9}^{2}} \\
    &\quad + \frac{y_{1} y_{3} y_{6} x_{2}^{3} x_{3} x_{7}}{x_{1} x_{4} x_{5} x_{6} x_{9}^{2}} + \frac{2 y_{1} y_{4} y_{5} x_{5}^{3} x_{8} x_{10}}{x_{2} x_{3} x_{4} x_{6} x_{9}^{2}} + \frac{2 y_{1} y_{4} y_{8} x_{4} x_{6} x_{9}^{3}}{x_{2} x_{5}^{2} x_{7} x_{8} x_{10}} + \frac{y_{1} y_{2} y_{6} y_{8} x_{2} x_{3} x_{9}}{x_{1} x_{4} x_{5} x_{7} x_{10}} + \frac{y_{3} y_{4} y_{5} y_{7} x_{2} x_{5} x_{7}}{x_{1} x_{3} x_{4} x_{9} x_{10}} + \frac{y_{3} y_{4} y_{7} x_{2} x_{4}^{2} x_{6} x_{7}}{x_{1} x_{3} x_{5}^{2} x_{9} x_{10}} \\
    &\quad + \frac{2 y_{1} y_{3} y_{6} x_{2}^{2} x_{3} x_{8} x_{10}}{x_{1} x_{4} x_{5} x_{6} x_{9}^{2}} + \frac{2 y_{2} y_{4} y_{7} x_{2}^{2} x_{4} x_{6} x_{7}}{x_{1} x_{5}^{2} x_{8} x_{9} x_{10}} + \frac{y_{1} y_{2} y_{6} x_{2} x_{3} x_{8} x_{10}^{2}}{x_{1} x_{4} x_{5} x_{7} x_{9}^{2}} + \frac{y_{1} y_{2} y_{4} y_{8} x_{2} x_{4}^{2} x_{6}^{2}}{x_{1} x_{3} x_{5}^{2} x_{8} x_{10}} + \frac{y_{1} y_{3} y_{4} y_{5} x_{2} x_{8}^{2} x_{10}^{2}}{x_{1} x_{4} x_{6} x_{7} x_{9}^{2}} \\
    &\quad + \frac{y_{1} y_{4} y_{8} x_{4}^{2} x_{6}^{2} x_{9}^{3}}{x_{2}^{2} x_{3} x_{5}^{2} x_{7} x_{8} x_{10}} + \frac{y_{1} y_{4} y_{5} x_{5}^{3} x_{8}^{2} x_{10}^{2}}{x_{2}^{2} x_{3} x_{4} x_{6} x_{7} x_{9}^{2}} + \frac{y_{1} y_{3} y_{4} y_{5} y_{8} x_{2} x_{8} x_{9}}{x_{1} x_{4} x_{6} x_{7} x_{10}} + \frac{y_{1} y_{2} y_{4} y_{5} y_{8} x_{2} x_{5} x_{6}}{x_{1} x_{3} x_{4} x_{8} x_{10}} + \frac{y_{2} y_{4} y_{7} x_{2} x_{4}^{2} x_{6}^{2} x_{7}}{x_{1} x_{3} x_{5}^{2} x_{8} x_{9} x_{10}} \\
    &\quad + \frac{y_{1} y_{3} y_{6} x_{2} x_{3} x_{8}^{2} x_{10}^{2}}{x_{1} x_{4} x_{5} x_{6} x_{7} x_{9}^{2}} + \frac{y_{1} y_{4} y_{5} y_{8} x_{5} x_{6} x_{9}^{3}}{x_{2}^{2} x_{3} x_{4} x_{7} x_{8} x_{10}} + \frac{y_{1} y_{4} y_{5} y_{8} x_{5}^{3} x_{8} x_{9}}{x_{2}^{2} x_{3} x_{4} x_{6} x_{7} x_{10}} + \frac{y_{1} y_{3} y_{6} y_{8} x_{2} x_{3} x_{8} x_{9}}{x_{1} x_{4} x_{5} x_{6} x_{7} x_{10}} + \frac{y_{2} y_{4} y_{5} y_{7} x_{2} x_{5} x_{6} x_{7}}{x_{1} x_{3} x_{4} x_{8} x_{9} x_{10}};
\end{align*}
\allowdisplaybreaks
\begin{align*}
x_3^{(1)} &= \frac{y_2 y_5}{x_1} + \frac{y_3 y_8}{x_1} + \frac{6 x_4 x_{10}}{x_2^2} + \frac{2 x_3 x_7}{x_6 x_8} + \frac{3 x_3 x_{10}}{x_2 x_6} + \frac{3 x_4 x_7}{x_2 x_8} + \frac{3 x_4^2 x_9}{x_2^2 x_5} + \frac{3 x_5 x_{10}^2}{x_2^2 x_9} + \frac{x_3^2 x_9^3}{x_5^3 x_8^2} + \frac{x_5^3 x_7^2}{x_6^2 x_9^3} + \frac{x_6 x_{10}^3}{x_2^3 x_7} + \frac{x_4^3 x_8}{x_2^3 x_3} + \frac{y_2 y_8 x_6}{x_1 x_8}  \\
&\quad + \frac{y_3 y_5 x_8}{x_1 x_6} + \frac{3 y_5 x_5^2 x_9}{x_2^3 x_3} + \frac{3 y_8 x_5 x_9^2}{x_2^3 x_7} + \frac{3 x_3 x_9 x_{10}}{x_2 x_5 x_8} + \frac{3 x_4 x_5 x_7}{x_2 x_6 x_9} + \frac{3 x_3 x_4 x_9^2}{x_2 x_5^2 x_8} + \frac{3 x_5^2 x_7 x_{10}}{x_2 x_6 x_9^2} + \frac{6 x_4^2 x_6 x_9^2}{x_2^2 x_5^2 x_8} + \frac{6 x_5^2 x_8 x_{10}^2}{x_2^2 x_6 x_9^2} + \frac{3 x_4^2 x_6^2 x_9^3}{x_2^2 x_5^3 x_8^2} \\
&\quad + \frac{3 x_5^3 x_8^2 x_{10}^2}{x_2^2 x_6^2 x_9^3} + \frac{3 x_5 x_8 x_{10}^3}{x_2^3 x_7 x_9} + \frac{3 x_4^3 x_6 x_9}{x_2^3 x_3 x_5} + \frac{y_5 x_5^3 x_8}{x_2^3 x_3 x_6} + \frac{y_5 x_6^2 x_9^3}{x_2^3 x_3 x_8^2} + \frac{y_8 x_5^3 x_8^2}{x_2^3 x_6^2 x_7} + \frac{y_8 x_6 x_9^3}{x_2^3 x_7 x_8} + \frac{9 y_3 x_4^2 x_6 x_{10}}{x_1 x_5^3} + \frac{9 y_2 x_4 x_8 x_{10}^2}{x_1 x_9^3} \\
&\quad + \frac{3 y_3 y_5 x_6 x_{10}}{x_1 x_2 x_3} + \frac{2 y_2 y_8 x_3 x_8}{x_1 x_6 x_7} + \frac{3 y_2 y_8 x_3 x_9}{x_1 x_5 x_7} + \frac{3 y_3 y_5 x_5 x_7}{x_1 x_3 x_9} + \frac{2 y_3 y_5 x_6 x_7}{x_1 x_3 x_8} + \frac{3 y_2 y_8 x_4 x_8}{x_1 x_2 x_7} + \frac{3 x_4 x_5 x_8 x_{10}}{x_2^2 x_6 x_9} + \frac{3 x_4 x_6 x_9 x_{10}}{x_2^2 x_5 x_8} \\
&\quad + \frac{6 y_3 x_2 x_4^2 x_7}{x_1 x_5^2 x_9} + \frac{6 y_2 x_2 x_3 x_{10}^2}{x_1 x_5 x_9^2} + \frac{6 y_2 x_4^2 x_6 x_{10}}{x_1 x_5^2 x_9} + \frac{6 y_3 x_4^2 x_8 x_{10}}{x_1 x_5^2 x_9} + \frac{6 y_2 x_4 x_6 x_{10}^2}{x_1 x_5 x_9^2} + \frac{6 y_3 x_4 x_8 x_{10}^2}{x_1 x_5 x_9^2} + \frac{2 y_3 y_8 x_3 x_8^2}{x_1 x_6^2 x_7} + \frac{2 y_2 y_5 x_6^2 x_7}{x_1 x_3 x_8^2} \\
&\quad + \frac{y_2^2 y_5 y_8 x_6^2}{x_1^2 x_3 x_7} + \frac{y_3^2 y_5 y_8 x_8^2}{x_1^2 x_3 x_7} + \frac{3 x_4^3 x_6^2 x_9^2}{x_2^3 x_3 x_5^2 x_8} + \frac{x_4^3 x_6^3 x_9^3}{x_2^3 x_3 x_5^3 x_8^2} + \frac{3 x_3 x_4 x_6 x_9^3}{x_2 x_5^3 x_8^2} + \frac{3 x_5^3 x_7 x_8 x_{10}}{x_2 x_6^2 x_9^3} + \frac{x_5^3 x_8^3 x_{10}^3}{x_2^3 x_6^2 x_7 x_9^3} + \frac{3 x_5^2 x_8^2 x_{10}^3}{x_2^3 x_6 x_7 x_9^2} \\
&\quad + \frac{3 y_5 x_5 x_6 x_9^2}{x_2^3 x_3 x_8} + \frac{2 y_3 x_2^3 x_3 x_7^2}{x_1 x_6^2 x_9^3} + \frac{3 y_8 x_5^2 x_8 x_9}{x_2^3 x_6 x_7} + \frac{2 y_2 x_2^3 x_3^2 x_7}{x_1 x_5^3 x_8^2} + \frac{3 y_2 x_2^2 x_3^2 x_{10}}{x_1 x_5^3 x_8} + \frac{3 y_3 x_2^2 x_3^2 x_{10}}{x_1 x_5^3 x_6} + \frac{3 y_2 x_2^2 x_4 x_7^2}{x_1 x_8 x_9^3} + \frac{3 y_3 x_2^2 x_4 x_7^2}{x_1 x_6 x_9^3} \\
&\quad + \frac{9 y_3 x_2 x_3 x_4 x_{10}}{x_1 x_5^3} + \frac{9 y_2 x_2 x_4 x_7 x_{10}}{x_1 x_9^3} + \frac{9 y_2 x_4^2 x_6^2 x_{10}}{x_1 x_5^3 x_8} + \frac{9 y_3 x_4 x_8^2 x_{10}^2}{x_1 x_6 x_9^3} + \frac{y_2^2 y_8 x_2^3 x_3^2}{x_1^2 x_5^3 x_7} + \frac{y_3^2 y_5 x_2^3 x_7^2}{x_1^2 x_3 x_9^3} + \frac{3 y_2^2 x_2^4 x_3^2 x_{10}^2}{x_1^2 x_5^3 x_9^3} + \frac{3 y_3^2 x_2^4 x_4^2 x_7^2}{x_1^2 x_5^3 x_9^3} \\
&\quad + \frac{3 y_2 x_2 x_3 x_4 x_{10}}{x_1 x_5^2 x_9} + \frac{3 y_3 x_2 x_4 x_7 x_{10}}{x_1 x_5 x_9^2} + \frac{3 y_2 y_5 x_6^2 x_{10}}{x_1 x_2 x_3 x_8} + \frac{3 y_3 y_8 x_4 x_8^2}{x_1 x_2 x_6 x_7} + \frac{3 y_3 x_4^3 x_6^2 x_{10}}{x_1 x_2 x_3 x_5^3} + \frac{2 y_2 x_2^3 x_3 x_7^2}{x_1 x_6 x_8 x_9^3} + \frac{2 y_3 x_2^3 x_3^2 x_7}{x_1 x_5^3 x_6 x_8} \\
&\quad + \frac{6 y_2 x_2 x_4^2 x_6^2 x_7}{x_1 x_5^3 x_8^2} + \frac{6 y_3 x_2 x_4^2 x_6 x_7}{x_1 x_5^3 x_8} + \frac{2 y_2 x_3 x_8^2 x_{10}^3}{x_1 x_6 x_7 x_9^3} + \frac{2 y_3 x_3 x_8^3 x_{10}^3}{x_1 x_6^2 x_7 x_9^3} + \frac{3 y_2 x_3 x_8 x_{10}^3}{x_1 x_5 x_7 x_9^2} + \frac{3 y_3 x_4^3 x_6 x_7}{x_1 x_3 x_5^2 x_9} + \frac{2 y_2 x_4^3 x_6^3 x_7}{x_1 x_3 x_5^3 x_8^2} + \frac{2 y_3 x_4^3 x_6^2 x_7}{x_1 x_3 x_5^3 x_8} \\
&\quad + \frac{6 y_3 x_2^2 x_3 x_4 x_7}{x_1 x_5^3 x_8} + \frac{6 y_2 x_2^2 x_3 x_7 x_{10}}{x_1 x_6 x_9^3} + \frac{6 y_2 x_2 x_3 x_8 x_{10}^2}{x_1 x_6 x_9^3} + \frac{6 y_3 x_2 x_3 x_8^2 x_{10}^2}{x_1 x_6^2 x_9^3} + \frac{3 y_2 x_4 x_8^2 x_{10}^3}{x_1 x_2 x_7 x_9^3} + \frac{9 y_2^2 x_2^2 x_4^2 x_6^2 x_{10}^2}{x_1^2 x_5^3 x_9^3} + \frac{9 y_3^2 x_2^2 x_4^2 x_8^2 x_{10}^2}{x_1^2 x_5^3 x_9^3} \\
&\quad + \frac{3 y_2^2 y_8 x_2 x_4^2 x_6^2}{x_1^2 x_5^3 x_7} + \frac{3 y_3^2 y_8 x_2 x_4^2 x_8^2}{x_1^2 x_5^3 x_7} + \frac{3 y_2^2 y_5 x_2 x_6^2 x_{10}^2}{x_1^2 x_3 x_9^3} + \frac{3 y_3^2 y_5 x_2 x_8^2 x_{10}^2}{x_1^2 x_3 x_9^3} + \frac{y_2^2 y_8 x_4^3 x_6^3}{x_1^2 x_3 x_5^3 x_7} + \frac{y_3^2 y_5 x_8^3 x_{10}^3}{x_1^2 x_3 x_7 x_9^3} + \frac{y_2^2 x_2^6 x_3^2 x_7^2}{x_1^2 x_5^3 x_8^2 x_9^3} + \frac{y_3^2 x_2^6 x_3^2 x_7^2}{x_1^2 x_5^3 x_6^2 x_9^3} \\
&\quad + \frac{3 y_3 y_8 x_4 x_8 x_9}{x_1 x_2 x_5 x_7} + \frac{3 y_2 y_5 x_5 x_6 x_{10}}{x_1 x_2 x_3 x_9} + \frac{3 y_3 y_5 x_5 x_8 x_{10}}{x_1 x_2 x_3 x_9} + \frac{3 y_3 y_8 x_3 x_8 x_9}{x_1 x_5 x_6 x_7} + \frac{3 y_2 y_5 x_5 x_6 x_7}{x_1 x_3 x_8 x_9} + \frac{3 y_2 y_8 x_4 x_6 x_9}{x_1 x_2 x_5 x_7} + \frac{6 y_2 x_2 x_4^2 x_6 x_7}{x_1 x_5^2 x_8 x_9} \\
&\quad + \frac{3 y_2 x_2^2 x_3 x_4 x_7}{x_1 x_5^2 x_8 x_9} + \frac{3 y_3 x_2^2 x_3 x_4 x_7}{x_1 x_5^2 x_6 x_9} + \frac{3 y_2 x_2^2 x_3 x_7 x_{10}}{x_1 x_5 x_8 x_9^2} + \frac{3 y_3 x_2^2 x_3 x_7 x_{10}}{x_1 x_5 x_6 x_9^2} + \frac{6 y_3 x_2 x_3 x_8 x_{10}^2}{x_1 x_5 x_6 x_9^2} + \frac{2 y_2 y_3 y_5 y_8 x_6 x_8}{x_1^2 x_3 x_7} + \frac{3 y_3 x_4 x_8^2 x_{10}^3}{x_1 x_2 x_5 x_7 x_9^2} \\
&\quad + \frac{3 y_2 x_4^3 x_6^3 x_{10}}{x_1 x_2 x_3 x_5^3 x_8} + \frac{3 y_2 x_4^3 x_6^2 x_{10}}{x_1 x_2 x_3 x_5^2 x_9} + \frac{3 y_3 x_3 x_8^2 x_{10}^3}{x_1 x_5 x_6 x_7 x_9^2} + \frac{3 y_2 x_4^3 x_6^2 x_7}{x_1 x_3 x_5^2 x_8 x_9} + \frac{6 y_2 x_2^2 x_3 x_4 x_6 x_7}{x_1 x_5^3 x_8^2} + \frac{6 y_3 x_2^2 x_3 x_7 x_8 x_{10}}{x_1 x_6^2 x_9^3} + \frac{9 y_2 x_2 x_3 x_4 x_6 x_{10}}{x_1 x_5^3 x_8} \\
&\quad + \frac{9 y_3 x_2 x_4 x_7 x_8 x_{10}}{x_1 x_6 x_9^3} + \frac{3 y_3 x_4 x_8^3 x_{10}^3}{x_1 x_2 x_6 x_7 x_9^3} + \frac{y_2^2 x_2^3 x_3^2 x_8 x_{10}^3}{x_1^2 x_5^3 x_7 x_9^3} + \frac{y_3^2 x_2^3 x_4^3 x_6 x_7^2}{x_1^2 x_3 x_5^3 x_9^3} + \frac{3 y_3^2 x_2 x_4^2 x_8^3 x_{10}^3}{x_1^2 x_5^3 x_7 x_9^3} + \frac{3 y_2^2 x_2 x_4^3 x_6^3 x_{10}^2}{x_1^2 x_3 x_5^3 x_9^3} \\
&\quad + \frac{9 y_2^2 x_2^3 x_3 x_4 x_6 x_{10}^2}{x_1^2 x_5^3 x_9^3} + \frac{9 y_3^2 x_2^3 x_4^2 x_7 x_8 x_{10}}{x_1^2 x_5^3 x_9^3} + \frac{y_3^2 y_8 x_2^3 x_3^2 x_8^2}{x_1^2 x_5^3 x_6^2 x_7} + \frac{y_2^2 y_5 x_2^3 x_6^2 x_7^2}{x_1^2 x_3 x_8^2 x_9^3} + \frac{3 y_3^2 y_5 x_2^2 x_7 x_8 x_{10}}{x_1^2 x_3 x_9^3} + \frac{y_3^2 y_8 x_4^3 x_6 x_8^2}{x_1^2 x_3 x_5^3 x_7} \\
&\quad + \frac{y_2^2 y_5 x_6^2 x_8 x_{10}^3}{x_1^2 x_3 x_7 x_9^3} + \frac{3 y_2^2 y_8 x_2^2 x_3 x_4 x_6}{x_1^2 x_5^3 x_7} + \frac{3 y_3^2 x_2^4 x_3^2 x_8^2 x_{10}^2}{x_1^2 x_5^3 x_6^2 x_9^3} + \frac{3 y_2^2 x_2^4 x_4^2 x_6^2 x_7^2}{x_1^2 x_5^3 x_8^2 x_9^3} + \frac{3 y_2^2 x_2^5 x_3^2 x_7 x_{10}}{x_1^2 x_5^3 x_8 x_9^3} + \frac{3 y_3^2 x_2^5 x_3 x_4 x_7^2}{x_1^2 x_5^3 x_6 x_9^3} \\
&\quad + \frac{3 y_3 x_2 x_3 x_4 x_8 x_{10}}{x_1 x_5^2 x_6 x_9} + \frac{3 y_2 x_2 x_4 x_6 x_7 x_{10}}{x_1 x_5 x_8 x_9^2} + \frac{3 y_3 x_4^3 x_6 x_8 x_{10}}{x_1 x_2 x_3 x_5^2 x_9} + \frac{3 y_2 x_4 x_6 x_8 x_{10}^3}{x_1 x_2 x_5 x_7 x_9^2} + \frac{y_3^2 x_2^3 x_8^3 x_3^2 x_{10}^3}{x_1^2 x_5^3 x_6^2 x_7 x_9^3} + \frac{y_2^2 x_2^3 x_6^3 x_4^3 x_7^2}{x_1^2 x_3 x_5^3 x_8^2 x_9^3} \\
&\quad + \frac{3 y_2^2 x_2 x_4^2 x_6^2 x_8 x_{10}^3}{x_1^2 x_5^3 x_7 x_9^3} + \frac{3 y_3^2 x_2 x_4^3 x_6 x_8^2 x_{10}^2}{x_1^2 x_3 x_5^3 x_9^3} + \frac{y_2^2 x_4^3 x_6^3 x_8 x_{10}^3}{x_1^2 x_3 x_5^3 x_7 x_9^3} + \frac{y_3^2 x_4^3 x_6 x_8^3 x_{10}^3}{x_1^2 x_3 x_5^3 x_7 x_9^3} + \frac{9 y_3^2 x_2^3 x_3 x_4 x_8^2 x_{10}^2}{x_1^2 x_5^3 x_6 x_9^3} + \frac{9 y_2^2 x_2^3 x_4^2 x_6^2 x_7 x_{10}}{x_1^2 x_5^3 x_8 x_9^3} \\
&\quad + \frac{3 y_3^2 x_2^2 x_3 x_4 x_8^3 x_{10}^3}{x_1^2 x_5^3 x_6 x_7 x_9^3} + \frac{18 y_2 y_3 x_2^3 x_3 x_4 x_8 x_{10}^2}{x_1^2 x_5^3 x_9^3} + \frac{18 y_2 y_3 x_2^3 x_4^2 x_6 x_7 x_{10}}{x_1^2 x_5^3 x_9^3} + \frac{18 y_2 y_3 x_2^2 x_4^2 x_6 x_8 x_{10}^2}{x_1^2 x_5^3 x_9^3} + \frac{3 y_2^2 y_5 x_2^2 x_6^2 x_7 x_{10}}{x_1^2 x_3 x_8 x_9^3} \\
&\quad + \frac{3 y_3^2 y_8 x_2^2 x_3 x_4 x_8^2}{x_1^2 x_5^3 x_6 x_7} + \frac{6 y_2 y_3 y_8 x_2^2 x_3 x_4 x_8}{x_1^2 x_5^3 x_7} + \frac{6 y_2 y_3 y_5 x_2^2 x_6 x_7 x_{10}}{x_1^2 x_3 x_9^3} + \frac{6 y_2 y_3 y_8 x_2 x_4^2 x_6 x_8}{x_1^2 x_5^3 x_7} + \frac{6 y_2 y_3 y_5 x_2 x_6 x_8 x_{10}^2}{x_1^2 x_3 x_9^3} \\
&\quad + \frac{2 y_2 y_3 y_8 x_4^3 x_6^2 x_8}{x_1^2 x_3 x_5^3 x_7} + \frac{2 y_2 y_3 y_5 x_6 x_8^2 x_{10}^3}{x_1^2 x_3 x_7 x_9^3} + \frac{2 y_2 y_3 y_8 x_2^3 x_3^2 x_8}{x_1^2 x_5^3 x_6 x_7} + \frac{2 y_2 y_3 y_5 x_2^3 x_6 x_7^2}{x_1^2 x_3 x_8 x_9^3} + \frac{18 y_2 y_3 x_2^4 x_3 x_4 x_7 x_{10}}{x_1^2 x_5^3 x_9^3} + \frac{6 y_2 y_3 x_2^4 x_3^2 x_8 x_{10}^2}{x_1^2 x_5^3 x_6 x_9^3} \\
&\quad + \frac{6 y_2 y_3 x_2^4 x_4^2 x_6 x_7^2}{x_1^2 x_5^3 x_8 x_9^3} + \frac{3 y_3^2 x_2^5 x_3^2 x_7 x_8 x_{10}}{x_1^2 x_5^3 x_6^2 x_9^3} + \frac{3 y_2^2 x_2^5 x_3 x_4 x_6 x_7^2}{x_1^2 x_5^3 x_8^2 x_9^3} + \frac{6 y_2 y_3 x_2^5 x_3^2 x_7 x_{10}}{x_1^2 x_5^3 x_6 x_9^3} + \frac{6 y_2 y_3 x_2^5 x_3 x_4 x_7^2}{x_1^2 x_5^3 x_8 x_9^3} + \frac{2 y_2 y_3 x_2^6 x_3^2 x_7^2}{x_1^2 x_5^3 x_6 x_8 x_9^3} \\
&\quad + \frac{3 y_2^2 x_2^2 x_4^3 x_6^3 x_7 x_{10}}{x_1^2 x_3 x_5^3 x_8 x_9^3} + \frac{3 y_3^2 x_2^2 x_4^3 x_6 x_7 x_8 x_{10}}{x_1^2 x_3 x_5^3 x_9^3} + \frac{3 y_2^2 x_2^2 x_3 x_4 x_6 x_8 x_{10}^3}{x_1^2 x_5^3 x_7 x_9^3} + \frac{6 y_2 y_3 x_2^2 x_3 x_4 x_8^2 x_{10}^3}{x_1^2 x_5^3 x_6 x_7 x_9^3} + \frac{6 y_2 y_3 x_2^2 x_4^3 x_6^2 x_7 x_{10}}{x_1^2 x_3 x_5^3 x_9^3} \\
&\quad + \frac{6 y_2 y_3 x_2 x_4^2 x_6 x_8^2 x_{10}^3}{x_1^2 x_5^3 x_7 x_9^3} + \frac{6 y_2 y_3 x_2 x_4^3 x_6^2 x_8 x_{10}^2}{x_1^2 x_3 x_5^3 x_9^3} + \frac{2 y_2 y_3 x_4^3 x_6^2 x_8^2 x_{10}^3}{x_1^2 x_3 x_5^3 x_7 x_9^3} + \frac{2 y_2 y_3 x_2^3 x_3^2 x_8^2 x_{10}^3}{x_1^2 x_5^3 x_6 x_7 x_9^3} + \frac{2 y_2 y_3 x_2^3 x_4^3 x_6^2 x_7^2}{x_1^2 x_3 x_5^3 x_8 x_9^3} \\
&\quad + \frac{9 y_2^2 x_2^4 x_3 x_4 x_6 x_7 x_{10}}{x_1^2 x_5^3 x_8 x_9^3} + \frac{9 y_3^2 x_2^4 x_3 x_4 x_7 x_8 x_{10}}{x_1^2 x_5^3 x_6 x_9^3}.\\
\end{align*}
}

For letting all variables $x_i = x$ and $y_j = y$, we can calculate the desired expansion formula for diagonal $\gamma' = (x_3^{(1)}, x_2^{(5)})$:
\[
\begin{aligned}
x_3^{(1)} = &\ 80 + 256y + 256y^2 + \frac{16y + 48y^2 + 64y^3}{x} + \frac{4y^4}{x^2};\\
x_2^{(5)} = &\ 16y^2 + \frac{24y^2 + 52y^3 + 16y^4}{x} + \frac{4y^2 + 16y^3 + 16y^4 + 4y^5}{x^2}.
\end{aligned}
\]
In general, for all vertices:
{\small\[
\begin{aligned}
x_1^{(1)} = &\ 2y;\\
x_2^{(5)} = &\ 16y^2 + \frac{24y^2 + 52y^3 + 16y^4}{x} + \frac{4y^2 + 16y^3 + 16y^4 + 4y^5}{x^2};\\
x_3^{(1)} = &\ 80 + 256y + 256y^2 + \frac{16y + 48y^2 + 64y^3}{x} + \frac{4y^4}{x^2};\\
x_4^{(2)} = &\ 96y^2 + 248y^3 + 96y^4 + 64xy^2 + \frac{56y^2 + 272y^3 + 432y^4 + 220y^5 + 32y^6}{x}+ \frac{8y^3 + 36y^4 + 56y^5 + 28y^6 + 4y^7}{x^2};\\
x_5^{(4)} = &\ 12y + 36y^2 + 16y^3 + 16xy + \frac{4y^2 + 10y^3 + 4y^4}{x};\\
x_6^{(1)} = &\ 8 + 24y + 32y^2 + \frac{4y^3}{x};\\
x_7^{(12)} = &\ 9216y^3 + 23552y^4 + 8192y^5 + 4096xy^3 + \frac{6912y^3 + 36288y^4 + 58784y^5 + 30208y^6 + 4608y^7}{x} \\
&+ \frac{1728y^3 + 14688y^4 + 44352y^5 + 58112y^6 + 32256y^7 + 7296y^8 + 512y^9}{x^2} \\
&+ \frac{640y^4 + 4576y^5 + 12128y^6 + 14528y^7 + 7792y^8 + 1768y^9 + 128y^{10}}{x^3} \\
&+ \frac{64y^5 + 416y^6 + 1008y^7 + 1112y^8 + 560y^9 + 120y^{10} + 8y^{11}}{x^4};\\
x_8^{(6)} = &\ 64y^4 + \frac{144y^4 + 312y^5 + 96y^6}{x} + \frac{80y^3 + 424y^4 + 904y^5 + 808y^6 + 284y^7 + 32y^8}{x^2} \\
&+ \frac{8y^4 + 48y^5 + 116y^6 + 104y^7 + 36y^8 + 4y^9}{x^3};\\
x_9^{(3)} = &\ 12y + 16y^2 + \frac{4y + 12y^2 + 6y^3}{x};\\
x_{10}^{(1)} = &\ 2y + 8y^2 + \frac{2y^2 + 2y^3}{x}.
\end{aligned}
\]
}
For the flip sequence $\mu_2$, we get the new expansion formulas of the vertices of the form:
\begin{gather*}
	(y_1,y_2,...,y_8,x_1,x_2,...,x_{10}) \rightarrow (y_1,y_2,...,y_8,x_1^{(1)},x_{2}^{(5)},x_3^{(1)},x_4^{(1)},x_{5}^{(3)},x_6^{(1)},x_{7}^{(12)},x_8^{(6)},x_9^{(4)},x_{10}^{(2)}).
\end{gather*}
We shall omit the results of the remaining vertices as the formulas are too long. We finally get the diagonal expansion formula after the flip:
\begin{gather*}
	BD = (x_3^{(1)}, x_2^{(5)})
\end{gather*}
where:
{\footnotesize
\begin{align*}
x_{2}^{(5)} &= \frac{y_{1} y_{6}}{x_{2}} + \frac{y_{4} y_{8}}{x_{2}} + \frac{y_{4} x_{5} x_{9}}{x_{2} x_{6}} + \frac{y_{3} y_{6} x_{10}}{x_{1} x_{5}} + \frac{y_{3} y_{6} x_{9}}{x_{1} x_{4}} + \frac{y_{4} x_{4}^{2} x_{10}}{x_{2}^{2} x_{3}} + \frac{2 y_{1} y_{4} x_{4} x_{9}}{x_{5} x_{10}} + \frac{3 y_{1} y_{6} y_{7} x_{6}}{x_{4} x_{8}} + \frac{y_{1} y_{4} y_{5} y_{8}}{x_{4} x_{9}} + \frac{y_{6} x_{5}^{2} x_{10}}{x_{2}^{2} x_{4} x_{6}} \\
&\quad + \frac{y_{4} y_{8}^{2} x_{3} x_{10}}{x_{5}^{2} x_{9}^{2}} + \frac{2 y_{1}^{2} y_{6} x_{6} x_{9}}{x_{5} x_{10}^{2}} + \frac{3 y_{1}^{2} y_{6} x_{6} x_{9}^{2}}{x_{4} x_{10}^{3}} + \frac{2 y_{4} y_{8} x_{4} x_{10}}{x_{2} x_{5} x_{9}} + \frac{2 y_{1} y_{6} x_{5} x_{9}}{x_{2} x_{4} x_{10}} + \frac{y_{2} y_{6} x_{6} x_{10}}{x_{1} x_{5} x_{7}} + \frac{y_{2} y_{6} x_{6} x_{9}}{x_{1} x_{4} x_{7}} + \frac{y_{1} y_{6} y_{8} x_{3}}{x_{4} x_{5} x_{9}} \\
&\quad + \frac{y_{1} y_{6} y_{8} x_{6}}{x_{2} x_{5} x_{9}} + \frac{2 y_{1} y_{6} y_{8} x_{6}}{x_{2} x_{4} x_{10}} + \frac{2 y_{4} x_{4} x_{7} x_{9}^{2}}{x_{5} x_{8} x_{10}} + \frac{2 y_{3} y_{4} x_{2} x_{4} x_{9}}{x_{1} x_{5}^{2}} + \frac{2 y_{4} y_{7} x_{4} x_{10}^{2}}{x_{5} x_{8} x_{9}} + \frac{2 y_{6} y_{7} x_{6} x_{7} x_{9}}{x_{4} x_{8}^{2}} + \frac{2 y_{6} y_{8} x_{5} x_{10}}{x_{2}^{2} x_{4} x_{9}} + \frac{y_{6} y_{8}^{2} x_{6} x_{10}}{x_{2}^{2} x_{4} x_{9}^{2}} \\
&\quad + \frac{3 y_{1} y_{4} y_{7} x_{2}^{2} x_{3}}{x_{5}^{2} x_{8}} + \frac{2 y_{1} y_{4} y_{8} x_{2} x_{3}}{x_{5}^{2} x_{10}} + \frac{4 y_{1} y_{4} y_{8} x_{4} x_{6}}{x_{5}^{2} x_{10}} + \frac{2 y_{1}^{2} y_{4} y_{5} x_{2} x_{9}}{x_{4} x_{10}^{2}} + \frac{3 y_{1}^{2} y_{4} x_{2}^{2} x_{3} x_{9}^{2}}{x_{5}^{2} x_{10}^{3}} + \frac{y_{6} y_{7}^{2} x_{6} x_{10}^{3}}{x_{4} x_{8}^{2} x_{9}^{2}} + \frac{y_{1}^{3} y_{6} x_{6} x_{8}}{x_{5} x_{7} x_{10}^{2}} + \frac{y_{6} x_{6} x_{7}^{2} x_{9}^{4}}{x_{4} x_{8}^{2} x_{10}^{3}} \\
&\quad + \frac{y_{1} y_{6} y_{7} x_{6} x_{10}}{x_{5} x_{8} x_{9}} + \frac{y_{2} y_{4} y_{5} x_{2} x_{10}}{x_{1} x_{4} x_{7}} + \frac{y_{3} y_{4} y_{5} x_{2} x_{10}}{x_{1} x_{4} x_{6}} + \frac{y_{3} y_{4} y_{5} x_{5} x_{9}}{x_{1} x_{3} x_{4}} + \frac{2 y_{6} x_{5} x_{7} x_{9}^{2}}{x_{2} x_{4} x_{8} x_{10}} + \frac{y_{2} y_{4} x_{2}^{2} x_{3} x_{9}}{x_{1} x_{5}^{2} x_{7}} + \frac{y_{3} y_{4} x_{2}^{2} x_{3} x_{9}}{x_{1} x_{5}^{2} x_{6}} \\
&\quad + \frac{y_{3} y_{4} x_{4}^{2} x_{6} x_{9}}{x_{1} x_{3} x_{5}^{2}} + \frac{2 y_{6} y_{7} x_{5} x_{10}^{2}}{x_{2} x_{4} x_{8} x_{9}} + \frac{2 y_{4} y_{7} x_{2}^{2} x_{3} x_{7} x_{9}}{x_{5}^{2} x_{8}^{2}} + \frac{y_{1} y_{6} x_{6} x_{7} x_{9}^{2}}{x_{5} x_{8} x_{10}^{2}} + \frac{2 y_{4} y_{8}^{2} x_{4} x_{6} x_{10}}{x_{2} x_{5}^{2} x_{9}^{2}} + \frac{2 y_{1}^{2} y_{6} x_{2} x_{3} x_{9}}{x_{4} x_{5} x_{10}^{2}} + \frac{6 y_{1} y_{4} y_{7} x_{2} x_{4} x_{6}}{x_{5}^{2} x_{8}} \\
&\quad + \frac{3 y_{1} y_{4} y_{7} x_{4}^{2} x_{6}^{2}}{x_{3} x_{5}^{2} x_{8}} + \frac{3 y_{1} y_{6} x_{6} x_{7} x_{9}^{3}}{x_{4} x_{8} x_{10}^{3}} + \frac{y_{4} y_{5} x_{5}^{3} x_{10}}{x_{2}^{2} x_{3} x_{4} x_{6}} + \frac{y_{4} y_{7}^{2} x_{2}^{2} x_{3} x_{10}^{3}}{x_{5}^{2} x_{8}^{2} x_{9}^{2}} + \frac{6 y_{1}^{2} y_{4} x_{2} x_{4} x_{6} x_{9}^{2}}{x_{5}^{2} x_{10}^{3}} + \frac{3 y_{1}^{2} y_{4} x_{4}^{2} x_{6}^{2} x_{9}^{2}}{x_{3} x_{5}^{2} x_{10}^{3}} + \frac{y_{1}^{3} y_{6} x_{6} x_{8} x_{9}}{x_{4} x_{7} x_{10}^{3}} \\
&\quad + \frac{y_{1}^{3} y_{4} y_{5} x_{2} x_{8}}{x_{4} x_{7} x_{10}^{2}} + \frac{y_{4} x_{2}^{2} x_{3} x_{7}^{2} x_{9}^{4}}{x_{5}^{2} x_{8}^{2} x_{10}^{3}} + \frac{2 y_{6} y_{8} x_{6} x_{7} x_{9}}{x_{2} x_{4} x_{8} x_{10}} + \frac{y_{2} y_{6} x_{2} x_{3} x_{10}}{x_{1} x_{4} x_{5} x_{7}} + \frac{y_{3} y_{6} x_{2} x_{3} x_{10}}{x_{1} x_{4} x_{5} x_{6}} + \frac{y_{1} y_{4} y_{5} y_{7} x_{2} x_{10}}{x_{4} x_{8} x_{9}} + \frac{3 y_{1} y_{4} y_{5} y_{7} x_{5} x_{6}}{x_{3} x_{4} x_{8}} \\
&\quad + \frac{2 y_{1} y_{4} x_{4}^{2} x_{6} x_{9}}{x_{2} x_{3} x_{5} x_{10}} + \frac{2 y_{2} y_{4} x_{2} x_{4} x_{6} x_{9}}{x_{1} x_{5}^{2} x_{7}} + \frac{y_{2} y_{4} x_{4}^{2} x_{6}^{2} x_{9}}{x_{1} x_{3} x_{5}^{2} x_{7}} + \frac{2 y_{4} y_{8} x_{4}^{2} x_{6} x_{10}}{x_{2}^{2} x_{3} x_{5} x_{9}} + \frac{4 y_{4} y_{7} x_{2} x_{4} x_{6} x_{7} x_{9}}{x_{5}^{2} x_{8}^{2}} + \frac{2 y_{4} y_{7} x_{4}^{2} x_{6}^{2} x_{7} x_{9}}{x_{3} x_{5}^{2} x_{8}^{2}} \\
&\quad + \frac{2 y_{4} y_{8} x_{2} x_{3} x_{7} x_{9}}{x_{5}^{2} x_{8} x_{10}} + \frac{4 y_{4} y_{8} x_{4} x_{6} x_{7} x_{9}}{x_{5}^{2} x_{8} x_{10}} + \frac{y_{4} y_{8}^{2} x_{4}^{2} x_{6}^{2} x_{10}}{x_{2}^{2} x_{3} x_{5}^{2} x_{9}^{2}} + \frac{y_{1} y_{4} y_{5} x_{2} x_{7} x_{9}^{2}}{x_{4} x_{8} x_{10}^{2}} + \frac{2 y_{6} y_{7} y_{8} x_{6} x_{10}^{2}}{x_{2} x_{4} x_{8} x_{9}^{2}} + \frac{2 y_{4} y_{5} y_{8} x_{5}^{2} x_{10}}{x_{2}^{2} x_{3} x_{4} x_{9}} \\
&\quad + \frac{2 y_{4} y_{7} y_{8} x_{2} x_{3} x_{10}^{2}}{x_{5}^{2} x_{8} x_{9}^{2}} + \frac{4 y_{4} y_{7} y_{8} x_{4} x_{6} x_{10}^{2}}{x_{5}^{2} x_{8} x_{9}^{2}} + \frac{2 y_{1} y_{4} y_{8} x_{4}^{2} x_{6}^{2}}{x_{2} x_{3} x_{5}^{2} x_{10}} + \frac{2 y_{1} y_{4} y_{5} x_{5}^{2} x_{9}}{x_{2} x_{3} x_{4} x_{10}} + \frac{3 y_{1} y_{4} x_{2}^{2} x_{3} x_{7} x_{9}^{3}}{x_{5}^{2} x_{8} x_{10}^{3}} + \frac{3 y_{1} y_{4} x_{4}^{2} x_{6}^{2} x_{7} x_{9}^{3}}{x_{3} x_{5}^{2} x_{8} x_{10}^{3}} \\
&\quad + \frac{2 y_{4} y_{7}^{2} x_{2} x_{4} x_{6} x_{10}^{3}}{x_{5}^{2} x_{8}^{2} x_{9}^{2}} + \frac{y_{4} y_{7}^{2} x_{4}^{2} x_{6}^{2} x_{10}^{3}}{x_{3} x_{5}^{2} x_{8}^{2} x_{9}^{2}} + \frac{y_{1}^{3} y_{4} x_{2}^{2} x_{3} x_{8} x_{9}}{x_{5}^{2} x_{7} x_{10}^{3}} + \frac{y_{1}^{3} y_{6} x_{2} x_{3} x_{8}}{x_{4} x_{5} x_{7} x_{10}^{2}} + \frac{3 y_{1}^{2} y_{4} y_{5} x_{5} x_{6} x_{9}^{2}}{x_{3} x_{4} x_{10}^{3}} + \frac{2 y_{4} x_{2} x_{4} x_{6} x_{7}^{2} x_{9}^{4}}{x_{5}^{2} x_{8}^{2} x_{10}^{3}} \\
&\quad + \frac{y_{4} x_{4}^{2} x_{6}^{2} x_{7}^{2} x_{9}^{4}}{x_{3} x_{5}^{2} x_{8}^{2} x_{10}^{3}} + \frac{y_{1} y_{6} y_{7} x_{2} x_{3} x_{10}}{x_{4} x_{5} x_{8} x_{9}} + \frac{y_{2} y_{4} y_{5} x_{5} x_{6} x_{9}}{x_{1} x_{3} x_{4} x_{7}} + \frac{2 y_{1} y_{4} y_{5} y_{8} x_{5} x_{6}}{x_{2} x_{3} x_{4} x_{10}} + \frac{2 y_{4} x_{4}^{2} x_{6} x_{7} x_{9}^{2}}{x_{2} x_{3} x_{5} x_{8} x_{10}} + \frac{2 y_{4} y_{7} x_{4}^{2} x_{6} x_{10}^{2}}{x_{2} x_{3} x_{5} x_{8} x_{9}} + \frac{2 y_{4} y_{5} x_{5}^{2} x_{7} x_{9}^{2}}{x_{2} x_{3} x_{4} x_{8} x_{10}} \\
&\quad + \frac{y_{1} y_{6} x_{2} x_{3} x_{7} x_{9}^{2}}{x_{4} x_{5} x_{8} x_{10}^{2}} + \frac{2 y_{4} y_{5} y_{7} x_{5} x_{6} x_{7} x_{9}}{x_{3} x_{4} x_{8}^{2}} + \frac{2 y_{4} y_{5} y_{7} x_{5}^{2} x_{10}^{2}}{x_{2} x_{3} x_{4} x_{8} x_{9}} + \frac{y_{4} y_{5} y_{8}^{2} x_{5} x_{6} x_{10}}{x_{2}^{2} x_{3} x_{4} x_{9}^{2}} + \frac{6 y_{1} y_{4} x_{2} x_{4} x_{6} x_{7} x_{9}^{3}}{x_{5}^{2} x_{8} x_{10}^{3}} + \frac{2 y_{1}^{3} y_{4} x_{2} x_{4} x_{6} x_{8} x_{9}}{x_{5}^{2} x_{7} x_{10}^{3}} \\
&\quad + \frac{y_{1}^{3} y_{4} x_{4}^{2} x_{6}^{2} x_{8} x_{9}}{x_{3} x_{5}^{2} x_{7} x_{10}^{3}} + \frac{y_{4} y_{5} y_{7}^{2} x_{5} x_{6} x_{10}^{3}}{x_{3} x_{4} x_{8}^{2} x_{9}^{2}} + \frac{y_{4} y_{5} x_{5} x_{6} x_{7}^{2} x_{9}^{4}}{x_{3} x_{4} x_{8}^{2} x_{10}^{3}} + \frac{2 y_{4} y_{8} x_{4}^{2} x_{6}^{2} x_{7} x_{9}}{x_{2} x_{3} x_{5}^{2} x_{8} x_{10}} + \frac{2 y_{4} y_{7} y_{8} x_{4}^{2} x_{6}^{2} x_{10}^{2}}{x_{2} x_{3} x_{5}^{2} x_{8} x_{9}^{2}} + \frac{3 y_{1} y_{4} y_{5} x_{5} x_{6} x_{7} x_{9}^{3}}{x_{3} x_{4} x_{8} x_{10}^{3}} \\
&\quad + \frac{y_{1}^{3} y_{4} y_{5} x_{5} x_{6} x_{8} x_{9}}{x_{3} x_{4} x_{7} x_{10}^{3}} + \frac{2 y_{4} y_{5} y_{8} x_{5} x_{6} x_{7} x_{9}}{x_{2} x_{3} x_{4} x_{8} x_{10}} + \frac{2 y_{4} y_{5} y_{7} y_{8} x_{5} x_{6} x_{10}^{2}}{x_{2} x_{3} x_{4} x_{8} x_{9}^{2}};\\
\end{align*}
\allowdisplaybreaks

}
For letting all variables $x_i = x$ and $y_j = y$, we can calculate the desired expansion formula for diagonal $\gamma' = (x_3^{(1)}, x_2^{(5)})$:
{\small\[
\begin{aligned}
x_3^{(1)} = &\ 4 + 11x + 12x^2 + 8x^3 + 98y + 440y^2 + 390y^3 + 132xy + 249xy^2 + 73x^2y \\
&+ \frac{1 + 32y + 300y^2 + 572y^3 + 285y^4}{x} + \frac{3y + 48y^2 + 225y^3 + 264y^4 + 102y^5}{x^2} \\
&+ \frac{3y^2 + 20y^3 + 51y^4 + 49y^5 + 17y^6}{x^3} + \frac{y^3 + y^4 + 3y^5 + 3y^6 + y^7}{x^4};\\
x_2^{(5)} = &\ 4y + 33y^2 + 34y^3 + 5xy + 21xy^2 + 4x^2y + \frac{y + 19y^2 + 39y^3 + 14y^4}{x} + \frac{2y^2 + 12y^3 + 8y^4 + 2y^5}{x^2} + \frac{y^3 + y^4}{x^3}.
\end{aligned}
\]
}
In general, for all vertices:
{\small\[
\begin{aligned}
x_1^{(1)} = &\ 2y;\\
x_2^{(5)} = &\ 4y + 33y^2 + 34y^3 + 5xy + 21xy^2 + 4x^2y + \frac{y + 19y^2 + 39y^3 + 14y^4}{x}  \frac{2y^2 + 12y^3 + 8y^4 + 2y^5}{x^2} + \frac{y^3 + y^4}{x^3};\\
x_3^{(1)} = &\ 4 + 11x + 12x^2 + 8x^3 + 98y + 440y^2 + 390y^3 + 132xy + 249xy^2 + 73x^2y \\
&+ \frac{1 + 32y + 300y^2 + 572y^3 + 285y^4}{x} + \frac{3y + 48y^2 + 225y^3 + 264y^4 + 102y^5}{x^2} \\
&+ \frac{3y^2 + 20y^3 + 51y^4 + 49y^5 + 17y^6}{x^3} + \frac{y^3 + y^4 + 3y^5 + 3y^6 + y^7}{x^4};\\
x_4^{(1)} = &\ y + \frac{2y^2 + y^3}{x};\\
x_5^{(3)} = &\ 4y^2 + 23y^3 + 20y^4 + 5y^5 + 4xy^2 + 4xy^3 + xy^4 + \frac{2y^2 + 26y^3 + 51y^4 + 32y^5 + 7y^6}{x} \\
&+ \frac{5y^3 + 20y^4 + 21y^5 + 8y^6 + y^7}{x^2} + \frac{4y^4 + 4y^5 + y^6}{x^3};\\
x_6^{(1)} = &\ 1 + 3x + 3x^2 + x^3 + 23y + 67y^2 + 45y^3 + 25xy + 30xy^2 + 9x^2y + \frac{3y + 40y^2 + 69y^3 + 30y^4}{x} \\
&+ \frac{3y^2 + 21y^3 + 25y^4 + 9y^5}{x^2} + \frac{y^3 + 3y^4 + 3y^5 + y^6}{x^3};\\
x_7^{(12)} = &\ 12y^3 + 90y^4 + 131y^5 + 99y^6 + 9y^7 + 8xy^3 + 12xy^4 + 14xy^5 + xy^6 \\
&+ \frac{11y^3 + 466y^5 + 557y^6 + 275y^7 + 30y^8}{x} + \frac{4y^3 + 104y^4 + 594y^5 + 1201y^6 + 1038y^7 + 373y^8 + 45y^9}{x^2} \\
&+ \frac{34y^4 + 355y^5 + 1085y^6 + 1420y^7 + 872y^8 + 258y^9 + 30y^{10}}{x^3} \\
&+ \frac{3y^4 + 54y^5 + 305y^6 + 629y^7 + 600y^8 + 80y^{10} + 9y^{11}}{x^4} \\
&+ \frac{3y^5 + 26y^6 + 110y^8 + 81y^9 + 36y^{10} + 9y^{11} + y^{12}}{x^5}+ \frac{y^6 + 3y^7 + 3y^8 + y^9}{x^6};
\end{aligned}
\]
}
{\small\[
\begin{aligned}
x_8^{(6)} = &\ 8y^4 + 12y^5 + 6y^6 + y^7 + \frac{12y^4 + 52y^5 + 64y^6 + 30y^7 + 5y^8}{x} + \frac{y^3 + 7y^4 + 70y^5 + 143y^6 + 121y^7 + 47y^8 + 7y^9}{x^2}\\
& + \frac{6y^4 + 27y^5 + 72y^6 + 86y^7 + 48y^8 + 12y^9 + y^{10}}{x^3}+ \frac{12y^5 + 36y^6 + 39y^7 + 18y^8 + 3y^9}{x^4} + \frac{8y^6 + 12y^7 + 6y^8 + y^9}{x^5};\\
x_9^{(4)} = &\ 1 + 2x + x^2 + 13y + 13y^2 + 7xy + \frac{2y + 6y^2 + 3y^3}{x} + \frac{y^2}{x^2};\\
x_{10}^{(2)} = &\ 3y + 3y^2 + xy + \frac{y^2 + y^3}{x}.
\end{aligned}
\]
}
Clearly, two given flips are different, so in case we do a general triangulation on a surface, we need to clarify which type of flips from each quadrilateral.

\appendix\section{Recursion formula via good lattice}\label{good_lattice}
The following section demonstrates further observation of the recurrence relation discussed in Section \ref{subsec:solverecur}, partially solving~Proposition \ref{prop:g_func}.

Note that from Propositions~\ref{prop:stairs_combinatorial} and \ref{prop:reversed_stairs}, we can rewrite the following values in case either $a=1$ or $b=1$:
\begin{align*}
	f(1,b,i,j) &= \sum_{\substack{(l_1+1),l_i,u_i,(u_k+1),k \in \mathbb{N}: \\ \sum_{i=1}^{k}(l_i+u_i) = b}} \left(x_{i-\sum_{s=1}^{k}l_s,j+\sum_{t=1}^{k}u_t} \prod_{v=1}^{k} \frac{x_{i+1-\sum_{s=1}^{v-1}l_s,j+\sum_{t=1}^{v-1}u_t}}{x_{i+1-\sum_{s=1}^{v}l_s,j+\sum_{t=1}^{v-1}u_t}} \frac{x_{i-\sum_{s=1}^{v}l_s,j-1+\sum_{t=1}^{v-1}u_t}}{x_{i-\sum_{s=1}^{v}l_s,j-1+\sum_{t=1}^{v}u_t}}\right);\\
	f(a,1,i,j) &= \sum_{\substack{(r_1+1),r_i,d_i,(d_k+1),k \in \mathbb{N}: \\ \sum_{i=1}^{k}(r_i+d_i) = a}} \left(x_{i+\sum_{s=1}^{k}r_s,j-\sum_{t=1}^{k}d_t} \prod_{v=1}^{k} \frac{x_{i-1+\sum_{s=1}^{v-1}r_s,j-\sum_{t=1}^{v-1}d_t}}{x_{i-1+\sum_{s=1}^{v}r_s,j-\sum_{t=1}^{v-1}d_t}} \frac{x_{i+\sum_{s=1}^{v}r_s,j+1-\sum_{t=1}^{v-1}d_t}}{x_{i+\sum_{s=1}^{v}r_s,j+1-\sum_{t=1}^{v}d_t}}\right)
\end{align*}
for all $a,b \geq 1$. 

We shall introduce the notion of \emph{good lattice} and all its related features.
\begin{Def}
For any $(a,b) \in \mathbb{N}^2$, denote $$\mathcal{L}(a,b) := \{(p,q) \in \mathbb{Z}^2| -a \leq \min \{-p,q,-p+q \} \leq \max \{-p,q,-p+q \} \leq b \} \setminus \{(a,b),(-b,-a) \}\subset\Z^2$$ called the \emph{good lattice of type $(a,b)$}. Additionally for $a,b\geq 3$ we also define the degenerate cases:
$$\mathcal{L}(0,0)= \mathcal{L}(0,0) = \mathcal{L}(1,0) = \mathcal{L}(2,0) = \mathcal{L}(0,1) = \mathcal{L}(0,2) :=\{(0,0) \},$$
$$\mathcal{L}(0,b):= \{(-p,q) \in \mathbb{Z}^2|p,q \in \mathbb{N}, p+q \leq b-1 \} \cup \{(0,0) \}$$ and $$\mathcal{L}(a,0):= \{(p,-q) \in \mathbb{Z}^2|p,q \in \mathbb{N}, p+q \leq a-1  \} \cup \{(0,0) \}.$$  Furthermore, for any $(a,b) \in \mathbb{N}^2$, we shall call the convex hull of a good lattice of type $(a,b)$ its \emph{boundary} and denote it by $\partial \mathcal{L}(a,b)$. See Figure~\ref{fig:lattice_example_35} for examples.
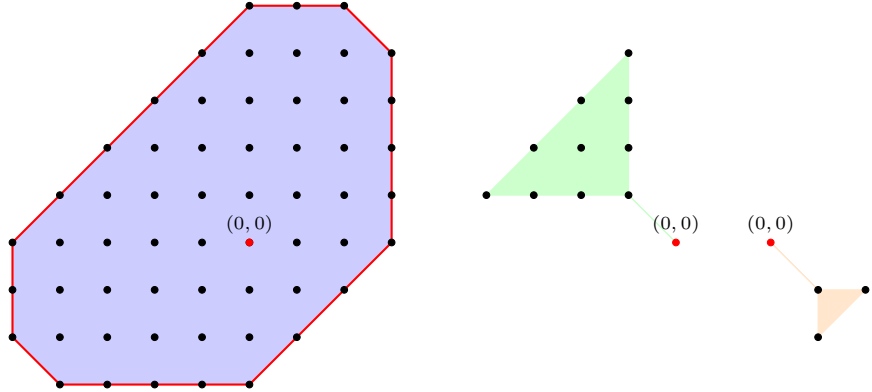
\begin{figure}[H]
\centering
\begin{tikzpicture}[scale=0.63,
    mid arrow/.style={
        postaction={decorate},
        decoration={
            markings,
            mark=at position 0.5 with {\arrow{>}}
        }
    }
]
	\filldraw[blue!20] (0,-3)--(3,0)--(3,4)--(2,5)--(0,5)--(-5,0)--(-5,-2)--(-4,-3)--cycle;
	\draw[thick,red] (0,-3)--(3,0)--(3,4)--(2,5)--(0,5)--(-5,0)--(-5,-2)--(-4,-3)--cycle;
	
	\foreach \i in {0,...,4} {
		\filldraw[black] (3,\i) circle (2pt);
		}
		\foreach \i in {0,...,6} {
		\filldraw[black] (2,\i-1) circle (2pt);
		}
	
		\foreach \i in {0,...,7} {
		\filldraw[black] (1,\i-2) circle (2pt);
		}
	
		\foreach \i in {0,...,8} {
		\filldraw[black] (0,\i-3) circle (2pt);
		}
	
		\foreach \i in {0,...,7} {
		\filldraw[black] (-1,\i-3) circle (2pt);
		}
	
		\foreach \i in {0,...,6} {
		\filldraw[black] (-2,\i-3) circle (2pt);
		}
	
		\foreach \i in {0,...,5} {
		\filldraw[black] (-3,\i-3) circle (2pt);
		}
	
		\foreach \i in {0,...,4} {
		\filldraw[black] (-4,\i-3) circle (2pt);
		}
	
		\foreach \i in {0,...,2} {
		\filldraw[black] (-5,\i-2) circle (2pt);
		}
	\filldraw[red] (0,0) circle (2pt);
	\node[above, font=\scriptsize\bfseries] at (0,0) {$(0,0)$};
	
	\filldraw[green!20] (9,0)--(8,1)--(5,1)--(8,4)--(8,1)--cycle;
	\filldraw[black] (5,1) circle (2pt);
	\filldraw[black] (6,1) circle (2pt);
	\filldraw[black] (6,2) circle (2pt);
	\filldraw[black] (7,1) circle (2pt);
	\filldraw[black] (7,2) circle (2pt);
	\filldraw[black] (7,3) circle (2pt);
	\filldraw[black] (8,1) circle (2pt);
	\filldraw[black] (8,2) circle (2pt);
	\filldraw[black] (8,3) circle (2pt);
	\filldraw[black] (8,4) circle (2pt);
	\filldraw[red] (9,0) circle (2pt);
	\node[above, font=\scriptsize\bfseries] at (9,0) {$(0,0)$};
	
	\filldraw[orange!20] (11,0)--(12,-1)--(13,-1)--(12,-2)--(12,-1)--cycle;
	\filldraw[black] (12,-1) circle (2pt);
	\filldraw[black] (12,-2) circle (2pt);
	\filldraw[black] (13,-1) circle (2pt);
	\filldraw[red] (11,0) circle (2pt);
	\node[above, font=\scriptsize\bfseries] at (11,0) {$(0,0)$};	
\end{tikzpicture}
\caption{All points in $\mathcal{L}(3,5)$ (left), $\mathcal{L}(0,6)$ (middle), $\mathcal{L}(4,0)$ (right) and the corresponding boundary of $\mathcal{L}(3,5)$ (red), using the grid labeling of Definition \ref{gridlabeling}.}
\label{fig:lattice_example_35}
\end{figure}
For each $(a,b) \in \mathbb{N}^2$, we further define subsets $\partial^{\Box} \mathcal{L}(a,b)$ with $\Box \in \{\leftarrow, \rightarrow, \downarrow, \uparrow \}$ of the boundary $\partial \mathcal{L}(a,b)$, where $\partial^{\leftarrow} \mathcal{L}(a,b)$ (resp. $\rightarrow, \downarrow, \uparrow$) is defined to be the set of all boundary points that are strictly to the left (resp. right, below, above) the point $(0,0)$, see Figure~\ref{fig:lattice_example_35_boundary} for examples.
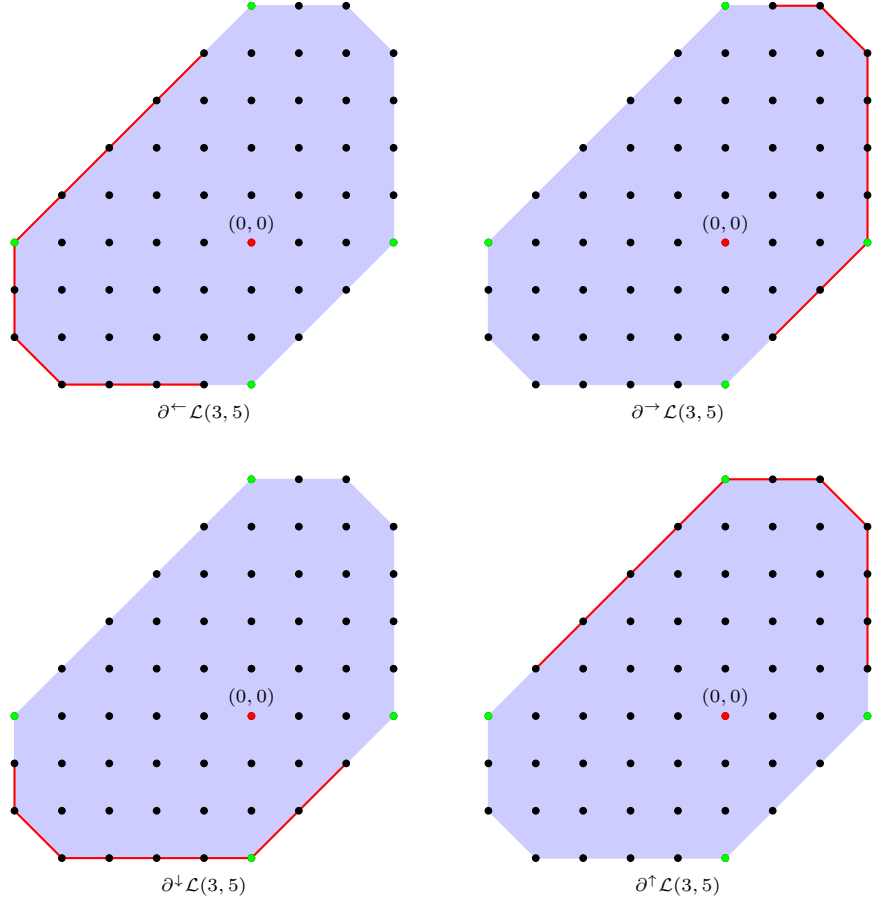
\begin{figure}[H]
\centering
\begin{tikzpicture}[scale=0.63,
    mid arrow/.style={
        postaction={decorate},
        decoration={
            markings,
            mark=at position 0.5 with {\arrow{>}}
        }
    }
]
	\filldraw[blue!20] (0,-3)--(3,0)--(3,4)--(2,5)--(0,5)--(-5,0)--(-5,-2)--(-4,-3)--cycle;
	\draw[thick,red] (-1,-3)--(-4,-3)--(-5,-2)--(-5,0)--(-1,4);
	
	\foreach \i in {0,...,4} {
		\filldraw[black] (3,\i) circle (2pt);
		}
		\foreach \i in {0,...,6} {
		\filldraw[black] (2,\i-1) circle (2pt);
		}
	
		\foreach \i in {0,...,7} {
		\filldraw[black] (1,\i-2) circle (2pt);
		}
	
		\foreach \i in {0,...,8} {
		\filldraw[black] (0,\i-3) circle (2pt);
		}
	
		\foreach \i in {0,...,7} {
		\filldraw[black] (-1,\i-3) circle (2pt);
		}
	
		\foreach \i in {0,...,6} {
		\filldraw[black] (-2,\i-3) circle (2pt);
		}
	
		\foreach \i in {0,...,5} {
		\filldraw[black] (-3,\i-3) circle (2pt);
		}
	
		\foreach \i in {0,...,4} {
		\filldraw[black] (-4,\i-3) circle (2pt);
		}
	
		\foreach \i in {0,...,2} {
		\filldraw[black] (-5,\i-2) circle (2pt);
		}
	\filldraw[red] (0,0) circle (2pt);
	\filldraw[green] (3,0) circle (2pt);
	\filldraw[green] (0,5) circle (2pt);
	\filldraw[green] (-5,0) circle (2pt);
	\filldraw[green] (0,-3) circle (2pt);
	\node[above, font=\scriptsize] at (0,0) {$(0,0)$};
	\node[above, font=\scriptsize] at (-1,-4) {$\partial^{\leftarrow} \mathcal{L}(3,5)$};
	
	\filldraw[blue!20] (0,-3-10)--(3,0-10)--(3,4-10)--(2,5-10)--(0,5-10)--(-5,0-10)--(-5,-2-10)--(-4,-3-10)--cycle;
	\draw[thick,red] (2,-1-10)--(0,-3-10)--(-4,-3-10)--(-5,-2-10)--(-5,-1-10);
	
	\foreach \i in {0,...,4} {
		\filldraw[black] (3,\i-10) circle (2pt);
		}
		\foreach \i in {0,...,6} {
		\filldraw[black] (2,\i-1-10) circle (2pt);
		}
	
		\foreach \i in {0,...,7} {
		\filldraw[black] (1,\i-2-10) circle (2pt);
		}
	
		\foreach \i in {0,...,8} {
		\filldraw[black] (0,\i-3-10) circle (2pt);
		}
	
		\foreach \i in {0,...,7} {
		\filldraw[black] (-1,\i-3-10) circle (2pt);
		}
	
		\foreach \i in {0,...,6} {
		\filldraw[black] (-2,\i-3-10) circle (2pt);
		}
	
		\foreach \i in {0,...,5} {
		\filldraw[black] (-3,\i-3-10) circle (2pt);
		}
	
		\foreach \i in {0,...,4} {
		\filldraw[black] (-4,\i-3-10) circle (2pt);
		}
	
		\foreach \i in {0,...,2} {
		\filldraw[black] (-5,\i-2-10) circle (2pt);
		}
	\filldraw[red] (0,0-10) circle (2pt);
	\filldraw[green] (3,0-10) circle (2pt);
	\filldraw[green] (0,5-10) circle (2pt);
	\filldraw[green] (-5,0-10) circle (2pt);
	\filldraw[green] (0,-3-10) circle (2pt);
	\node[above, font=\scriptsize] at (0,0-10) {$(0,0)$};
	\node[above, font=\scriptsize] at (-1,-4-10) {$\partial^{\downarrow} \mathcal{L}(3,5)$};	

	\filldraw[blue!20] (0+10,-3)--(3+10,0)--(3+10,4)--(2+10,5)--(0+10,5)--(-5+10,0)--(-5+10,-2)--(-4+10,-3)--cycle;
	\draw[thick,red] (1+10,5)--(2+10,5)--(3+10,4)--(3+10,0)--(1+10,-2);
	
	\foreach \i in {0,...,4} {
		\filldraw[black] (3+10,\i) circle (2pt);
		}
		\foreach \i in {0,...,6} {
		\filldraw[black] (2+10,\i-1) circle (2pt);
		}
	
		\foreach \i in {0,...,7} {
		\filldraw[black] (1+10,\i-2) circle (2pt);
		}
	
		\foreach \i in {0,...,8} {
		\filldraw[black] (0+10,\i-3) circle (2pt);
		}
	
		\foreach \i in {0,...,7} {
		\filldraw[black] (-1+10,\i-3) circle (2pt);
		}
	
		\foreach \i in {0,...,6} {
		\filldraw[black] (-2+10,\i-3) circle (2pt);
		}
	
		\foreach \i in {0,...,5} {
		\filldraw[black] (-3+10,\i-3) circle (2pt);
		}
	
		\foreach \i in {0,...,4} {
		\filldraw[black] (-4+10,\i-3) circle (2pt);
		}
	
		\foreach \i in {0,...,2} {
		\filldraw[black] (-5+10,\i-2) circle (2pt);
		}
	\filldraw[red] (0+10,0) circle (2pt);
	\filldraw[green] (3+10,0) circle (2pt);
	\filldraw[green] (0+10,5) circle (2pt);
	\filldraw[green] (-5+10,0) circle (2pt);
	\filldraw[green] (0+10,-3) circle (2pt);
	\node[above, font=\scriptsize] at (0+10,0) {$(0,0)$};
	\node[above, font=\scriptsize] at (-1+10,-4) {$\partial^{\rightarrow} \mathcal{L}(3,5)$};
	
	\filldraw[blue!20] (0+10,-3-10)--(3+10,0-10)--(3+10,4-10)--(2+10,5-10)--(0+10,5-10)--(-5+10,0-10)--(-5+10,-2-10)--(-4+10,-3-10)--cycle;
	\draw[thick,red] (-4+10,1-10)--(0+10,5-10)--(2+10,5-10)--(3+10,4-10)--(3+10,1-10);
	
	\foreach \i in {0,...,4} {
		\filldraw[black] (3+10,\i-10) circle (2pt);
		}
		\foreach \i in {0,...,6} {
		\filldraw[black] (2+10,\i-1-10) circle (2pt);
		}
	
		\foreach \i in {0,...,7} {
		\filldraw[black] (1+10,\i-2-10) circle (2pt);
		}
	
		\foreach \i in {0,...,8} {
		\filldraw[black] (0+10,\i-3-10) circle (2pt);
		}
	
		\foreach \i in {0,...,7} {
		\filldraw[black] (-1+10,\i-3-10) circle (2pt);
		}
	
		\foreach \i in {0,...,6} {
		\filldraw[black] (-2+10,\i-3-10) circle (2pt);
		}
	
		\foreach \i in {0,...,5} {
		\filldraw[black] (-3+10,\i-3-10) circle (2pt);
		}
	
		\foreach \i in {0,...,4} {
		\filldraw[black] (-4+10,\i-3-10) circle (2pt);
		}
	
		\foreach \i in {0,...,2} {
		\filldraw[black] (-5+10,\i-2-10) circle (2pt);
		}
	\filldraw[red] (0+10,0-10) circle (2pt);
	\filldraw[green] (3+10,0-10) circle (2pt);
	\filldraw[green] (0+10,5-10) circle (2pt);
	\filldraw[green] (-5+10,0-10) circle (2pt);
	\filldraw[green] (0+10,-3-10) circle (2pt);
	\node[above, font=\scriptsize] at (0+10,0-10) {$(0,0)$};
	\node[above, font=\scriptsize] at (-1+10,-4-10) {$\partial^{\uparrow} \mathcal{L}(3,5)$};	
\end{tikzpicture}
\caption{All 4 sets $\partial^{\Box} \mathcal{L}(3,5)$ (red) for $\Box \in \{\leftarrow, \rightarrow, \downarrow, \uparrow \}$}
\label{fig:lattice_example_35_boundary}
\end{figure}
\end{Def}
We can directly compute the following values for any $a,b \geq 1$:
\begin{align*}
	|\mathcal{L}(a,b)| &= ab+\frac{(a+b+1)(a+b+2)}{2}-2; \quad |\partial \mathcal{L}(a,b)| = 3(a+b)-2;\\
	|\mathcal{L}(0,b)| &= 1+\frac{(b-1)(b-2)}{2}; \quad |\mathcal{L}(a,0)|=1+\frac{(a-1)(a-2)}{2};\\
	|\partial^{\leftarrow} \mathcal{L}(a,b)| &= a+2b-2; \quad |\partial^{\rightarrow} \mathcal{L}(a,b)| = 2a+b-2; \quad |\partial^{\downarrow} \mathcal{L}(a,b)| = 2a+b-2; \quad |\partial^{\uparrow} \mathcal{L}(a,b)| = a+2b-2.
\end{align*}

For $m\in\N$, we consider the standard lexicographical ordering on $\Z^m$: for any $(a_1,a_2,...,a_m), (b_1,b_2,...,b_m)$ $\in \mathbb{Z}^m$, we have $(a_1,a_2,...,a_m) > (b_1,b_2,...,b_m)$ if and only if there exists $1\leq k\leq m$ such that $a_i = b_i$ for $i = 1,2,...,k-1$ and $a_k > b_k$. Now we can write each value $g(a,b,i,j)$ (for $a,b \in \mathbb{Z}_{\geq 0}$ and $i,j \in \mathbb{Z}$) of form:
\begin{gather*}
	g(a,b,i,j) = \sum_{N=1}^{2^{ab}} \prod_{(k,l) \in \mathcal{L}(a,b)} x_{i+k,j+l}^{c_{a,b}^{N,k,l}}
\end{gather*}
where we have a total of $2^{ab}$ \emph{exponential vectors} $\textbf{c}^{N,a,b} = (c_{a,b}^{N,a,b-1}, c_{a,b}^{N,a,b-2},...,c_{a,b}^{N,-b,-a+1})$ that are arranged in the lexicographical order of $\mathbb{Z}^{ab+\frac{(a+b+1)(a+b+2)}{2}-2}$, with all terms $c_{a,b}^{N,k,l}$ arranged in the lexicographical order of $\mathbb{Z}^2$ for pairs $(k,l)$ in the good lattice of $(a,b)$ for each $N = 1,2,...,2^{ab}$. We also define $$\textbf{x}_{i,j}^{N,a,b} := \prod_{(k,l) \in \mathcal{L}(a,b)} x_{i+k,j+l}^{c_{a,b}^{N,k,l}}$$ for all $(a,b,i,j) \in \mathbb{Z}_{\geq 0}^2 \times \mathbb{Z}^2$ and $N = 1,2,...,2^{ab}$. We have the following particular values:
\begin{align*}
	\textbf{c}^{1,0,b} &= (1,1,...,1); \quad \textbf{c}^{1,a,0} = (1,1,...,1);\\
	\textbf{c}^{1,1,1} &= (1,0,0,0,1); \quad \textbf{c}^{2,1,1} = (0,1,0,1,0);\\
	\textbf{c}^{1,1,2} &= (1, 0, 0, 0, 0, 1, 1, 0, 1, 0); \quad \textbf{c}^{2,1,2} = (0, 1, 0, 1, 0, 0, 1, 0, 1, 0);\\
	\textbf{c}^{3,1,2} &= (0, 0, 1, 1, 1, 0, 0, 0, 0, 1); \quad \textbf{c}^{4,1,2} = (0, 0, 1, 0, 1, 1, 0, 1, 0, 0);\\
	\textbf{c}^{1,2,1} &= (1, 0, 0, 0, 0, 1, 1, 1, 0, 0); \quad \textbf{c}^{2,2,1} = (0, 1, 0, 1, 1, 0, 0, 0, 0, 1);\\
	\textbf{c}^{3,2,1} &= (0, 1, 0, 1, 0, 0, 1, 0, 1, 0); \quad \textbf{c}^{4,2,1} = (0, 0, 1, 0, 1, 1, 0, 1, 0, 0);
\end{align*}
For all $a,b \geq 1$, we can compute:
\begin{align*}
    g(a+1,b+1,i,j) &= \sum_{P=1}^{2^{(a+1)(b+1)}} \prod_{(k,l) \in \mathcal{L}(a+1,b+1)} x_{i+k,j+l}^{c_{a+1,b+1}^{P,k,l}} \\
    &= \sum_{P=1}^{2^{(a+1)(b+1)}} x_{i+a,j+b}^{c_{a+1,b+1}^{P,a,b}} x_{i-b,j-a}^{c_{a+1,b+1}^{P,-b,-a}} 
    \cdot \left( \prod_{(k,l) \in \mathcal{L}(a,b)} x_{i+k,j+l}^{c_{a+1,b+1}^{P,k,l}} \right) 
    \cdot \left( \prod_{(k,l) \in \partial \mathcal{L}(a+1,b+1)} x_{i+k,j+l}^{c_{a+1,b+1}^{P,k,l}} \right);
\end{align*}

\begin{align*}
    &g(a+1,b,i-1,j)g(a,b+1,i+1,j) = \left(\sum_{M=1}^{2^{(a+1)b}} \prod_{(k,l) \in \mathcal{L}(a+1,b)} x_{i-1+k,j+l}^{c_{a+1,b}^{M,k,l}}\right) \cdot \left(\sum_{N=1}^{2^{a(b+1)}} \prod_{(k,l) \in \mathcal{L}(a,b+1)} x_{i+1+k,j+l}^{c_{a,b+1}^{N,k,l}}\right) \\
    &= \left(\sum_{M=1}^{2^{(a+1)b}} x_{i-b,j-a}^{c_{a+1,b}^{M,-b,-a}} \cdot \left(\prod_{(k,l) \in \mathcal{L}(a,b)} x_{i+k,j+l}^{c_{a+1,b}^{M,k,l}}\right) \cdot \left(\prod_{(k,l) \in \partial^{\leftarrow} \mathcal{L}(a+1,b+1)} x_{i+k,j+l}^{c_{a+1,b}^{M,k,l}}\right)\right) \\
    &\quad \cdot \left(\sum_{N=1}^{2^{a(b+1)}} x_{i+a,j+b}^{c_{a,b+1}^{N,a,b}} \cdot \left(\prod_{(k,l) \in \mathcal{L}(a,b)} x_{i+k,j+l}^{c_{a,b+1}^{N,k,l}}\right) \cdot \left(\prod_{(k,l) \in \partial^{\rightarrow} \mathcal{L}(a+1,b+1)} x_{i+k,j+l}^{c_{a,b+1}^{N,k,l}}\right)\right);
\end{align*}

\begin{align*}
    &g(a,b+1,i,j-1)g(a+1,b,i,j+1) = \left(\sum_{N=1}^{2^{a(b+1)}} \prod_{(k,l) \in \mathcal{L}(a,b+1)} x_{i+k,j-1+l}^{c_{a,b+1}^{N,k,l}}\right) \cdot \left(\sum_{M=1}^{2^{(a+1)b}} \prod_{(k,l) \in \mathcal{L}(a+1,b)} x_{i+k,j+1+l}^{c_{a+1,b}^{M,k,l}}\right) \\
    &= \left(\sum_{N=1}^{2^{a(b+1)}} x_{i-b,j-a}^{c_{a,b+1}^{N,-b,-a}} \cdot \left(\prod_{(k,l) \in \mathcal{L}(a,b)} x_{i+k,j+l}^{c_{a,b+1}^{N,k,l}}\right) \cdot \left(\prod_{(k,l) \in \partial^{\downarrow} \mathcal{L}(a+1,b+1)} x_{i+k,j+l}^{c_{a,b+1}^{N,k,l}}\right)\right) \\
    &\quad \cdot \left(\sum_{M=1}^{2^{(a+1)b}} x_{i+a,j+b}^{c_{a+1,b}^{M,a,b}} \cdot \left(\prod_{(k,l) \in \mathcal{L}(a,b)} x_{i+k,j+l}^{c_{a+1,b}^{M,k,l}}\right) \cdot \left(\prod_{(k,l) \in \partial^{\uparrow} \mathcal{L}(a+1,b+1)} x_{i+k,j+l}^{c_{a+1,b}^{M,k,l}}\right)\right).
\end{align*}
Then by the Proposition~\ref{prop:g_func}, we have the following computations:
\begin{align*}
	\text{LHS} &= g(a+1,b+1,i,j)g(a,b,i,j) \\
	&= \left(\sum_{Q=1}^{2^{ab}} \prod_{(k,l) \in \mathcal{L}(a,b)} x_{i+k,j+l}^{c_{a,b}^{Q,k,l}}\right) \\
	&\cdot \left(\sum_{P=1}^{2^{(a+1)(b+1)}} x_{i+a,j+b}^{c_{a+1,b+1}^{P,a,b}} x_{i-b,j-a}^{c_{a+1,b+1}^{P,-b,-a}} 
    \cdot \left( \prod_{(k,l) \in \mathcal{L}(a,b)} x_{i+k,j+l}^{c_{a+1,b+1}^{P,k,l}} \right) 
    \cdot \left( \prod_{(k,l) \in \partial \mathcal{L}(a+1,b+1)} x_{i+k,j+l}^{c_{a+1,b+1}^{P,k,l}} \right)\right) \\
    &= \sum_{Q=1}^{2^{ab}} \sum_{P=1}^{2^{(a+1)(b+1)}}x_{i+a,j+b}^{c_{a+1,b+1}^{P,a,b}} x_{i-b,j-a}^{c_{a+1,b+1}^{P,-b,-a}} 
    \cdot \left( \prod_{(k,l) \in \mathcal{L}(a,b)} x_{i+k,j+l}^{c_{a+1,b+1}^{P,k,l}+c_{a,b}^{Q,k,l}} \right) 
    \cdot \left( \prod_{(k,l) \in \partial \mathcal{L}(a+1,b+1)} x_{i+k,j+l}^{c_{a+1,b+1}^{P,k,l}} \right);
\end{align*}

\begin{align*}
   \text{RHS} &= x_{i,j-a}x_{i,j+b} \cdot g(a+1,b,i-1,j)g(a,b+1,i+1,j) + x_{i+a,j}x_{i-b,j} \cdot g(a,b+1,i,j-1)g(a+1,b,i,j+1) \\ 
   &= x_{i,j-a}x_{i,j+b} \cdot \Bigg(\sum_{M=1}^{2^{(a+1)b}} x_{i-b,j-a}^{c_{a+1,b}^{M,-b,-a}} \cdot \Bigg(\prod_{(k,l) \in \mathcal{L}(a,b)} x_{i+k,j+l}^{c_{a+1,b}^{M,k,l}}\Bigg) \cdot \Bigg(\prod_{(k,l) \in \partial^{\leftarrow} \mathcal{L}(a+1,b+1)} x_{i+k,j+l}^{c_{a+1,b}^{M,k,l}}\Bigg)\Bigg) \\
   &\quad \cdot \Bigg(\sum_{N=1}^{2^{a(b+1)}} x_{i+a,j+b}^{c_{a,b+1}^{N,a,b}} \cdot \Bigg(\prod_{(k,l) \in \mathcal{L}(a,b)} x_{i+k,j+l}^{c_{a,b+1}^{N,k,l}}\Bigg) \cdot \Bigg(\prod_{(k,l) \in \partial^{\rightarrow} \mathcal{L}(a+1,b+1)} x_{i+k,j+l}^{c_{a,b+1}^{N,k,l}}\Bigg)\Bigg) \\
   &\quad+ x_{i+a,j}x_{i-b,j} \cdot \Bigg(\sum_{N=1}^{2^{a(b+1)}} x_{i-b,j-a}^{c_{a,b+1}^{N,-b,-a}} \cdot \Bigg(\prod_{(k,l) \in \mathcal{L}(a,b)} x_{i+k,j+l}^{c_{a,b+1}^{N,k,l}}\Bigg) \cdot \Bigg(\prod_{(k,l) \in \partial^{\downarrow} \mathcal{L}(a+1,b+1)} x_{i+k,j+l}^{c_{a,b+1}^{N,k,l}}\Bigg)\Bigg) \\
   &\quad \cdot \Bigg(\sum_{M=1}^{2^{(a+1)b}} x_{i+a,j+b}^{c_{a+1,b}^{M,a,b}} \cdot \Bigg(\prod_{(k,l) \in \mathcal{L}(a,b)} x_{i+k,j+l}^{c_{a+1,b}^{M,k,l}}\Bigg) \cdot \Bigg(\prod_{(k,l) \in \partial^{\uparrow} \mathcal{L}(a+1,b+1)} x_{i+k,j+l}^{c_{a+1,b}^{M,k,l}}\Bigg)\Bigg) \\
   &= x_{i,j-a}x_{i,j+b} \cdot \Bigg(\sum_{M=1}^{2^{(a+1)b}} \sum_{N=1}^{2^{a(b+1)}} x_{i+a,j+b}^{c_{a,b+1}^{N,a,b}} x_{i-b,j-a}^{c_{a+1,b}^{M,-b,-a}}  \cdot \Bigg(\prod_{(k,l) \in \mathcal{L}(a,b)} x_{i+k,j+l}^{c_{a+1,b}^{M,k,l}+c_{a,b+1}^{N,k,l}}\Bigg) \\
   &\cdot \Bigg(\prod_{(k,l) \in \partial^{\leftarrow} \mathcal{L}(a+1,b+1)} x_{i+k,j+l}^{c_{a+1,b}^{M,k,l}}\Bigg) \cdot \Bigg(\prod_{(k,l) \in \partial^{\rightarrow} \mathcal{L}(a+1,b+1)} x_{i+k,j+l}^{c_{a,b+1}^{N,k,l}}\Bigg) \Bigg) \\
   &+ x_{i+a,j}x_{i-b,j} \cdot \Bigg(\sum_{M=1}^{2^{(a+1)b}} \sum_{N=1}^{2^{a(b+1)}} x_{i+a,j+b}^{c_{a+1,b}^{M,a,b}} x_{i-b,j-a}^{c_{a,b+1}^{N,-b,-a}}  \cdot \Bigg(\prod_{(k,l) \in \mathcal{L}(a,b)} x_{i+k,j+l}^{c_{a+1,b}^{M,k,l}+c_{a,b+1}^{N,k,l}}\Bigg) \\
   &\cdot \Bigg(\prod_{(k,l) \in \partial^{\downarrow} \mathcal{L}(a+1,b+1)} x_{i+k,j+l}^{c_{a,b+1}^{N,k,l}}\Bigg) \cdot \Bigg(\prod_{(k,l) \in \partial^{\uparrow} \mathcal{L}(a+1,b+1)} x_{i+k,j+l}^{c_{a+1,b}^{M,k,l}}\Bigg) \Bigg).
\end{align*}

\section{Maple code for computation of recurrence relation}\label{maple}
The following Maple code computes the recurrence relation discussed in Section \ref{subsec:solverecur}.
\begin{verbatim}
restart;

# Function f: discrete 4D Hirota-Miwa equation
f := proc(a, b, i, j) 
    local res;
    option remember;
    if a = 0 or b = 0 then
        res := x[i, j];
    elif a = 1 and b = 1 then
        res := (x[i - 1, j]*x[i + 1, j] + x[i, j - 1]*x[i, j + 1])/x[i, j];
    else
        res := expand(factor(
            (f(a, b - 1, i - 1, j)*f(a - 1, b, i + 1, j) 
            + f(a - 1, b, i, j - 1)*f(a, b - 1, i, j + 1))
            / f(a - 1, b - 1, i, j)
        ));
    end if;
    return res;
end proc:

# Function X_{i,j}^{a,b}: product over the specified range
X := proc(a, b, i, j)
    local p, q, product_result, min_val, max_val;
    option remember;
    
    # Handle the (0,0) case
    if a = 0 and b = 0 then
        return x[i, j];
    end if;
    
    product_result := 1;
    
    # For (a,b) \neq (0,0), compute the product over integers p,q 
    # with 1-a \leq min{p,q,p+q} \leq max{p,q,p+q} \leq b-1
    for p from 1-a to b-1 do
        for q from 1-a to b-1 do
            min_val := min(p, q, p+q);
            max_val := max(p, q, p+q);
            
            if min_val >= 1-a and max_val <= b-1 then
                product_result := product_result * x[i - p, j + q];
            end if;
        end do;
    end do;
    
    return product_result;
end proc:

# Function g(a,b,i,j) = X_{i,j}^{a,b} * f(a,b,i,j)
g := proc(a, b, i, j)
    local x_product, f_value;
    option remember;
    
    x_product := X(a, b, i, j);
    f_value := f(a, b, i, j);
    
    return expand(x_product * f_value);
end proc:

# Function to verify the recurrence relation for g
verify_g_recurrence := proc(a, b, i, j)
    local LHS, RHS, term1, term2;
    
    # Left-hand side of the recurrence
    LHS := g(a + 1, b + 1, i, j) * g(a, b, i, j);
    
    # Right-hand side: two terms
    term1 := x[i, j - a] * x[i, j + b] * g(a + 1, b, i - 1, j) * g(a, b + 1, i + 1, j);
    term2 := x[i + a, j] * x[i - b, j] * g(a, b + 1, i, j - 1) * g(a + 1, b, i, j + 1);
    
    RHS := expand(term1 + term2);
    
    # Check if LHS equals RHS
    if simplify(LHS - RHS) = 0 then
        return true;
    else
        return false, LHS, RHS;
    end if;
end proc:
\end{verbatim}
From the code above, using Proposition~\ref{prop:g_func} yields the validity of the recurrence relation for $g$, i.e. we always get ``\verb|true|'' as output from the call \verb|verify_g_recurrence|.
\begin{Thm}
	For every $(a, b, i, j) \in \mathbb{Z}_{\geq 0}^2 \times \mathbb{Z}^2$, the function \verb|verify_g_recurrence| returns \verb|true|.
\end{Thm}

\end{document}